\documentclass[12pt,
	]{amsart}
\usepackage[margin=1in, marginparwidth=60pt,
	]{geometry}
\usepackage{graphicx}%\includegraphics{}
\usepackage[mathscr]{euscript}%variant of \mathcal{}
\usepackage{subcaption}%\begin{subfigure}
\usepackage{amssymb}%\geqslant variant of \geq
\usepackage{enumerate} %changing the enumerate counters
\usepackage{float}%\begin{figure}[H]
\usepackage{setspace}
\usepackage{pdfpages} % \includepdf
\usepackage{hyperref}%to prevent bugs, should be last package
\newtheorem{theorem}{Theorem}
\newtheorem{conjecture}[theorem]{Conjecture}
\newtheorem{proposition}[theorem]{Proposition}
\newtheorem{lemma}[theorem]{Lemma}
\newtheorem{compl}[theorem]{Complement}
\newtheorem{fact}[theorem]{Fact}
\theoremstyle{definition}
\newtheorem{definition}[theorem]{Definition}
\newtheorem{remark}[theorem]{Remark}
\renewcommand{\vec}[1]{\mathbf{#1}}
\renewcommand\subset{\subseteq}
\renewcommand\emptyset{\varnothing}
\renewcommand{\geq}{\geqslant}
\renewcommand{\leq}{\leqslant}
\renewcommand{\bar}{\overline}
\makeatletter
\newcommand{\addresseshere}{%
  \enddoc@text\let\enddoc@text\relax
}
\makeatother
\title[Quantum traces]{Quantum traces for $\mathrm{SL}_n(\mathbb{C})$: The case $n=3$}
\author[D. C. Douglas]{Daniel C. Douglas}
\address{Department of Mathematics, Virginia Tech, 225 Stanger Street, Blacksburg, VA 24061}
\email{dcdouglas@vt.edu}
\date{\today}
\thanks{This work was partially supported by the U.S. National Science Foundation grants DMS-1406559 and DMS-1711297}
\begin{document}
%\layout%toggle for formatting  
\begin{abstract}
We generalize Bonahon--Wong's $\mathrm{SL}_2(\mathbb{C})$-quantum trace map to the setting of $\mathrm{SL}_3(\mathbb{C})$.  More precisely, given a non-zero complex parameter $q=e^{2 \pi i \hbar}$, we associate to each isotopy class of framed oriented links $K$ in a thickened punctured surface $\mathfrak{S} \times (0, 1)$ a Laurent polynomial $\mathrm{Tr}_\lambda^q(K) = \mathrm{Tr}_\lambda^q(K)(X_i^q)$ in $q$-deformations $X_i^q$ of the Fock--Goncharov $\mathcal{X}$-coordinates for higher Teichm\"{u}ller space.  This construction depends on a choice $\lambda$ of ideal triangulation of the surface $\mathfrak{S}$.  Along the way, we propose a definition for a $\mathrm{SL}_n(\mathbb{C})$-version of this invariant.  
\end{abstract}
\maketitle

\section{Introduction}
\label{sec:introduction}

For a finitely generated group $\Gamma$ and a suitable Lie group $G$, a primary object of study in low-dimensional geometry and topology is the $G$-character variety
\begin{equation*}
	\mathscr{R}_G(\Gamma) = \left\{  \rho : \Gamma \to G  \right\} /\!\!/  
	G,
\end{equation*}
consisting of group homomorphisms $\rho$ from $\Gamma$ to $G$, considered up to conjugation.  Here, the quotient is taken in the algebraic geometric sense of geometric invariant theory \cite{Mumford94}.  Character varieties can be explored using a wide variety of mathematical skill sets.  Some examples include the Higgs bundle approach of Hitchin \cite{HitchinTopology92}, the dynamics approach of Labourie \cite{LabourieInvent06}, and the representation theory approach of Fock--Goncharov \cite{FockIHES06}.

In the case where the group $\Gamma = \pi_1(\mathfrak{S})$ is the fundamental group of a punctured surface $\mathfrak{S}$ of finite topological type with negative Euler characteristic, and where the Lie group $G = \mathrm{SL}_n(\mathbb{C})$ is the special linear group, we are interested in studying a relationship between two competing deformation quantizations of the character variety $\mathscr{R}_{\mathrm{SL}_n(\mathbb{C})}(\pi_1(\mathfrak{S}))$, which we denote simply by $\mathscr{R}_{\mathrm{SL}_n(\mathbb{C})}(\mathfrak{S})$.  Here, a deformation quantization $\{ \mathscr{R}^q \}_q$ of a Poisson space $\mathscr{R}$ is a family of non-commutative algebras $\mathscr{R}^q$ parametrized by a non-zero complex parameter $q = e^{2 \pi i \hbar}$, such that the lack of commutativity in $\mathscr{R}^q$ is infinitesimally measured in the semi-classical limit $\hbar \to 0$ by the Poisson bracket of the space $\mathscr{R}$.  In the case where  $\mathscr{R} = \mathscr{R}_{\mathrm{SL}_n(\mathbb{C})}(\mathfrak{S})$ is the character variety, the bracket is provided by the Goldman Poisson structure on  $\mathscr{R}_{\mathrm{SL}_n(\mathbb{C})}(\mathfrak{S})$ \cite{GoldmanAdvMath84, GoldmanInvent86, BiswasInternatjmath93}.  

The first quantization of the character variety is the $\mathrm{SL}_n(\mathbb{C})$-skein algebra $\mathscr{S}^q_n(\mathfrak{S})$ of the surface $\mathfrak{S}$; see \cite{Turaev89, WittenCommMathPhys89, PrzytyckiBullPolishAcad91, BullockKnotTheory99, KuperbergCommMathPhys96, SikoraAlgGeomTop05, CautisMathAnn14}.  The skein algebra is motivated by the classical algebraic geometric approach to studying the character variety $\mathscr{R}_{\mathrm{SL}_n(\mathbb{C})}(\mathfrak{S})$ by means of its commutative algebra of regular functions $\mathbb{C}[\mathscr{R}_{\mathrm{SL}_n(\mathbb{C})}(\mathfrak{S})]$.  An example of a regular function is the trace function $\mathrm{Tr}_\gamma : \mathscr{R}_{\mathrm{SL}_n(\mathbb{C})}(\mathfrak{S}) \to \mathbb{C}$ associated to a closed curve $\gamma \in \pi_1(\mathfrak{S})$ sending a representation $\rho : \pi_1(\mathfrak{S}) \to \mathrm{SL}_n(\mathbb{C})$ to the trace $\mathrm{Tr}(\rho(\gamma)) \in \mathbb{C}$ of the matrix $\rho(\gamma) \in \mathrm{SL}_n(\mathbb{C})$.  A theorem of classical invariant theory, due to Procesi \cite{ProcesiAdvMath76}, implies that the trace functions $\mathrm{Tr}_\gamma$ generate the algebra of functions $\mathbb{C}[\mathscr{R}_{\mathrm{SL}_n(\mathbb{C})}(\mathfrak{S})]$ as an algebra.  According to the philosophy of Turaev and Witten, quantizations of the character variety should be of a 3-dimensional nature.  Indeed, elements of the skein algebra $\mathscr{S}^q_n(\mathfrak{S})$ are represented by (formal linear combinations of) knots (or links) $K$ in the thickened surface $\mathfrak{S} \times (0, 1)$.  The skein algebra $\mathscr{S}^q_n(\mathfrak{S})$ has the advantage of being natural, but is difficult to work with in practice.  

The second quantization of the $\mathrm{SL}_n(\mathbb{C})$-character variety is the Fock--Goncharov quantum space $\mathscr{T}_n^q(\mathfrak{S})$; see \cite{Fock99TheorMathPhys, Kashaev98LettMathPhys, FockENS09}.  At the classical level, Fock--Goncharov \cite{FockIHES06}     introduced a framed version $\mathscr{R}_{\mathrm{PSL}_n(\mathbb{C})}(\mathfrak{S})_{\mathrm{fr}}$ (also called the $\mathcal{X}$-moduli space) of the $\mathrm{PSL}_n(\mathbb{C})$-character variety, which, roughly speaking, consists of representations $\rho : \pi_1(\mathfrak{S}) \to \mathrm{PSL}_n(\mathbb{C})$ equipped  with additional linear algebraic data  attached to the punctures of $\mathfrak{S}$.  Associated to each ideal triangulation $\lambda$ of the punctured surface $\mathfrak{S}$ is a $\lambda$-coordinate chart $U_\lambda$ for $\mathscr{R}_{\mathrm{PSL}_n(\mathbb{C})}(\mathfrak{S})_{\mathrm{fr}}$ parametrized by $N$ non-zero complex coordinates $X_1, X_2, \dots, X_N$ where the integer $N$ depends only on the topology of the surface $\mathfrak{S}$ and the rank of the Lie group $\mathrm{SL}_n(\mathbb{C})$.  When written in terms of these coordinates $X_i$ the trace functions $\mathrm{Tr}_\gamma$ on the character variety take the form of Laurent polynomials $\widetilde{\mathrm{Tr}}_\gamma(X_i^{1/n})$ in $n$-roots of the $X_i$ (a subtlety being that $\widetilde{\mathrm{Tr}}_\gamma(X_i^{1/n})$ depends on the regular homotopy class of $\gamma$, represented by immersed curves, rather than the homotopy class of $\gamma$).  At the quantum level, there are $q$-deformed versions $X_i^q$ of these coordinates, which no longer commute but $q$-commute with each other, according to the underlying Poisson structure.  The quantized character variety $\mathscr{T}_n^q(\mathfrak{S})$ is obtained by gluing together quantum tori $\mathscr{T}_n^q(\sigma)$, including one for each triangulation $\sigma=\lambda$ consisting of Laurent polynomials in the quantized Fock--Goncharov coordinates  $X_i^q$.  The quantum character variety $\mathscr{T}_n^q(\mathfrak{S})$ has the advantage of being easier to work with than the skein algebra $\mathscr{S}^q_n(\mathfrak{S})$, but is less intrinsic.  

We seek $q$-deformed versions $\mathrm{Tr}_\gamma^q$ of the trace functions, associating to a closed curve $\gamma$ a Laurent polynomial in the quantized Fock--Goncharov coordinates with respect to a fixed triangulation $\lambda$. Turaev and Witten's philosophy leads us from the 2-dimensional setting of curves $\gamma$ on the surface $\mathfrak{S}$ to the $3$-dimensional setting of knots $K$ in the thickened surface $\mathfrak{S} \times (0, 1)$.  More precisely:
\begin{conjecture}[$\mathrm{SL}_n(\mathbb{C})$-quantum trace map] 
\label{conj:quantum-trace}
Fix a $(2*n^2)$-root $\omega^{1/2} = q^{1/(2*n^2)} \in \mathbb{C} - \{ 0 \}$. For each ideal triangulation $\lambda$ of the punctured  surface $\mathfrak{S}$ (with empty boundary, $\partial \mathfrak{S}=\emptyset$), there exists an injective algebra homomorphism
\begin{equation*}
	\mathrm{Tr}^\omega_\lambda : \mathscr{S}^q_n(\mathfrak{S}) \hookrightarrow \mathscr{T}_n^\omega(\lambda),
\end{equation*}
such that if $\omega^{1/2}=1$, then for every blackboard-framed oriented knot $K$ in the thickened surface $\mathfrak{S} \times (0, 1)$ projecting to an immersed closed curve $\gamma$ on the surface $\mathfrak{S}$,
\begin{equation*}
	\mathrm{Tr}_\lambda^1(K) 
	=  \widetilde{\mathrm{Tr}}_\gamma(X_i^{1/n}).  
\end{equation*} 
\end{conjecture}
This last equation says that the Fock--Goncharov classical trace polynomial associated to the curve $\gamma$ is recovered in the classical limit.  Moreover, the $\mathrm{SL}_n(\mathbb{C})$-quantum trace map should be natural, appropriately interpreted, with respect to the choice of triangulation $\lambda$; see \cite{KimArxiv21}.  

Conjecture \ref{conj:quantum-trace} is due to Chekhov--Fock \cite{FockArxiv97, ChekhovCzechJPhys00} in the case $n=2$, and was proved `by hand' in that case by Bonahon--Wong \cite{BonahonGT11}.  One of the motivations of the present work is to develop a  conceptual understanding of their construction.  

We prove the following, slightly weaker, version of Conjecture \ref{conj:quantum-trace} in the case $n=3$:
\begin{theorem}[Theorem \ref{thm:second-theorem}, $\mathrm{SL}_3(\mathbb{C})$-quantum trace polynomials]
\label{thm:intro-theorem-1}
	Fix a $(2*3^2)$-root $\omega^{1/2} = q^{1/(2*3^2)}=q^{1/18} \in \mathbb{C} - \{ 0 \}$.  For each ideal triangulation $\lambda$ of the punctured  surface $\mathfrak{S}$  (with possibly non-empty boundary), there exists a function 
\begin{equation*}
	\mathrm{Tr}_\lambda^\omega :
	\left\{  \textnormal{isotopy classes of (stated) framed oriented links } K \textnormal{ in } \mathfrak{S} \times (0, 1)  
	\right\}
	\to 
	\mathscr{T}_3^\omega(\lambda)
\end{equation*}
such that if $\omega^{1/2}=1$, then for every blackboard-framed oriented link $K$ whose components $K_1, K_2, \dots, K_\ell$ project to immersed closed oriented curves $\gamma_1, \gamma_2, \dots, \gamma_\ell$ in $\mathfrak{S}$,
\begin{equation*}
	\mathrm{Tr}_\lambda^1(K) 
	= \prod_{j=1}^\ell \widetilde{\mathrm{Tr}}_{\gamma_j}(X_i^{1/3}).
\end{equation*}
Moreover, this invariant satisfies the $q$-evaluated \textnormal{HOMFLYPT} relation \cite{FreydBullAmerMathSoc85, Przytycki87ProcAmercMathSoc}, as well as the unknot and framing skein relations, for $n=3$; see Figures {\upshape\ref{fig:HOMFLYPT-relation}}, {\upshape\ref{fig:unknot-relation}}, {\upshape\ref{fig:framing-skein-relations}}.  In the figures, $[n]_q = (q^n - q^{-n})/(q-q^{-1})$ is the quantum integer, and $\bar{\zeta}_n = (-1)^{n-1} q^{(1-n^2)/n}$ is (essentially) the (co)ribbon element of the quantum special linear group $\mathrm{SL}_n^q$ (see Appendix {\upshape\ref{sec:proof-of-reshetikhin-turaev-invariant}}).  \qed
\end{theorem}

\begin{figure}[H]
	\centering
	\includegraphics[width=.7\textwidth]{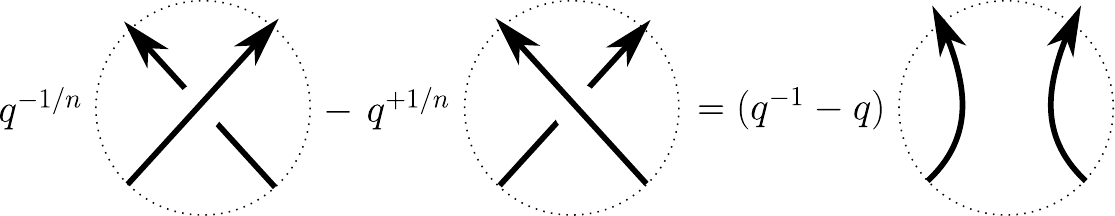}
	\caption{HOMFLYPT skein relation.}
	\label{fig:HOMFLYPT-relation}
\end{figure}

\begin{figure}[H]
	\centering
	\includegraphics[width=.39\textwidth]{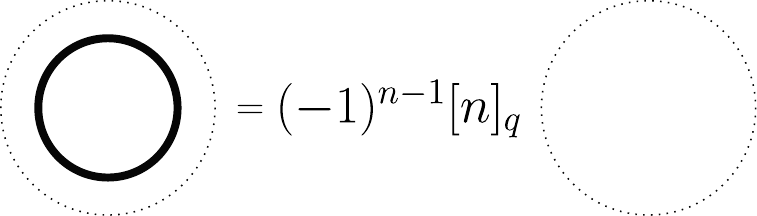}
	\caption{Unknot skein relation.}
	\label{fig:unknot-relation}
\end{figure}

\begin{figure}[H]
     \centering
\begin{subfigure}{0.49\textwidth}
	\centering
	\includegraphics[width=\textwidth]{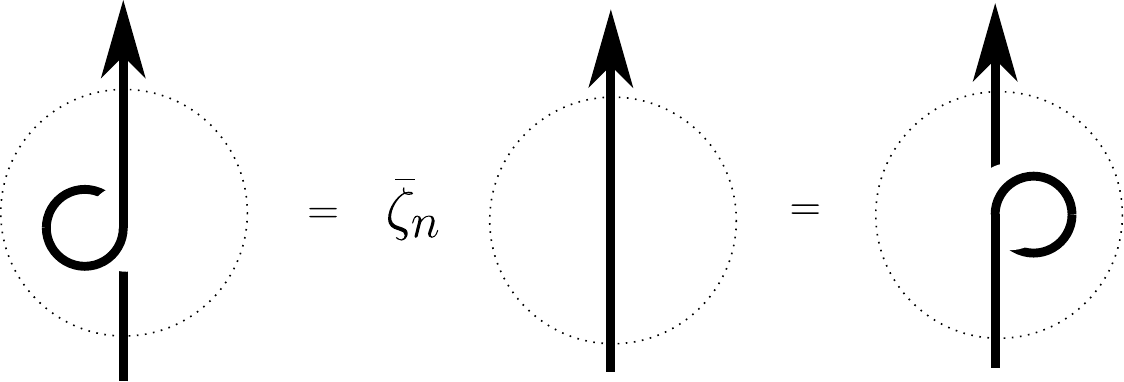}
	\caption{Positive kinks}
	\label{fig:positive-kink-skein-relation}
\end{subfigure} 
\hfill
\begin{subfigure}{0.49\textwidth}
	\centering
	\includegraphics[width=\textwidth]{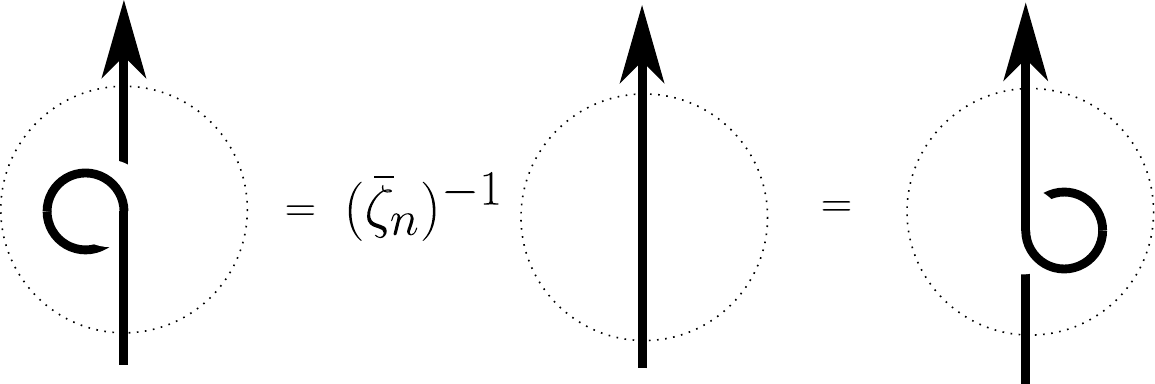}
	\caption{Negative kinks}
	\label{fig:negative-kink-skein-relation}
\end{subfigure}
        \caption{Framing skein relations.}
        \label{fig:framing-skein-relations}
\end{figure}  

In particular, the isotopy invariance property of Theorem \ref{thm:intro-theorem-1} can be thought of as the main step toward proving Conjecture \ref{conj:quantum-trace} in the case $n=3$.  

Theorem \ref{thm:intro-theorem-1} was originally proved as part of \cite{DouglasThesis20}.  Our proof is also `by hand', generalizing the strategy of \cite{BonahonGT11}, and relies on computer assistance for some of the  calculations; see Appendix \ref{sec:the-appendix}.  

Along the way, we propose a definition for a $\mathrm{SL}_n(\mathbb{C})$-version of the quantum trace polynomials; see \S \ref{sec:def-of-quantum-trace}.  That this construction is well-defined for general $n$ is expected to follow from recent related work \cite{chekhovMR4597214}, which shares some overlap with \cite{DouglasA}; see Remark \ref{rem:chekhov-shapiro}.  

The solution of \cite{BonahonGT11} in the $n=2$ case is implicitly related to the theory of the quantum group $\mathrm{U}_q(\mathfrak{sl}_2)$ or, more precisely, of its Hopf dual $\mathrm{SL}_2^q$; see for instance \cite{Kassel95}.  For general  $n$, we make this relationship more explicit; see \S \ref{sec:quantum-SLn} and Appendix \ref{sec:proof-of-reshetikhin-turaev-invariant}.  

In addition to the HOMFLYPT relation, the skein algebra $\mathscr{S}^q_n(\mathfrak{S})$ has other relations, best expressed as identities among certain $n$-valent graphs $W$ in $\mathfrak{S} \times (0, 1)$ called webs \cite{KuperbergCommMathPhys96, SikoraTrans01, CautisMathAnn14}.  It is therefore desirable to extend the definition of the quantum trace polynomials from links $K$ to webs $W$.  Building on our construction for links, Kim \cite{KimArxiv20, KimArxiv21} defined a $\mathrm{SL}_3(\mathbb{C})$-quantum trace map for webs, which is natural with respect to the choice of ideal triangulation $\lambda$.  Combined with Kim's work, the results of \cite{DS24, DouglasArxiv20b} lead to a proof of the injectivity property of Conjecture \ref{conj:quantum-trace} in the case $n=3$, which is closely related to the study of linear bases of skein algebras; see \cite[\S 9.3]{DS24}.  

As another application, Kim \cite{KimArxiv20} constructed a $\mathrm{SL}_3(\mathbb{C})$-quantum Fock--Goncharov duality map \cite{FockENS09} (of the bangle, rather than bracelet, form \cite{ThurstonPNAS14}), generalizing much of the $n=2$ solution of \cite{AllegrettiAdvMath17}; see also \cite{AllegrettiSelectaMath19, ChoCommMathPhys20}.  For other related studies in the $\mathrm{SL}_3(\mathbb{C})$-setting, see \cite{higginsMR4609753, Frohman22MathZ, ishibashiMR4552137}.

L\^{e}--Yu \cite{LeArxiv23} constructed a $\mathrm{SL}_n(\mathbb{C})$-quantum trace map for webs, agreeing at the level of links with the definition proposed in this paper.  Their construction fits into a theory of $\mathrm{SL}_n(\mathbb{C})$-stated skein algebras \cite{LeQuantumTopology18, costantinoMR4493620, LeArxiv21}.  

Quantum traces also appear in the 
context of spectral networks \cite{GabellaCommMathPhys17, NeitzkeJHighEnerPhys20}.  Empirical computations performed together with A. Neitzke (see \cite{neitzkeMR4482975}) suggest that, at least for simple curves, the $\mathrm{SL}_3(\mathbb{C})$-quantum trace map defined in this paper agrees with that constructed in \cite{neitzkeMR4482975}; see also \cite{kimMR4568006} in the case $n=2$.  

The quantum trace map is a tool for studying the representation theory of skein algebras \cite{BonahonInvent16, FrohmanInvent19}, relevant to topological quantum field theories \cite{WittenCommMathPhys89, BlanchetTopology95}. 

\section*{Acknowledgements}

This work would not have been possible without the guidance of my Ph.D. advisor Francis Bonahon.  Many thanks go out to Sasha Shapiro and Thang L\^e for informing me about related research and for enjoyable conversations, as well as to Dylan Allegretti and Zhe Sun who helped me fine-tune my ideas.  I am also grateful to Andy Neitzke for sharing his enthusiasm for experiment, as well as to the referee for their helpful comments.  Last but not least, Person D would like to thank Person G for their limitless support (and artistic inspiration).  

\section{Classical trace polynomials for \texorpdfstring{$\mathrm{SL}_n$}{SLn}}
\label{sec:classical-fock-goncharov-coordinates}

In this section, we will associate a Laurent polynomial $\widetilde{\mathrm{Tr}}_\gamma(X_i^{1/n})$ in commuting formal $n$-roots $X_1^{\pm 1/n}$, $X_2^{\pm 1/n}$, $\dots$,  $X_N^{\pm 1/n}$ to each immersed oriented closed curve $\gamma$ transverse to a fixed ideal triangulation $\lambda$ of a punctured surface $\mathfrak{S}$, where $N$ depends only on the topology of $\mathfrak{S}$ and the rank of the Lie group $\mathrm{SL}_n(\mathbb{C})$.  

\subsection{Topological setup}
\label{ssec:topological-setting}

Let $\mathfrak{S}$ be an oriented \textit{punctured surface} of finite topological type, namely $\mathfrak{S}$ is obtained by removing a finite subset $P$, called the set of \textit{punctures}, from a compact oriented surface $\overline{\mathfrak{S}}$.   In particular, note that $\mathfrak{S}$ may have non-empty boundary, $\partial \mathfrak{S} \neq \emptyset$.  We require that there is at least one puncture, that each component of $\partial \overline{\mathfrak{S}}$ is punctured (that is, intersects $P$), and that the Euler characteristic $\chi(\mathfrak{S})$ of the punctured surface $\mathfrak{S}$ satisfies $\chi(\mathfrak{S}) < d/2$ where $d$ is the number of components of $\partial \mathfrak{S}$.  Note  that each component of $\partial \mathfrak{S}$ is an \textit{ideal arc}.  These topological conditions guarantee the existence of an \textit{ideal triangulation} $\lambda$ of the punctured  surface $\mathfrak{S}$, namely a triangulation $\overline{\lambda}$ of the  surface $\overline{\mathfrak{S}}$ whose vertex set is equal to the set of punctures $P$.  See Figure \ref{fig:example-ideal-triangulations} for some examples of ideal triangulations.  The ideal triangulation $\lambda$ consists of $\epsilon = -3 \chi(\mathfrak{S}) + 2d$ edges $E$ and $\tau =-2 \chi(\mathfrak{S}) + d$ triangles $\mathfrak{T}$.  

For simplicity, we always assume that $\lambda$ does not contain any \textit{self-folded triangles}.  Consequently, each triangle $\mathfrak{T}$ of $\lambda$ has three distinct edges.  Such an ideal triangulation $\lambda$ always exists  (except when $\mathfrak{S}$ is a disk with one internal puncture and one puncture on the boundary).  Our results should generalize to allow for self-folded triangles, requiring only minor adjustments (one would need to modify Definition \ref{def:FG-algebra-of-the-whole-surface}--compare \cite[\S 2.1]{BonahonGT11}, which makes use of the Weyl quantum ordering (\S \ref{sssec:weyl-ordering})--but the main definition, Definition \ref{def:main-definition-qtracemap}, is un-changed).

\begin{figure}[htb]
     \centering
     \begin{subfigure}{0.45\textwidth}
         \centering
         \includegraphics[width=.4\textwidth]{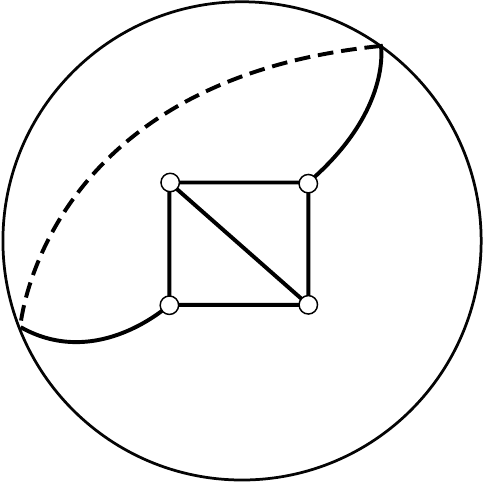}
         \caption{Four times punctured sphere}
         \label{fig:four-times-punctured-sphere-triang}
     \end{subfigure}     
\hfill
     \begin{subfigure}{0.45\textwidth}
         \centering
         \includegraphics[width=.75\textwidth]{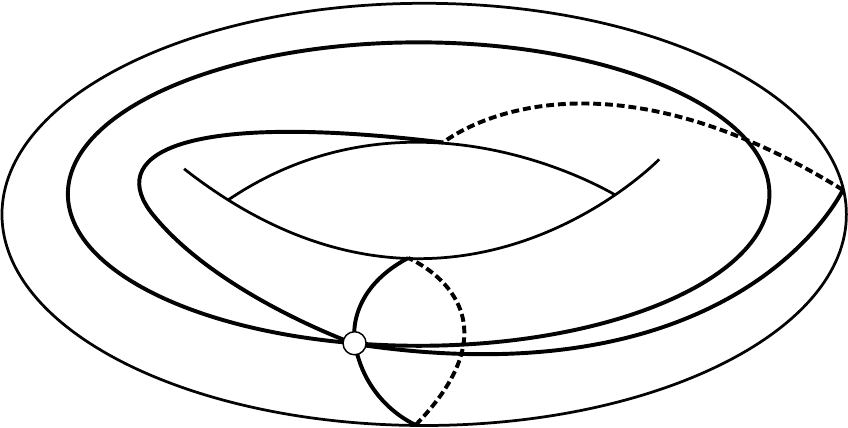}
         \caption{Once punctured torus}
         \label{fig:once-punctured-torus}
     \end{subfigure}
        \caption{Ideal triangulations ($\partial \mathfrak{S} = \emptyset$).}
        \label{fig:example-ideal-triangulations}
\end{figure}  

\subsection{Discrete triangle}
\label{subsec:the-discrete-triangle}

The \textit{discrete $n$-triangle} $\Theta_n$ is
\begin{equation*}
	\Theta_n \overset{\text{def}}{=} \left\{ (a, b, c) \in \mathbb{Z}^3 ;  a, b, c \geq 0,   a+b+c=n \right\},
\end{equation*}
as shown in Figure \ref{fig:n-discrete-triangle}.   The \textit{interior} of the $n$-discrete triangle is
\begin{equation*}
	\mathrm{int}(\Theta_n) \overset{\text{def}}{=} \left\{ (a, b, c) \in \mathbb{Z}^3 ;  a, b, c > 0,  a+b+c=n \right\}.
\end{equation*}

\begin{figure}[htb]
\centering
\includegraphics[scale=.55]{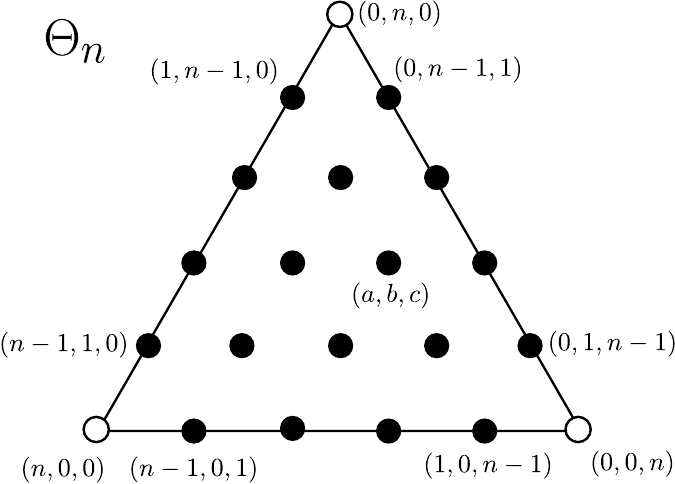}
\caption{Discrete triangle ($n=5$).}
\label{fig:n-discrete-triangle}    
\end{figure}

\subsection{Dotted ideal triangulations}
\label{ssec:dotted-ideal-triangulations}

Let the punctured  surface $\mathfrak{S}$ be equipped with an ideal triangulation $\lambda$, and let $N = \epsilon(n-1) + \tau(n-1)(n-2)/2$; see \S \ref{ssec:topological-setting}.  

The associated \textit{dotted ideal triangulation} consists of $\lambda$ together with $N$ distinct black \textit{dots} attached to the edges $E$ and triangles $\mathfrak{T}$ of $\lambda$, where there are $n-1$ \textit{edge-dots} attached to each edge $E$ and $(n-1)(n-2)/2$ \textit{triangle-dots} attached to each triangle $\mathfrak{T}$ (punctures, that is, triangle vertices, are always drawn as white dots).  For each triangle $\mathfrak{T}$ including its three boundary edges $E_1$, $E_2$, $E_3$, these dots are arranged  as the vertices of the discrete $n$-triangle $\Theta_n$ (minus its three corner vertices) overlaid on top of the ideal triangle $\mathfrak{T}$; see Figures \ref{fig:example-ideal-triangulations-dotted} and \ref{fig:coordinates-on-a-triangulation}.  We  talk about \textit{boundary-dots} or \textit{interior-dots} depending on whether the dots are on the boundary of interior of the surface.

Given a triangle $\mathfrak{T}$ of $\lambda$, which acquires an orientation from the orientation of $\mathfrak{S}$, and given an edge $E$ of $\mathfrak{T}$, it makes sense to say that an edge-dot on $E$ is \textit{to the left} or \textit{to the right} of another edge-dot on $E$ as viewed from the triangle $\mathfrak{T}$; see Figure \ref{fig:dotted-triangle}.  

We always assume that we have chosen an ordering for the $N$ dots lying on the dotted ideal triangulation $\lambda$, so we can talk about the $i$-th dot, $i=1,2,\dots,N$.  

\begin{figure}[htb]
     \centering
     \begin{subfigure}{0.43\textwidth}
         \centering
         \includegraphics[width=.45\textwidth]{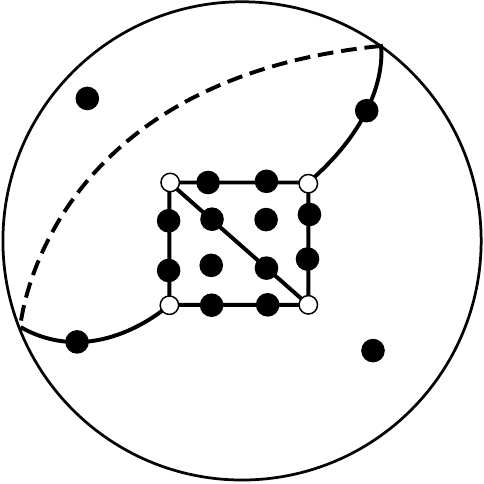}
         \caption{Four times punctured sphere}
         \label{fig:four-times-punctured-sphere-triang-dotted}
     \end{subfigure}     
\hfill
     \begin{subfigure}{0.43\textwidth}
         \centering
         \includegraphics[width=.48\textwidth]{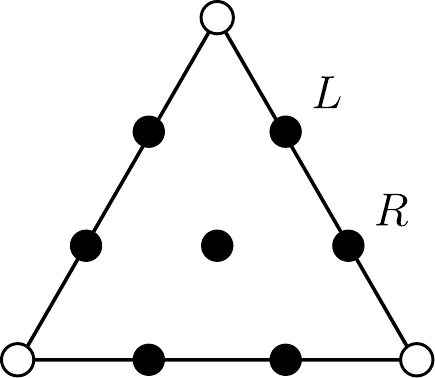}
         \caption{Ideal triangle}
         \label{fig:dotted-triangle}
     \end{subfigure}
        \caption{Dotted ideal triangulations ($n=3$).}
        \label{fig:example-ideal-triangulations-dotted}
\end{figure}

\begin{figure}[htb]
	\centering
	\includegraphics[scale=.55]{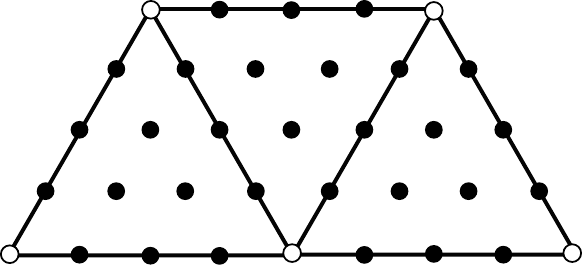}
	\caption{Dotted ideal triangulation ($n=4$).}
	\label{fig:coordinates-on-a-triangulation}
\end{figure}

\subsection{Classical polynomial algebra}
\label{sec:Fock-Goncharov-coordinates-on-a-triangulated-surface}

Let the punctured  surface $\mathfrak{S}$ be equipped with a dotted ideal triangulation $\lambda$.  

\begin{definition}
\label{def:generators}
The \textit{classical polynomial algebra} $\mathscr{T}_n^1(\lambda) = \mathbb{C}[X_1^{\pm 1/n}, X_2^{\pm 1/n}, \dots, X_N^{\pm 1/n}]$ associated to the dotted ideal triangulation $\lambda$ is the commutative algebra of Laurent polynomials generated by formal $n$-roots $X_i^{1/n}$ and their inverses.  We think of the generator $X_i^{1/n}$ as associated to the $i$-th dot lying on $\lambda$.  As for dots, see \S \ref{ssec:dotted-ideal-triangulations}, we speak of \textit{edge-} and \textit{triangle-generators} as well as \textit{boundary-} and \textit{interior-generators}.  Elements $X_i^{\pm 1} = (X_i^{\pm 1/n})^n$ of $\mathscr{T}_n^1(\lambda)$ are called \textit{coordinates}.  We often indicate edge-coordinates with the letter $Z$ instead of $X$.  
\end{definition}

\begin{remark}
\label{rem:boundary-coordinates}
The algebraic coordinates $X_i$ in the classical polynomial algebra correspond to Fock--Goncharov's geometric coordinates for the framed $\mathrm{PSL}_n(\mathbb{C})$-character variety $\mathscr{R}_{\mathrm{PSL}_n(\mathbb{C})}(\mathfrak{S})_{\mathrm{fr}}$; see \S \ref{sec:introduction}.  As a caveat,  in the classical geometric setting the Fock--Goncharov coordinates $X_i$ are  associated only  to the interior-dots (not to the boundary-dots), while in the quantum algebraic setting there are coordinates $X_i$ associated to the boundary-dots as well.  
	
In the language of cluster algebras \cite{FominJAmerMathSoc02}, these boundary-variables are also called \textit{frozen variables} \cite{Fock06, FockIHES06, FominArxiv16}.  From the classical geometric point of view, one can think of the frozen boundary-coordinates as having the potential to become `actual' un-frozen interior-coordinates if the surface-with-boundary $\mathfrak{S}$ were included inside the interior of a larger surface $\mathfrak{S}^\prime$.  At the quantum algebraic level, the inclusion of boundary-coordinates is an essential  step in order to observe the local quantum properties; see, for instance, Theorem \ref{thm:first-theorem}.  

In the case of $\mathrm{SL}_2(\mathbb{R})$ or $\mathrm{SL}_2(\mathbb{C})$, the Fock--Goncharov coordinates $X_i$ coincide with the shear or shear-bend coordinates for Teichm\"{u}ller space due to Thurston \cite{Thurston97}; see \cite{bonahonMR1413855,FockArxiv97, Fock07,Hollands16LettMathPhys} for more details.  There is also a geometric interpretation of Fock--Goncharov's coordinates in the case $n=3$, where  the coordinates parametrize convex projective structures on the surface $\mathfrak{S}$; see \cite{FockAdvMath07, CasellaMathSocJapMem20}.  
\end{remark}

\subsection{Elementary edge and triangle matrices}
\label{sec:theorem-3}

Let $\mathrm{M}_n(\mathscr{T}_n^1(\lambda))$ denote the   algebra of $n \times n$ matrices with coefficients in the commutative classical polynomial algebra $\mathscr{T}_n^1(\lambda)$ $=$ $\mathbb{C}[X_1^{\pm 1/n}, X_2^{\pm 1/n}, \dots, X_N^{\pm 1/n}]$; see Definition  \ref{def:matrix-algebra}.  Let the special linear group $\mathrm{SL}_n(\mathscr{T}_n^1(\lambda))$  be the subset of $\mathrm{M}_n(\mathscr{T}_n^1(\lambda))$ consisting of the matrices with determinant equal to $1$.  
  
Let $Z = X_i^{\pm 1}$ be an edge-coordinate in the classical polynomial algebra $\mathscr{T}_n^1(\lambda)$.  For $j = 1, 2, \dots, n-1$ define the \textit{$j$-th elementary edge matrix} $\vec{S}^\mathrm{edge}_j(Z) \in \mathrm{SL}_n(\mathscr{T}_n^1(\lambda))$ by 
\begin{equation*}
\label{eq:shearing-matrix}
	\vec{S}^\mathrm{edge}_j(Z) \overset{\text{def}}{=} Z^{-j/n}
	\left(\begin{smallmatrix}
		Z&&&&&&&\\
		&Z&&&&&&\\
		&&\ddots&&&&&\\
		&&&Z&&&&\\
		&&&&1&&&\\
		&&&&&1&&\\
		&&&&&&\ddots&\\
		&&&&&&&1
	\end{smallmatrix}\right)
	  \in \mathrm{SL}_n(\mathscr{T}_n^1(\lambda))\,\,
	  \left(Z \text{ appears } j \text{ times}\right).
\end{equation*}
Note the normalizing factor $Z^{-j/n}$ multiplying the matrix on the left (or on the right).  
Similarly, for any triangle-coordinate $X = X_i^{\pm 1}$ in $\mathscr{T}_n^1(\lambda)$ and for any index $j = 1, 2, \dots, n-1$ define the \textit{$j$-th left elementary triangle matrix} $\vec{S}^\mathrm{left}_j (X) \in \mathrm{SL}_n(\mathscr{T}_n^1(\lambda))$ by
\begin{equation*}
\label{eq:upper-triangular-matrix}
	\vec{S}^\mathrm{left}_j(X) \overset{\text{def}}{=} X^{-(j-1)/n}
	\left(\begin{smallmatrix}
		X&&&&&&&\\
		&\ddots&&&&&&\\
		&&X&&&&&\\
		&&&1&1&&&\\
		&&&&1&&&\\
		&&&&&1&&\\
		&&&&&&\ddots&\\
		&&&&&&&1
	\end{smallmatrix}\right)
	\in  \mathrm{SL}_n(\mathscr{T}_n^1(\lambda))\,\,
	  \left(X \text{ appears } j-1 \text{ times}\right),
\end{equation*}
and define the \textit{$j$-th right elementary triangle matrix} $\vec{S}^\mathrm{right}_j (X) \in \mathrm{SL}_n(\mathscr{T}_n^1(\lambda))$ by
\begin{equation*}
\label{eq:lower-triang-matrix}
	\vec{S}^\mathrm{right}_j(X) \overset{\text{def}}{=}
	X^{(j-1)/n}
	\left(\begin{smallmatrix}
		1&&&&&&&\\
		&\ddots&&&&&&\\
		&&1&&&&&\\
		&&&1&&&\\
		&&&1&1&&&\\
		&&&&&X^{-1}&&\\
		&&&&&&\ddots&\\
		&&&&&&&X^{-1}
	\end{smallmatrix}\right)	
	 \in \mathrm{SL}_n(\mathscr{T}_n^1(\lambda))\,\,
	  \left(X \text{ appears } j-1 \text{ times}\right).  
\end{equation*}
Note that $\vec{S}^\mathrm{left}_1(X)$ and $\vec{S}^\mathrm{right}_1(X)$ do not actually involve the variable $X$, so we will denote these matrices  by $\vec{S}^\mathrm{left}_1$ and $\vec{S}^\mathrm{right}_1$, respectively.  

\begin{remark}
\label{rem:classical-coordinates-snakes}
In the theory of Fock--Goncharov, these elementary matrices for edges and triangles are called \textit{snake-move matrices}.  Each such matrix is the coordinate transformation matrix passing between a pair of compatibly normalized projective coordinate systems associated to a pair of adjacent \textit{snakes}.  For a framed local system in $\mathscr{R}_{\mathrm{PSL}_n(\mathbb{C})}(\mathfrak{S})_{\mathrm{fr}}$ with  coordinates $X_i$, computing the monodromy of the local system around a curve $\gamma$ amounts to multiplying together a sequence of snake-move matrices along the direction of the curve.  For more details, see \cite[\S 9]{FockIHES06}, \cite[Appendix A]{GaiottoAnnHenriPoincare14}, and \cite[Chapter 2]{DouglasThesis20}.   
\end{remark}

\subsection{Local monodromy matrices}
\label{sec:local-monodromy-matrices}
	
 Let $\mathfrak{T}$ be a dotted ideal triangle, which we think of as sitting inside a larger dotted ideal triangulation $\lambda$; see Figure \ref{fig:example-ideal-triangulations-dotted}.  We assign $n \times n$ matrices with coefficients in the classical polynomial algebra $\mathscr{T}_n^1(\lambda) = \mathbb{C}[X_1^{\pm 1/n}, X_2^{\pm 1/n}, \dots, X_N^{\pm 1/n}]$ to various `short' oriented arcs lying on the surface $\mathfrak{S}$.  

Recall that we think of the $n$-discrete triangle $\Theta_n$ (\S \ref{subsec:the-discrete-triangle}) as overlaid on top of the triangle $\mathfrak{T}$, so that there is a one-to-one correspondence between coordinates $X_i = X_{abc}^{}$ in $\mathscr{T}_n^1(\mathfrak{T}) \subset \mathscr{T}_n^1(\lambda)$ and vertices $(a,b,c)$ in $\Theta_n - \{ (n,0,0), (0,n,0), (0,0,n) \}$ (in other words, in the $n$-discrete triangle $\Theta_n$ minus its three corner vertices).  Note that $X_{abc}^{}$ is a triangle-coordinate if and only if $(a, b, c)$ is an interior point $(a, b, c) \in \mathrm{int}(\Theta_n)$, otherwise $X_{abc}^{}$ is an edge-coordinate.  

We will use the following notational convention.  Given an arbitrary family $\vec{M}_i$ of $n \times n$ matrices, put
\begin{gather*}
	\prod_{i=M}^N \vec{M}_i \overset{\text{def}}{=} \vec{M}_{M} \vec{M}_{M+1} \cdots \vec{M}_N,\,\,
	  \prod_{i=N+1}^M \vec{M}_i \overset{\text{def}}{=} 1\,\,
	  \left(  M \leq N  \right),
\\	\coprod_{i=N}^M \vec{M}_i \overset{\text{def}}{=} \vec{M}_N \vec{M}_{N-1} \cdots \vec{M}_M,\,\,
	  \coprod_{i=M-1}^N \vec{M}_i \overset{\text{def}}{=} 1\,\,
	  \left(  M \leq N  \right).
\end{gather*}

First, consider a \textit{left-moving arc} $\overline{\gamma}$, as shown in Figure \ref{fig:left-matrix}.  We assume $\overline{\gamma}$ has no kinks (Figures \ref{fig:positive-kink-skein-relation} and \ref{fig:negative-kink-skein-relation}).  Let $\vec{X}=(X_i)$ be a vector consisting of the triangle-coordinates $X_i$ (Definition \ref{def:generators}).  Define the associated \textit{left matrix} $\vec{M}^\mathrm{left}(\vec{X})$ in $\mathrm{SL}_n(\mathscr{T}_n^1(\lambda))$ by

\begin{equation*}
\label{eq:explicit-quantum-left-matrix}
\vec{M}^\mathrm{left}(\vec{X}) \overset{\text{def}}{=} 
\coprod_{i=n-1}^1 \left( \vec{S}^\mathrm{left}_1 \prod_{j=2}^i \vec{S}^\mathrm{left}_j(X_{(j-1)(n-i)(i-j+1)}) \right)
\in \mathrm{SL}_n(\mathscr{T}_n^1(\lambda)),
\end{equation*}
where the matrix $\vec{S}_j^\mathrm{left}(X_{abc})$ is the $j$-th left elementary triangle matrix; see \S \ref{sec:theorem-3}.  (The dots colored red in the figure correspond to the coordinates appearing in the expression of the matrix associated to the curve.)

Next, consider a \textit{right-moving arc} $\overline{\gamma}$, as shown in Figure \ref{fig:right-matrix}.  We assume $\overline{\gamma}$ has no kinks.  We define the associated \textit{right matrix} $\vec{M}^\mathrm{right}(\vec{X}) $ in $ \mathrm{SL}_n(\mathscr{T}_n^1(\lambda))$ by
\begin{equation*}
\label{eq:explicit-quantum-right-matrices}
\vec{M}^\mathrm{right}(\vec{X}) \overset{\text{def}}{=} 
 \coprod_{i=n-1}^1 \left( \vec{S}^\mathrm{right}_1 \prod_{j=2}^i \vec{S}^\mathrm{right}_j(X_{(i-j+1)(n-i)(j-1)}) \right)
   \in \mathrm{SL}_n(\mathscr{T}_n^1(\lambda)),
\end{equation*}
where the matrix $\vec{S}_j^\mathrm{right}(X_{abc})$ is the $j$-th right elementary triangle matrix; see \S \ref{sec:theorem-3}.  

Next, consider an \textit{edge-crossing arc} $\overline{\gamma}$, as shown on the left hand or right hand side of Figure \ref{fig:edge-matrix}.    Let $Z_j$, $j=1,2,\dots,n-1$, be the $j$-th edge-coordinate (Definition \ref{def:generators}), measured from right to left as seen by the triangle out of which the arc is moving.  Let $\vec{Z}=(Z_j)$ be a vector consisting of the $Z_j$'s.  Define the associated \textit{edge matrix} $\vec{M}^\mathrm{edge}(\vec{Z}) $ in $ \mathrm{SL}_n(\mathscr{T}_n^1(\lambda))$ by
\begin{equation*}
\label{eq:edge-matrix}
	\vec{M}^\mathrm{edge}(\vec{Z}) 
	\overset{\text{def}}{=} \prod_{j=1}^{n-1} \vec{S}^\mathrm{edge}_j(Z_j)
  \in \mathrm{SL}_n(\mathscr{T}_n^1(\lambda)),
\end{equation*}
where the matrix $\vec{S}_j^\mathrm{edge}(Z_j)$ is the $j$-th  elementary edge matrix; see \S \ref{sec:theorem-3}.  
Note that if the orientation of the edge-crossing arc is reversed, then the edge matrix changes by permuting the coordinates $Z_j$ by $Z_j \leftrightarrow Z_{n-j}$; see Figure \ref{fig:edge-matrix}.  

Observe that the order in which the elementary matrices $\vec{S}_j$ are multiplied does not matter in the formula for the edge matrix $\vec{M}^\mathrm{edge}(\vec{Z})$, since they are diagonal, but does matter in the formulas for the triangle matrices $\vec{M}^\mathrm{left}(\vec{X})$ and $\vec{M}^\mathrm{right}(\vec{X})$.

Lastly, define the \textit{clockwise U-turn matrix} $\vec{U} $ in $ \mathrm{SL}_n(\mathbb{C})$ by 
\begin{equation*}
\label{eq:U-turn-matrix}
	\vec{U}
	\overset{\text{def}}{=}  \left(\begin{smallmatrix}
		&&&&\\
		&&&&(-1)^{n-1}\\
		&&&\reflectbox{$\ddots$}&\\
		&&+1&&\\
		&-1&&&\\
		+1&&&&
	\end{smallmatrix}\right)
  \in  \mathrm{SL}_n(\mathbb{C}).
\end{equation*}
We associate to a \textit{clockwise U-turn arc} (resp. \textit{counterclockwise U-turn arc}) $\overline{\gamma}$, as shown on the left hand (resp. right hand) side of Figure \ref{fig:U-turn-matrices},  the U-turn matrix $\vec{U}$ (resp. transpose $\vec{U}^\mathrm{T}$ of the U-turn matrix).  Again, we have assumed that $\overline{\gamma}$ has no kinks.  Note that $\vec{U}^\mathrm{T} = - \vec{U}$ (resp. $=\vec{U}$) when $n$ is even (resp. odd).  

\begin{figure}[htb]
	\centering
	\includegraphics[scale=.5]{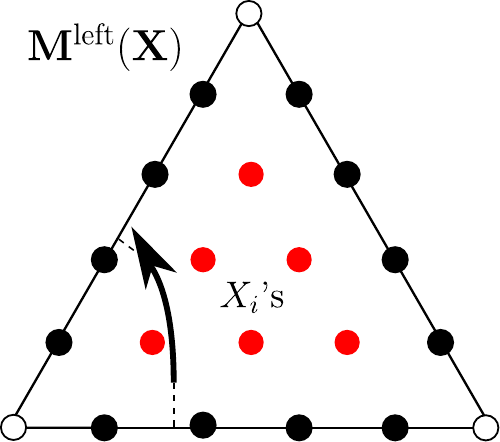}
	\caption{Left matrix ($n=5$).}
	\label{fig:left-matrix}
\end{figure}

\begin{figure}[htb]
	\centering
	\includegraphics[scale=.5]{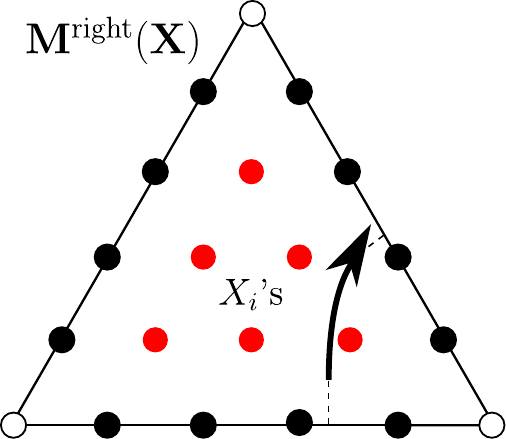}
	\caption{Right matrix ($n=5$).}
	\label{fig:right-matrix}
\end{figure}

\begin{figure}[htb]
	\centering
	\includegraphics[scale=.8]{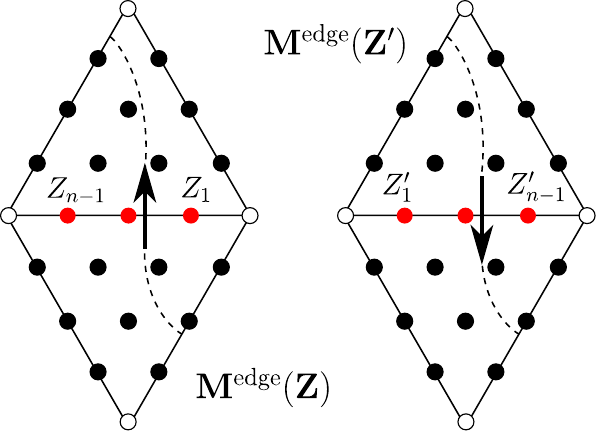}
	\caption{Edge matrices $(n=4)$.}
	\label{fig:edge-matrix}
\end{figure}

\begin{figure}[htb]
     \centering
     \begin{subfigure}{0.49\textwidth}
         \centering
         \includegraphics[scale=.6]{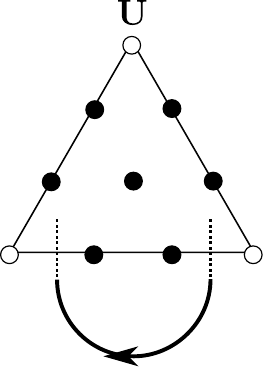}
         \caption{Clockwise U-turn}
         \label{fig:clockwise-U-turn}
     \end{subfigure}     
\hfill
     \begin{subfigure}{0.49\textwidth}
         \centering
         \includegraphics[scale=.6]{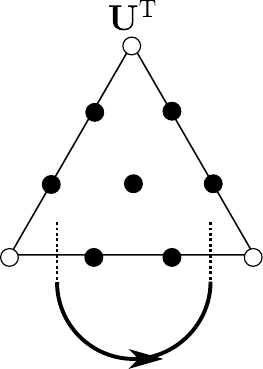}
         \caption{Counterclockwise U-turn}
         \label{fig:counter-clockwise-U-turn}
     \end{subfigure}
     	\caption{U-turn matrices $(n=3)$.}
        \label{fig:U-turn-matrices}
\end{figure}

\subsection{Definition of the \texorpdfstring{$\mathrm{SL}_n$}{SLn}-classical trace polynomials}
\label{sec:computing-the-classical-trace-polynomial}

Let $\gamma$ be an immersed oriented closed curve in the surface $\mathfrak{S}$ such that $\gamma$ is transverse to the ideal triangulation $\lambda$.  We want to calculate the classical trace polynomial $\widetilde{\mathrm{Tr}}_\gamma(X_i^{1/n})$ in $\mathscr{T}_n^1(\lambda)$ $=$ $\mathbb{C}[X_1^{\pm 1/n}, X_2^{\pm 1/n}, \dots, X_N^{\pm 1/n}]$ associated to the immersed curve $\gamma$.  We say that the polynomial $\widetilde{\mathrm{Tr}}_\gamma(X_i^{1/n})$ is obtained from a `state sum', `local-to-global', or `transfer matrices' construction.  

More precisely, as we travel along the curve $\gamma$ according to its orientation, assume $\gamma$ crosses edges $E_{j_k}$ for $k=1, 2, \dots, K$ in that order, and assume $\gamma$ crosses triangles $\mathfrak{T}_{i_k}$ for $k =1, 2, \dots, K$ in that order.  
As the curve $\gamma$ crosses the edge $E_{j_k}$, moving out of the triangle $\mathfrak{T}_{i_{k-1}}$ into the triangle $ \mathfrak{T}_{i_k}$, this defines an edge-crossing arc $\overline{\gamma}_{j_k}$; see \S \ref{sec:local-monodromy-matrices} and Figure \ref{fig:edge-matrix}.  Put $\vec{Z}_{j_k} = ((Z_{j_k})_{j^\prime})$ and put
\begin{equation*}
	\vec{M}^\mathrm{edge}_{j_k} \overset{\text{def}}{=} \vec{M}^\mathrm{edge}(\vec{Z}_{j_k})  \in \mathrm{SL}_n(\mathscr{T}_n^1(\lambda)),
\end{equation*} 
 the associated edge matrix, where the $(Z_{j_k})_{j^\prime}\text{'s}$ are the $j^\prime = 1, \dots, n-1$ edge-coordinates attached to the edge $E_{j_k}$ measured from right to left as seen from $\mathfrak{T}_{i_{k-1}}$.    
As $\gamma$ traverses the triangle $\mathfrak{T}_{i_k}$ between two edges $E_{j_k}$ and $E_{j_{k+1}}$, it does one of three things:
\begin{itemize}
	\item  the curve $\gamma$ turns left, ending on $E_{j_{k+1}} \neq E_{j_k}$, see Figure \ref{fig:left-matrix};
	\item  or $\gamma$ turns right, ending on $E_{j_{k+1}} \neq E_{j_k}$, see Figure \ref{fig:right-matrix};
	\item  or $\gamma$ does a U-turn, thereby returning to the same edge $E_{j_{k+1}} = E_{j_{k}}$, see Figure \ref{fig:U-turn-matrices}.
\end{itemize}
We also keep track of the following winding information:  for the first and second items above, the number $t_k$   of full turns  to the right that the curve $\gamma$ makes while traversing the triangle $\mathfrak{T}_{i_k}$; and for the third item above, the number $2 t_k + 1$   of half turns  to the right that the curve $\gamma$ makes before coming back to the same edge $E_{j_k}$.  Note $t_k \in \mathbb{Z}$.  Note that the turning integer $t_k$ associated to the curve $\gamma$ as it traverses the triangle $\mathfrak{T}_{i_k}$ will only be relevant when  $n$  is even.  Let the $(X_{i_k})_{i^\prime}\text{'s}$ be the $i^\prime = 1, \dots, (n-1)(n-2)/2$ triangle-coordinates attached to the triangle $\mathfrak{T}_{i_k}$, and put $\vec{X}_{i_k} = ((X_{i_k})_{i^\prime})$; see Figure \ref{fig:left-matrix}.  Let $\gamma^\prime$ be the curve obtained by `pulling tight' the kinks of $\gamma$.  (In particular, $\gamma$ and $\gamma^\prime$ are homotopic, but not, in general, regularly homotopic.)  Corresponding to the three items above:
\begin{itemize}
	\item  the curve $\gamma^\prime$ turns left, defining a left-moving arc $\overline{\gamma}^\prime$ and an associated left matrix $\vec{M}^\prime_{i_k} = \vec{M}^{\mathrm{left}}(\vec{X}_{i_k})$, see \S \ref{sec:local-monodromy-matrices} and Figure \ref{fig:left-matrix};
	\item  or the curve $\gamma^\prime$ turns right, defining a right-moving arc $\overline{\gamma}^\prime$ and an associated right matrix $\vec{M}^\prime_{i_k} = \vec{M}^{\mathrm{right}}(\vec{X}_{i_k})$, see \S \ref{sec:local-monodromy-matrices} and Figure \ref{fig:right-matrix};
	\item  or the curve $\gamma^\prime$ does a clockwise or counterclockwise U-turn, thereby returning to the same edge $E_{j_{k+1}} = E_{j_k}$ and defining a U-turn arc $\overline{\gamma}^\prime$, see \S \ref{sec:local-monodromy-matrices} and Figure \ref{fig:U-turn-matrices}.  \end{itemize}
In the first two cases, where $\overline{\gamma}^\prime$ is either left- or right-moving, put
\begin{equation*}
	\vec{M}_{i_k} \overset{\text{def}}{=} (-1)^{(n-1)t_k} \vec{M}^{\prime}_{i_k}
	 \in \mathrm{SL}_n(\mathscr{T}_n^1(\lambda)),
\end{equation*}
and in the third case, where $\overline{\gamma}^\prime$ is a U-turn, put
\begin{equation*}	
	\vec{M}_{i_k} \overset{\text{def}}{=} (-1)^{(n-1)t_k} \vec{U}
	 \in \mathrm{SL}_n(\mathbb{C}),
\end{equation*}
where $\vec{U}$ is the U-turn matrix defined in \S \ref{sec:local-monodromy-matrices}.  Note, in the third case, that this is consistent with what was said in \S \ref{sec:local-monodromy-matrices}, where, in the case $\overline{\gamma}$ has no kinks, $\overline{\gamma}^\prime = \overline{\gamma}$ is associated to $\vec{U}$ (resp. $\vec{U}^\mathrm{T} = (-1)^{n-1} \vec{U}$) if $\overline{\gamma}$ travels clockwise hence $t_k = 0$ (resp. counterclockwise hence $t_k = -1$).  

\begin{definition}
\label{def:second-classical-trace}
	The \textit{$\mathrm{SL}_n$-classical trace polynomial} $\widetilde{\mathrm{Tr}}_\gamma(X_i^{1/n}) \in \mathscr{T}_n^1(\lambda)$ associated to the immersed oriented closed curve $\gamma$, transverse to the ideal triangulation $\lambda$, is defined by
	\begin{equation*}
		\widetilde{\mathrm{Tr}}_\gamma(X_i^{1/n}) 
		\overset{\text{def}}{=}  \mathrm{Tr}  (
		\vec{M}^\mathrm{edge}_{j_1} \vec{M}_{i_1} \vec{M}^\mathrm{edge}_{j_2} \vec{M}_{i_2} \cdots \vec{M}^\mathrm{edge}_{j_K} \vec{M}_{i_K}
		)
		 \in \mathscr{T}_n^1(\lambda)
		=
		\mathbb{C}[X_1^{\pm 1/n}, \dots, X_N^{\pm 1/n}],
	\end{equation*}
	where on the right hand side we have taken the usual matrix trace.  Note that this is independent of where one starts along the curve $\gamma$, by the conjugation invariance of the trace.  
\end{definition}

\subsection{Relation to Fock--Goncharov theory}	

In this subsection, we assume (for simplicity) the surface $\mathfrak{S}$ has empty boundary, $\partial \mathfrak{S} = \emptyset$.    A \textit{complete flag} $E\in\mathrm{Flag}(\mathbb{C}^n)$ is a maximal nested sequence of distinct sub-spaces $\left\{0\right\}=E^{(0)} \subsetneq E^{(1)} \subsetneq E^{(2)} \subsetneq \cdots \subsetneq E^{(n)} = \mathbb{C}^n$.    The group $\mathrm{PSL}_n(\mathbb{C})$ acts on the set of complete flags by matrix multiplication.  

In \cite{FockIHES06}, Fock--Goncharov define the \textit{moduli space of framed local systems}, also called the \textit{$\mathscr{X}$-moduli space}, whose complex points have been denoted in this paper by $\mathscr{R}_{\mathrm{PSL}_n(\mathbb{C})}(\mathfrak{S})_\mathrm{fr}$.  
A framed local system over $\mathbb{C}$ is, roughly speaking, a pair $(\rho, \xi)$ where $\rho: \pi_1(\mathfrak{S}) \to \mathrm{PSL}_n(\mathbb{C})$ is a group homomorphism and $\xi$  assigns to each puncture $p$  a complete flag $\xi(p)$ invariant under the monodromy of any peripheral curve $\gamma_p$ around $p$, that is, $\rho(\gamma_p)(\xi(p))=\xi(p) \in \mathrm{Flag}(\mathbb{C}^n)$.  Framed local systems are considered up to equivalence under the action of $\mathrm{PSL}_n(\mathbb{C})$, which in particular acts on representations $\rho$ by conjugation.

Fock--Goncharov associate to each ideal triangulation $\lambda$ of $\mathfrak{S}$ a coordinate chart $U_\lambda$ for the moduli space $\mathscr{R}_{\mathrm{PSL}_n(\mathbb{C})}(\mathfrak{S})_\mathrm{fr}$ parametrized by non-zero complex coordinates $X_i$.       Here,  $N$ is the number of dots  associated to the ideal triangulation $\lambda$; see \S \ref{ssec:dotted-ideal-triangulations},  \ref{sec:Fock-Goncharov-coordinates-on-a-triangulated-surface}.     In particular, choosing $\lambda$ and $N$-many coordinates $X_i$ determines a representation $\rho=\rho(X_i)$ up to conjugation.    Since $\rho$ is valued in $\mathrm{PSL}_n(\mathbb{C})$, the trace $\mathrm{Tr}(\rho(\gamma))$ for any $\gamma \in \pi_1(\mathfrak{S})$ is well-defined only up to multiplication by a $n$-root of unity.  

\begin{theorem}[\cite{FockIHES06}, $\mathrm{SL}_n$-classical trace polynomials]  Fix an ideal triangulation $\lambda$ of the punctured surface $\mathfrak{S}$.  Let $\rho=\rho(X_i)$ be a  representation $\rho: \pi_1(\mathfrak{S}) \to \mathrm{PSL}_n(\mathbb{C})$ given by Fock--Goncharov coordinates $X_i$ as above.  Moreover, choose arbitrary $n$-roots $X_i^{1/n} \in \mathbb{C}-\{0\}$.  Then for each immersed oriented closed curve $\gamma$ in $\mathfrak{S}$ transverse to the ideal triangulation $\lambda$, the trace $\mathrm{Tr}(\rho(\gamma))$ equals the evaluated classical trace polynomial $\widetilde{\mathrm{Tr}}_\gamma(X_i^{1/n}) \in \mathbb{C}$ (Definition {\upshape\ref{def:second-classical-trace}}) up to multiplication by a $n$-root of unity.  \qed
\end{theorem}

\section{Quantum matrices for \texorpdfstring{$\mathrm{SL}_n$}{SLn}}
\label{sec:quantum-matrices}

In this section, we will define quantum versions of the classical polynomial algebra and the classical local monodromy matrices, and relate them to the quantum special linear group $\mathrm{SL}_n^q$.  
Throughout, let $q \in \mathbb{C} - \left\{ 0 \right\}$ and $\omega = q^{1/n^2}$ be a $n^2$-root of $q$.  Technically, also choose $\omega^{1/2}$.  

\subsection{Quantum tori, matrix algebras, and the Weyl quantum ordering}
\label{ssec:quantum-tori-matrix-algebras-and-weyl-ordering}

\subsubsection{Quantum tori}
\label{sssec:quantum-tori}
	
For a natural number $N^\prime>0$, let $\vec{P}$ (for `Poisson') be an integer $N^\prime \times N^\prime$ anti-symmetric matrix.  

\begin{definition}
The \textit{quantum torus (with $n$-roots)} $\mathscr{T}^\omega(\vec{P})$ associated to $\vec{P}$ is the quotient of the free algebra $\mathbb{C}\{ X_1^{1/n}, X_1^{-1/n}, \dots, X_{N^\prime}^{1/n}, X_{N^\prime}^{-1/n}\}$ in the indeterminates $X_i^{\pm 1/n}$ by the two-sided ideal generated by the relations
\begin{equation*}	
\label{q-commutation-relations}
	X_i^{m_i/n} X_j^{m_j/n} 
	= \omega^{\vec{P}_{ij} m_i m_j} X_j^{m_j/n} X_i^{m_i/n}
	\,\, (m_i, m_j \in \mathbb{Z}),\,\,
	X_i^{m/n} X_i^{-{m/n}} = X_i^{-{m/n}} X_i^{m/n} = 1
	\,\,	(m \in \mathbb{Z}). 
\end{equation*}
Put $X_i^{\pm 1} = (X_i^{\pm 1/n})^n$.  We refer to the $X_i^{\pm 1/n}$ as \textit{generators}, and the $X_i$ as \textit{quantum coordinates}, or just \textit{coordinates}.  Define the subset of fractions
\begin{equation*}
	\mathbb{Z}/n \overset{\text{def}}{=} \left\{ {m/n};  m \in \mathbb{Z}\right\}  \subset \mathbb{Q}.
\end{equation*}  
Written in terms of the coordinates $X_i$ and the fractions $r \in \mathbb{Z}/n$, the relations above become
\begin{equation*}
\label{eq:easier-to-read-q-commutation-relations}
	X_i^{r_i} X_j^{r_j} = q^{\vec{P}_{ij} r_i r_j} X_j^{r_j} X_i^{r_i}
	\,\,  (r_i, r_j \in \mathbb{Z}/n),
\,\,
	X_i^r X_i^{-r} = X_i^{-r} X_i^r = 1
	\,\,	(r \in \mathbb{Z}/n). 
\end{equation*}
\end{definition}

\subsubsection{Matrix algebras}
\label{sssec:matrix-algebras}
	
\begin{definition}
\label{def:matrix-algebra}
	Let $\mathscr{T}$ be a, possibly non-commutative, algebra, and let $n^\prime$ be a positive integer.  The \textit{matrix algebra with coefficients in $\mathscr{T}$}, denoted $\mathrm{M}_{n^\prime}(\mathscr{T})$, is the complex vector space of $n^\prime \times n^\prime$ matrices, equipped with the usual multiplicative structure.  Specifically, the product $\vec{M} \vec{N}$ of two matrices $\vec{M}$ and $\vec{N}$ is defined entry-wise by 
\begin{equation*}
	(\vec{M} \vec{N})_{ij} \overset{\text{def}}{=} \sum_{k=1}^{n^\prime} \vec{M}_{ik} \vec{N}_{kj}  \in \mathscr{T}
	\,\,  \left(  1 \leq i, j \leq n^\prime  \right).
\end{equation*}  
As usual, the entry $\vec{M}_{ij}$ of a matrix $\vec{M}$ is the entry in the $i$-th row and $j$-th column.  Note that the order of $\vec{M}_{ik}$ and $\vec{N}_{kj}$ in the above equation matters since these elements might not commute in $\mathscr{T}$.  
\end{definition}

\subsubsection{Weyl quantum ordering}
\label{sssec:weyl-ordering}
	
If $\mathscr{T}$ is a quantum torus (\S \ref{sssec:quantum-tori}), then there is a linear map
\begin{equation*}
\label{eq:weyl-ordering-eq-1}
	[-] \colon \mathbb{C}\{ X_1^{\pm 1/n}, \dots, X_{N^\prime}^{\pm 1/n}\} \to \mathscr{T},
\end{equation*}
from the free algebra to $\mathscr{T}$, called the \textit{Weyl quantum ordering}, defined by the property that a word
	$X_{i_1}^{r_1} X_{i_2}^{r_2} \cdots X_{i_k}^{r_k}$ for $r_a \in \mathbb{Z}/n$ (note $i_a$ may equal $i_b$ if $a \neq b$)
 is mapped to
\begin{equation*}
\label{eq:weyl-quantum-ordering-def}
	[X_{i_1}^{r_1} X_{i_2}^{r_2} \cdots X_{i_k}^{r_k}] \overset{\text{def}}{=} \left( q^{-\frac{1}{2} \sum_{1 \leq a < b \leq k} \vec{P}_{i_a i_b} r_a r_b} \right) X_{i_1}^{r_1} X_{i_2}^{r_2} \cdots X_{i_k}^{r_k}\in\mathscr{T}.
\end{equation*}
Also, the empty word is mapped to $1$.  Note the Weyl ordering $[-]$ depends on the choice of $\omega^{1/2}$; see the beginning of \S \ref{sec:quantum-matrices}.  
The Weyl ordering is specially designed to satisfy the symmetry
\begin{equation*}
\label{eq:symmetric-property-of-the-Weyl-ordering}
	[X_{i_1}^{r_1} \cdots X_{i_k}^{r_k}] 
	= [X_{i_{\sigma(1)}}^{r_{\sigma(1)}} \cdots X_{i_{\sigma(k)}}^{r_{\sigma(k)}}]\in\mathscr{T},
\end{equation*}
for every permutation $\sigma$ of $\left\{1, \dots, k\right\}$.  Also, $[X_i^{1/n} X_i^{-1/n}]=1$.  Let
\begin{equation*}
	[-] \colon \mathbb{C}[ X_1^{\pm 1/n}, \dots, X_{N^\prime}^{\pm 1/n}] \to \mathscr{T},
\end{equation*}
be the induced linear map from the  commutative Laurent polynomial algebra to $\mathscr{T}$.  This determines a linear map of matrix algebras
\begin{equation*}
	[-] \colon \mathrm{M}_{n^\prime}( \mathbb{C}[ X_1^{\pm 1/n}, \dots, X_{N^\prime}^{\pm 1/n}]) \to \mathrm{M}_{n^\prime}(\mathscr{T}), 
	\,\, 	[\vec{M}]_{ij} \overset{\text{def}}{=} [\vec{M}_{ij}]  \in \mathscr{T}.
\end{equation*} 

\subsection{Fock--Goncharov quantum torus for a triangle}	
\label{sec:Fock-Goncharov-algebra-for-a-triangle}
	
Let $\Gamma(\Theta_n)$ denote the set of corner vertices $\Gamma(\Theta_n)=\left\{ (n, 0, 0), (0, n, 0), (0, 0, n) \right\}$ of the discrete triangle $\Theta_n$; see \S \ref{subsec:the-discrete-triangle}.  

Define a function 
\begin{equation*}
	\vec{P} : (\Theta_n - \Gamma(\Theta_n))
	\times 
	(\Theta_n - \Gamma(\Theta_n))
	 \to \left\{-2, -1, 0, 1, 2\right\},
\end{equation*}
using the \textit{quiver} with vertex set $\Theta_n - \Gamma(\Theta_n)$ illustrated in Figure \ref{fig:Fg-quiver}.  The function $\vec{P}$ is defined by sending the ordered tuple $(v_1, v_2)$ of vertices of $\Theta_n - \Gamma(\Theta_n)$ to $2$ (resp. $-2$) if there is a solid arrow pointing from $v_1$ to $v_2$ (resp. $v_2$ to $v_1$), to $1$ (resp. $-1$) if there is a dotted arrow pointing from $v_1$ to $v_2$ (resp. $v_2$ to $v_1$), and to $0$ if there is no arrow connecting $v_1$ and $v_2$.  Note that internal arrows are solid, and boundary arrows are dotted.  By labeling the vertices of $\Theta_n - \Gamma(\Theta_n)$ by their coordinates $(a, b, c)$ we may think of the function $\vec{P}$ as a $N \times N$ anti-symmetric matrix $\vec{P} = (\vec{P}_{abc, a^\prime b^\prime c^\prime})$ called the \textit{Poisson matrix} associated to the quiver.  Here, $N = 3(n-1) + (n-1)(n-2)/2$; see \S \ref{ssec:dotted-ideal-triangulations}.  

\begin{definition}
The \textit{Fock--Goncharov quantum torus} 
	$\mathscr{T}_n^\omega(\mathfrak{T})$,  also denoted $\mathbb{C}[X_1^{\pm 1/n}, \dots, X_{N}^{\pm 1/n}]^\omega$,
associated to the triangle $\mathfrak{T}$ is defined to be the quantum torus $\mathscr{T}^\omega(\vec{P})$ determined by the $N \times N$ Poisson matrix $\vec{P}$, with generators $X_i^{\pm 1/n} = X_{abc}^{\pm 1/n}$ for all $(a, b, c) \in \Theta_n - \Gamma(\Theta_n)$.  Note that when $q=\omega=1$ this  recovers the classical polynomial algebra $\mathscr{T}_n^1(\mathfrak{T})$ for  $\mathfrak{T}$; see \S \ref{sec:Fock-Goncharov-coordinates-on-a-triangulated-surface}.  

As a notational convention, for $j = 1, 2, \dots, n-1$ we write $Z_{j}^{\pm 1/n}$ (resp. $Z_{j}^{\prime \pm 1/n}$ and $Z_j^{\prime\prime \pm 1/n}$) in place of $X_{j0(n-j)}^{\pm 1/n}$ (resp. $X_{j(n-j)0}^{\pm 1/n}$ and $X_{0j(n-j)}^{\pm 1/n}$); see Figure   \ref{fig:left-and-right-quantum-matrices}.  So, triangle-coordinates will be denoted $X_i = X_{abc}$ for $(a, b, c) \in \mathrm{Int}(\Theta_n)$ while edge-coordinates will be denoted $Z_j, Z^\prime_j, Z^{\prime\prime}_j$.  
\end{definition}

\begin{remark}
\label{rem:half-of-an-edge-coordinate}
Intuitively speaking, we think of the $Z$-coordinates as quantizations of `half' of their corresponding classical edge-coordinates.  This is because the `other half' of each coordinate lives in an adjacent triangle.  When viewed inside an ideal triangulation $\lambda$, an edge $E$ of $\lambda$ `splits' this classical edge-coordinate into its two `quantum halves'.    Compare Figure \ref{fig:edge-coordinates-as-tensor-products}.
\end{remark} 

\subsection{Quantum left and right matrices}
\label{ssec:quantum-left-and-right-matrices}

Although the  general (global) definition of the $\mathrm{SL}_n$-quantum trace polynomials (\S \ref{sec:def-of-quantum-trace}) is somewhat technical (requiring that one  keep track of the ordering of the non-commuting quantum torus variables), the extension of the local  monodromy matrices (\S \ref{sec:local-monodromy-matrices}) to the quantum setting is more straightforward, just  using the Weyl quantum ordering (\S \ref{sssec:weyl-ordering}) to symmetrize the variables.  

\subsubsection{Weyl quantum ordering for the Fock--Goncharov quantum torus}
\label{sssec:weyl-ordering-continued}
	
Let $\mathscr{T} =  \mathscr{T}_n^\omega(\mathfrak{T})$ be the Fock--Goncharov quantum torus (\S $\ref{sec:Fock-Goncharov-algebra-for-a-triangle})$.  Then the Weyl ordering $[-]$ of \S \ref{sssec:weyl-ordering} gives a map
\begin{equation*}
	[-] : 
	\mathrm{M}_n(\mathscr{T}_n^1(\mathfrak{T})) 
	 \to \mathrm{M}_n(\mathscr{T}_n^\omega(\mathfrak{T})),
\end{equation*}
where we have used the identification $\mathscr{T}_n^1(\mathfrak{T}) \cong \mathbb{C}[X_1^{\pm 1/n}, X_2^{\pm 1/n}, \dots, X_{N}^{\pm 1/n}]$ discussed in \S \ref{sec:Fock-Goncharov-algebra-for-a-triangle}.  

\begin{figure}[htb]
	\centering
	\includegraphics[scale=.78]{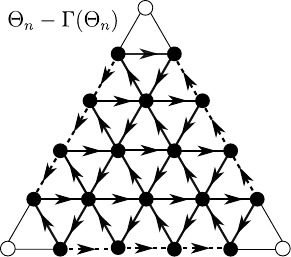}
	\caption{Quiver defining the Fock--Goncharov quantum torus $(n=5)$.}
	\label{fig:Fg-quiver}
\end{figure}	

\subsubsection{Quantum left and right matrices}
\label{sssec:quantum-left-and-right-matrices}

Let $\mathfrak{T}$ be a triangle.  An \textit{extended left-moving arc} $\overline{\gamma}$ is similar to a left-moving arc, from \S \ref{sec:local-monodromy-matrices}, except that it extends all the way to the two distinct edges of the triangle $\mathfrak{T}$; see Figure \ref{fig:left-and-right-quantum-matrices}.     We think of an extended left-moving arc $\overline{\gamma}$ as the concatenation of `half' of an edge-crossing arc $\overline{\gamma}^{1/2}_1$ together with a left-moving arc $\overline{\gamma}_2$ together with another half of an edge-crossing arc $\overline{\gamma}^{1/2}_3$, as indicated in Figure \ref{fig:left-and-right-quantum-matrices}; compare Remark \ref{rem:half-of-an-edge-coordinate}.  We refer to these halves of edge-crossing arcs as \textit{half-edge-crossing arcs}.  Similarly, we define \textit{extended right-moving arcs} $\overline{\gamma}$.  

Defined  as in \S \ref{sec:local-monodromy-matrices} are left matrices $\vec{M}^\mathrm{left}(\vec{X})$,  right matrices $\vec{M}^\mathrm{right}(\vec{X})$, and edge matrices $\vec{M}^\mathrm{edge}(\vec{Z})$ in $ \mathrm{SL}_n(\mathscr{T}_n^1(\mathfrak{T}))$ associated to non-extended left-moving arcs (Figure \ref{fig:left-matrix}), non-extended right-moving arcs (Figure \ref{fig:right-matrix}), and half-edge-crossing arcs, respectively.  

\begin{definition}
\label{def:left-and-right-quantum-matrices}
Put vectors $\vec{X}=(X_i)$, $\vec{Z}=(Z_j)$, $\vec{Z}^\prime=(Z^\prime_j)$, and $\vec{Z}^{\prime\prime}=(Z^{\prime\prime}_j)$ as in Figure \ref{fig:left-and-right-quantum-matrices}.
To an extended left-moving arc $\overline{\gamma}$, as in Figure \ref{fig:left-and-right-quantum-matrices},  we associate a \textit{quantum left matrix} $\vec{L}^\omega$ in $\mathrm{M}_n(\mathscr{T}_n^\omega(\mathfrak{T}))$ by the formula
\begin{equation*}
\label{eq:quantum-left-matrix}
	 \vec{L}^\omega 
	 \overset{\text{def}}{=} \vec{L}^\omega(\vec{Z}, \vec{X},  \vec{Z}^\prime)
	\overset{\text{def}}{=} [  \vec{M}^\mathrm{edge}(\vec{Z}) \vec{M}^\mathrm{left}(\vec{X}) \vec{M}^\mathrm{edge}(\vec{Z}^\prime) ]
	  \in  \mathrm{M}_n(\mathscr{T}_n^\omega(\mathfrak{T})),
\end{equation*}
where we have applied the Weyl quantum ordering $[-]$ discussed in \S \ref{sssec:weyl-ordering-continued} to the product  $\vec{M}^\mathrm{edge}(\vec{Z}) \vec{M}^\mathrm{left}(\vec{X}) \vec{M}^\mathrm{edge}(\vec{Z}^\prime)$ of classical matrices in $\mathrm{M}_n(\mathscr{T}_n^1(\mathfrak{T})) $  (actually, in $\mathrm{SL}_n(\mathscr{T}_n^1(\mathfrak{T}))$).  This just means that we apply the Weyl ordering to each entry of the classical matrix.
Similarly, to an extended right-moving arc $\overline{\gamma}$, as in Figure \ref{fig:left-and-right-quantum-matrices}, we associate a \textit{quantum right matrix} $\vec{R}^\omega$ in $\mathrm{M}_n(\mathscr{T}_n^\omega(\mathfrak{T}))$ by the formula
\begin{equation*}
\label{eq:quantum-right-matrix}
	 \vec{R}^\omega 
	 \overset{\text{def}}{=} \vec{R}^\omega(\vec{Z}, \vec{X}, \vec{Z}^{\prime\prime})
	\overset{\text{def}}{=}  [  \vec{M}^\mathrm{edge}(\vec{Z}) \vec{M}^\mathrm{right}(\vec{X}) \vec{M}^\mathrm{edge}(\vec{Z}^{\prime\prime}) ]
	  \in  \mathrm{M}_n(\mathscr{T}_n^\omega(\mathfrak{T})).  
\end{equation*}
\end{definition}

\begin{figure}[htb]
	\centering
	\includegraphics[scale=.5]{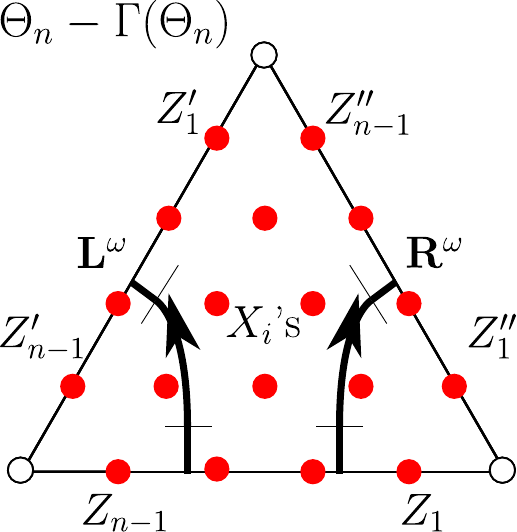}
	\caption{Quantum left and right matrices $(n=5)$.}
	\label{fig:left-and-right-quantum-matrices}
\end{figure}

\subsection{Quantum \texorpdfstring{$\mathrm{SL}_n$}{SLn} and its points: first result}
\label{sec:quantum-SLn}
	
(For a more theoretical discussion about quantum $\mathrm{SL}_n$, see Appendix \ref{sec:proof-of-reshetikhin-turaev-invariant}.)  Let $\mathscr{T}$ be a, possibly non-commutative, algebra.  

\begin{definition}
A $2 \times 2$ matrix $\vec{M} = \left(\begin{smallmatrix} a & b \\ c & d \end{smallmatrix}\right)$ in $\mathrm{M}_2(\mathscr{T})$ is a \textit{$\mathscr{T}$-point of the quantum matrix algebra $\mathrm{M}_2^q$}, denoted $\vec{M} \in \mathrm{M}_2^q(\mathscr{T}) \subset \mathrm{M}_2(\mathscr{T})$, if 
\begin{equation*}
\tag{$\ast$}
\label{eq:quantum-group-2by2-relations}
	ba = q ab, 
	\,\,  dc = qcd, 
	\,\,  ca = qac, 
	\,\,  db = qbd, 
	\,\,  bc = cb,
	\,\,  da - ad = (q - q^{-1}) bc
	  \in  \mathscr{T}.  
\end{equation*}
A matrix $\vec{M} \in \mathrm{M}_2(\mathscr{T})$ is a \textit{$\mathscr{T}$-point of the quantum special linear group $\mathrm{SL}_2^q$}, denoted $\vec{M} \in \mathrm{SL}_2^q(\mathscr{T}) \subset \mathrm{M}_2^q(\mathscr{T}) \subset \mathrm{M}_2(\mathscr{T})$, if 
	$\vec{M} \in \mathrm{M}_2^q(\mathscr{T})$
and the \textit{quantum determinant}
\begin{equation*}
	\mathrm{Det}^q(\vec{M})   \overset{\text{def}}{=} ad - q^{-1}bc  = 1
	  \in \mathscr{T}.
\end{equation*}
\end{definition}

These notions are also defined for $n \times n$ matrices, as follows.  

\begin{definition}
A matrix $\vec{M} \in \mathrm{M}_n(\mathscr{T})$ is a \textit{$\mathscr{T}$-point of the quantum matrix algebra $\mathrm{M}_n^q$}, denoted $\vec{M} \in \mathrm{M}_n^q(\mathscr{T}) \subset \mathrm{M}_n(\mathscr{T})$, if every $2 \times 2$ submatrix $\left( \begin{smallmatrix} \vec{M}_{ik}&\vec{M}_{im}\\\vec{M}_{jk}&\vec{M}_{jm} \end{smallmatrix} \right)$ of $\vec{M}$ is a $\mathscr{T}$-point of $\mathrm{M}_2^q$.  That is,
\begin{gather*}		\label{defining-relations-of-quantum-coord-ring}
	\textbf{M}_{im} \textbf{M}_{ik} = q \textbf{M}_{ik} \textbf{M}_{im}, \,\,
	\vec{M}_{jm} \vec{M}_{jk}=q\vec{M}_{jk} \vec{M}_{jm},	\,\,
	\vec{M}_{jk}\vec{M}_{ik} = q \vec{M}_{ik} \vec{M}_{jk},	
\,\,  	\textbf{M}_{jm} \textbf{M}_{im} = q \textbf{M}_{im} \textbf{M}_{jm},  	\\	
	\textbf{M}_{im} \textbf{M}_{jk} = \textbf{M}_{jk} \textbf{M}_{im}, \,\,
	\textbf{M}_{jm} \textbf{M}_{ik} - \textbf{M}_{ik} \textbf{M}_{jm} = (q - q^{-1}) \textbf{M}_{im} \textbf{M}_{jk},     
\end{gather*}	
for all $i < j$ and $k < m$, where $1 \leq i, j, k, m \leq n$.  
A matrix $\vec{M} \in \mathrm{M}_n(\mathscr{T})$ is a \textit{$\mathscr{T}$-point of the quantum special linear group $\mathrm{SL}_n^q$}, denoted $\vec{M} \in \mathrm{SL}_n^q(\mathscr{T}) \subset \mathrm{M}_n^q(\mathscr{T}) \subset \mathrm{M}_n(\mathscr{T})$, if both $\vec{M} \in \mathrm{M}_n^q(\mathscr{T})$ and  $\mathrm{Det}^q(\vec{M}) = 1$.  
Here, the \textit{quantum determinant} $\mathrm{Det}^q(\vec{M}) \in \mathscr{T}$ of a matrix $\vec{M} \in \mathrm{M}_n(\mathscr{T})$ is
\begin{equation*}
	\mathrm{Det}^q(\vec{M})\overset{\text{def}}{=}\sum_{\sigma\in S_n} (-q^{-1})^{\ell(\sigma)} \vec{M}_{1\sigma(1)} \vec{M}_{2\sigma(2)} \cdots \vec{M}_{n\sigma(n)}\in\mathscr{T},     
\end{equation*}
where the length $\ell(\sigma)$ of the permutation $\sigma\in S_n$ is the minimum number of factors appearing in a decomposition of $\sigma$ as a product of adjacent transpositions $(i, i+1)$; see, for example, \cite[Chapter I.2]{Brown02}.
\end{definition}

\begin{remark}
Note that the definitions satisfy the property that if a $\mathscr{T}$-point $\vec{M} \in \mathrm{M}_n^q(\mathscr{T}) \subset \mathrm{M}_n(\mathscr{T})$ is a triangular matrix, then the diagonal entries $\vec{M}_{ii} \in \mathscr{T}$ commute, and $\mathrm{Det}^q(\vec{M}) = \prod_i \vec{M}_{ii} \in \mathscr{T}$.

	 Note also that the subsets $\mathrm{M}_n^q(\mathscr{T}) \subset \mathrm{M}_n(\mathscr{T})$ and $\mathrm{SL}_n^q(\mathscr{T}) \subset \mathrm{M}_n(\mathscr{T})$ are generally not closed under matrix multiplication.  
\end{remark}

Take $\mathscr{T} = \mathscr{T}_n^\omega(\mathfrak{T})$ to be the Fock--Goncharov quantum torus for the triangle $\mathfrak{T}$, as defined in \S \ref{sec:Fock-Goncharov-algebra-for-a-triangle}.  Let $\vec{L}^\omega$ and $\vec{R}^\omega$ in $\mathrm{M}_n(\mathscr{T}_n^\omega(\mathfrak{T}))$ be the quantum left and right matrices, respectively, as defined in Definition \ref{def:left-and-right-quantum-matrices}.  In a companion paper, we prove:

\begin{theorem}[\cite{DouglasA}, $\mathrm{SL}_n$-quantum matrices]
\label{thm:first-theorem}
	The quantum left and right  matrices,
\begin{equation*}
		\vec{L}^\omega 
	 = \vec{L}^\omega(\vec{Z}, \vec{X},  \vec{Z}^\prime) \,\, \text{and } \,\,  \vec{R}^\omega 
	 = \vec{R}^\omega(\vec{Z}, \vec{X}, \vec{Z}^{\prime\prime})
		  \in 
  \mathrm{M}_n(\mathscr{T}_n^\omega(\mathfrak{T})),
\end{equation*}
	 are $\mathscr{T}_n^\omega(\mathfrak{T})$-points of the quantum special linear group $\mathrm{SL}_n^q$.  That is, $\vec{L}^\omega, \vec{R}^\omega \in \mathrm{SL}_n^q(\mathscr{T}_n^\omega(\mathfrak{T})) 
		  \subset  \mathrm{M}_n(\mathscr{T}_n^\omega(\mathfrak{T}))$.  \qed
\end{theorem}	

 The proof uses a quantum version of  Fock--Goncharov snakes; see Remark \ref{rem:classical-coordinates-snakes}.  See also Remark \ref{rem:chekhov-shapiro} for recent related work.  

\subsection{Examples}
\label{sec:concrete-formulas}

\subsubsection{$\mathrm{SL}_3$ example.}
\label{sssec:n=3-example}

Consider the case $n=3$; see Figure \ref{fig:n=3-left-and-right-matrices-example}.  On the right hand side, we show the quiver defining the commutation relations in the quantum torus $\mathscr{T}_3^\omega(\mathfrak{T})$, recalling Figure \ref{fig:Fg-quiver} and the definitions of \S \ref{sssec:quantum-tori} and \ref{sec:Fock-Goncharov-algebra-for-a-triangle}.  For instance, the following are some sample commutation relations:
\begin{equation*}
	X Z^\prime = q^2 Z^\prime X,
	\,\,
	X W^\prime = q^{-2} W^\prime X,
	\,\,
	Z W = q W Z,
	\,\,
	Z W^\prime = q^2 W^\prime Z.
\end{equation*}
Then, the quantum left and right matrices are computed as
\begin{gather*}
	\vec{L}^\omega = 
	\left[
W^{-\frac{1}{3}} Z^{-\frac{2}{3}} 
\left(\begin{smallmatrix}
WZ&&\\
&Z&\\
&&1
\end{smallmatrix}\right)
\left(\begin{smallmatrix}
1&1&\\
&1&\\
&&1
\end{smallmatrix}\right)
X^{-\frac{1}{3}}
\left(\begin{smallmatrix}
X&&\\
&1&1\\
&&1
\end{smallmatrix}\right)
\left(\begin{smallmatrix}
1&1&\\
&1&\\
&&1
\end{smallmatrix}\right)
Z^{\prime -\frac{1}{3}} W^{\prime -\frac{2}{3}}	
\left(\begin{smallmatrix}
Z^\prime W^\prime&&\\
&W^\prime&\\
&&1
\end{smallmatrix}\right)
	\right]
\\= 
	\left(\begin{smallmatrix}
\left[W^\frac{2}{3} Z^\frac{1}{3} X^\frac{2}{3} Z^{\prime \frac{2}{3}} W^{\prime \frac{1}{3}}\right]
&\left[W^\frac{2}{3} Z^\frac{1}{3} X^\frac{2}{3} Z^{\prime -\frac{1}{3}} W^{\prime \frac{1}{3}}\right]+\left[ W^\frac{2}{3} Z^\frac{1}{3} X^{-\frac{1}{3}} Z^{\prime -\frac{1}{3}} W^{\prime \frac{1}{3}}\right]
&\left[W^\frac{2}{3} Z^\frac{1}{3} X^{-\frac{1}{3}} Z^{\prime -\frac{1}{3}} W^{\prime -\frac{2}{3}}\right]
\\0
&\left[W^{-\frac{1}{3}} Z^\frac{1}{3} X^{-\frac{1}{3}} Z^{\prime -\frac{1}{3}} W^{\prime \frac{1}{3}}\right]
&\left[W^{-\frac{1}{3}} Z^\frac{1}{3} X^{-\frac{1}{3}} Z^{\prime -\frac{1}{3}} W^{\prime -\frac{2}{3}}\right]
\\0
&0
&\left[W^{-\frac{1}{3}} Z^{-\frac{2}{3}} X^{-\frac{1}{3}} Z^{\prime -\frac{1}{3}} W^{\prime -\frac{2}{3}}\right]
	\end{smallmatrix}\right),
\end{gather*}
and 
\begin{gather*}
\label{eq:example-4by4-right-matrixSL3}
	\vec{R}^\omega = 
	\left[
W^{\prime -\frac{1}{3}} Z^{\prime -\frac{2}{3}} 
\left(\begin{smallmatrix}
W^\prime Z^\prime&&\\
&Z^\prime&\\
&&1
\end{smallmatrix}\right)
\left(\begin{smallmatrix}
1&&\\
&1&\\
&1&1
\end{smallmatrix}\right)
X^{+\frac{1}{3}}
\left(\begin{smallmatrix}
1&&\\
1&1&\\
&&X^{-1}
\end{smallmatrix}\right)
\left(\begin{smallmatrix}
1&&\\
&1&\\
&1&1
\end{smallmatrix}\right)
Z^{ -\frac{1}{3}} W^{ -\frac{2}{3}}	
\left(\begin{smallmatrix}
Z W&&\\
&W&\\
&&1
\end{smallmatrix}\right)
	\right]
\\
=
	\left(\begin{smallmatrix}
\left[ W^{\prime \frac{2}{3}} Z^{\prime \frac{1}{3}} X^\frac{1}{3} Z^\frac{2}{3} W^\frac{1}{3}\right]
&0
&0
\\ \left[ W^{\prime -\frac{1}{3}} Z^{\prime \frac{1}{3}} X^\frac{1}{3} Z^\frac{2}{3} W^\frac{1}{3}\right]
&\left[ W^{\prime -\frac{1}{3}} Z^{\prime \frac{1}{3}} X^\frac{1}{3} Z^{-\frac{1}{3}} W^\frac{1}{3}\right]
&0
\\ \left[ W^{\prime -\frac{1}{3}} Z^{\prime -\frac{2}{3}} X^\frac{1}{3} Z^\frac{2}{3} W^\frac{1}{3}\right]
&\left[ W^{\prime -\frac{1}{3}} Z^{\prime -\frac{2}{3}} X^\frac{1}{3} Z^{-\frac{1}{3}} W^\frac{1}{3}\right] + \left[ W^{\prime -\frac{1}{3}} Z^{\prime -\frac{2}{3}} X^{-\frac{2}{3}} Z^{-\frac{1}{3}} W^\frac{1}{3}\right]
&\left[ W^{\prime -\frac{1}{3}} Z^{\prime -\frac{2}{3}} X^{-\frac{2}{3}} Z^{-\frac{1}{3}} W^{-\frac{2}{3}} \right]
\end{smallmatrix}\right).  
\end{gather*}
Theorem \ref{thm:first-theorem} says that these two matrices are elements of $\mathrm{SL}_3^q(\mathscr{T}_3^\omega(\mathfrak{T}))$.  For instance, the entries $a,b,c,d$ of the $2 \times 2$ sub-matrix of $\vec{L}^\omega$,
\begin{equation*}
\label{eq:quantum-left-matrix-2x2-sub-matrixSL3}
\begin{split}
	\left(\begin{smallmatrix}
		a&b\\
		c&d
	\end{smallmatrix}\right)
	=
	\left(\begin{smallmatrix}
		\vec{L}^\omega_{12}&\vec{L}^\omega_{13}\\
		\vec{L}^\omega_{22}&\vec{L}^\omega_{23}
	\end{smallmatrix}\right)
	=
	\left(\begin{smallmatrix}
\left[W^\frac{2}{3} Z^\frac{1}{3} X^\frac{2}{3} Z^{\prime -\frac{1}{3}} W^{\prime \frac{1}{3}}\right]+\left[ W^\frac{2}{3} Z^\frac{1}{3} X^{-\frac{1}{3}} Z^{\prime -\frac{1}{3}} W^{\prime \frac{1}{3}}\right]
&
\left[W^\frac{2}{3} Z^\frac{1}{3} X^{-\frac{1}{3}} Z^{\prime -\frac{1}{3}} W^{\prime -\frac{2}{3}}\right]
\\
\left[W^{-\frac{1}{3}} Z^\frac{1}{3} X^{-\frac{1}{3}} Z^{\prime -\frac{1}{3}} W^{\prime \frac{1}{3}}\right]
&
	\left[W^{-\frac{1}{3}} Z^\frac{1}{3} X^{-\frac{1}{3}} Z^{\prime -\frac{1}{3}} W^{\prime -\frac{2}{3}}\right]
	\end{smallmatrix}\right),
\end{split}
\end{equation*}
satisfy Equation \eqref{eq:quantum-group-2by2-relations}.  For a computer verification of this, see Appendix \ref{sec:the-appendix}. We also demonstrate in the appendix that Equation \eqref{eq:quantum-group-2by2-relations} is satisfied by the entries $a,b,c,d$ of the $2 \times 2$ sub-matrix  of $\vec{R}^\omega$,
\begin{equation*}
\label{eq:quantum-right-matrix-2x2-sub-matrixSL3}
\begin{split}
	\left(\begin{smallmatrix}
		a&b\\
		c&d
	\end{smallmatrix}\right)
	\!=\!
	\left(\begin{smallmatrix}
		\vec{R}^\omega_{21}&\vec{R}^\omega_{22}\\
		\vec{R}^\omega_{31}&\vec{R}^\omega_{32}
	\end{smallmatrix}\right)
	\!=\!
	\left(\begin{smallmatrix}
\left[ W^{\prime -\frac{1}{3}} Z^{\prime \frac{1}{3}} X^\frac{1}{3} Z^\frac{2}{3} W^\frac{1}{3}\right]
&
\left[ W^{\prime -\frac{1}{3}} Z^{\prime \frac{1}{3}} X^\frac{1}{3} Z^{-\frac{1}{3}} W^\frac{1}{3}\right]
\\
\left[ W^{\prime -\frac{1}{3}} Z^{\prime -\frac{2}{3}} X^\frac{1}{3} Z^\frac{2}{3} W^\frac{1}{3}\right]
&
\left[ W^{\prime -\frac{1}{3}} Z^{\prime -\frac{2}{3}} X^\frac{1}{3} Z^{-\frac{1}{3}} W^\frac{1}{3}\right] + \left[ W^{\prime -\frac{1}{3}} Z^{\prime -\frac{2}{3}} X^{-\frac{2}{3}} Z^{-\frac{1}{3}} W^\frac{1}{3}\right]
	\end{smallmatrix}\right).
\end{split}
\end{equation*}

\begin{figure}[htb]
     \centering
     \begin{subfigure}[b]{0.465\textwidth}
         \centering
         \includegraphics[width=.49\textwidth]{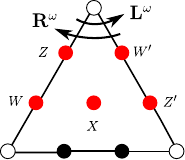}
         \caption{Quantum left and right matrices}
         \label{fig:n=3-left-and-right-matrices-example-subfigure}
     \end{subfigure}     
\hfill
     \begin{subfigure}[b]{0.465\textwidth}
         \centering
         \includegraphics[width=.49\textwidth]{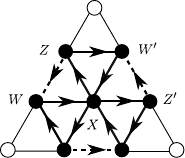}
         \caption{Fock--Goncharov quiver}
         \label{fig:n=3-quiver}
     \end{subfigure}
     	\caption{Example in the case $n=3$.}
        \label{fig:n=3-left-and-right-matrices-example}
\end{figure}

\subsubsection{$\mathrm{SL}_4$ example.}
\label{sssec:n=4-example}

Consider the case $n=4$; see Figure \ref{fig:n=4-left-and-right-matrices-example}.  On the right hand side, we show the quiver defining the commutation relations in the quantum torus $\mathscr{T}_4^\omega(\mathfrak{T})$, recalling Figure \ref{fig:Fg-quiver} and the definitions of \S \ref{sssec:quantum-tori} and \ref{sec:Fock-Goncharov-algebra-for-a-triangle}.  For instance, the following are some sample commutation relations:
\begin{equation*}
	X_3 Z^{\prime\prime}_2 = q^2 X_3 Z^{\prime\prime}_2,
	\,\,
	X_3 X_1 = q^{-2} X_1 X_3,
	\,\,
	Z_3 Z_2 = q Z_2 Z_3,
	\,\,
	Z_3 Z^\prime_3 = q^2 Z^\prime_3 Z_3.
\end{equation*}
Then, the quantum left and right matrices are computed as
\begin{equation*}
\begin{split}
	\vec{L}^\omega = 
	&\Bigg[
Z_1^{-\frac{1}{4}} Z_2^{-\frac{2}{4}} Z_3^{-\frac{3}{4}}	
\left(\begin{smallmatrix}
Z_1 Z_2 Z_3&&&\\
&Z_2 Z_3&&\\
&&Z_3&\\
&&&1
\end{smallmatrix}\right)
\left(\begin{smallmatrix}
1&1&&\\
&1&&\\
&&1&\\
&&&1
\end{smallmatrix}\right)
X_1^{-\frac{1}{4}}
\left(\begin{smallmatrix}
X_1&&&\\
&1&1&\\
&&1&\\
&&&1
\end{smallmatrix}\right)
X_2^{-\frac{2}{4}}
\left(\begin{smallmatrix}
X_2&&&\\
&X_2&&\\
&&1&1\\
&&&1
\end{smallmatrix}\right)
\\
&\left(\begin{smallmatrix}
1&1&&\\
&1&&\\
&&1&\\
&&&1
\end{smallmatrix}\right)
X_3^{-\frac{1}{4}}
\left(\begin{smallmatrix}
X_3&&&\\
&1&1&\\
&&1&\\
&&&1
\end{smallmatrix}\right)
\left(\begin{smallmatrix}
1&1&&\\
&1&&\\
&&1&\\
&&&1
\end{smallmatrix}\right)
Z_1^{\prime-\frac{1}{4}} Z_2^{\prime-\frac{2}{4}} Z_3^{\prime-\frac{3}{4}}	
\left(\begin{smallmatrix}
Z^\prime_1 Z^\prime_2 Z^\prime_3&&&\\
&Z^\prime_2 Z^\prime_3&&\\
&&Z^\prime_3&\\
&&&1
\end{smallmatrix}\right)
	\Bigg],
\end{split}
\end{equation*}
and
\begin{equation*}
\label{eq:example-4by4-right-matrix}
\begin{split}
	\vec{R}^\omega = 
	&\Bigg[
Z_1^{-\frac{1}{4}} Z_2^{-\frac{2}{4}} Z_3^{-\frac{3}{4}}	
\left(\begin{smallmatrix}
Z_1 Z_2 Z_3&&&\\
&Z_2 Z_3&&\\
&&Z_3&\\
&&&1
\end{smallmatrix}\right)
\left(\begin{smallmatrix}
1&&&\\
&1&&\\
&&1&\\
&&1&1
\end{smallmatrix}\right)
X_2^{+\frac{1}{4}}
\left(\begin{smallmatrix}
1&&&\\
&1&&\\
&1&1&\\
&&&X_2^{-1}
\end{smallmatrix}\right)
X_1^{+\frac{2}{4}}
\left(\begin{smallmatrix}
1&&&\\
1&1&&\\
&&X_1^{-1}&\\
&&&X_1^{-1}
\end{smallmatrix}\right)
\\
&\left(\begin{smallmatrix}
1&&&\\
&1&&\\
&&1&\\
&&1&1
\end{smallmatrix}\right)
X_3^{+\frac{1}{4}}
\left(\begin{smallmatrix}
1&&&\\
&1&&\\
&1&1&\\
&&&X_3^{-1}
\end{smallmatrix}\right)
\left(\begin{smallmatrix}
1&&&\\
&1&&\\
&&1&\\
&&1&1
\end{smallmatrix}\right)
Z_1^{\prime\prime-\frac{1}{4}} Z_2^{\prime\prime-\frac{2}{4}} Z_3^{\prime\prime-\frac{3}{4}}	
\left(\begin{smallmatrix}
Z^{\prime\prime}_1 Z^{\prime\prime}_2 Z_3^{\prime\prime}&&&\\
&Z^{\prime\prime}_2 Z^{\prime\prime}_3&&\\
&&Z^{\prime\prime}_3&\\
&&&1
\end{smallmatrix}\right)
	\Bigg].
\end{split}
\end{equation*}
Theorem \ref{thm:first-theorem} says that these two matrices are elements of $\mathrm{SL}_4^q(\mathscr{T}_4^\omega(\mathfrak{T}))$.  For instance, the entries $a,b,c,d$ of the $2 \times 2$ sub-matrix (arranged as a $4 \times 1$ matrix) of $\vec{L}^\omega$,  $	\left(\begin{smallmatrix}
		a\\b\\
		c\\d
	\end{smallmatrix}\right)
	=
	\left(\begin{smallmatrix}
		\vec{L}^\omega_{13}\\\vec{L}^\omega_{14}\\
		\vec{L}^\omega_{23}\\\vec{L}^\omega_{24}
	\end{smallmatrix}\right)=$
\begin{equation*}
\label{eq:quantum-left-matrix-2x2-sub-matrix}
\begin{split}
	\left(\begin{smallmatrix}
		[Z_3^\frac{1}{4} Z_2^\frac{2}{4} Z_1^\frac{3}{4} 
		Z_3^{\prime\frac{1}{4}} Z_2^{\prime-\frac{2}{4}} Z_1^{\prime-\frac{1}{4}} 
		X_1^{-\frac{1}{4}} X_2^{-\frac{2}{4}} X_3^{-\frac{1}{4}}]
	+	
		[Z_3^\frac{1}{4} Z_2^\frac{2}{4} Z_1^\frac{3}{4} 
		Z_3^{\prime\frac{1}{4}} Z_2^{\prime-\frac{2}{4}} Z_1^{\prime-\frac{1}{4}} 
		X_1^{-\frac{1}{4}} X_2^\frac{2}{4} X_3^{-\frac{1}{4}}]
	+	
		[Z_3^\frac{1}{4} Z_2^\frac{2}{4} Z_1^\frac{3}{4} 
		Z_3^{\prime\frac{1}{4}} Z_2^{\prime-\frac{2}{4}} Z_1^{\prime-\frac{1}{4}} 
		X_1^\frac{3}{4} X_2^\frac{2}{4} X_3^{-\frac{1}{4}}]
\\
		[Z_3^\frac{1}{4} Z_2^\frac{2}{4} Z_1^\frac{3}{4} 
		Z_3^{\prime-\frac{3}{4}} Z_2^{\prime-\frac{2}{4}} Z_1^{\prime-\frac{1}{4}} 
		X_1^{-\frac{1}{4}} X_2^{-\frac{2}{4}} X_3^{-\frac{1}{4}}]
\\	
		[Z_3^\frac{1}{4} Z_2^\frac{2}{4} Z_1^{-\frac{1}{4}} 
		Z_3^{\prime\frac{1}{4}} Z_2^{\prime-\frac{2}{4}} Z_1^{\prime-\frac{1}{4}} 
		X_1^{-\frac{1}{4}} X_2^{-\frac{2}{4}} X_3^{-\frac{1}{4}}]	
	+	
		[Z_3^\frac{1}{4} Z_2^\frac{2}{4} Z_1^{-\frac{1}{4}} 
		Z_3^{\prime\frac{1}{4}} Z_2^{\prime-\frac{2}{4}} Z_1^{\prime-\frac{1}{4}} 
		X_1^{-\frac{1}{4}} X_2^\frac{2}{4} X_3^{-\frac{1}{4}}]
\\
		[Z_3^\frac{1}{4} Z_2^\frac{2}{4} Z_1^{-\frac{1}{4}} 
		Z_3^{\prime-\frac{3}{4}} Z_2^{\prime-\frac{2}{4}} Z_1^{\prime-\frac{1}{4}} 
		X_1^{-\frac{1}{4}} X_2^{-\frac{2}{4}} X_3^{-\frac{1}{4}}]	
	\end{smallmatrix}\right),
\end{split}
\end{equation*}
satisfy Equation \eqref{eq:quantum-group-2by2-relations}.  For a computer verification of this, see Appendix \ref{sec:the-appendix}.  We also demonstrate in the appendix that Equation \eqref{eq:quantum-group-2by2-relations} is satisfied by the entries $a,b,c,d$ of the $2 \times 2$ sub-matrix (arranged as a $4 \times 1$ matrix) of $\vec{R}^\omega$, 
$  	\left(\begin{smallmatrix}
		a\\b\\
		c\\d
	\end{smallmatrix}\right)
	\!=\!
	\left(\begin{smallmatrix}
		\vec{R}^\omega_{31}\\\vec{R}^\omega_{32}\\
		\vec{R}^\omega_{41}\\\vec{R}^\omega_{42}
	\end{smallmatrix}\right)
	\!=\!$
\begin{equation*}
\label{eq:quantum-right-matrix-2x2-sub-matrix}
\hspace{-1in}
\begin{split}
	\left(\begin{smallmatrix}
		[Z_3^\frac{1}{4} Z_2^{-\frac{1}{2}}  Z_1^{-\frac{1}{4}}  X_2^\frac{1}{4}  X_1^{\frac{1}{2}}  X_3^{\frac{1}{4}}  Z_3^{\prime\prime \frac{1}{4}} 
		Z_2^{\prime\prime \frac{1}{2}}   Z_1^{\prime\prime \frac{3}{4}}
			]
	\\
		[Z_3^\frac{1}{4} Z_2^{-\frac{1}{2}}  Z_1^{-\frac{1}{4}}  X_2^\frac{1}{4}  X_1^{-\frac{1}{2}}  X_3^{\frac{1}{4}}  Z_3^{\prime\prime \frac{1}{4}}  
		Z_2^{\prime\prime \frac{1}{2}}  Z_1^{\prime\prime -\frac{1}{4}}
			]	
	+	
		[Z_3^\frac{1}{4} Z_2^{-\frac{1}{2}}  Z_1^{-\frac{1}{4}}  X_2^\frac{1}{4}  X_1^{\frac{1}{2}}  X_3^{\frac{1}{4}}  Z_3^{\prime\prime \frac{1}{4}}  
		Z_2^{\prime\prime \frac{1}{2}}  Z_1^{\prime\prime -\frac{1}{4}}
			]
	\\
		[Z_3^{-\frac{3}{4}} Z_2^{-\frac{1}{2}}  Z_1^{-\frac{1}{4}}  X_2^\frac{1}{4}  X_1^{\frac{1}{2}}  X_3^{\frac{1}{4}}  Z_3^{\prime\prime \frac{1}{4}}  
		Z_2^{\prime\prime \frac{1}{2}}  Z_1^{\prime\prime \frac{3}{4}}
			]
	\\		
			[Z_3^{-\frac{3}{4}} Z_2^{-\frac{1}{2}}  Z_1^{-\frac{1}{4}}  X_2^{-\frac{3}{4}}  X_1^{-\frac{1}{2}}  X_3^{\frac{1}{4}}  Z_3^{\prime\prime \frac{1}{4}}  
		Z_2^{\prime\prime \frac{1}{2}}  Z_1^{\prime\prime -\frac{1}{4}}
				]
		+	
			[Z_3^{-\frac{3}{4}} Z_2^{-\frac{1}{2}}  Z_1^{-\frac{1}{4}}  X_2^\frac{1}{4}  X_1^{-\frac{1}{2}}  X_3^{\frac{1}{4}}  Z_3^{\prime\prime \frac{1}{4}}  
		Z_2^{\prime\prime \frac{1}{2}}  Z_1^{\prime\prime -\frac{1}{4}}
				]+
			[Z_3^{-\frac{3}{4}} Z_2^{-\frac{1}{2}}  Z_1^{-\frac{1}{4}}  X_2^\frac{1}{4}  X_1^{\frac{1}{2}}  X_3^{\frac{1}{4}}  Z_3^{\prime\prime \frac{1}{4}}  
		Z_2^{\prime\prime \frac{1}{2}}  Z_1^{\prime\prime -\frac{1}{4}}
				]
	\end{smallmatrix}\right).
\end{split}
\end{equation*}

\begin{figure}[htb]
     \centering
     \begin{subfigure}[b]{0.465\textwidth}
         \centering
         \includegraphics[width=.49\textwidth]{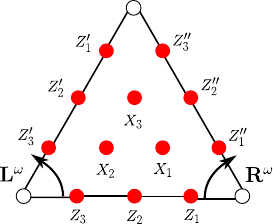}
         \caption{Quantum left and right matrices}
         \label{fig:n=4-left-and-right-matrices-example-subfigure}
     \end{subfigure}     
\hfill
     \begin{subfigure}[b]{0.49\textwidth}
         \centering
         \includegraphics[width=.44\textwidth]{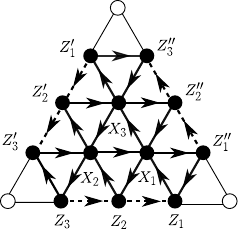}
         \caption{Fock--Goncharov quiver}
         \label{fig:n=4-quiver}
     \end{subfigure}
     	\caption{Example in the case $n=4$.}
        \label{fig:n=4-left-and-right-matrices-example}
\end{figure}

\subsection{Quantum tori for surfaces}
\label{ssec:quantum-tori-for-surfaces}
	
For  a dotted ideal triangulation $\lambda$ of $\mathfrak{S}$, in \S \ref{sec:Fock-Goncharov-coordinates-on-a-triangulated-surface} we defined the classical polynomial algebra $\mathscr{T}_n^1(\lambda) =
\mathbb{C}[X_1^{\pm 1/n}, X_2^{\pm 1/n}, \dots, X_N^{\pm 1/n}]$ where there is one generator $X_i^{\pm 1/n}$ associated to every dot on  $\lambda$.  In the case where $\mathfrak{S}=\mathfrak{T}$ is an ideal triangle, in \S \ref{sec:Fock-Goncharov-algebra-for-a-triangle} we deformed the classical polynomial algebra $\mathscr{T}_n^1(\mathfrak{T})$ to a quantum torus $\mathscr{T}_n^\omega(\mathfrak{T})$.  We now generalize the quantum torus $\mathscr{T}_n^\omega(\mathfrak{T})$ to a quantum torus  $\mathscr{T}_n^\omega(\lambda)$ associated to the triangulated  surface $(\mathfrak{S}, \lambda)$ which deforms the classical polynomial algebra $\mathscr{T}_n^1(\lambda)$.  

For each dotted triangle $\mathfrak{T}$ of $\lambda$, associate a copy $\widehat{\mathfrak{T}}$ of $\mathfrak{T}$, which is also a dotted triangle, such as that shown in Figure \ref{fig:dotted-triangle}.  Note that the boundary $\partial \widehat{\mathfrak{T}}$ consists of three ideal edges.  The dotted ideal triangulation $\lambda$ can be reconstructed from the individual triangles $\widehat{\mathfrak{T}}$ by supplying additional gluing data.  To each dotted triangle $\widehat{\mathfrak{T}}$ associate the Fock--Goncharov quantum torus $\mathscr{T}_n^\omega(\widehat{\mathfrak{T}})$ of the triangle $\widehat{\mathfrak{T}}$, whose coordinates we will denote by $\widehat{X}$.  
Recall that a generator $\widehat{X}_{abc}^{ \pm 1/n}$ of the quantum torus $\mathscr{T}_n^\omega(\widehat{\mathfrak{T}})$ is either a triangle-generator or an edge-generator.  If $\widehat{X}_{abc}^{ \pm 1/n}$ is an edge-generator, then there are two cases:
\begin{itemize}
	\item  the corresponding generator $X_i^{\pm 1/n}$ in the classical polynomial algebra $\mathscr{T}_n^1(\lambda)$ for the glued surface $(\mathfrak{S}, \lambda)$ is a boundary-generator;
	\item  the corresponding  generator $X_i^{\pm 1/n}$ in $\mathscr{T}_n^1(\lambda)$ is an interior-generator.  
\end{itemize}
In the second case, the corresponding interior-generator $X_i^{\pm 1/n}$ in $\mathscr{T}_n^1(\lambda)$ lies on an internal edge $E$ of the ideal triangulation $\lambda$.  So, there exists a triangle $\mathfrak{T}^\prime$ adjacent to $\mathfrak{T}$ along the edge $E$.  Moreover, there exists a unique edge-generator $\widehat{X}_{a^\prime b^\prime c^\prime}^{\prime \pm 1/n}$ in the quantum torus $\mathscr{T}_n^\omega(\widehat{\mathfrak{T}}^\prime)$ for the triangle $\widehat{\mathfrak{T}}^\prime$ that also corresponds to the  interior-generator $X_i^{\pm 1/n}$ in $\mathscr{T}_n^1(\lambda)$ lying on the internal edge $E$.  Therefore, we may say that the two quantum generators $\widehat{X}_{abc}^{ \pm 1/n}$ in $\mathscr{T}_n^\omega(\widehat{\mathfrak{T}})$ and $\widehat{X}_{a^\prime b^\prime c^\prime}^{\prime \pm 1/n}$ in  $\mathscr{T}_n^\omega(\widehat{\mathfrak{T}}^\prime)$ \textit{correspond to one another}; see Figure \ref{fig:edge-coordinates-as-tensor-products}.  
	
\begin{definition}
\label{def:FG-algebra-of-the-whole-surface}
	The \textit{Fock--Goncharov quantum torus} $\mathscr{T}_n^\omega(\lambda)$ associated to the surface $\mathfrak{S}$ equipped with the dotted ideal triangulation $\lambda$ is the sub-algebra,
	\begin{equation*}
		\mathscr{T}_n^\omega(\lambda)
		\subset
		\bigotimes_{\text{copies } \widehat{\mathfrak{T}} \text{ of triangles } \mathfrak{T} \text{ of }\lambda} \mathscr{T}_n^\omega(\widehat{\mathfrak{T}}),
	\end{equation*}
	of the tensor product of the Fock--Goncharov quantum tori $\mathscr{T}_n^\omega(\widehat{\mathfrak{T}})$ associated to the copies $\widehat{\mathfrak{T}}$ of the dotted triangles $\mathfrak{T}$ of the ideal triangulation $\lambda$, generated:
	\begin{itemize}
		\item  by triangle-generators $\widehat{X}_{abc}^{\pm 1/n}$ in $\mathscr{T}_n^\omega(\widehat{\mathfrak{T}})$; 
		\item  by tensor products $\widehat{X}_{abc}^{\pm 1/n} \otimes \widehat{X}_{a^\prime b^\prime c^\prime}^{\prime \pm 1/n}$ in $\mathscr{T}_n^\omega(\widehat{\mathfrak{T}}) \otimes \mathscr{T}_n^\omega(\widehat{\mathfrak{T}}^{\prime})$ of corresponding edge-generators associated to a common internal edge $E$ lying between two triangles $\mathfrak{T}$ and $\mathfrak{T}^{\prime}$ in $\lambda$; 
		\item  and by edge-generators $\widehat{X}_{abc}^{\pm 1/n}$ in $\mathscr{T}_n^\omega(\widehat{\mathfrak{T}})$ associated to boundary edges $E \subset \partial \mathfrak{S}$ of $\lambda$.
	\end{itemize}
	In particular, when $q=\omega=1$, the Fock--Goncharov quantum torus $\mathscr{T}_n^1(\lambda)$ is  naturally isomorphic to the classical polynomial algebra $\mathbb{C}[X_1^{\pm 1/n}, X_2^{\pm 1/n}, \dots, X_N^{\pm 1/n}]$ (as indicated by the notation).
(Going forward, we will omit the `hat' symbol in the notation, naturally identifying triangles $\mathfrak{T}$ with $\widehat{\mathfrak{T}}$.)   
\end{definition}

\begin{remark}
An important difference between the local quantum tori  $\mathscr{T}_n^\omega(\mathfrak{T})$ for the triangles $\mathfrak{T}$ and the global quantum torus $\mathscr{T}_n^\omega(\lambda)$ for the triangulated surface $(\mathfrak{S}, \lambda)$ is that two edge-generators $X_{abc}^{\pm 1/n}$ and $X_{ABC}^{\pm 1/n}$ in $\mathscr{T}_n^\omega(\mathfrak{T})$ lying on the same boundary edge of $\mathfrak{T}$ may not commute, rather may $q$-commute, while the corresponding interior-generators $X_{abc}^{\pm 1/n} \otimes X_{a^\prime b^\prime c^\prime}^{\prime \pm 1/n}$ and $X_{ABC}^{\pm 1/n} \otimes X_{A^\prime B^\prime C^\prime}^{\prime \pm 1/n}$ in $\mathscr{T}_n^\omega(\lambda)$ always commute.  This is because the orientations of the two triangles' $\mathfrak{T}$ and $\mathfrak{T}^\prime$ quivers  go against each other (Figures \ref{fig:Fg-quiver}, \ref{fig:edge-coordinates-as-tensor-products}).  Intuitively, the local $q$-commutation relations on the boundary are created upon `splitting the edge-coordinates in half' at the quantum level (Remarks \ref{rem:boundary-coordinates}, \ref{rem:half-of-an-edge-coordinate}).  This phenomenon does not occur for $\mathrm{SL}_2$ because there each edge carries only one coordinate.  
\end{remark}	

\begin{figure}[htb]
     \centering
     \begin{subfigure}{0.50\textwidth}
         \centering
         \includegraphics[width=.7\textwidth]{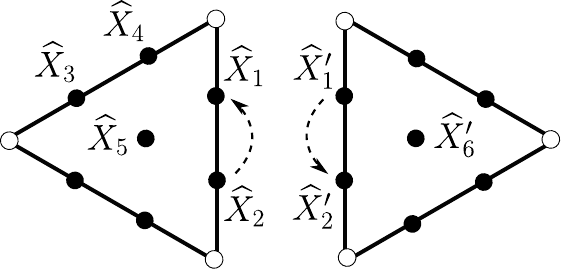}
         \caption{Before gluing}
         \label{fig:edge-coordinates-as-tensor-products-before-gluing}
     \end{subfigure}     
\hfill
     \begin{subfigure}{0.40\textwidth}
         \centering
         \includegraphics[width=.7\textwidth]{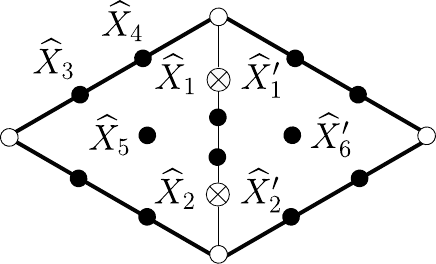}
         \caption{After gluing}
         \label{fig:edge-coordinates-as-tensor-products-after-gluing}
     \end{subfigure}
        \caption{Interior-generators as tensor products of local edge-generators, shown in the case $n=3$.}
        \label{fig:edge-coordinates-as-tensor-products}
\end{figure}

\section{Main theorem: quantum trace polynomials for \texorpdfstring{$\mathrm{SL}_3$}{SL3}}
\label{sec:main-theorem-quantum-trace-map-for-SL3}

\subsection{Framed oriented links in thickened surfaces}
\label{sec:framed-oriented-links}
	
So far, we have been working in the 2-dimensional setting of the punctured surface $\mathfrak{S}$.  We now turn to the 3-dimensional setting of the thickened surface $\mathfrak{S} \times (0, 1)$.  
We will follow \cite[\S 3.1]{BonahonGT11}, the only difference being that we consider oriented links.  

\begin{definition}
	A \textit{framed oriented link} $K$ in the thickened surface $\mathfrak{S} \times (0, 1)$ is a compact oriented one-dimensional manifold, possibly-with-boundary, $K \subset \mathfrak{S} \times (0, 1)$ that is embedded in $\mathfrak{S} \times (0, 1)$ and is equipped with a framing (see below), satisfying the following properties:  
\begin{itemize}
	\item  we have $\partial K = K \cap ((\partial \mathfrak{S}) \times (0, 1))$;   
	\item  the framing at a boundary point of $K$ is \textit{vertical}, meaning parallel to the $(0, 1)$ axis and pointing in the $1$ direction (or, in pictures, toward the eye of the reader);  
	\item  for each boundary component $k$ of $\mathfrak{S}$, the finitely many points $(\partial K) \cap (k \times (0, 1))$ have distinct \textit{heights}, meaning that the coordinates with respect to $(0, 1)$ are distinct.  
\end{itemize}
Here, by a \textit{framing}, we mean the choice of a smooth assignment along the link $K$ of unit  vectors in the tangent spaces of $\mathfrak{S} \times (0, 1)$ such that this vector field on $K$ is everywhere orthogonal to $K$.  
A \textit{framed oriented knot} $K$ is a closed framed oriented link (namely, a framed oriented link with empty boundary $\partial K = \emptyset$)  with one connected component.  
Two framed oriented links $K$ and $K^\prime$ are \textit{isotopic} if $K$ can be smoothly deformed to $K^\prime$ through the class of framed oriented links.  
By possibly introducing \textit{kinks} (Figures \ref{fig:positive-kink-skein-relation} and \ref{fig:negative-kink-skein-relation}), one can always isotope a framed link so that it has \textit{blackboard framing}, meaning constant vertical framing in the $1$ direction  (with respect to the $(0,1)$ coordinate).  
\end{definition}

\begin{remark}
\label{rem:higher-lower}
We display links in figures by their \textit{diagrams}, namely their projections onto the surface $\mathfrak{S}\cong\mathfrak{S}\times\{1/2\}$ equipped with over/under crossing information.  By convention, all link diagrams represent blackboard-framed links. 

Instead of using the picture conventions of \cite[\S 3.5]{BonahonGT11}, in our diagrams we will indicate explicitly which  points lying on a single boundary component $k \times (0, 1)$ of $\mathfrak{S}\times(0,1)$ are higher or lower with respect to the $(0, 1)$ direction.  (Note, importantly, that two points of $\partial K$ on a single boundary component $k \times (0,1)$ cannot exchange heights during an isotopy of the link.)

One can think of a framed link $K$ as a `ribbon', namely an oriented annulus (that is, oriented as a surface, not to be confused with link orientations)  embedded in $\mathfrak{S} \times (0, 1)$ where the framing  is perpendicular to the annulus and determined by the orientation.  
\end{remark}

\subsection{Stated links}
\label{ssec:stated-links}

\begin{definition}
\label{def:compatible}
	A ($n$-)\textit{stated framed oriented link} $(K, s)$ is  a framed oriented link $K$ equipped with a  function
	\begin{equation*}
		s : \partial K \to \{ 1, 2, \dots, n \},
	\end{equation*}
	called the \textit{state}, assigning to each element of the boundary of the link a  \textit{state-number} in $\{ 1, 2, \dots, n  \}$  (we often confuse  `state' with these state-numbers).  	
Note that a stated closed link is the same thing as a closed link.  
As for links, there is the corresponding notion of \textit{isotopy} of stated links. 

Let $(K, s)$ be a stated framed oriented link in a triangulated surface $(\mathfrak{S}, \lambda)$ obtained by gluing together two triangulated surfaces $(\mathfrak{S}_1, \lambda_1)$ and $(\mathfrak{S}_2, \lambda_2)$ along edges of the triangulations.  Let $K_1$ and $K_2$ be the associated links in $\mathfrak{S}_1$ and $\mathfrak{S}_2$.  We say states $s_1$ and $s_2$ for stated links $(K_1,s_1)$ and $(K_2,s_2)$ such that $s_1$ and $s_2$ agree with $s$ on $\partial K$ are \textit{compatible} if their values agree on the common boundaries of $K_1$ and $K_2$ (resulting from cutting $K$).  
\end{definition}

\subsection{Main result}
\label{ssec:main-theorem}

In this subsection, we  restrict to the case $n=3$.  Let the surface $\mathfrak{S}$ be equipped with a dotted ideal triangulation $\lambda$.  Recall the Fock--Goncharov quantum torus $\mathscr{T}_3^\omega(\lambda) \subset \bigotimes_{\textnormal{triangles } \mathfrak{T} \text{ of } \lambda} \mathscr{T}_3^\omega(\mathfrak{T})$ associated to this data; see Definition \ref{def:FG-algebra-of-the-whole-surface}.  

Note that if $K \subset \mathfrak{S} \times (0, 1)$ is a blackboard-framed oriented knot (meaning, in particular, that it is closed), and if $\pi : \mathfrak{S} \times (0, 1) \to \mathfrak{S} \times \{  1/2  \} \cong \mathfrak{S}$ is the natural projection, then, possibly after an arbitrarily small perturbation of the knot $K$, we have that $\gamma = \pi(K)$ is an immersed oriented closed curve in $\mathfrak{S}$, so we may consider the classical trace polynomial $\widetilde{\mathrm{Tr}}_\gamma(X_i^{1/3})$ in $\mathscr{T}_3^1(\lambda) = \mathbb{C}[X_1^{\pm 1/3}, X_2^{\pm 1/3}, \dots, X_N^{\pm 1/3}]$ associated to $\gamma$; see Definition \ref{def:second-classical-trace}.  

We now give a more detailed version of Theorem \ref{thm:intro-theorem-1} from \S \ref{sec:introduction}.  Technically, our solution involves choosing a square root $\omega^{1/2}$ of the parameter $\omega$.  Proofs will be given in \S \ref{sec:proof-of-the-main-theorem}, \ref{sec:computer-check-of-local-moves}.  

\begin{theorem}[$\mathrm{SL}_3$-quantum trace polynomials]
\label{thm:second-theorem}
	Let $q \in \mathbb{C} - \{ 0 \}$ be a non-zero complex number, and let $\omega = q^{1/3^2}$ be a $3^2$-root of $q$; choose also $\omega^{1/2}$.  There is a function 
	\begin{equation*}
		\mathrm{Tr}^\omega_\lambda  :
		\left\{  
			\textnormal{stated framed oriented links $(K,s)$ in } \mathfrak{S} \times (0, 1)
		\right\} 
		\to 
		\mathscr{T}_3^\omega(\lambda),
	\end{equation*}
satisfying the following properties:
\begin{enumerate}[\normalfont(A)]
	\item  the element $\mathrm{Tr}_\lambda^\omega(K, s) \in \mathscr{T}_3^\omega(\lambda)$ is invariant under isotopy of stated framed oriented links;
	\item  the $\mathrm{SL}_3$-HOMFLYPT skein relation (Figure {\upshape\ref{fig:HOMFLYPT-relation}} with $n=3$) holds;
	\item the $\mathrm{SL}_3$-quantum unknot and framing relations (Figures {\upshape\ref{fig:unknot-relation}} and {\upshape\ref{fig:framing-skein-relations}} with $n=3$) hold.  
\end{enumerate}
\end{theorem}

\begin{compl}
\label{compl}
Moreover, this invariant satisfies the following additional properties.  
\begin{itemize} 
	\item  (Classical Trace Property)  Let $q = \omega = \omega^{1/2} = 1$ and let $K$ be a closed blackboard-framed oriented knot.  Then,
	\begin{equation*}
	\label{eq:quantum-trace-classical-property}
		\mathrm{Tr}^1_\lambda(K) = \widetilde{\mathrm{Tr}}_{\gamma}(X_i^{1/3})
		  \in  
		\mathscr{T}_3^1(\lambda),
	\end{equation*}
	where $\gamma$ is the immersed oriented closed curve obtained by projecting $K$ to $\mathfrak{S}$.  
	\item  (Multiplication Property)  Let $(K,s) = (K_1,s_1) \cup (K_2,s_2) \cup \dots \cup (K_\ell,s_\ell)$ be a stated framed oriented link, written as a disjoint union of links $K_j$.  Assume in addition that $K_{j-1}$ lies entirely below $K_{j}$ in $\mathfrak{S} \times (0, 1)$, with respect to the height coordinate.  Then, 
	\begin{equation*}
	\label{eq:quantum-trace-multiplicative-property}
		\mathrm{Tr}^\omega_\lambda(K, s) = \mathrm{Tr}^\omega_\lambda(K_1, s_1) \mathrm{Tr}^\omega_\lambda(K_2, s_2) \cdots \mathrm{Tr}^\omega_\lambda(K_\ell, s_\ell)
		  \in  
		\mathscr{T}_3^\omega(\lambda).
	\end{equation*}
	Note that the order of multiplication matters, since $\mathscr{T}_3^\omega(\lambda)$ is non-commutative.
	\item  (State Sum Property)  Let $(K, s)$ be a stated framed oriented link in a triangulated surface $(\mathfrak{S}, \lambda)$ obtained by gluing together two triangulated surfaces $(\mathfrak{S}_1, \lambda_1)$ and $(\mathfrak{S}_2, \lambda_2)$.  Let $K_1$ and $K_2$ be the associated links in $\mathfrak{S}_1$ and $\mathfrak{S}_2$.  Then,
\begin{equation*}
	\mathrm{Tr}^\omega_\lambda(K,s)=\sum_{\textnormal{compatible } s_1, s_2} \mathrm{Tr}^\omega_{\lambda_1}(K_1, s_1) \otimes \mathrm{Tr}^\omega_{\lambda_2}(K_2, s_2)  \in \mathscr{T}^\omega_3(\lambda).
\end{equation*} 
\end{itemize}
\end{compl}

\section{Quantum trace polynomials for \texorpdfstring{$\mathrm{SL}_n$}{SLn}}
\label{sec:proof-of-the-main-theorem}

The corresponding version of Theorem \ref{thm:second-theorem} and Complement \ref{compl} should hold in the case of $\mathrm{SL}_n$, by replacing $3$ with $n$ everywhere in the statement.  In this section, we will construct the quantum trace map $\mathrm{Tr}_\lambda^\omega$ for general $n$. However, we only give a proof that it is well-defined for $n=3$.   For concreteness, along the way we will give explicit formulas for the case $n=3$.   When $n=2$, our construction coincides with that in \cite{BonahonGT11}.  In particular, our construction gives a way to think of their construction, which was defined for un-oriented links, in terms of oriented links.  
Throughout, fix $q \in \mathbb{C} - \{ 0 \}$ and a $n^2$-root $\omega = q^{1/n^2} \in \mathbb{C} - \{ 0 \}$ of $q$.  Technically, also choose $\omega^{1/2}$.  
\subsection{Matrix conventions}
\label{ssec:matrix-conventions}

We will need to display $3 \times 3$ and $3^2 \times 3^2$ matrices.  Lower indices will indicate rows and upper indices will indicate columns.  A $3 \times 3$ matrix $\vec{M} = (M_i^j)$ will be displayed in the general form
\begin{equation*}
	\vec{M} = 
	\left(\begin{smallmatrix}
		M_1^1&M_1^2&M_1^3\\
		M_2^1&M_2^2&M_2^3\\
		M_3^1&M_3^2&M_3^3
	\end{smallmatrix}\right).
\end{equation*}
A $3^2 \times 3^2$ matrix $\vec{M} = (M_{i_1 i_2}^{j_1 j_2})$ will be displayed in the general form
\begin{equation*}
\label{eq:matrix-conventions}
	\vec{M} = 
	\left(\begin{smallmatrix}
		M_{11}^{11}&M_{11}^{12}&M_{11}^{13}&
			M_{11}^{21}&M_{11}^{22}&M_{11}^{23}&
			M_{11}^{31}&M_{11}^{32}&M_{11}^{33}\\
		M_{12}^{11}&M_{12}^{12}&M_{12}^{13}&
			M_{12}^{21}&M_{12}^{22}&M_{12}^{23}&
			M_{12}^{31}&M_{12}^{32}&M_{12}^{33}\\		
		M_{13}^{11}&M_{13}^{12}&M_{13}^{13}&
			M_{13}^{21}&M_{13}^{22}&M_{13}^{23}&
			M_{13}^{31}&M_{13}^{32}&M_{13}^{33}\\
		M_{21}^{11}&M_{21}^{12}&M_{21}^{13}&
			M_{21}^{21}&M_{21}^{22}&M_{21}^{23}&
			M_{21}^{31}&M_{21}^{32}&M_{21}^{33}\\
		M_{22}^{11}&M_{22}^{12}&M_{22}^{13}&
			M_{22}^{21}&M_{22}^{22}&M_{22}^{23}&
			M_{22}^{31}&M_{22}^{32}&M_{22}^{33}\\
		M_{23}^{11}&M_{23}^{12}&M_{23}^{13}&
			M_{23}^{21}&M_{23}^{22}&M_{23}^{23}&
			M_{23}^{31}&M_{23}^{32}&M_{23}^{33}\\
		M_{31}^{11}&M_{31}^{12}&M_{31}^{13}&
			M_{31}^{21}&M_{31}^{22}&M_{31}^{23}&
			M_{31}^{31}&M_{31}^{32}&M_{31}^{33}\\
		M_{32}^{11}&M_{32}^{12}&M_{32}^{13}&
			M_{32}^{21}&M_{32}^{22}&M_{32}^{23}&
			M_{32}^{31}&M_{32}^{32}&M_{32}^{33}\\
		M_{33}^{11}&M_{33}^{12}&M_{33}^{13}&
			M_{33}^{21}&M_{33}^{22}&M_{33}^{23}&
			M_{33}^{31}&M_{33}^{32}&M_{33}^{33}\\
	\end{smallmatrix}\right).
\end{equation*}
If $V$ and $W$ are finite-dimensional complex vector spaces with bases $\{ v_1, \dots, v_m \}$ and $\{ w_1, \dots, w_p \}$ and if $T : V \to W$ is a linear map, we define the $p \times m$ matrix $[T] \in \mathrm{M}_{p,m}(\mathbb{C})$ associated to $T$ and these bases of $V$ and $W$ by the property
\begin{equation*}
	T(v_j) = \sum_{i=1}^p [T]_i^j w_i
	\,\,  (j=1, 2, \dots, m).
\end{equation*}

\subsection{Biangle quantum trace map}
\label{ssec:biangles-and-the-reshetikhin-turaev-invariant}
	
A \textit{biangle} $\mathfrak{B}$ is a closed disk with two punctures on its boundary.  Biangles do not admit ideal triangulations, so $\mathfrak{S}$ is never a biangle.  However, we may still consider stated framed oriented links $(K, s)$ in the thickened biangle $\mathfrak{B} \times (0, 1)$ defined just as before.  In this subsection, we will (implicitly) use the Reshetikhin--Turaev construction \cite{ReshetikhinCommMathPhys90} to provide an analogue of Theorem \ref{thm:second-theorem} and Complement \ref{compl} for biangles, valued in the complex numbers $\mathbb{C}$; see Appendix \ref{sec:proof-of-reshetikhin-turaev-invariant} for the explicit (and more conceptual) connection to \cite{ReshetikhinCommMathPhys90}.  Note that this subsection does not require the choice of square root $\omega^{1/2}$.  

Parametrize the thickened biangle $\mathfrak{B} \times (0, 1) \cong [0, 1] \times \mathbb{R} \times (0, 1)$ such that, in Figure \ref{fig:U-turn-dec-cw} say, the first coordinate points along the page to the right, the second coordinate points along the page up, and the third coordinate points out of the page toward the eye of the reader.  Note this parametrization is not canonical: there are two possibilities, related by `turning the biangle on its head'.  The construction will be independent of this choice of parametrization (see the comments after Proposition \ref{prop:reshetikhin-turaev}).

In order to state the result, we first define some elementary matrices associated to certain local link diagrams, namely various U-turns and crossings.  

\subsubsection{U-turns}
\label{sssec:U-turns}

  In Figures \ref{fig:decreasing-U-turns} and \ref{fig:increasing-U-turns}, we show the four possible U-turns, which are in particular stated framed oriented links with the blackboard framing.  
In agreement with our picture conventions (see Remark \ref{rem:higher-lower}), the boundary point of the link that is labeled `Higher' or `H' is higher, namely has a greater coordinate with respect to the $(0, 1)$ direction, than the boundary point of the link that is labeled `Lower' or `L'.  

\begin{definition}
\label{def:ribbon-element}  $ $
\begin{itemize}
\item	The \textit{$\mathrm{SL}_n$-coribbon element} is
\begin{equation*}
	\bar{\zeta}_n \overset{\text{def}}{=} (-1)^{n-1} q^{(1-n^2)/n} 
	\,\, \left( = (-1)^{n-1} \omega^{n(1-n^2)} \right)
	\,\,  \in \mathbb{C}-\{0\}.
\end{equation*}
(Note the overbar notation in $\bar{\zeta}_n$ does not mean `complex conjugation'.)  For the categorical context, see \S \ref{ssec:coribbon element}.
\item   Define the \textit{square root of the (signed) coribbon element} by
\begin{equation*}
	\bar{\sigma}_n \overset{\text{def}}{=} + q^{(1-n^2)/2n}
	\,\, \left( = +\omega^{n(1-n^2)/2} \right)
	\,\, \in \mathbb{C}-\{0\}.
\end{equation*}
We require the parenthetical `signed' because we can only say $\bar{\sigma}_n^2 = (-1)^{n-1} \bar{\zeta}_n$.  
Note that $n(1-n^2)$ is always even, so here we do not need to choose a square root $\omega^{1/2}$.  
\item  Define a $n \times n$ matrix $\vec{U}^q$ over the complex numbers by
\begin{equation*}
	\vec{U}^q \overset{\text{def}}{=} \bar{\sigma}_n
	\left(\begin{smallmatrix}
		&&&&(-1)^{n-1} q^{(1-n)/2}\\
		&&&\reflectbox{$\ddots$}&\\
		&&+q^{(n-5)/2}&&\\
		&-q^{(n-3)/2}&&&\\
		+q^{(n-1)/2}&&&&
	\end{smallmatrix}\right)
	  \in  \mathrm{M}_n(\mathbb{C}),
\end{equation*}
where it is implicit that we are putting $q=\omega^{n^2}$ as above, so $q^{(n-1-2k)/2}$ is defined.  Note that the common ratio between adjacent entries in the matrix is equal to $-q$.  Note also that putting $q=\omega=1$ recovers the classical U-turn matrix $\vec{U}$ from \S \ref{sec:local-monodromy-matrices}.    For example, in the case $n=3$, the  $3 \times 3$ matrix $\vec{U}^q$ is
\begin{equation*}
\label{eq:non-standard-duality}
	\vec{U}^q =
	+q^{-4/3} \left(\begin{smallmatrix}
		0&0&+q^{-1}\\
		0&-1&0\\
		+q&0&0
	\end{smallmatrix}\right)
	  \in \mathrm{M}_3(\mathbb{C}).
\end{equation*}
\end{itemize}
\end{definition}
\begin{definition}

For each pair of states $s_1, s_2 \in \{  1, 2, \dots, n  \}$, define four complex numbers 
\begin{equation*}
	\mathrm{Tr}^\omega_\mathfrak{B}(U_\text{dec}^\text{cw})_{s_1}^{s_2}, \,\, \mathrm{Tr}^\omega_\mathfrak{B}(U_\text{dec}^\text{ccw})_{s_1}^{s_2}, \,\,
\mathrm{Tr}^\omega_\mathfrak{B}(U_\text{inc}^\text{ccw})_{s_1}^{s_2}, \,\, \mathrm{Tr}^\omega_\mathfrak{B}(U_\text{inc}^\text{cw})_{s_1}^{s_2}  \in \mathbb{C},
\end{equation*}
by the matrix equations
\begin{gather*}
	(\mathrm{Tr}^\omega_\mathfrak{B}(U_\text{dec}^\text{cw})_{s_1}^{s_2}) \overset{\text{def}}{=} \vec{U}^q
	\in  \mathrm{M}_n(\mathbb{C}),\,\,
	(\mathrm{Tr}^\omega_\mathfrak{B}(U_\text{dec}^\text{ccw})_{s_1}^{s_2})
	\overset{\text{def}}{=}  (\bar{\zeta}_n)^{-1} \vec{U}^q
	\in  \mathrm{M}_n(\mathbb{C}),
\\
	(\mathrm{Tr}^\omega_\mathfrak{B}(U_\text{inc}^\text{ccw})_{s_1}^{s_2}) \overset{\text{def}}{=} (\vec{U}^q)^\mathrm{T}
	\in  \mathrm{M}_n(\mathbb{C}),\,\,
	(\mathrm{Tr}^\omega_\mathfrak{B}(U_\text{inc}^\text{cw})_{s_1}^{s_2})
	\overset{\text{def}}{=}  (\bar{\zeta}_n)^{-1}  (\vec{U}^q)^\mathrm{T}
	\in  \mathrm{M}_n(\mathbb{C}),
\end{gather*}
see \S \ref{ssec:matrix-conventions}.  Here, the superscript $\mathrm{T}$ indicates that we are taking the matrix transpose.  
For example, in the case $n=3$, in the above formulas $(\bar{\zeta}_3)^{-1} = +q^{+8/3}$.
\end{definition}

\begin{remark}
When $n=2$, these formulas agree with those in \cite[Proposition 13.2.b]{BonahonGT11} for the underlying un-oriented link, taking $\alpha = -\omega^{-5}$ and $\beta = \omega^{-1}$; see \cite[Prop. 26]{BonahonGT11}.  
\end{remark}

\begin{figure}[htb]
     \centering
     \begin{subfigure}{0.49\textwidth}
         \centering
         \includegraphics[width=.275\textwidth]{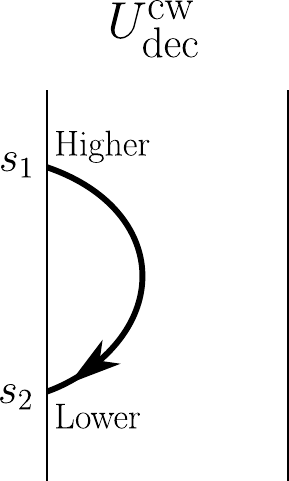}
         \caption{Clockwise}
         \label{fig:U-turn-dec-cw}
     \end{subfigure}     
\hfill
     \begin{subfigure}{0.49\textwidth}
         \centering
         \includegraphics[width=.275\textwidth]{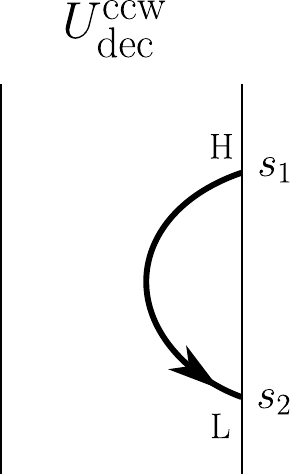}
         \caption{Counterclockwise}
         \label{fig:U-turn-dec-ccw}
     \end{subfigure}
     	\caption{Decreasing U-turns.}
        \label{fig:decreasing-U-turns}
\end{figure}

\begin{figure}[htb]
     \centering
     \begin{subfigure}{0.49\textwidth}
         \centering
         \includegraphics[width=.275\textwidth]{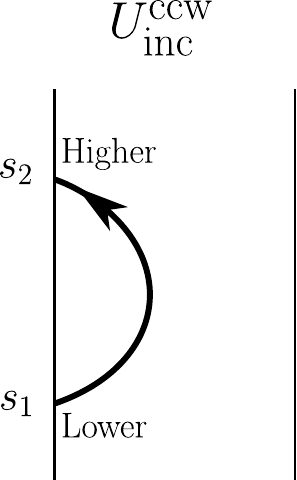}
         \caption{Counterclockwise}
         \label{fig:U-turn-inc-ccw}
     \end{subfigure}     
\hfill
     \begin{subfigure}{0.49\textwidth}
         \centering
         \includegraphics[width=.275\textwidth]{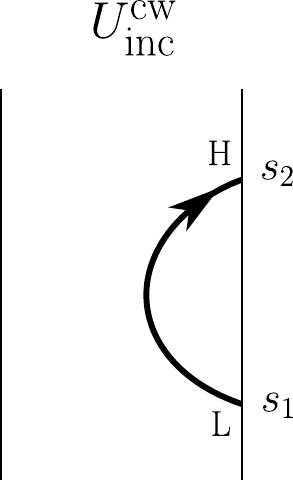}
         \caption{Clockwise}
         \label{fig:U-turn-inc-cw}
     \end{subfigure}
     	\caption{Increasing U-turns.}
        \label{fig:increasing-U-turns}
\end{figure}

\subsubsection{Crossings}
\label{sssec:crossings}
	
Shown in Figures \ref{fig:same-direction-crossings} and \ref{fig:opposite-direction-crossings} are the eight possible crossings, with blackboard framing and  the usual picture conventions as above.  

Let $V$ be a $n$-dimensional complex vector space, and let $V^*$ be the complex vector space dual to $V$.  Choose a linear basis $\{ e^1, e^2, \dots, e^n \}$ for $V$, and let $\{ e_1^*, e_2^*, \dots, e_n^* \}$ be the corresponding dual basis for $V^*$.  Define four linear isomorphisms
\begin{gather*}
	\bar{c}_{V, V} : V \otimes V \to V \otimes V,
\,\,	\bar{c}_{V^*, V^*} : V^* \otimes V^* \to V^* \otimes V^*,
	\\  \bar{c}_{V^*, V} : V^* \otimes V \to V \otimes V^*,
\,\,	  \bar{c}_{V, V^*} : V \otimes V^* \to V^* \otimes V,
\end{gather*} 
by extending linearly the following assignments for tensor product basis elements
\begin{gather*}		
	\bar{c}_{V, V}(e^i \otimes e^j) 
\overset{\text{def}}{=} q^{+1/n} 
\begin{cases}
q^{-1} e^i \otimes e^i,  &  i = j,  \\
(q^{-1} - q) e^i \otimes e^j     +                e^j \otimes e^i,                         &  i < j,   \\
e^j \otimes e^i,       &  i > j,
\end{cases}
\\		
	\bar{c}_{V^*, V^*}(e_i^* \otimes e_j^*) 
\overset{\text{def}}{=} q^{+1/n} 
\begin{cases}
q^{-1} e_i^* \otimes e_i^*,  &  i = j,  \\
(q^{-1} - q) e_i^* \otimes e_j^*     +                e_j^* \otimes e_i^*,                         &  i > j,   \\
e_j^* \otimes e_i^*,       &  i < j,
\end{cases}
\\
	\bar{c}_{V^*, V}(e^*_i \otimes e^j) 
	\overset{\text{def}}{=} q^{-1/n} 
	\begin{cases}
		q e^i \otimes e^*_i + (q - q^{-1}) \sum_{1 \leq k < i} e^k \otimes e^*_k,  &  i = j,  \\
		e^j \otimes e^*_i,  &  i \neq j,
	\end{cases}	
\\
	\bar{c}_{V, V^*}(e^i \otimes e^*_j) 
\overset{\text{def}}{=}  q^{-1/n} 
\begin{cases}
	q e^*_i \otimes e^i + (q - q^{-1}) \sum_{i < k \leq n} q^{2(k-i)} e^*_k \otimes e^k,  &  i = j,  \\
	e^*_j \otimes e^i,  &  i \neq j.  
	\end{cases}
\end{gather*}
Define bases $\beta_{V, V}$, $\beta_{V^*, V^*}$, $\beta_{V^*, V}$, and $\beta_{V, V^*}$ of $V \otimes V$, $V^* \otimes V^*$, $V^* \otimes V$, and $V \otimes V^*$ by
\begin{gather*}
	(\beta_{V, V})_{ij}
	\overset{\text{def}}{=}  e^i \otimes e^j,	
\,\,
	(\beta_{V^*, V^*})_{ij}
	\overset{\text{def}}{=} (-q)^{n-i} e^*_{n-i+1} \otimes (-q)^{n-j} e^*_{n-j+1},
\\
	(\beta_{V^*, V})_{ij} \overset{\text{def}}{=}
		 (-q)^{n-i} e^*_{n-i+1} \otimes e^j,  
\,\,
	(\beta_{V, V^*})_{ij} \overset{\text{def}}{=}
		 e^i \otimes (-q)^{n-j} e^*_{n-j+1}.
\end{gather*}
For example, when $n=2$, the ordered basis $\beta_{V, V}$ is $ \left\{ e^1 \otimes e^1, e^1 \otimes e^2, e^2 \otimes e^1, e^2 \otimes e^2 \right\}$.  

The following fact, a simple calculation from the above definitions, motivates the definitions of the matrices $\vec{C}^q_\mathrm{same}$ and $\vec{C}^q_\mathrm{opp}$  below (and will be used in Appendix \ref{sec:proof-of-reshetikhin-turaev-invariant}).

\begin{fact}
\label{fact:braiding-matrices}
	We have the following equalities of matrices
\begin{gather*}
	\vec{C}^q_\mathrm{same} \overset{\mathrm{def}}{=} [\bar{c}_{V, V}] = [\bar{c}_{V^*, V^*}]
	  \in  \mathrm{M}_{n^2}(\mathbb{C}),
\,\,	\vec{C}^q_\mathrm{opp} \overset{\mathrm{def}}{=} [\bar{c}_{V^*, V}] = [\bar{c}_{V, V^*}]
  \in  \mathrm{M}_{n^2}(\mathbb{C}),
\end{gather*}
representing the linear isomorphisms $\bar{c}_{V, V}$, $\bar{c}_{V^*, V^*}$, $\bar{c}_{V^*, V}$ and $\bar{c}_{V^*, V}$ when expressed in terms of the bases $\beta_{V, V}$, $\beta_{V^*, V^*}$, $\beta_{V^*, V}$ and $\beta_{V, V^*}$.  Also, these matrices are symmetric.  \qed
\end{fact}

For example, in the case $n=3$, these two $3^2 \times 3^2$ matrices $\vec{C}_\text{same}^q$ and $\vec{C}_\text{opp}^q$ are given by
\begin{gather*}
\label{eq:same-R-matrix}
	\vec{C}_\text{same}^q  = q^{+1/3}
	\left(\begin{smallmatrix}
q^{-1}&0&0&0&0&0&0&0&0\\
0&q^{-1}-q&0&1&0&0&0&0&0\\		
0&0&q^{-1}-q&0&0&0&1&0&0\\	
0&1&0&0&0&0&0&0&0\\	
0&0&0&0&q^{-1}&0&0&0&0\\	
0&0&0&0&0&q^{-1}-q&0&1&0\\	
0&0&1&0&0&0&0&0&0\\	
0&0&0&0&0&1&0&0&0\\	
0&0&0&0&0&0&0&0&q^{-1}
	\end{smallmatrix}\right)
	  \in  \mathrm{M}_{3^2}(\mathbb{C}),
\\
\label{eq:opp-R-matrix}
	\vec{C}_\text{opp}^q  = q^{+2/3}
	\left(\begin{smallmatrix}
q^{-1}&0&0&0&0&0&0&0&0\\
0&0&0&q^{-1}&0&0&0&0&0\\
0&0&q^2-1&0&q^{-1}-q&0&1&0&0\\
0&q^{-1}&0&0&0&0&0&0&0\\
0&0&q^{-1}-q&0&1&0&0&0&0\\
0&0&0&0&0&0&0&q^{-1}&0\\
0&0&1&0&0&0&0&0&0\\
0&0&0&0&0&q^{-1}&0&0&0\\
0&0&0&0&0&0&0&0&q^{-1}
	\end{smallmatrix}\right)
	  \in  \mathrm{M}_{3^2}(\mathbb{C}).
\end{gather*}	

See Remark \ref{rem:comment-on-coribbon-element} for a discussion of how Fact \ref{fact:braiding-matrices} relates to the  quantum group $\mathrm{SL}_n^q$.
	
An observation is that, for general $n$, when $q=\omega=1$ then the two matrices $\vec{C}_\text{same}^1$ and $\vec{C}_\text{opp}^1$ are identical.  For another property of these  \textit{$R$-matrices}, see \S  \ref{sec:HOMFLYPT-skein-relation}.  

\begin{definition}
For each quadruple of states $s_1, s_2, s_3, s_4 \in \{  1, 2, \dots, n  \}$, define eight complex numbers 
\begin{gather*}
	\mathrm{Tr}^\omega_\mathfrak{B}(C_\text{pos-same}^\text{over-to-lower})_{s_1 s_2}^{s_3 s_4}
	, \,\, \mathrm{Tr}^\omega_\mathfrak{B}(C_\text{neg-same}^\text{over-to-higher})_{s_1 s_2}^{s_3 s_4}
	, \,\, \mathrm{Tr}^\omega_\mathfrak{B}(C_\text{pos-same}^\text{over-to-higher})_{s_1 s_2}^{s_3 s_4}
	, \,\, \mathrm{Tr}^\omega_\mathfrak{B}(C_\text{neg-same}^\text{over-to-lower})_{s_1 s_2}^{s_3 s_4}
	,
\\
	\mathrm{Tr}^\omega_\mathfrak{B}(C_\text{neg-opp}^\text{over-to-lower})_{s_1 s_2}^{s_3 s_4}
	, \,\, \mathrm{Tr}^\omega_\mathfrak{B}(C_\text{pos-opp}^\text{over-to-higher})_{s_1 s_2}^{s_3 s_4}
	, \,\, \mathrm{Tr}^\omega_\mathfrak{B}(C_\text{neg-opp}^\text{over-to-higher})_{s_1 s_2}^{s_3 s_4}
	, \,\, \mathrm{Tr}^\omega_\mathfrak{B}(C_\text{pos-opp}^\text{over-to-lower})_{s_1 s_2}^{s_3 s_4},
\end{gather*}
by the matrix equations
\begin{gather*}
	(\mathrm{Tr}^\omega_\mathfrak{B}(C_\text{pos-same}^\text{over-to-lower})_{s_1 s_2}^{s_3 s_4}
	) 
	\overset{\text{def}}{=} \vec{C}_\text{same}^q
	  \in  \mathrm{M}_{n^2}(\mathbb{C}),
\,\,
	(\mathrm{Tr}^\omega_\mathfrak{B}(C_\text{neg-same}^\text{over-to-higher})_{s_1 s_2}^{s_3 s_4}
	) 
	\overset{\text{def}}{=} (\vec{C}_\text{same}^q)^{-1}
	  \in  \mathrm{M}_{n^2}(\mathbb{C}),
\\
	(\mathrm{Tr}^\omega_\mathfrak{B}(C_\text{pos-same}^\text{over-to-higher})_{s_1 s_2}^{s_3 s_4}
	)
	\overset{\text{def}}{=}  
	\vec{C}_\text{same}^q
	  \in  \mathrm{M}_{n^2}(\mathbb{C}),
\,\,
	(\mathrm{Tr}^\omega_\mathfrak{B}(C_\text{neg-same}^\text{over-to-lower})_{s_1 s_2}^{s_3 s_4}
	)
	\overset{\text{def}}{=}  
	(\vec{C}_\text{same}^q)^{-1}
	  \in  \mathrm{M}_{n^2}(\mathbb{C}),
\\
	(\mathrm{Tr}^\omega_\mathfrak{B}(C_\text{neg-opp}^\text{over-to-lower})_{s_1 s_2}^{s_3 s_4}
	) 
	\overset{\text{def}}{=} \vec{C}_\text{opp}^q
	  \in  \mathrm{M}_{n^2}(\mathbb{C}),
\,\,
	(\mathrm{Tr}^\omega_\mathfrak{B}(C_\text{pos-opp}^\text{over-to-higher})_{s_1 s_2}^{s_3 s_4}
	) 
	\overset{\text{def}}{=} (\vec{C}_\text{opp}^q)^{-1}
	  \in  \mathrm{M}_{n^2}(\mathbb{C}),
\\
	(\mathrm{Tr}^\omega_\mathfrak{B}(C_\text{neg-opp}^\text{over-to-higher})_{s_1 s_2}^{s_3 s_4}
	)
	\overset{\text{def}}{=}  
	\vec{C}_\text{opp}^q
	  \in  \mathrm{M}_{n^2}(\mathbb{C}),
\,\,
	(\mathrm{Tr}^\omega_\mathfrak{B}(C_\text{pos-opp}^\text{over-to-lower})_{s_1 s_2}^{s_3 s_4}
	)
	\overset{\text{def}}{=}  
	(\vec{C}_\text{opp}^q)^{-1}
	  \in  \mathrm{M}_{n^2}(\mathbb{C}).
\end{gather*}
\end{definition}

\begin{remark}  \label{rem:comment-on-coribbon-element}  In the case $n=2$, these formulas agree with those in \cite[Lemma 22]{BonahonGT11}, for the underlying un-oriented link, taking $A = \omega^{-2}$, $\alpha = -\omega^{-5}$ and $\beta = \omega^{-1}$ (see \cite[Proposition 26]{BonahonGT11}).  In particular, as another indication of the un-oriented nature of $\mathrm{SL}_2$, when $n=2$ the two matrices $\vec{C}^q_\mathrm{same}$ and $\vec{C}^q_\mathrm{opp}$ are identical for all $q$ and $\omega$ (we saw above that, for general $n$, this is only true for $q = \omega=1$).  This can be explained conceptually as follows.  For any $n$, the vector spaces $V$ and $V^*$ can be given the structure of a \textit{right $\mathrm{SL}_n^q$-comodule}; see Appendix \ref{sec:proof-of-reshetikhin-turaev-invariant}.  When $n=2$, the linear isomorphism $V \to V^*$, $e^1 \mapsto -q e_2^*$, $e^2 \mapsto e_1^*$ is an isomorphism of right $\mathrm{SL}_2^q$-comodules, but this is not true for $n > 2$.  This is why, loosely speaking, the choices above for the bases $\beta_{V, V}$, $\beta_{V^*, V^*}$, $\beta_{V^*, V}$ and $\beta_{V, V^*}$ are `preferred'.  

The linear isomorphisms $\bar{c}_{V, V}$, $\bar{c}_{V^*, V^*}$, $\bar{c}_{V^*, V}$ and $\bar{c}_{V^*, V}$ arise naturally as braidings in the ribbon category of finite-dimensional right $\mathrm{SL}_n^q$-comodules, where the categorical coribbon element $\bar{\zeta}_{\mathrm{SL}_n^q}$ is essentially given by Definition \ref{def:ribbon-element}; see Appendix \ref{sec:proof-of-reshetikhin-turaev-invariant}.  Possibly of interest, we have implicitly taken a `symmetric' duality, which is more fitting for the current  setting.  In the notation of \cite{Kassel95} (compare \cite[Chapter XIV.2, Example 1]{Kassel95}), these symmetric dualities $b_V$ and $d_V$ are related to the usual ones by $b_V= \nu b_V^\textnormal{Kassel}$ and  $d_V= \nu^{-1} d_V^\textnormal{Kassel}$, where $\nu = q^{(1-n)/2n} = \bar{\sigma}_n q^{(n-1)/2}$; see Appendix \ref{sec:proof-of-reshetikhin-turaev-invariant}.  Note that $\nu$ is the bottom left entry of $\vec{U}^q$, see Definition \ref{def:ribbon-element}.
\end{remark}

\begin{figure}[htb]
     \centering
     \begin{subfigure}{0.24\textwidth}
         \centering
         \includegraphics[width=.75\textwidth]{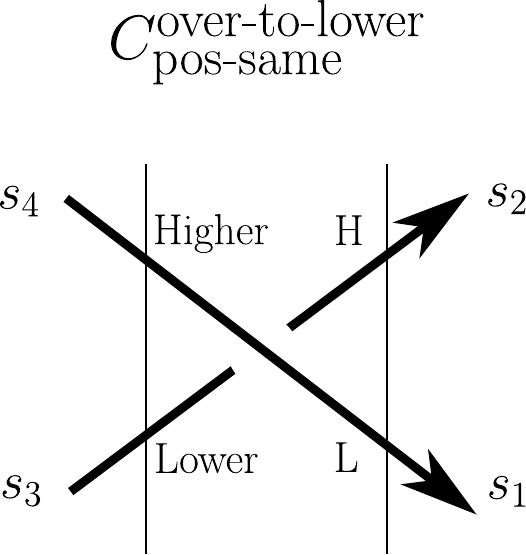}
         \caption{positive crossing, over strand higher to lower}
         \label{fig:cross-pos-same-over-to-lower}
     \end{subfigure}     
\hfill
     \begin{subfigure}{0.24\textwidth}
         \centering
         \includegraphics[width=.75\textwidth]{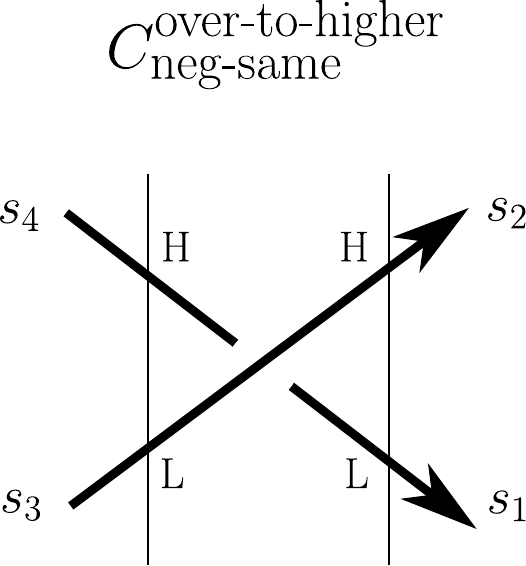}
         \caption{negative crossing, over strand lower to higher}
         \label{fig:cross-neg-same-over-to-higher}
     \end{subfigure}
\hfill
     \begin{subfigure}{0.24\textwidth}
         \centering
         \includegraphics[width=.75\textwidth]{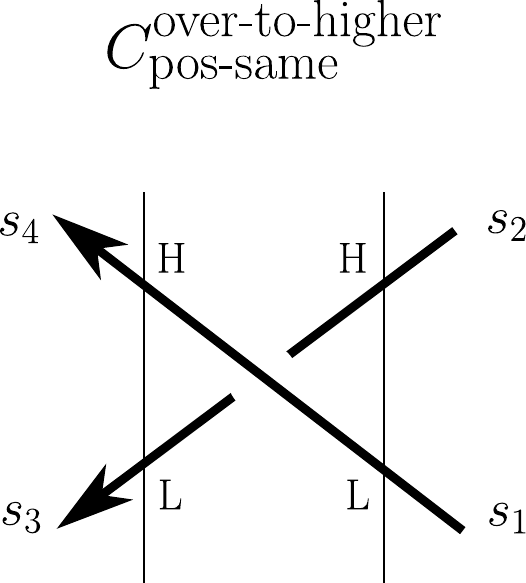}
         \caption{positive crossing, over strand lower to higher}
         \label{fig:cross-pos-same-over-to-higher}
     \end{subfigure}
\hfill
     \begin{subfigure}{0.24\textwidth}
         \centering
         \includegraphics[width=.75\textwidth]{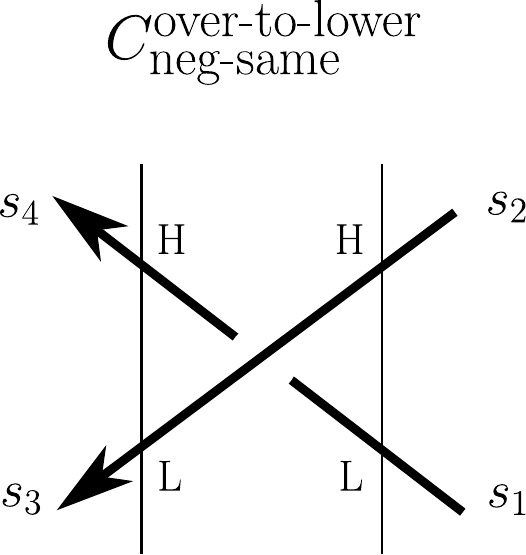}
         \caption{negative crossing, over strand higher to lower}
         \label{fig:cross-neg-same-over-to-lower}
     \end{subfigure}
     	\caption{Same direction crossings.}
        \label{fig:same-direction-crossings}
\end{figure}

\begin{figure}[htb]
     \centering
     \begin{subfigure}{0.24\textwidth}
         \centering
         \includegraphics[width=.75\textwidth]{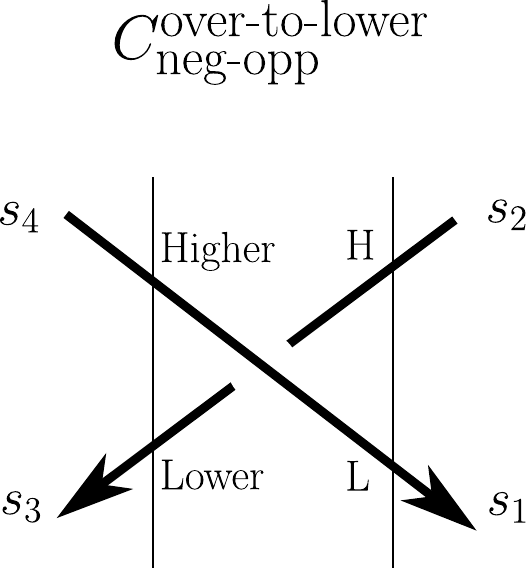}
         \caption{negative crossing, over strand higher to lower}
         \label{fig:cross-neg-opp-over-to-lower}
     \end{subfigure}     
\hfill
     \begin{subfigure}{0.24\textwidth}
         \centering
         \includegraphics[width=.75\textwidth]{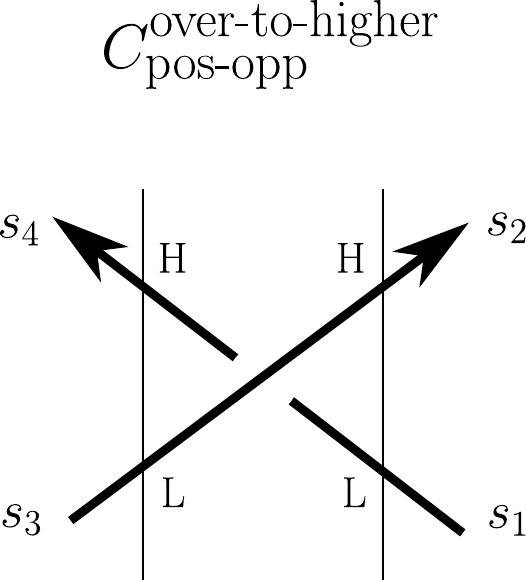}
         \caption{positive crossing, over strand lower to higher}
         \label{fig:cross-pos-opp-over-to-higher}
     \end{subfigure}
\hfill
     \begin{subfigure}{0.24\textwidth}
         \centering
         \includegraphics[width=.75\textwidth]{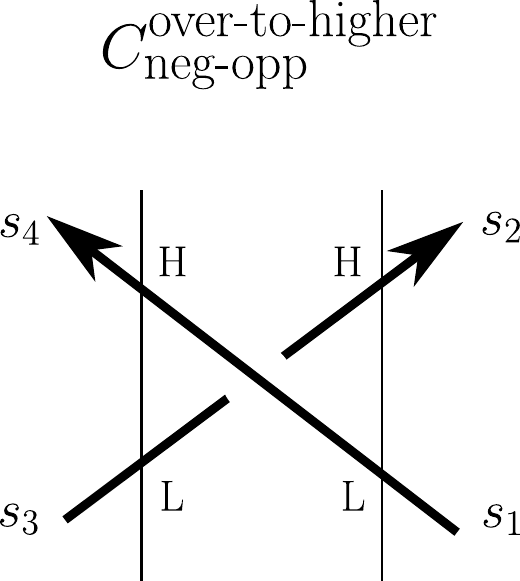}
         \caption{negative crossing, over strand lower to higher}
         \label{fig:cross-neg-opp-over-to-higher}
     \end{subfigure}
\hfill
     \begin{subfigure}{0.24\textwidth}
         \centering
         \includegraphics[width=.75\textwidth]{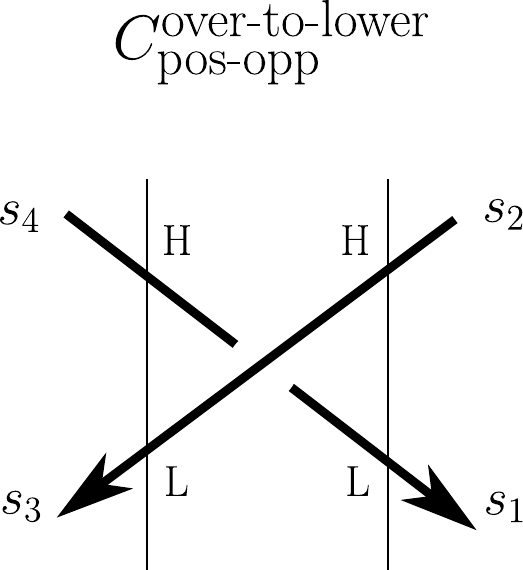}
         \caption{positive crossing, over strand higher to lower}
         \label{fig:cross-pos-opp-over-to-lower}
     \end{subfigure}
     	\caption{Opposite direction crossings.}
        \label{fig:opposite-direction-crossings}
\end{figure}

\subsubsection{Trivial strand}
\label{sssec:trivial-strand}
	
Consider a single strand crossing from one boundary edge of the biangle to the other boundary edge, as shown in Figure \ref{fig:single-strand-crossing-the-biangle}.  Note that the height of the strand with respect to the $(0, 1)$ component does not play a role in this particular case. 

This trivial strand corresponds to the $n \times n$ identity matrix.  That is, define for each pair of states $s_1, s_2 \in \{  1, 2, \dots, n  \}$ the complex number $\mathrm{Tr}^\omega_\mathfrak{B}(I)_{s_1}^{s_2}$ by the matrix equation
\begin{equation*}
\label{eq:skein-relation-last-third}
	( \mathrm{Tr}^\omega_\mathfrak{B}(I)_{s_1}^{s_2} ) \overset{\text{def}}{=} \vec{Id}_{n}
	  \in \mathrm{M}_{n}(\mathbb{C}).
\end{equation*}

\begin{figure}[htb]
	\centering
	\includegraphics[scale=.4]{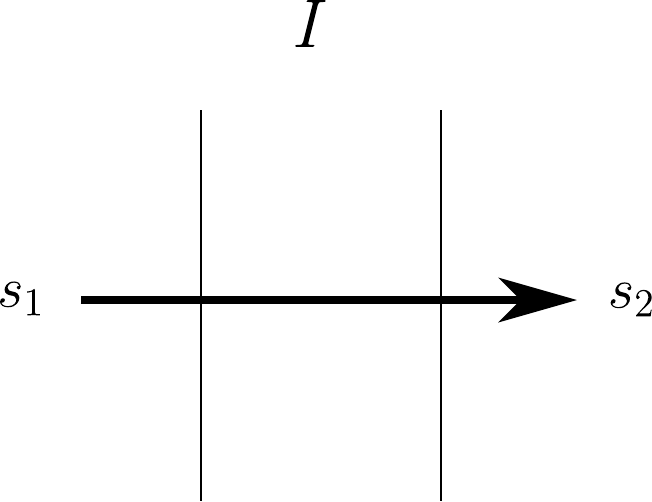}
	\caption{Trivial strand.}
	\label{fig:single-strand-crossing-the-biangle}
\end{figure}

\subsubsection{Kinks and the biangle quantum trace map}
\label{sssec:kinks-and-the-Reshetikhin-Turaev-invariant}
	
Up to this point, we have assigned complex numbers $\mathrm{Tr}_\mathfrak{B}^\omega(K,s)$ to a handful of small stated blackboard-framed oriented links $(K,s)$ (but not their isotopy classes) in the parametrized (see the beginning of \S \ref{ssec:biangles-and-the-reshetikhin-turaev-invariant}) thickened biangle $\mathfrak{B} \times (0, 1)$.  We now provide the general assignment.  

Assume first that the blackboard-framed link $K$ has no kinks.
For the moment, we also assume that higher points of $K \cap (\left\{0\right\} \times \mathbb{R} \times (0, 1))$ (resp. $K \cap (\left\{1\right\} \times \mathbb{R} \times (0, 1))$) have larger second coordinates with respect to the parametrization of $\mathfrak{B} \times (0, 1) \cong [0, 1] \times \mathbb{R} \times (0, 1)$; see, for example, Figures \ref{fig:decreasing-U-turns}, \ref{fig:increasing-U-turns}, \ref{fig:same-direction-crossings}, \ref{fig:opposite-direction-crossings}.    (Recall, in particular, Remark \ref{rem:higher-lower}.)
Fixing endpoints, isotope (without introducing kinks) $K$ into an arbitrary \textit{bridge position}.  This means that, after isotopy, there exists a partition $0=x_0<x_1<\dots<x_p=1$ of $[0,1]$ such that $K \cap ([x_i, x_{i+1}] \times \mathbb{R} \times (0, 1)) = K_i = \cup_{\ell} K_{i,\ell}$ is a disjoint union of links $K_{i,\ell}$ satisfying:
\begin{itemize}
	\item  the higher points of $K_i \cap (\left\{ x_i \right\} \times \mathbb{R} \times (0, 1))$ have a larger second coordinate;
	\item  for each $i$, there is a single $\ell$ such that $K_{i,\ell}$ is either a U-turn (Figures \ref{fig:decreasing-U-turns}, \ref{fig:increasing-U-turns}) or a crossing (Figures \ref{fig:same-direction-crossings}, \ref{fig:opposite-direction-crossings}),  and the other $K_{i,\ell}$'s are trivial (Figure \ref{fig:single-strand-crossing-the-biangle} and its inverse, with opposite orientation).
\end{itemize}
(Compare \cite[\S 4, proof of Lemma 15]{BonahonGT11}.)  
For each $i=0,1,\dots,p-1$ and any state $s_i$ on $K_i$, define 
\begin{equation*}
     \mathrm{Tr}_\mathfrak{B}^\omega(K_i,s_i) \overset{\text{def}}{=} \prod_{\ell} \mathrm{Tr}_\mathfrak{B}^\omega(K_{i,\ell}, s_i|_{K_{i,\ell}}) \in \mathbb{C},     
\end{equation*}
      see \S \ref{sssec:U-turns}, \ref{sssec:crossings}, \ref{sssec:trivial-strand}.  We then define the number     
\begin{equation*}
	\mathrm{Tr}_\mathfrak{B}^\omega(K,s) \overset{\text{def}}{=}
	\sum_{\textnormal{compatible } s_0, s_1, \dots, s_{p-1}}
	\prod_{i=0,1,\dots,p-1} \mathrm{Tr}_\mathfrak{B}^\omega(K_i,s_i)
	  \in  \mathbb{C}.     
\end{equation*}
Here, the states $s_i$ are compatible (Definition \ref{def:compatible}) if $s|_{K \cap (\left\{0\right\} \times \mathbb{R} \times (0, 1))}=s_0|_{K_0 \cap (\left\{0\right\} \times \mathbb{R} \times (0, 1))}$ and \\ $s_i|_{K_i \cap (\left\{ x_{i+1}\right\} \times \mathbb{R} \times (0, 1))}= s_{i+1}|_{K_{i+1} \cap (\left\{ x_{i+1}\right\} \times \mathbb{R} \times (0, 1))}$ and $s_{p-1}|_{K_{p-1} \cap (\left\{ 1\right\} \times \mathbb{R} \times (0, 1))}=s|_{K \cap (\left\{ 1 \right\} \times \mathbb{R} \times (0, 1))}$.

Define $\mathrm{Tr}^\omega_\mathfrak{B}(K, s) \in \mathbb{C}$ for a general stated blackboard-framed oriented link $(K,s)$, possibly with kinks (\S \ref{sec:framed-oriented-links}), as follows.    By isotopy, sliding the link horizontally along the boundary of $\mathfrak{B} \times (0,1)$ while preserving blackboard framing throughout, we can arrange that the boundary $\partial K$ satisfies the `higher point, larger second coordinate' condition assumed just above.  (Note this sliding might introduce or remove kinks.)  Let $K^\prime$ be the  blackboard-framed link without kinks obtained by removing the kinks of $K$ (more precisely, pulling tight by homotopy--not isotopy--the kinks in the un-framed link underlying $K$), and further isotoped into a bridge position as above.  Then we have defined a complex number $\mathrm{Tr}^\omega_\mathfrak{B}(K^\prime, s) \in \mathbb{C}$.  Define $\mathrm{Tr}^\omega_\mathfrak{B}(K, s)\in\mathbb{C}$ by modifying $\mathrm{Tr}^\omega_\mathfrak{B}(K^\prime, s)$ according to the un-kinking `skein relations' shown in Figures \ref{fig:positive-kink-skein-relation} and \ref{fig:negative-kink-skein-relation}.    (For example, in the case $n=3$, $\bar{\zeta}_3 = +q^{-8/3}$.)  More precisely, for $P$ (resp. $N$) the number of positive (resp. negative) kinks of $K$,
\begin{equation*}
     \mathrm{Tr}^\omega_\mathfrak{B}(K, s)\overset{\text{def}}{=}(\bar{\zeta}_n)^{P-N} \mathrm{Tr}^\omega_\mathfrak{B}(K^\prime, s) \in\mathbb{C}.     
\end{equation*}

\begin{proposition}[$\mathrm{SL}_n$-biangle quantum trace map]
\label{prop:reshetikhin-turaev}
	Let $\mathfrak{B}\times (0,1)$ be the (non-parametrized) thickened biangle.  Then the construction of the current section determines a function
	\begin{equation*}
		\mathrm{Tr}^\omega_\mathfrak{B}  :
		\left\{  
			\textnormal{stated framed oriented links } (K,s) \textnormal{ in } \mathfrak{B} \times (0, 1)
		\right\} 
		\to 
		\mathbb{C},
	\end{equation*}
satisfying the following properties:  
\begin{enumerate}[\normalfont(A)]
	\item  the number $\mathrm{Tr}_\mathfrak{B}^\omega(K, s) \in \mathbb{C}$ is invariant under isotopy of stated framed oriented links;
	\item  the $\mathrm{SL}_n$-HOMFLYPT skein relation (Figure {\upshape\ref{fig:HOMFLYPT-relation}}) holds;
	\item the $\mathrm{SL}_n$-quantum unknot and framing relations (Figures {\upshape\ref{fig:unknot-relation}} and {\upshape\ref{fig:framing-skein-relations}}) hold.  
\end{enumerate}
Moreover, this invariant satisfies the Multiplication Property 
\begin{equation*}
	\label{eq:quantum-trace-multiplicative-propertybigon}
		\mathrm{Tr}^\omega_\mathfrak{B}(K, s) = \prod_{j=1}^\ell \mathrm{Tr}^\omega_\mathfrak{B}(K_j, s_j)
		  \in  
		\mathbb{C},     
	\end{equation*}
     in the case where $K$ is the disjoint union of links $K_j$ having mutually non-overlapping heights (note the order of multiplication is immaterial, in contrast to Complement {\upshape\ref{compl}}), as well as 	
 the State Sum Property     
\begin{equation*}
	\mathrm{Tr}_\mathfrak{B}^\omega(K,s)=
	\sum_{\textnormal{compatible } s_1, s_2}
	\mathrm{Tr}_{\mathfrak{B}_1}^\omega(K_1, s_1) \mathrm{Tr}_{\mathfrak{B}_2}^\omega(K_2, s_2)
	  \in  \mathbb{C},     
\end{equation*}
where $K_1$ and $K_2$ are the links obtained by cutting the biangle $\mathfrak{B}$ into two biangles $\mathfrak{B}_1$ and $\mathfrak{B}_2$.
\end{proposition}

For the proof of part (A) of the proposition, namely the isotopy invariance, we refer the reader to Appendix \ref{sec:proof-of-reshetikhin-turaev-invariant}, which makes use of a `symmetric' specialization of the Reshetikhin--Turaev invariant (see Figure \ref{fig:DandG}).  The skein relations, parts (B) and (C), are local calculations; see \S \ref{sec:HOMFLYPT-skein-relation}.

Note that (assuming isotopy invariance) the State Sum Property is immediate from the construction.  It suffices to establish the Multiplication Property when $\ell=2$ (assuming $K_1$ lies below $K_2$, say).  This follows from the State Sum Property and the definition for the trivial strand (Figure \ref{fig:single-strand-crossing-the-biangle}) by choosing a parametrization of $\mathfrak{B}\times(0,1)\cong [0, 1] \times \mathbb{R} \times (0, 1)$, the partition $0<1/2<1$ of $[0,1]$, and a position for the links such that $(K_1,s_1)\subset[0,1]\times\mathbb{R}\times(0,1/2)$ (resp. $(K_2,s_2)\subset[0,1]\times\mathbb{R}\times(1/2,1)$) with only trivial strands in $[1/2,1]\times\mathbb{R}\times(0,1/2)$ (resp. $[0,1/2]\times\mathbb{R}\times(1/2,1)$).  (Compare \cite[Lemma 19]{BonahonGT11}.)

We note, in particular, that the biangle quantum trace map is independent of the choice of parametrization of the biangle $\mathfrak{B}$, discussed at the beginning of \S \ref{ssec:biangles-and-the-reshetikhin-turaev-invariant}.   Indeed, this property follows either from the properties of the symmetric specialization of the Reshetikhin--Turaev invariant (Appendix \ref{sec:proof-of-reshetikhin-turaev-invariant}), or directly from the symmetries of the matrices corresponding to the links displayed in Figures \ref{fig:decreasing-U-turns}, \ref{fig:increasing-U-turns}, \ref{fig:same-direction-crossings}, \ref{fig:opposite-direction-crossings}, \ref{fig:single-strand-crossing-the-biangle}.  

\begin{remark}\label{rem:goingtowebs}
	The Reshetikhin--Turaev invariant can be defined more generally for \textit{ribbon graphs}, including so-called \textit{webs}  \cite{KuperbergCommMathPhys96,SikoraAlgGeomTop05}.  In \cite{BonahonGT11},  the $\mathrm{SL}_2$-quantum trace is defined by splitting the edges of the ideal triangulation $\lambda$ to form biangles and then ``pushing all of the complexities of the link into the biangles,'' \cite[p.1596]{BonahonGT11} leaving only flat arcs lying over the triangles.  In order to construct the $\mathrm{SL}_n$-quantum trace for webs, one can perform the same procedure, in particular pushing all of the vertices of the web into the biangles.  Then, the Reshetikhin--Turaev invariant can be applied to the webs in the biangles and (as we will see in the next section) the Fock--Goncharov matrices can be associated to the arcs lying over the triangles. This is essentially the strategy employed in \cite{KimArxiv20} in the case $n=3$.  
\end{remark}
	
\subsection{Definition of the \texorpdfstring{$\mathrm{SL}_n$}{SLn(C)}-quantum trace polynomials}
\label{sec:def-of-quantum-trace}

Our construction of the quantum trace map in the  general $n$ case will follow exactly the same procedure as explained in \cite[\S 3.4-6]{BonahonGT11} for the case $n=2$, where our Proposition \ref{prop:reshetikhin-turaev}  plays the role of Proposition 13 in \cite[\S 4]{BonahonGT11}.  It remains to discuss how Property (2)(a) of Theorem 11 in \cite[\S 3.4]{BonahonGT11}, concerning the values of the quantum traces for arcs in triangles (see Remark \ref{rem:goingtowebs}), generalizes to our setting, which we have essentially already done.  (In particular, we follow the `ordered lower to higher, multiply left to right' convention of \cite{BonahonGT11} for ordering the non-commutative variables associated to different arc components lying over a single triangle.)  After this, the rest of the construction is identical to \cite[\S 6, pp. 1600-1601]{BonahonGT11}, where the quantum trace for a general triangulated surface $(\mathfrak{S},\lambda)$ is defined as a state sum over the triangles of the ideal triangulation $\lambda$ of $\mathfrak{S}$.   We proceed below to spell all of this out in greater detail.  
Now is the point where we require the choice of square root $\omega^{1/2}$; see Remark \ref{rem:choice-of-square-root}. 

\subsubsection{Arcs in a triangle}  
\label{sssec:arcs-in-a-triangle}

Generalizing Property (2)(a) of Theorem 11 in \cite[\S 3.4]{BonahonGT11} to the case of general $n$ is accomplished by using the quantum left and right matrices $\vec{L}^\omega$ and $\vec{R}^\omega$, with coefficients in the Fock--Goncharov quantum torus $\mathscr{T}_n^\omega(\mathfrak{T})$ for a triangle $\mathfrak{T}$ in the ideal triangulation $\lambda$, appearing earlier in  Theorem \ref{thm:first-theorem}.  

Consider a single extended left-moving or right-moving arc crossing the triangle between two distinct boundary edges, such as those shown in Figure \ref{fig:left-and-right-quantum-matrices}; see \S \ref{sssec:quantum-left-and-right-matrices}.  
For example, in the case $n=3$, these extended left-moving and right-moving arcs are displayed  in Figures \ref{fig:quantum-left-matrix} and \ref{fig:quantum-right-matrix}.  Using the notation from these figures, the $n=3$ quantum left and right matrices $\vec{L}^\omega = \vec{L}^\omega(W, Z, W^\prime, Z^\prime, X)$ and $\vec{R}^\omega = \vec{R}^\omega(W, Z, W^\prime, Z^\prime, X)$ are given by 
\begin{equation*}
\label{eq:quantum-left-matrix-n=3}
	\vec{L}^\omega(W,Z,W^\prime,Z^\prime,X) \overset{\text{def}}{=} 
	\left(\begin{smallmatrix}
[D_L^{-1/3} W Z X Z^\prime W^\prime]
&[D_L^{-1/3} W Z X W^\prime]+[D_L^{-1/3} W Z W^\prime]
&[D_L^{-1/3} W Z]
\\0
&[D_L^{-1/3} Z W^\prime]
&[D_L^{-1/3} Z]
\\0
&0
&[D_L^{-1/3}]
	\end{smallmatrix}\right)\in\mathrm{SL}_3^q(\mathscr{T}^\omega_3(\mathfrak{T})),
\end{equation*}
where $D_L^{-1/3}$ in $\mathscr{T}_3^\omega(\mathfrak{T})$ is defined by
\begin{equation*}
	D_L^{-1/3} \overset{\text{def}}{=} W^{-1/3} Z^{-2/3} X^{-1/3} Z^{\prime-1/3} W^{\prime-2/3}\in\mathscr{T}_3^\omega(\mathfrak{T}),
\end{equation*}
and 
\begin{equation*}
\label{eq:quantum-right-matrix-n=3}
	\vec{R}^\omega(W,Z,W^\prime,Z^\prime,X) \overset{\text{def}}{=} 
	\left(\begin{smallmatrix}
[D_R^{-1/3} W^\prime Z^\prime Z W]
&0
&0
\\ [D_R^{-1/3} Z^\prime Z W]
&[D_R^{-1/3} Z^\prime W]
&0
\\ [D_R^{-1/3} Z W]
&[D_R^{-1/3} W] + [D_R^{-1/3} X^{-1} W]
&[D_R^{-1/3} X^{-1}]
\end{smallmatrix}\right)\in\mathrm{SL}_3^q(\mathscr{T}^\omega_3(\mathfrak{T})),
\end{equation*}
where $D_R^{-1/3}$ in $\mathscr{T}_3^\omega(\mathfrak{T})$ is defined by
\begin{equation*}
	D_R^{-1/3} \overset{\text{def}}{=} W^{\prime-1/3} Z^{\prime-2/3} X^{1/3} Z^{-1/3} W^{-2/3}\in\mathscr{T}_3^\omega(\mathfrak{T}).
\end{equation*}
(This is the result of multiplying out the snake-move matrices in the case $n=3$; compare  \S \ref{sssec:n=3-example}.)

\begin{definition}
For general $n$, define for each pair of states $s_1, s_2 \in \{  1, 2, \dots, n  \}$ two elements in the quantum torus $\mathscr{T}_n^\omega(\mathfrak{T})$
\begin{equation*}
	\mathrm{Tr}^\omega_\mathfrak{T}(L)_{s_1}^{s_2}, \,\, \mathrm{Tr}^\omega_\mathfrak{T}(R)_{s_1}^{s_2}  \in \mathscr{T}_n^\omega(\mathfrak{T}),
\end{equation*}
by the matrix equations (see \S \ref{ssec:matrix-conventions})
\begin{gather*}
	(\mathrm{Tr}^\omega_\mathfrak{T}(L)_{s_1}^{s_2}) \overset{\text{def}}{=} \vec{L}^\omega
	  \in  \mathrm{SL}_n^q(\mathscr{T}_n^\omega(\mathfrak{T}))
\subset	
	\mathrm{M}_n(\mathscr{T}_n^\omega(\mathfrak{T})),
\\
	(\mathrm{Tr}^\omega_\mathfrak{T}(R)_{s_1}^{s_2}) \overset{\text{def}}{=} \vec{R}^\omega
	  \in  
	\mathrm{SL}_n^q(\mathscr{T}_n^\omega(\mathfrak{T}))
	\subset \mathrm{M}_n(\mathscr{T}_n^\omega(\mathfrak{T})).
\end{gather*} 
\end{definition}

\begin{remark}
\label{rem:choice-of-square-root}
	In the above matrices, we recall that the square brackets surrounding the monomials indicate that we are taking the Weyl quantum ordering, which depends on the quiver defining the $q$-commutation relations in the Fock--Goncharov quantum torus $\mathscr{T}_n^\omega(\mathfrak{T})$; see \S \ref{sssec:weyl-ordering} and Figure \ref{fig:Fg-quiver}.  It is here that the choice of $\omega^{1/2}$ enters into the construction.

In Theorem \ref{thm:first-theorem}, we saw that the quantum left and right matrices $\vec{L}^\omega$ and $\vec{R}^\omega$ are points of the quantum special linear group $\mathrm{SL}_n^q$.  Note that, in order for these matrices to satisfy even just the relations required to be in the quantum matrix algebra $\mathrm{M}_n^q$, they had to be normalized by `dividing out' their determinants.  For example, the above $n=3$ version of the matrix $\vec{L}^\omega$ would not satisfy the $q$-commutation relations required to be a point of $\mathrm{M}_3^q$ if we had instead put $D_L=1$.  
\end{remark}

\begin{figure}[htb]
     \centering
     \begin{subfigure}{0.49\textwidth}
         \centering
         \includegraphics[width=.49\textwidth]{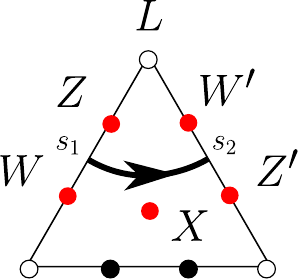}
         \caption{Left}
         \label{fig:quantum-left-matrix}
     \end{subfigure}     
\hfill
     \begin{subfigure}{0.49\textwidth}
         \centering
         \includegraphics[width=.49\textwidth]{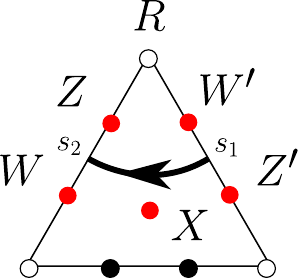}
         \caption{Right}
         \label{fig:quantum-right-matrix}
     \end{subfigure}
     	\caption{Quantum left and right matrices for $n=3$.}
        \label{fig:quantum-left-and-right-matrices}
\end{figure}

\subsubsection{Good position of a link}
\label{sssec:good-position}

Fix an ideal triangulation $\lambda$ of $\mathfrak{S}$.  Form the corresponding \textit{split ideal triangulation} $\widehat{\lambda}$ of $\mathfrak{S}$ by `splitting' each edge $E$ of $\lambda$ into a biangle $\mathfrak{B}_E$.  (Compare \cite[\S 5]{BonahonGT11}.)  See Figure \ref{fig:split-ideal-triangulation}.  For notational simplicity, we identify the triangles of $\lambda$ with the triangles of $\widehat{\lambda}$. 
A framed link $K$ is said to be in \textit{good position} with respect to the split ideal triangulation $\widehat{\lambda}$ if:
\begin{itemize}
	\item  the link $K$ is transverse to $\widehat{E} \times (0, 1)$ for each edge $\widehat{E}$ of $\widehat{\lambda}$;
	\item  for each triangle $\mathfrak{T}_j$ of $\widehat{\lambda}$, the intersection $K \cap (\mathfrak{T}_j \times (0, 1)) = K_j = \cup_{\ell} K_{j,\ell}$ consists of a disjoint union of arcs $K_{j,\ell}$, each connecting distinct sides of $\mathfrak{T}_j \times (0, 1)$;
	\item  the arcs $K_{j,\ell}$ are `flat', in the sense that each arc has a constant height with respect to the vertical coordinate of $\mathfrak{T}_j \times (0, 1)$ and has the blackboard framing.  
\end{itemize}
(Compare \cite[Lemma 23]{BonahonGT11}.)  In particular, when in good position, all of the `complexity' of the link $K$ resides in the thickened biangles $\cup_i (\mathfrak{B}_i \times (0, 1))$.  
A \textit{good position move} between framed oriented links $K_1$ and $K_2$ in good position is one of the oriented versions of the local moves depicted in Figures 15-19 in \cite[\S 5]{BonahonGT11}.  In the present article, these moves are displayed in Figures \ref{fig:Move-I}, \ref{fig:Move-Ib}, \ref{fig:Move-Ic}, \ref{fig:Move-Id} (Move I); Figures \ref{fig:Move-II}, \ref{fig:Move-IIbcorrect} (Move II); Figures \ref{fig:Move-III}, \ref{fig:Move-IIIb}, \ref{fig:Move-IIIc}, \ref{fig:Move-IIId} (Move III); Figures \ref{fig:Move-IV}, \ref{fig:Move-IVb}, \ref{fig:Move-IVc}, \ref{fig:Move-IVd} (Move IV); and, Figure \ref{fig:Move-V} (Move V).  

The proof of the following fact is the same as the proof of the corresponding un-oriented version (\cite[Lemma 24]{BonahonGT11}).

\begin{fact}
	Any framed oriented link $K$ has a good position with respect to the split ideal triangulation $\widehat{\lambda}$.  Any two isotopic framed oriented links $K_1$ and $K_2$  in good position  are related by a sequence of good position moves and their inverses (and isotopies through framed oriented links in good position).  \qed
\end{fact}

\begin{figure}[htb]
	\centering
	\includegraphics[scale=.65]{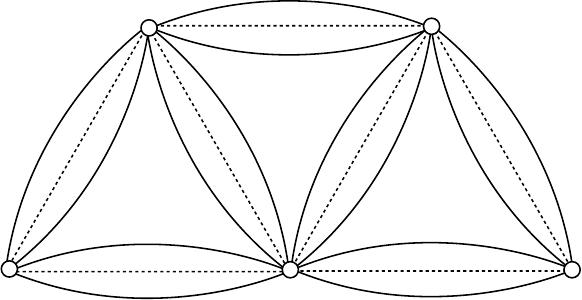}
	\caption{Split ideal triangulation}
	\label{fig:split-ideal-triangulation}
\end{figure}

\subsubsection{General case}
\label{sssec:general-case}

Let $K$ be any  blackboard-framed oriented link (recall that a framed link can be isotoped to have the blackboard framing by possibly introducing kinks).  By \S \ref{sssec:good-position}, we may assume that $K$ is in good position with respect to the split ideal triangulation $\widehat{\lambda}$.  Let the biangles of $\widehat{\lambda}$ be denoted $\mathfrak{B}_i$ for $i=1,2,\dots,p$ and let the triangles be denoted $\mathfrak{T}_j$ for $j=1,2,\dots,m$.  Put $L_i = K \cap (\mathfrak{B}_i \times (0, 1))$ and $K_j = K \cap (\mathfrak{T}_j \times (0, 1))$.  By definition of good position, $K_j = K_{j,1} \cup K_{j,2} \cup \dots \cup K_{j,{\ell_j}}$ where each component $K_{j,\ell}$ is a flat oriented arc connecting distinct sides of $\mathfrak{T}_j \times (0, 1)$.  Choose indices such  that $K_{j,\ell}$ lies below $K_{j,{\ell+1}}$ with respect to the height order of $\mathfrak{T}_j \times (0, 1)$.  
For any state $s_j$ on $K_j$,  by \S \ref{sssec:arcs-in-a-triangle}, the triangle quantum torus elements $\mathrm{Tr}_\mathfrak{T}^\omega(K_{j,\ell}, s_j|_{K_{j,\ell}}) \in \mathscr{T}_n^\omega(\mathfrak{T}_j)$ are defined for $\ell=1,2,\dots,\ell_j$.  Assign such a quantum torus element to the  stated link $(K_j,s_j)$ by
\begin{equation*}
	\mathrm{Tr}_{\mathfrak{T}_j}^\omega(K_j, s_j) \overset{\text{def}}{=}
	\mathrm{Tr}_{\mathfrak{T}_j}^\omega(K_{j,1}, s_j|_{K_{j,1}})
	\mathrm{Tr}_{\mathfrak{T}_j}^\omega(K_{j,2}, s_j|_{K_{j,2}}) 
	\cdots 
	\mathrm{Tr}_{\mathfrak{T}_j}^\omega(K_{j,{\ell_j}}, s_j|_{K_{j,{\ell_j}}})
	  \in  \mathscr{T}_n^\omega(\mathfrak{T}_j).
\end{equation*}
Note, importantly, the order in which the non-commuting elements $\mathrm{Tr}_\mathfrak{T}^\omega(K_{j,\ell}, s_j|_{K_{j,\ell}})$ are multiplied, the convention being `ordered lower to higher, multiply left to right'.  
For any state $t_i$ on $L_i$, let the numbers $\mathrm{Tr}^\omega_{\mathfrak{B}_i}(L_i, t_i) \in \mathbb{C}$ be defined by Proposition \ref{prop:reshetikhin-turaev}.  

\begin{definition}
\label{def:main-definition-qtracemap}
	Let $(K,s)$ be a stated blackboard-framed oriented link in $\mathfrak{S}\times(0,1)$  in good position with respect to the split ideal triangulation $\widehat{\lambda}$.  The \textit{$\mathrm{SL}_n$-quantum trace polynomial} $\mathrm{Tr}_\lambda^\omega(K,s)$  is defined by
\begin{equation*}
	\mathrm{Tr}_\lambda^\omega(K,s) \overset{\text{def}}{=} 
	\sum_{\text{compatible }t_1,t_2,\dots,t_p,s_1,s_2,\dots,s_m}  \left(
	\prod_{i=1}^p \mathrm{Tr}^\omega_{\mathfrak{B}_i}(L_i, t_i)
	\right)
	\left( 
	\bigotimes_{j=1}^m \mathrm{Tr}^\omega_{\mathfrak{T}_j}(K_j, s_j)
	\right)  
	\in \bigotimes_{\textnormal{triangles } \mathfrak{T}_j } \mathscr{T}_n^\omega(\mathfrak{T}_j),
\end{equation*}
where the compatibility condition (Definition \ref{def:compatible}) for the states $t_i$ and $s_j$ with respect to the state $s$ and the split triangulation $\widehat{\lambda}$ is analogous to that in \S \ref{sssec:kinks-and-the-Reshetikhin-Turaev-invariant}.
(Compare \cite[\S 6, p. 1600-1601]{BonahonGT11}.)  Note  that the quantities $\mathrm{Tr}^\omega_{\mathfrak{T}_j}(K_j, s_j)$ commute in the tensor product, since they lie in different tensor factors $\mathscr{T}_n^\omega(\mathfrak{T}_j) \subset \bigotimes_\mathfrak{T} \mathscr{T}_n^\omega(\mathfrak{T})$.  
\end{definition}

This completes the construction of the $\mathrm{SL}_n$-quantum trace map  for links.  One would 
 still need to  show it is well-defined, that is, independent of the choice of good position (equivalently, independent of isotopy); see \S \ref{sec:computer-check-of-local-moves} for a proof in the case $n=3$.  
	
\subsection{Properties}
\label{ssec:properties}

Assuming isotopy invariance, we conclude this section with a few observations.  

The above state sum definition of the quantum trace polynomial  takes as input a stated framed oriented link $(K, s)$ and outputs an element of the tensor product $\bigotimes_{\mathfrak{T}} \mathscr{T}^\omega_n(\mathfrak{T})$.  We had indicated earlier (Theorem \ref{thm:second-theorem}) that the image should lie in the Fock--Goncharov quantum torus sub-algebra $\mathscr{T}^\omega_n(\lambda) \subset \bigotimes_{\mathfrak{T}} \mathscr{T}^\omega_n(\mathfrak{T})$; see \S \ref{ssec:quantum-tori-for-surfaces}.  The following fact is justified by a straightforward analysis of the structure of the local U-turn, crossing, left, and right matrices.  (Compare \cite[Lemma 25]{BonahonGT11}.  See also \cite{KimArxiv20}, where a  stronger property is established.)

\begin{fact}
	The \textit{quantum trace polynomial} $\mathrm{Tr}_\lambda^\omega(K,s)$ is an element of $\mathscr{T}^\omega_n(\lambda)$.  \qed
\end{fact} 

\begin{proof}[Proof of (the general $n$ version of) Complement {\upshape\ref{compl}}]
The Classical Trace Property is by construction, comparing with the classical matrices of \S \ref{sec:classical-fock-goncharov-coordinates}.  The State Sum Property is immediate from the construction (assuming isotopy invariance).  The Multiplication Property follows from the State Sum Property by the corresponding property for biangles (Proposition \ref{prop:reshetikhin-turaev}), together with the definitions of good position and the quantities $\mathrm{Tr}^\omega_{\mathfrak{T}_j}(K_j, s_j)\in\mathscr{T}_n^\omega(\mathfrak{T}_j)$.  (Compare \cite[\S 6, p.1609]{BonahonGT11}.)  
\end{proof}

We remark that the quantum trace $\mathrm{Tr}^\omega_\lambda(K, s)$ of a stated framed oriented link $(K, s)$ can be thought of as a tensor having dimension equal to the number of boundary points $p_i \in \partial K$ of the link, each associated to a state $s_i$.  If the states $s_i$ are  partitioned into two groups $s_{i_1}, \dots, s_{i_\ell}$ and $s_{j_1}, \dots, s_{j_m}$, then the quantum trace of the link can be written as a matrix $(\mathrm{Tr}^\omega_\lambda(K, s)_{s_{i_1},\dots,s_{i_\ell}}^{s_{j_1},\dots,s_{j_m}})$ with coefficients in $\mathscr{T}_n^\omega(\lambda)$; see \S \ref{sec:computer-check-of-local-moves} for examples.  

\subsubsection{Skein relations}
\label{sec:HOMFLYPT-skein-relation}
	
We  justify parts (B)-(C) in (the general $n$ version of) Theorem \ref{thm:second-theorem}.  

The first skein relation is the well-known ($q$-evaluated) HOMFLYPT relation from knot theory \cite{FreydBullAmerMathSoc85,Przytycki87ProcAmercMathSoc}.  The $R$-matrices for the quantum group $\mathrm{SL}_n^q$ satisfy this skein relation.  For us, this relation appears with the normalization displayed in Figure \ref{fig:HOMFLYPT-relation}.  One can check from, say, Figures \ref{fig:cross-pos-same-over-to-lower}, \ref{fig:cross-neg-same-over-to-higher}, \ref{fig:single-strand-crossing-the-biangle} together with the definitions of \S \ref{sssec:crossings} and \S \ref{sssec:trivial-strand} that the quantum trace map $\mathrm{Tr}^\omega_\lambda$ satisfies this skein relation, translating to the matrix equation
\begin{equation*}
	q^{-1/n} \vec{C}^q_\text{same} 
	- q^{+1/n} (\vec{C}^q_\text{same})^{-1}
	=
	(q^{-1}-q)  \vec{Id}_{n^2}
	  \in  \mathrm{M}_{n^2}(\mathbb{C}).  
\end{equation*}
The second skein relation, coming from the U-turn `duality' matrices (Figures \ref{fig:decreasing-U-turns} and \ref{fig:increasing-U-turns}), says that the contractible untwisted unknot $K$ evaluates to $(-1)^{n-1}$ times the quantum integer $[n]_q = (q^n - q^{-n})/(q-q^{-1}) = \sum_{k=1}^n q^{2k -n -1}$; see Figure \ref{fig:unknot-relation}.  The third skein relation consists of the positive and negative framing relations; see Figure \ref{fig:framing-skein-relations}.  
	
\section{Isotopy invariance: proof of the main theorem}
\label{sec:computer-check-of-local-moves}
	
\begin{proof}[Proof of Theorem {\upshape\ref{thm:second-theorem}}]
Parts (B)-(C) were discussed  above.  In this section, we will establish part (A):  the $\mathrm{SL}_3$-quantum trace map is invariant under isotopy.   It suffices to check the oriented good position moves; see \S \ref{sssec:good-position}.  We do this `by hand', using computer assistance.  
\end{proof}

The more difficult moves are those of type (II) and (IV).  Indeed, (I) can be computed directly from the definitions (although it is still somewhat non-trivial), (III) is essentially equivalent to Theorem \ref{thm:first-theorem}, and (V)  is equivalent to the kink-removing skein relations appearing in Figure \ref{fig:framing-skein-relations}.  However, below we will justify moves (I), (III), and (V) as well.  
	
\begin{remark}
\label{rem:chekhov-shapiro}
	In the general case of $\mathrm{SL}_n$, a  proof of essentially these same algebraic identities (including Theorem \ref{thm:first-theorem}), which are equivalent to the local isotopy moves discussed in this section, is given in \cite{chekhovMR4597214} (motivated by \cite{SchraderInvent19, SchraderArxiv17} and earlier by \cite{Fock06, GekhtmanSelMathNewSer09}) in the context of quantum integrable systems; see also \cite{GoncharovArxiv19}.  Consequently, these works can be applied to finish the proof of the general $n$ version of Theorem \ref{thm:second-theorem}.
\end{remark}

\subsection{Notation}
\label{ssec:proofnotation}  

  Throughout this section, we will be considering a single triangle with 7 coordinates, denoted as in Figure \ref{fig:Move-II}.  (Note that the coordinates we are currently labeling as $W_2$, $Z_2$, $W_3$, $Z_3$, $X$ were labeled, respectively, $W$, $Z$, $W^\prime$, $Z^\prime$, $X$ in \S \ref{sssec:arcs-in-a-triangle}.)  We define matrices $\vec{L}^\omega(W_2, Z_2, W_3, Z_3, X)$ and $\vec{R}^\omega(W_2, Z_2, W_3, Z_3, X)$ in $\mathrm{SL}_3^q(\mathscr{T}^\omega_3(\mathfrak{T}))\subset\mathrm{M}_3(\mathscr{T}^\omega_3(\mathfrak{T})$ by the same formulas as in \S \ref{sssec:arcs-in-a-triangle}.  These matrices are considered as functions of the ordered $5$-tuple $(W_2, Z_2, W_3, Z_3, X)$.  For example, we may also consider a matrix $\vec{R}^\omega(W_3, Z_3, W_1, Z_1, X)$ corresponding to the right turn in Figure \ref{fig:Move-II}.  
	
\subsection{Move (I)}
\label{sssec:move-1}
	
In Figure \ref{fig:Move-I}, we show one of the oriented versions of Move (I).  Let $K$ be the link on the left, and $K^\prime$ the link on the right.  According to the definition of the quantum trace (\S \ref{sec:def-of-quantum-trace}) as a State Sum Formula, the equality expressing Move (I) can be interpreted as an equality of $3 \times 3$ matrices.  Specifically, the claim is that the matrix,
\begin{gather*}
\label{eq:move-i-example}
\tag{$I$}
	( \mathrm{Tr}^\omega_\lambda(K)_{s_1}^{s_2}) =
	\left(\begin{smallmatrix}
		a_1&b_1&c_1\\
		0&e_1&f_1\\
		0&0&i_1
	\end{smallmatrix}\right)
\overbrace{q^{-4/3}}^{\bar{\sigma}_3} \left(\begin{smallmatrix}
		0&0&+q^{-1}\\
		0&-1&0\\
		+q&0&0
	\end{smallmatrix}\right)
	\left(\begin{smallmatrix}
		A_1&0&0\\
		D_1&E_1&0\\
		G_1&H_1&I_1
	\end{smallmatrix}\right)=
\\
	=
	\left(\begin{smallmatrix}
	q^{-1/3} A_1 c_1 - q^{-4/3} D_1 b_1 + q^{-7/3} G_1 a_1 &
	-q^{-4/3} E_1 b_1 + q^{-7/3} H_1 a_1 &
	q^{-7/3} I_1 a_1 \\
	q^{-1/3} A_1 f_1 - q^{-4/3}D_1 e_1&
	-q^{-4/3} E_1 e_1 &
	0\\
	q^{-1/3} A_1 i_1 &
	0&
	0
	\end{smallmatrix}\right)\in\mathrm{M}_3(\mathscr{T}_3^\omega(\mathfrak{T})),
\end{gather*}
is equal to the matrix
\begin{equation*} 
\overbrace{q^{-4/3}}^{\bar{\sigma}_3} \left(\begin{smallmatrix}
		0&0&+q^{-1}\\
		0&-1&0\\
		+q&0&0
	\end{smallmatrix}\right)
	= ( \mathrm{Tr}^\omega_\lambda(K^\prime)_{s_1}^{s_2})
	 \in \mathrm{M}_3(\mathscr{T}_3^\omega(\mathfrak{T})),
\end{equation*}
where we have used Figure \ref{fig:U-turn-dec-cw} and the matrix $(\mathrm{Tr}^\omega_\mathfrak{B}(U_\text{dec}^\text{cw})_{s_1}^{s_2}) = \vec{U}^q$ from \S \ref{sssec:U-turns} for the middle matrix, and where we have put (\S \ref{ssec:proofnotation})
\begin{equation*}
	\left(\begin{smallmatrix}
		A_1&0&0\\
		D_1&E_1&0\\
		G_1&H_1&I_1
	\end{smallmatrix}\right)
	\overset{\text{def}}{=}  \vec{R}^\omega(W_2,Z_2,W_3,Z_3,X),\,\,
	\left(\begin{smallmatrix}
		a_1&b_1&c_1\\
		0&e_1&f_1\\
		0&0&i_1
	\end{smallmatrix}\right)
	\overset{\text{def}}{=}  \vec{L}^\omega(W_2, Z_2, W_3, Z_3, X)
	  \in  \mathrm{M}_3(\mathscr{T}_3^\omega(\mathfrak{T})).
\end{equation*}
See Appendix \ref{sec:the-appendix} for a computer check of the above equality of $3 \times 3$ matrices in $ \mathrm{M}_3(\mathscr{T}_3^\omega(\mathfrak{T}))$ representing this oriented Move (I) example.  
Also checked in Appendix \ref{sec:the-appendix} are the other three oriented versions of Move (I), whose equivalent matrix formulations are displayed in Figures \ref{fig:Move-Ib}, \ref{fig:Move-Ic}, \ref{fig:Move-Id}.  (Note that the reversal of order of the non-commuting variables in the case of Moves (I) and (I.c) is due to  the `ordered lower to higher, multiply left to right' rule; see \S \ref{sec:def-of-quantum-trace}.)  

\begin{gather*}
\label{eq:move-ib-example}
\tag{$I.b$}
	( \mathrm{Tr}^\omega_\lambda(K)_{s_1}^{s_2}) =
	\left(\begin{smallmatrix}
		a_1&b_1&c_1\\
		0&e_1&f_1\\
		0&0&i_1
	\end{smallmatrix}\right) 
	\left(\begin{smallmatrix}
		0&0&+q^{-1/3}\\
		0&-q^{-4/3}&0\\
		+q^{-7/3}&0&0
	\end{smallmatrix}\right)
	\left(\begin{smallmatrix}
		A_1&0&0\\
		D_1&E_1&0\\
		G_1&H_1&I_1
	\end{smallmatrix}\right)=
\\
	=
	\left(\begin{smallmatrix}
	q^{-7/3} c_1 A_1 - q^{-4/3} b_1 D_1 + q^{-1/3} a_1 G_1 &
	-q^{-4/3} b_1 E_1 + q^{-1/3} a_1 H_1 &
	q^{-1/3} a_1 I_1 \\
	q^{-7/3} f_1 A_1 - q^{-4/3}e_1 D_1&
	-q^{-4/3} e_1 E_1 &
	0\\
	q^{-7/3} i_1 A_1 &
	0&
	0
	\end{smallmatrix}\right)
\\
	\overset{?}{=} 
	\left(\begin{smallmatrix}
		0&0&+q^{-1/3}\\
		0&-q^{-4/3}&0\\
		+q^{-7/3}&0&0
	\end{smallmatrix}\right)
	= ( \mathrm{Tr}^\omega_\lambda(K^\prime)_{s_1}^{s_2})
	 \in \mathrm{M}_3(\mathscr{T}_3^\omega(\mathfrak{T})).
\end{gather*}

\begin{gather*}
\label{eq:move-ic-example}
\tag{$I.c$}
	( \mathrm{Tr}^\omega_\lambda(K)_{s_1}^{s_2}) =
	\left(\begin{smallmatrix}
		A_1&0&0\\
		D_1&E_1&0\\
		G_1&H_1&I_1
	\end{smallmatrix}\right)
	\left(\begin{smallmatrix}
		0&0&+q^{1/3}\\
		0&-q^{4/3}&0\\
		+q^{7/3}&0&0
	\end{smallmatrix}\right)
	\left(\begin{smallmatrix}
		a_1&b_1&c_1\\
		0&e_1&f_1\\
		0&0&i_1
	\end{smallmatrix}\right) 
=
\\
	=
	\left(\begin{smallmatrix}
0&0&q^{1/3} i_1 A_1
\\
0&-q^{4/3} e_1 E_1&-q^{4/3}f_1 E_1 + q^{1/3} i_1 D_1
\\
q^{7/3}a_1 I_1 & q^{7/3} b_1 I_1 - q^{4/3} e_1 H_1 & 
q^{7/3} c_1 I_1 - q^{4/3} f_1 H_1 + q^{1/3} i_1 G_1
	\end{smallmatrix}\right)
\\
	\overset{?}{=} 
	\left(\begin{smallmatrix}
		0&0&+q^{1/3}\\
		0&-q^{4/3}&0\\
		+q^{7/3}&0&0
	\end{smallmatrix}\right)
	= ( \mathrm{Tr}^\omega_\lambda(K^\prime)_{s_1}^{s_2})
	 \in \mathrm{M}_3(\mathscr{T}_3^\omega(\mathfrak{T})).
\end{gather*}

\begin{gather*}
\label{eq:move-id-example}
\tag{$I.d$}
	( \mathrm{Tr}^\omega_\lambda(K)_{s_1}^{s_2}) =
	\left(\begin{smallmatrix}
		A_1&0&0\\
		D_1&E_1&0\\
		G_1&H_1&I_1
	\end{smallmatrix}\right)
	\left(\begin{smallmatrix}
		0&0&+q^{7/3}\\
		0&-q^{4/3}&0\\
		+q^{1/3}&0&0
	\end{smallmatrix}\right)
	\left(\begin{smallmatrix}
		a_1&b_1&c_1\\
		0&e_1&f_1\\
		0&0&i_1
	\end{smallmatrix}\right) 
=
\\
	=
	\left(\begin{smallmatrix}
0 & 0 & q^{7/3} A_1 i_1
\\
0 & -q^{4/3} E_1 e_1 & -q^{4/3} E_1 f_1 + q^{7/3} D_1 i_1
\\
q^{1/3} I_1 a_1 & q^{1/3} I_1 b_1 - q^{4/3} H_1 e_1 & q^{1/3} I_1 c_1 - q^{4/3} H_1 f_1 + q^{7/3} G_1 i_1
	\end{smallmatrix}\right)
\\
	\overset{?}{=} 
	\left(\begin{smallmatrix}
		0&0&+q^{7/3}\\
		0&-q^{4/3}&0\\
		+q^{1/3}&0&0
	\end{smallmatrix}\right)
	= ( \mathrm{Tr}^\omega_\lambda(K^\prime)_{s_1}^{s_2})
	 \in \mathrm{M}_3(\mathscr{T}_3^\omega(\mathfrak{T})).
\end{gather*}

\begin{figure}[htb]
	\centering
	\includegraphics[width=.65\textwidth]{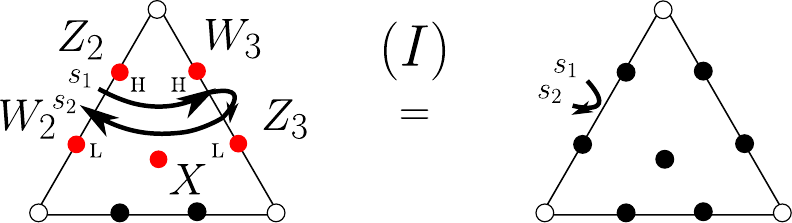}
	\caption{One of the oriented versions of Move (I).}
	\label{fig:Move-I}
\end{figure}

\begin{figure}[htb]
	\centering
	\includegraphics[width=.55\textwidth]{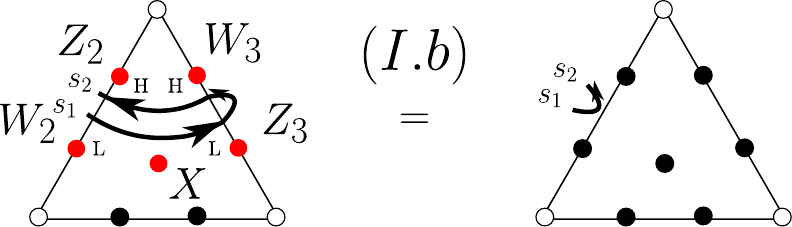}
	\caption{Move (I.b).}
	\label{fig:Move-Ib}
\end{figure}

\begin{figure}[htb]
	\centering
	\includegraphics[width=.55\textwidth]{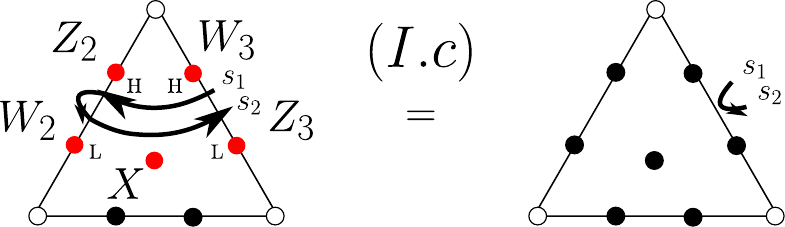}
	\caption{Move (I.c).}
	\label{fig:Move-Ic}
\end{figure}

\begin{figure}[htb]
	\centering
	\includegraphics[width=.55\textwidth]{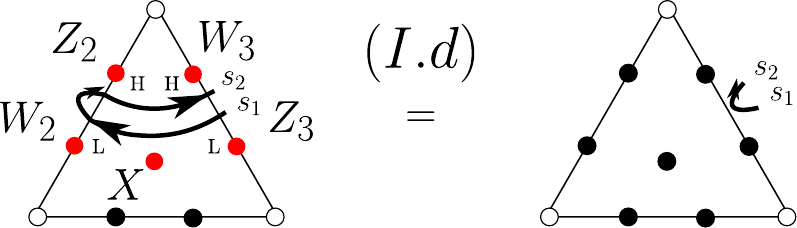}
	\caption{Move (I.d).}
	\label{fig:Move-Id}
\end{figure}

\subsection{Move (II)}
\label{sssec:move-2}
	
In Figure \ref{fig:Move-II}, we show one of the oriented versions of Move (II).  Let $K$ be the link on the left, and $K^\prime$ the link on the right.  According to the definition of the quantum trace (\S \ref{sec:def-of-quantum-trace}) as a State Sum Formula, the equality expressing Move (II) can be interpreted as an equality of $3 \times 3$ matrices.  Specifically, the claim is that the matrix,
\begin{gather*}
\label{eq:move-ii-example}
\tag{$II$}
	( \mathrm{Tr}^\omega_\lambda(K)_{s_1}^{s_2}) =
	\left(\begin{smallmatrix}
		A_2&0&0\\
		D_2&E_2&0\\
		G_2&H_2&I_2
	\end{smallmatrix}\right)
\overbrace{q^{-4/3}}^{\bar{\sigma}_3} \left(\begin{smallmatrix}
		0&0&+q\\
		0&-1&0\\
		+q^{-1}&0&0
	\end{smallmatrix}\right)
	\left(\begin{smallmatrix}
		A_1&0&0\\
		D_1&E_1&0\\
		G_1&H_1&I_1
	\end{smallmatrix}\right)=
\\
	=
	\left(\begin{smallmatrix}
	q^{-1/3} A_2 G_1 &
	q^{-1/3} A_2 H_1 &
	q^{-1/3} A_2 I_1 \\
	- q^{-4/3} E_2 D_1 + q^{-1/3} D_2 G_1&
	-q^{-4/3} E_2 E_1 + q^{-1/3} D_2 H_1 &
	q^{-1/3} D_2 I_1\\
	q^{-7/3} I_2 A_1 - q^{-4/3} H_2 D_1 + q^{-1/3} G_2 G_1&
	- q^{-4/3} H_2 E_1 + q^{-1/3} G_2 H_1&
	q^{-1/3} G_2 I_1
	\end{smallmatrix}\right)\in\mathrm{M}_3(\mathscr{T}_3^\omega(\mathfrak{T})),
\end{gather*}
is equal to the matrix
\begin{equation*}
\left(\begin{smallmatrix}
		a_3&b_3&c_3\\
		0&e_3&f_3\\
		0&0&i_3
	\end{smallmatrix}\right)
	= ( \mathrm{Tr}^\omega_\lambda(K^\prime)_{s_1}^{s_2})
	 \in \mathrm{M}_3(\mathscr{T}_3^\omega(\mathfrak{T})),
\end{equation*}
where we have used Figure \ref{fig:U-turn-inc-ccw} and the matrix $(\mathrm{Tr}^\omega_\mathfrak{B}(U_\text{inc}^\text{ccw})_{s_1}^{s_2}) = (\vec{U}^q)^\mathrm{T}$ from \S \ref{sssec:U-turns} for the middle matrix, and where we have put (\S \ref{ssec:proofnotation})
\begin{gather*}
	\left(\begin{smallmatrix}
		A_2&0&0\\
		D_2&E_2&0\\
		G_2&H_2&I_2
	\end{smallmatrix}\right)
	\overset{\text{def}}{=}  \vec{R}^\omega(W_3,Z_3,W_1,Z_1,X),\,\,
	\left(\begin{smallmatrix}
		A_1&0&0\\
		D_1&E_1&0\\
		G_1&H_1&I_1
	\end{smallmatrix}\right)
	\overset{\text{def}}{=}  \vec{R}^\omega(W_2,Z_2,W_3,Z_3,X),
\\
	\left(\begin{smallmatrix}
		a_3&b_3&c_3\\
		0&e_3&f_3\\
		0&0&i_3
	\end{smallmatrix}\right)
	\overset{\text{def}}{=}  \vec{L}^\omega(W_1, Z_1, W_2, Z_2, X)
	  \in  \mathrm{M}_3(\mathscr{T}_3^\omega(\mathfrak{T})).
\end{gather*}
See Appendix \ref{sec:the-appendix} for a computer check of the above equality of $3 \times 3$ matrices in $ \mathrm{M}_3(\mathscr{T}_3^\omega(\mathfrak{T}))$ representing this oriented Move (II) example.  
Also checked in Appendix \ref{sec:the-appendix} is the other oriented version of Move (II), whose equivalent matrix formulation is displayed in Figure \ref{fig:Move-IIbcorrect}.  

\begin{gather*}
\label{eq:move-iib-example}
\tag{$II.b$}
	( \mathrm{Tr}^\omega_\lambda(K)_{s_1}^{s_2}) =
	\left(\begin{smallmatrix}
		a_1&b_1&c_1\\
		0&e_1&f_1\\
		0&0&i_1
	\end{smallmatrix}\right) \left(\begin{smallmatrix}
		0&0&+q^{-7/3}\\
		0&-q^{-4/3}&0\\
		+q^{-1/3}&0&0
	\end{smallmatrix}\right)
	\left(\begin{smallmatrix}
		a_2&b_2&c_2\\
		0&e_2&f_2\\
		0&0&i_2
	\end{smallmatrix}\right)=
\\
	=
	\left(\begin{smallmatrix}
q^{-1/3} a_2 c_1 & q^{-1/3} b_2 c_1 - q^{-4/3} e_2 b_1 & q^{-1/3} c_2 c_1 - q^{-4/3} f_2 b_1 + q^{-7/3} i_2 a_1
\\
q^{-1/3} a_2 f_1 & q^{-1/3} b_2 f_1 - q^{-4/3} e_2 e_1 & q^{-1/3} c_2 f_1 - q^{-4/3} f_2 e_1
\\
q^{-1/3} a_2 i_1 & q^{-1/3} b_2 i_1 & q^{-1/3} c_2 i_1
	\end{smallmatrix}\right)
\\
	\overset{?}{=} \left(\begin{smallmatrix}
		A_3&0&0\\
		D_3&E_3&0\\
		G_3&H_3&I_3
	\end{smallmatrix}\right)
	= ( \mathrm{Tr}^\omega_\lambda(K^\prime)_{s_1}^{s_2})
	 \in \mathrm{M}_3(\mathscr{T}_3^\omega(\mathfrak{T})).
\end{gather*}

\begin{figure}[htb]
	\centering
	\includegraphics[width=.65\textwidth]{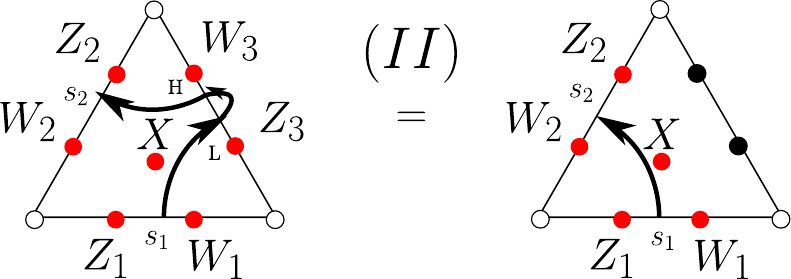}
	\caption{One of the oriented versions of Move (II).}
	\label{fig:Move-II}
\end{figure}

\begin{figure}[htb]
	\centering
	\includegraphics[width=.55\textwidth]{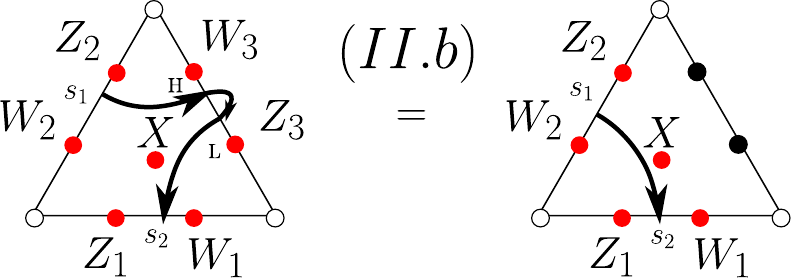}
	\caption{Move (II.b).}
	\label{fig:Move-IIbcorrect}
\end{figure}

\subsection{Move (III)}
\label{sssec:move-3}
	
In Figure \ref{fig:Move-III}, we show one of the oriented versions of Move (III).  Let $K$ be the link on the left, and $K^\prime$ the link on the right.  According to the definition of the quantum trace (\S \ref{sec:def-of-quantum-trace}) as a State Sum Formula, the equality expressing Move (III) can be interpreted as an equality of $3^2 \times 3^2$ matrices (\S \ref{ssec:matrix-conventions}).  Specifically, the claim is that the matrix,
\begin{equation*}
\label{eq:move-iii-example}
\tag{$III$}
	( \mathrm{Tr}^\omega_\lambda(K)_{s_1 s_2}^{s_3 s_4}) =
	\left(\begin{smallmatrix}
a_1^2&0&0		&
0&0&0  		&
0&0&0
\\
a_1 b_1&a_1 e_1&0		&
0&0&0		&
0&0&0
\\
a_1 c_1&a_1 f_1&a_1 i_1		&
0&0&0		&
0&0&0
\\
b_1 a_1 &0&0		&
e_1 a_1&0&	0	&
0&0&0
\\
b_1^2 & b_1 e_1&0		&
e_1 b_1 & e_1^2&0		&
0&0&0
\\
b_1 c_1 & b_1 f_1 & b_1 i_1 		&
e_1 c_1 & e_1 f_1 & e_1 i_1 		&
0&0&0
\\
c_1 a_1 & 0&0		&
f_1 a_1 & 0&0		&
i_1 a_1 & 0&0		
\\
c_1 b_1 & c_1 e_1 &0		&
f_1 b_1 & f_1 e_1 & 0		&
i_1 b_1 & i_1 e_1 & 0		
\\
c_1^2 & c_1 f_1 & c_1 i_1 		&
f_1 c_1 & f_1^2 & f_1 i_1 		&
i_1 c_1 & i_1 f_1 & i_1^2
	\end{smallmatrix}\right)\in  \mathrm{M}_{3^2}(\mathscr{T}_3^\omega(\mathfrak{T})),
\end{equation*}
is equal to the matrix
\begin{gather*}
	 \left(\begin{smallmatrix}
a_1^2 & 0 & 0 		&
0&0&0		&
0&0&0
\\
q b_1 a_1 \\+ (1-q^2) a_1 b_1& e_1 a_1 & 0		&
(q-q^{-1}) (e_1 a_1 - a_1 e_1) & 0 & 0 		&
0&0&0
\\
\begin{smallmatrix}q c_1 a_1 \\+ (1-q^2) a_1 c_1\end{smallmatrix} & f_1 a_1 & i_1 a_1 		&
(q-q^{-1})(f_1 a_1 - a_1 f_1) & 0 & 0 		&
\begin{smallmatrix}(q-q^{-1})(i_1 a_1 \\- a_1 i_1)\end{smallmatrix} & 0 & 0
\\
q a_1 b_1 & 0 & 0 		&
a_1 e_1 & 0 & 0 		&
0&0&0
\\
b_1^2 & q^{-1} e_1 b_1 & 0 		&
q^{-1} b_1 e_1 + (1-q^{-2}) e_1 b_1 & e_1^2 & 0 		&
0&0&0
\\
\begin{smallmatrix}q c_1 b_1 \\+ (1-q^2) b_1 c_1\end{smallmatrix} & \begin{smallmatrix}f_1 b_1 \\+ (q^{-1}-q)e_1 c_1\end{smallmatrix} & i_1 b_1 		&
\begin{smallmatrix}(-q^2+2-q^{-2}) e_1 c_1 \\+ c_1 e_1 \\+ (q-q^{-1})(f_1 b_1 - b_1 f_1)\end{smallmatrix} & \begin{smallmatrix}q f_1 e_1 \\+ (1-q^2) e_1 f_1\end{smallmatrix} & i_1 e_1 		&
\begin{smallmatrix}(q-q^{-1})(i_1 b_1 \\- b_1 i_1)\end{smallmatrix} & \begin{smallmatrix}(q-q^{-1})(i_1 e_1 \\- e_1 i_1)\end{smallmatrix} & 0
\\
q a_1 c_1 & 0 & 0 		&
a_1 f_1 & 0 & 0		&
a_1 i_1 & 0 & 0
\\
q b_1 c_1 & e_1 c_1 & 0 		&
b_1 f_1 + (q-q^{-1}) e_1 c_1 & q e_1 f_1 & 0 		&
b_1 i_1 & e_1 i_1 & 0
\\
c_1^2 & q^{-1} f_1 c_1 & q^{-1} i_1 c_1 		&
q^{-1} c_1 f_1 + (1-q^{-2}) f_1 c_1 & f_1^2 & q^{-1} i_1 f_1 		&
\begin{smallmatrix}q^{-1} c_1 i_1 \\+ (1-q^{-2}) i_1 c_1\end{smallmatrix} & \begin{smallmatrix}q^{-1} f_1 i_1 \\+ (1-q^{-2}) i_1 f_1\end{smallmatrix} & i_1^2
	\end{smallmatrix}\right)
\\
	= ( \mathrm{Tr}^\omega_\lambda(K^\prime)_{s_1 s_2}^{s_3 s_4})
	  \in  \mathrm{M}_{3^2}(\mathscr{T}_3^\omega(\mathfrak{T})),
\end{gather*}
where we have used Figures \ref{fig:cross-neg-same-over-to-higher}, \ref{fig:cross-pos-same-over-to-lower} and the matrices $(\mathrm{Tr}^\omega_\mathfrak{B}(C_\text{neg-same}^\text{over-to-higher})_{s_1 s_2}^{s_3 s_4}
	)
	=  
	(\vec{C}_\text{same}^q)^{-1}$ 
and
$(\mathrm{Tr}^\omega_\mathfrak{B}(C_\text{pos-same}^\text{over-to-lower})_{s_1 s_2}^{s_3 s_4}
	)
	=  
	\vec{C}_\text{same}^q$, respectively,
from \S \ref{sssec:crossings} as part of the computation for the matrix on the right, and where we have put (\S \ref{ssec:proofnotation})
\begin{equation*}
	\left(\begin{smallmatrix}
		a_1&b_1&c_1\\
		0&e_1&f_1\\
		0&0&i_1
	\end{smallmatrix}\right)
	\overset{\text{def}}{=}  \vec{L}^\omega(W_2, Z_2, W_3, Z_3, X)\in\mathrm{M}_3(\mathscr{T}_3^\omega(\mathfrak{T})).
\end{equation*}
See Appendix \ref{sec:the-appendix} for a computer check of the above equality of $3^2 \times 3^2$ matrices in $\mathrm{M}_{3^2}(\mathscr{T}_3^\omega(\mathfrak{T}))$ representing this oriented Move (III) example.
In Figures \ref{fig:Move-IIIb}, \ref{fig:Move-IIIc}, and \ref{fig:Move-IIId} we prove the remaining oriented versions of Move (III), in terms of the moves already established.  

\begin{figure}[htb]
	\centering
	\includegraphics[width=.65\textwidth]{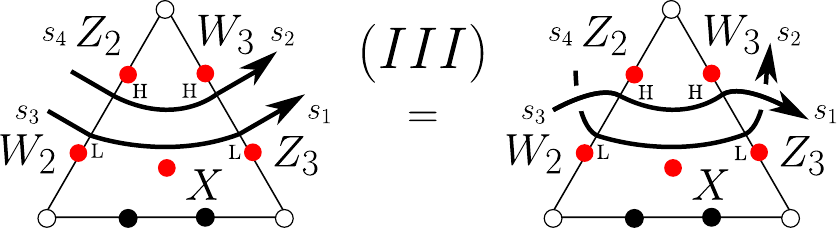}
	\caption{One of the oriented versions of Move (III).}
	\label{fig:Move-III}
\end{figure}	

\begin{figure}[htb]
	\centering
	\includegraphics[width=.95\textwidth]{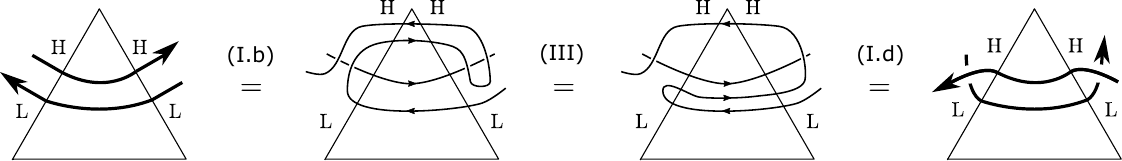}
	\caption{Move (III.b) and its proof.}
	\label{fig:Move-IIIb}
\end{figure}

\begin{figure}[htb]
	\centering
	\includegraphics[width=.95\textwidth]{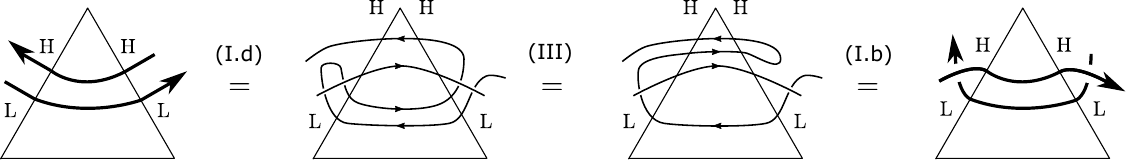}
	\caption{Move (III.c) and its proof.}
	\label{fig:Move-IIIc}
\end{figure}

\begin{figure}[htb]
	\centering
	\includegraphics[width=.95\textwidth]{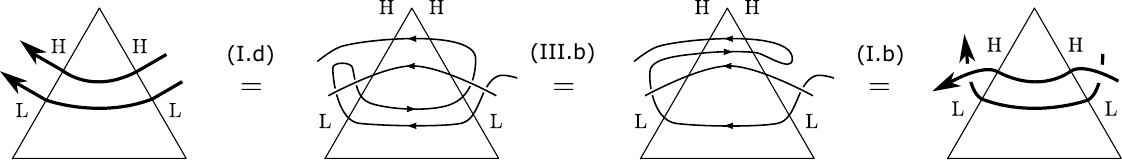}
	\caption{Move (III.d) and its proof.}
	\label{fig:Move-IIId}
\end{figure}

\subsection{Move (IV)}
\label{sssec:move-4}
	
In Figure \ref{fig:Move-IV}, we show one of the oriented versions of Move (IV).  Let $K$ be the link on the left, and $K^\prime$ the link on the right.  According to the definition of the quantum trace (\S \ref{sec:def-of-quantum-trace}) as a State Sum Formula, the equality expressing Move (IV) can be interpreted as an equality of $3^2 \times 3^2$ matrices (\S \ref{ssec:matrix-conventions}).  Specifically, the claim is that the matrix,
\begin{equation*}
\label{eq:move-iv-example}
\tag{$IV$}
	( \mathrm{Tr}^\omega_\lambda(K)_{s_1 s_2}^{s_3 s_4}) =
	\left(\begin{smallmatrix}
a_3 A_2&0&0&
b_3 A_2&0&0&
c_3 A_2&0&0
\\
a_3 D_2&a_3 E_2&0&
b_3 D_2&b_3 E_2&0&
c_3 D_2&c_3 E_2&0
\\
a_3 G_2&a_3 H_2&a_3 I_2&
b_3 G_2&b_3 H_2&b_3 I_2&
c_3 G_2&c_3 H_2&c_3 I_2
\\
0&0&0&
e_3 A_2&0&0&
f_3 A_2&0&0
\\
0&0&0&
e_3 D_2&e_3 E_2&0&
f_3 D_2&f_3 E_2&0
\\
0&0&0&
e_3 G_2&e_3 H_2&e_3 I_2&
f_3 G_2&f_3 H_2&f_3 I_2&
\\
0&0&0&
0&0&0&
i_3 A_2&0&0
\\
0&0&0&
0&0&0&
i_3 D_2&i_3 E_2&0
\\
0&0&0&
0&0&0&
i_3 G_2&i_3 H_2&i_3 I_2
\\
	\end{smallmatrix}\right)\in  \mathrm{M}_{3^2}(\mathscr{T}_3^\omega(\mathfrak{T})),
\end{equation*}
is equal to the matrix
\begin{gather*}
	q^{+1/3} \left(\begin{smallmatrix}
q^{-1} A_2 a_3&0&0&
q^{-1} A_2 b_3&0&0&
q^{-1} A_2 c_3&0&0
\\
D_2 a_3&E_2 a_3&0&
D_2 b_3 + (q^{-1}-q) A_2 e_3&E_2 b_3&0&
D_2 c_3 + (q^{-1}-q) A_2 f_3&E_2 c_3&0
\\
G_2 a_3&H_2 a_3&I_2 a_3&
G_2 b_3&H_2 b_3&I_2 b_3&
G_2 c_3+(q^{-1}-q)A_2 i_3&H_2 c_3&I_2 c_3
\\
0&0&0&
A_2 e_3&0&0&
A_2 f_3&0&0
\\
0&0&0&
q^{-1}D_2 e_3&q^{-1}E_2 e_3&0&
q^{-1} D_2 f_3&q^{-1}E_2 f_3&0
\\
0&0&0&
G_2 e_3&H_2 e_3&I_2 e_3&
G_2 f_3+(q^{-1}-q)D_2 i_3&H_2 f_3+(q^{-1}-q)E_2 i_3&I_2 f_3
\\
0&0&0&
0&0&0&
A_2 i_3&0&0
\\
0&0&0&
0&0&0&
D_2 i_3&E_2 i_3&0
\\
0&0&0&
0&0&0&
q^{-1}G_2 i_3&q^{-1}H_2 i_3&q^{-1}I_2 i_3
\\
	\end{smallmatrix}\right)
\\
	= ( \mathrm{Tr}^\omega_\lambda(K^\prime)_{s_1 s_2}^{s_3 s_4})
	  \in  \mathrm{M}_{3^2}(\mathscr{T}_3^\omega(\mathfrak{T})),
\end{gather*}
where we have used Figure \ref{fig:cross-pos-same-over-to-higher} and the matrix $(\mathrm{Tr}^\omega_\mathfrak{B}(C_\text{pos-same}^\text{over-to-higher})_{s_1 s_2}^{s_3 s_4}
	)
	=  
	\vec{C}_\text{same}^q$ from \S \ref{sssec:crossings} as part of the computation for the matrix on the right, and where we have put (\S \ref{ssec:proofnotation})
\begin{equation*}
	\left(\begin{smallmatrix}
		A_2&0&0\\
		D_2&E_2&0\\
		G_2&H_2&I_2
	\end{smallmatrix}\right)
	\overset{\text{def}}{=}  \vec{R}^\omega(W_3,Z_3,W_1,Z_1,X),\,\,
	\left(\begin{smallmatrix}
		a_3&b_3&c_3\\
		0&e_3&f_3\\
		0&0&i_3
	\end{smallmatrix}\right)
	\overset{\text{def}}{=}  \vec{L}^\omega(W_1, Z_1, W_2, Z_2, X)\in\mathrm{M}_3(\mathscr{T}_3^\omega(\mathfrak{T})).
\end{equation*}
See Appendix \ref{sec:the-appendix} for a computer check of the above equality of $3^2 \times 3^2$ matrices in $\mathrm{M}_{3^2}(\mathscr{T}_3^\omega(\mathfrak{T}))$ representing this oriented Move (IV) example.  
In Figures \ref{fig:Move-IVb}, \ref{fig:Move-IVc}, and \ref{fig:Move-IVd} we prove the remaining oriented versions of Move (IV).  (Note that Move (I.b$^\prime$) and Move (I.d$^\prime$), used here, are proved in \S \ref{ssec:auxiliarymoves}.)

\begin{figure}[htb]
	\centering
	\includegraphics[width=.65\textwidth]{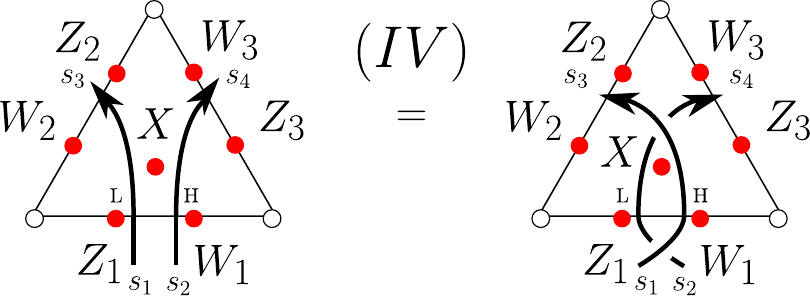}
	\caption{One of the oriented versions of Move (IV).}
	\label{fig:Move-IV}
\end{figure}	

\begin{figure}[htb]
	\centering
	\includegraphics[width=.95\textwidth]{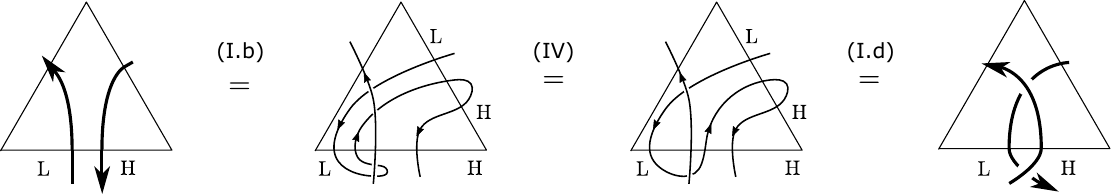}
	\caption{Move (IV.b) and its proof.}
	\label{fig:Move-IVb}
\end{figure}

\begin{figure}[htb]
	\centering
	\includegraphics[width=.95\textwidth]{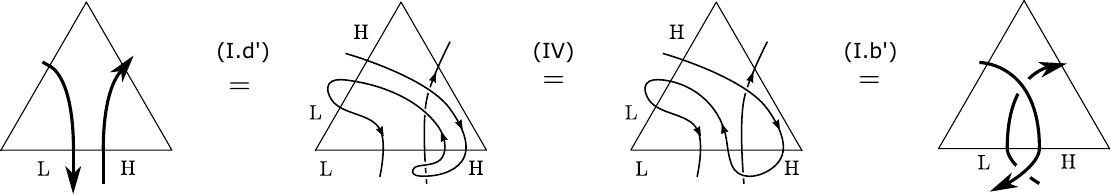}
	\caption{Move (IV.c) and its proof.}
	\label{fig:Move-IVc}
\end{figure}

\begin{figure}[htb]
	\centering
	\includegraphics[width=.95\textwidth]{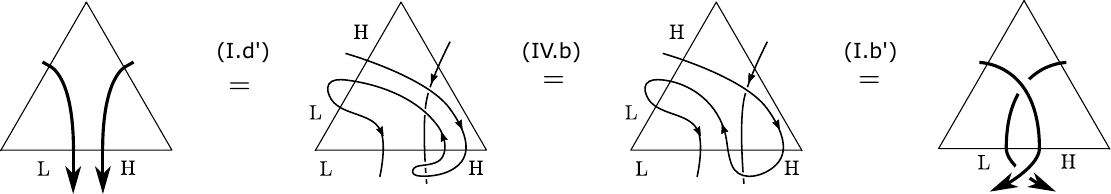}
	\caption{Move (IV.d) and its proof.}
	\label{fig:Move-IVd}
\end{figure}

\subsection{Move (V)}
\label{ssec:moveV}
	
See Figure \ref{fig:Move-V}.  This move is  implied by the kink-removing skein relations appearing in Figure \ref{fig:framing-skein-relations}; see \S \ref{sec:HOMFLYPT-skein-relation}.  
	
\begin{figure}[htb]
	\centering
	\includegraphics[width=.65\textwidth]{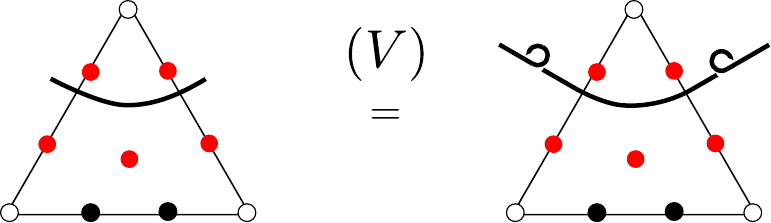}
	\caption{Move (V), valid for either orientation.}
	\label{fig:Move-V}
\end{figure}
	
\subsection{Auxiliary moves}
\label{ssec:auxiliarymoves}
	
Move (I.b$^\prime$) and Move (I.d$^\prime$) were used to establish Moves (IV.c) and (IV.d).  The proof of Move (I.b$^\prime$) is shown in Figure \ref{fig:Move-I-prime}.  Here $\sim$ denotes an isotopy preserving good position.  (See also the second to last paragraph of the proof of Lemma 24 in \cite{BonahonGT11}.)  The proof of Move (I.d$^\prime$) is obtained from that of Move (I.b$^\prime$) by horizontal reflection.

\begin{figure}[htb]
	\centering
	\includegraphics[width=\textwidth]{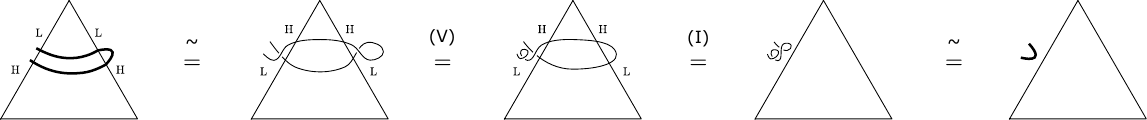}
	\caption{Move (I$^\prime$) and its proof, valid for either orientation.}
	\label{fig:Move-I-prime}
\end{figure}

\appendix

\section{Proof of Proposition \ref{prop:reshetikhin-turaev}}
\label{sec:proof-of-reshetikhin-turaev-invariant}

This is an application of Theorem XIV.5.1 in \cite{Kassel95}.  We will closely follow the definitions, notations, and conventions of \cite{Kassel95}, informing  otherwise.  Essentially all of the following theory is standard.  See, for instance, \cite{Kassel95,Brown02,joyalMR1173027,Majid95}.

Our first goal is to define the ribbon category $\mathcal{C}_V$ of interest.  In particular, this requires defining the objects $V$, morphisms $f : U \to V$, tensor products $V\otimes W$, tensor unit $I$, braiding morphisms $c_{V, W} : V \otimes W \to W \otimes V$ (see Remark \ref{rem:atypical-braiding}), dual objects $V^*$, left duality morphisms $b_V : I \to V \otimes V^*$ and $d_V : V^* \otimes V \to I$, 
 twist morphisms $\theta_V : V \to V$,  and right duality morphisms $b^\prime_V : I \to V^* \otimes V$ and $d^\prime_V : V \otimes V^* \to I$.
As above, fix $n \in \mathbb{Z}$, $n>1$, as well as $q$ and $\omega=q^{1/n^2}$ in $\mathbb{C} - \left\{ 0 \right\}$.  This section does not require a choice of square root $\omega^{1/2}$; compare \S \ref{ssec:biangles-and-the-reshetikhin-turaev-invariant}.  All vector spaces are over $\mathbb{C}$.  

\subsection{Quantum special linear group \texorpdfstring{$H = \mathrm{SL}_n^q$}{H=SLnq}}
\label{ssec:quantum-special-linear-group}
	
The \textit{quantum matrix algebra} $\mathrm{M}_n^q$  is the quotient of the free algebra in generators $(T_i^j)_{1 \leq i, j \leq n}$ by the relations
\begin{gather*}		
	T_i^m T_i^k = q T_i^k T_i^m, \,\,
	T_j^m T_i^m = q T_i^m T_j^m,	\,\,
	T_i^m T_j^k = T_j^k T_i^m, \,\,
	T_j^m T_i^k - T_i^k T_j^m = (q - q^{-1}) T_i^m T_j^k,
\end{gather*}	
for $i < j$ and $k < m$ (we think of lower indices indicating rows and upper indices columns).     The \textit{quantum determinant} $\mathrm{Det}^q\in\mathrm{M}_n^q$ is  (compare \S \ref{sec:quantum-SLn})
\begin{equation*}
	\mathrm{Det}^q\overset{\text{def}}{=}\sum_{\sigma\in S_n} (-q^{-1})^{\ell(\sigma)} T_{1}^{\sigma(1)} T_{2}^{\sigma(2)} \cdots T_{n}^{\sigma(n)}.
\end{equation*}
The \textit{quantum special linear group} $\mathrm{SL}_n^q$ is the quotient
\begin{equation*}
	H \overset{\text{def}}{=} \mathrm{SL}_n^q \overset{\text{def}}{=} \mathrm{M}_n^q / (\mathrm{Det}^q - 1).
\end{equation*}
The algebra $H$ is a Hopf algebra, meaning it is equipped with linear  maps
\begin{equation*}
\mu_H: H \otimes H \to H, \,\, \eta_H: \mathbb{C} \to H, \,\,
 \Delta_H : H \to H \otimes H, \,\, \epsilon_H : H \to \mathbb{C}, \,\, S_H : H \to H,
\end{equation*}
namely the product, unit, coproduct, counit, and antipode.  Specifically, if $\delta_{ij}$ denotes the \textit{Kronecker delta} (equals $1$ if $i=j$ and $0$ else),  then the coproduct $\Delta_H$, counit $\epsilon_H$, and antipode $S_H$ are defined by (abusing notation by using the same symbol for elements of $\mathrm{M}^q_n$ as their images in $H$)
\begin{equation*}
	\Delta_H(T_i^j) \overset{\text{def}}{=} \sum_{k=1}^n T_i^k \otimes T_k^j, \,\,
	\epsilon_H(T_i^j) \overset{\text{def}}{=} \delta_{ij}, \,\,    S_H(T_i^j)\overset{\text{def}}{=}(-q)^{j-i} A_j^i,     
\end{equation*}
where the \textit{quantum minor} $A_j^i$ is the quantum determinant of the subalgebra, isomorphic to $\mathrm{M}^q_{n-1}$, of $\mathrm{M}^q_n$ generated by the $T_k^\ell$ with $k\neq j$ and $\ell\neq i$.  (We write simply $\mu_H(x\otimes y)=xy$ and $\eta_H(z)=z1_H$.)

\subsection{Braided tensor category \texorpdfstring{$H$-$\mathrm{Comod}$}{H-Comod} of right \texorpdfstring{$H$}{H}-comodules}
\label{ssec:braided-category}

\subsubsection{Right $H$-comodules}
\label{ssec:right-h-comodules}

A vector space $V$ is a right $H$-comodule if it is equipped with a linear map
$\Delta_V : V \to V \otimes H$, namely a (right) coaction, satisfying certain properties.    The tensor product $V \otimes W$ of two right $H$-comodules is a right $H$-comodule,  with coaction
\begin{equation*}
     \Delta_{V\otimes W}(v\otimes w)
\overset{\mathrm{def}}{=}  \sum_{(v),(w)}v_V\otimes w_W\otimes v_H w_H.
\end{equation*}
Here, we have used Sweedler's notation for the coactions $\Delta_V$ and $\Delta_W$.  The trivial right $H$-comodule $\mathbb{C}$ has coaction $\Delta_\mathbb{C}(z)=z\otimes 1_H$.

Let $H$-$\mathrm{Comod}$ denote the tensor category whose objects $V$ are right $H$-comodules, morphisms $f:U\to V$ are homomorphisms of right $H$-comodules, and with tensor products $V\otimes W$ and tensor unit $I=\mathbb{C}$ as above.  

\subsubsection{Braidings}
\label{ssec:braiding}

The bialgebra $H$ is cobraided, meaning it is equipped with linear maps $r_H : H \otimes H \to \mathbb{C}$ and its inverse (with respect to the convolution operator $\star$ defined in \S \ref{ssec:coribbon element}) $\bar{r}_H : H \otimes H \to \mathbb{C}$, namely the universal $R$-forms.  Specifically, if $E_i^{j} : H \to \mathbb{C}$ is the linear map $E_i^{j}(T_k^\ell) = \delta_{i k} \delta_{j \ell}$, then 
\begin{gather*}		
	r_H \overset{\text{def}}{=} q^{-1/n} \left( \sum_{1 \leq i \neq j \leq n} E_i^{i } \otimes E_j^{j } + q \sum_{i=1}^n E_i^{i } \otimes E_i^{i } + (q - q^{-1}) \sum_{1 \leq i < j \leq n} E_i^{j } \otimes E_j^{i } \right),		\\		
	\bar{r}_H \overset{\text{def}}{=} q^{+1/n} \left( \sum_{1 \leq i \neq j \leq n} E_i^{i } \otimes E_j^{j } + q^{-1} \sum_{i=1}^n E_i^{i } \otimes E_i^{i } + (q^{-1} - q) \sum_{1 \leq i < j \leq n} E_i^{j } \otimes E_j^{i } \right).
\end{gather*}

Consequently, the tensor category $H$-$\mathrm{Comod}$ of right $H$-comodules is braided, with  braiding morphisms $c_{V, W} : V \otimes W \to W \otimes V$ and inverse braidings $\bar{c}_{V, W} : V \otimes W \to W \otimes V$ (so, $c_{V,W}^{-1}=\bar{c}_{W,V}$) defined by
\begin{equation*}
	c_{V, W}(v \otimes w) \overset{\text{def}}{=} \sum_{(v), (w)} (w_W \otimes v_V) r_H (v_H \otimes w_H), \,\,
	\bar{c}_{V, W}(v \otimes w) \overset{\text{def}}{=} \sum_{(v), (w)} (w_W \otimes v_V) \bar{r}_H (w_H \otimes v_H).
\end{equation*}
 
\begin{remark}
\label{rem:atypical-braiding}
	By symmetry, one can just as well take the inverse braidings $\bar{c}_{V, W}$ to be `the' braidings for the category.  For technical reasons, we will prefer this choice going forward.  (Note that, in order to compute with $\bar{c}_{V,W}$, the formulas,
\begin{equation*}
     \bar{r}(xy\otimes z)=\sum_{(z)}\bar{r}(y\otimes z^\prime)\bar{r}(x\otimes z^{\prime\prime}),
\,\,
\bar{r}(x\otimes yz)=\sum_{(x)}\bar{r}(x^\prime\otimes y)\bar{r}(x^{\prime\prime}\otimes z),     
\end{equation*}
can be helpful, here using Sweedler's notation for the coproduct $\Delta_H$.)
\end{remark}

\subsection{Ribbon sub-category \texorpdfstring{$H$-$\mathrm{Comod}_f$}{H-Comod-f} of finite-dimensional right \texorpdfstring{$H$}{H}-comodules}
\label{ssec:ribbon-category}

\subsubsection{Left dualities}
\label{ssec:left-duality}

Let $V$ be a right $H$-comodule of dimension $N<\infty$.  The  dual space $V^*=\mathrm{Hom}_\mathbb{C}(V, \mathbb{C})$  is a right $H$-comodule as follows.   Choose a basis $e^1, e^2, \dots, e^N$ for $V$ with corresponding dual basis $e_1^*, e_2^*, \dots, e_N^*$ for $V^*$.  Let $h_i^j \in H$ for $1 \leq i, j \leq N$ satisfy
\begin{equation*}
\Delta_V(e^j) = \sum_{k=1}^N e^k \otimes h_k^j.
\end{equation*}
 The coaction $\Delta_{V^*} : V^* \to V^* \otimes H$ is defined by
\begin{equation*}
\Delta_{V^*}(e_i^*) \overset{\text{def}}{=} \sum_{k=1}^N e_k^* \otimes S_H h_i^k.
\end{equation*}

Let $H$-$\mathrm{Comod}_f$ denote the braided sub-category of $H$-$\mathrm{Comod}$ consisting of finite-dimensional right $H$-comodules.  Then $H$-$\mathrm{Comod}_f$ has left duality, the dual objects $V^*$ being defined as above, and with left duality morphisms $b_V : \mathbb{C} \to V \otimes V^*$ and $d_V : V^* \otimes V \to \mathbb{C}$ defined by
\begin{equation*}
	b_V(1) \overset{\text{def}}{=} \nu \sum_{k=1}^N e^k \otimes e^*_k, \,\,  d_V(e^*_i \otimes e^j) \overset{\text{def}}{=} \nu^{-1} \delta_{ij}.
\end{equation*}
Here, $\nu$ is a fixed complex \textit{duality parameter}.  

\begin{remark}
\label{rem:non-standard-duality}
	One possible choice for the duality parameter is $\nu = 1$.  However, it is better for our purposes to take $\nu = q^{(1-n)/2n} = \bar{\sigma}_n q^{(n-1)/2}$, where $\bar{\sigma}_n =  q^{(1-n^2)/2n}$ is the square root of the (signed) coribbon element (Definition \ref{def:ribbon-element}); see Lemma \ref{fact:matrices-for-dualities} and Remarks \ref{rem:comment-on-coribbon-element} and \ref{rem:discussion-of-duality-parameter-choice}.  
\end{remark}

\subsubsection{Twists}
\label{ssec:coribbon element}

The following has been adapted to our purposes from \cite[Chapter XIV, Exercises 5-6]{Kassel95}.  

The \textit{convolution operator} $\star:H^*\otimes H^*\to H^*$ on the dual space $H^*=\mathrm{Hom}_\mathbb{C}(H,\mathbb{C})$ is defined by
\begin{equation*}
(f\star g)(x)
\overset{\mathrm{def}}{=}\sum_{(x)}f(x^\prime)g(x^{\prime\prime}).
\end{equation*}
This operation makes $H^*$ into an algebra, with multiplicative unit $1_{H^*}=\epsilon_H$ the counit for $H$.
Similarly, $\star$ operates on $(H\otimes H)^*=\mathrm{Hom}_\mathbb{C}(H\otimes H,\mathbb{C})$ by
\begin{equation*}
(r\star s)(x\otimes y)
\overset{\mathrm{def}}{=}\sum_{(x),(y)}r(x^\prime\otimes y^\prime)s(x^{\prime\prime}\otimes y^{\prime\prime}).
\end{equation*}

The cobraided Hopf algebra $H$ is \textit{coribbon}, meaning there exists an invertible central element $\zeta_H$ in $H^*$ such that
\begin{equation*}  \label{coribbon-algebra-equations}
	\zeta_H \circ \mu_H = (r_H \circ \tau_{H, H}) \star r_H \star (\zeta_H \otimes \zeta_H),
	\,\,  \zeta_H(1_H) = 1, 
	\,\,  \zeta_H \circ S_H = \zeta_H.
\end{equation*}
Here, $\tau_{H, H} : H \otimes H \to H \otimes H$ is the swapping map $x\otimes y\mapsto y\otimes x$.  Specifically,   $\zeta_H \in H^*$ and its convolution inverse $\bar{\zeta}_H \in H^*$ are defined by 
\begin{gather*}  
		\zeta_H(T_i^j) \overset{\text{def}}{=} \zeta_n \delta_{ij}, \,\, \zeta_n \overset{\text{def}}{=} (-1)^{n-1} q^{(n^2-1)/n}
	\,\, \left( = (-1)^{n-1} \omega^{n(n^2-1)} \right),  \\
		\bar{\zeta}_H(T_i^j) \overset{\text{def}}{=} \bar{\zeta}_n \delta_{ij}, \,\, \bar{\zeta}_n \overset{\text{def}}{=} (-1)^{n-1} q^{(1-n^2)/n}
	\,\, \left( = (-1)^{n-1} \omega^{n(1-n^2)} \right).
\end{gather*}

Consequently, the braided category with left duality $H$-$\mathrm{Comod}_f$ of finite-dimensional right $H$-comodules (with braidings $\bar{c}_{V,W}$, see Remark \ref{rem:atypical-braiding}) is ribbon, with twist morphisms $\theta_V : V \to V$ defined by
\begin{equation*}
	\theta_V(v) \overset{\text{def}}{=} \sum_{(v)} v_V \bar{\zeta}_H(v_H).
\end{equation*}

\begin{remark}
\label{rem:atypical-coribbon-element}  
	Note that $\bar{\zeta}_n\in\mathbb{C}-\{0\}$ is what we previously called the `coribbon element' in  Definition \ref{def:ribbon-element}; compare Remarks \ref{rem:comment-on-coribbon-element} and \ref{rem:non-standard-duality}.  We also refer to $\bar{\zeta}_H\in H^*$ as the \textit{coribbon element}.  
\end{remark}

\subsubsection{Right dualities}
\label{ssec:right-duality}

Moreover, the ribbon category $H$-$\mathrm{Comod}_f$ of finite-dimensional right $H$-comodules (with braidings $\bar{c}_{V,W}$, see Remark \ref{rem:atypical-braiding}) has right duality, with right duality morphisms  $b^\prime_V : \mathbb{C} \to V^* \otimes V$ and $d^\prime_V : V \otimes V^* \to \mathbb{C}$ defined by 
\begin{equation*}
	b^\prime_V \overset{\text{def}}{=} (\mathrm{id}_{V^*} \otimes \theta_V) \circ \bar{c}_{V, V^*} \circ b_V, 
\,\,
	d^\prime_V \overset{\text{def}}{=} d_V \circ \bar{c}_{V, V^*} \circ (\theta_V \otimes \mathrm{id}_{V^*}).
\end{equation*}

\subsection{Ribbon sub-category \texorpdfstring{$\mathcal{C}_V$}{C-V} of \texorpdfstring{$H$-$\mathrm{Comod}_f$}{H-Comod-f} coming from the quantum row-space}
\label{ssec:ribbon-cat-coming-from-row-space}

\subsubsection{Quantum row-space}
\label{ssec:quantum-row-space}

The \textit{quantum row-space} $A$ is the  quotient of the free algebra in generators $(e^j)_{1\leq j\leq n}$ by the relations
$
e^j e^i = q e^i e^j
$
for $j > i$.  The (infinite-dimensional) algebra $A$ is a  right $H$-comodule (in fact, a right  $H$-comodule-algebra), with coaction $\Delta_A : A \to A \otimes H$ defined by
\begin{equation*}
	\Delta_A(e^j) \overset{\text{def}}{=} \sum_{k=1}^n e^k \otimes T_k^j.
\end{equation*}
For each integer $d \geq 0$, let $V_d \subset A$ denote the (finite-dimensional) sub-space of $A$ consisting of homogeneous polynomials of degree $d$.  Then  $V_d$ is a right $H$-sub-comodule of $A$.  For the remainder of this appendix, define $V\subset A$ by
\begin{equation*}
V \overset{\text{def}}{=} V_1 = \mathrm{span}_\mathbb{C}(e^1, e^2, \dots, e^n).
\end{equation*} 

Let  $\mathcal{C}_V$ denote the ribbon sub-category of $H$-$\mathrm{Comod}_f$ generated by $V$.  In particular, objects of $\mathcal{C}_V$ are finite tensor products of $V$ and its dual $V^*$.  

\subsubsection{Morphism formulas}
\label{ssec:formulas}

We will give explicit formulas for the braidings $\overline{c}_{V, W}$ (see Remark \ref{rem:atypical-braiding}), left dualities $b_V$, $d_V$,  twists $\theta_V$, and right dualities $b^\prime_V$, $d^\prime_V$.  Since we are working in the sub-category $\mathcal{C}_V$, it suffices to compute the formulas for $V$ and $V^*$.  

It can be shown that the braidings $\bar{c}_{V, V}$, $\bar{c}_{V^*, V^*}$, $\bar{c}_{V^*, V}$ and $\bar{c}_{V, V^*}$ are calculated by the same formulas as those provided in \S \ref{sssec:crossings}; see also Remark \ref{rem:comment-on-coribbon-element}.  
The left dualities $b_V: \mathbb{C} \to V \otimes V^*$ and $d_V: V^* \otimes V \to \mathbb{C}$ are calculated by the formulas in \S \ref{ssec:left-duality}, taking the \textit{symmetric duality parameter} $\nu = q^{(1-n)/2n} = \bar{\sigma}_n q^{(n-1)/2}$; see Remark \ref{rem:non-standard-duality} and Lemma \ref{fact:matrices-for-dualities}.  
The twists   $\theta_V : V \to V$ and $\theta_{V^*}: V^* \to V^*$ can be computed by the formulas in \S \ref{ssec:coribbon element} as
\begin{equation*}
	\theta_V(e^j) = \bar{\zeta}_n e^j, 
\,\,
	\theta_{V^*}(e^*_i) = \bar{\zeta}_n e^*_i.
\end{equation*}
  The right dualities $b^\prime_V: \mathbb{C} \to V^* \otimes V$ and $d^\prime_V: V \otimes V^* \to \mathbb{C}$ can be computed by the formulas in \S \ref{ssec:right-duality} as
\begin{equation*}
	b^\prime_V(1) = (-1)^{n-1} \nu \sum_{k=1}^n q^{2k-n-1} e^*_k \otimes e^k, 
\,\,
	d^\prime_V(e^i \otimes e^*_j) = (-1)^{n-1} \nu^{-1} q^{n-2i+1} \delta_{ij}.
\end{equation*}

\subsubsection{Matrix formulas}
\label{ssec:matrix-formulas}

We will make an observation for the duality morphisms similar to Fact \ref{fact:braiding-matrices}.  This will require   our choice of the symmetric duality parameter $\nu =\bar{\sigma}_n q^{(n-1)/2}$; see Remark \ref{rem:non-standard-duality}.  

Recall  the matrix $\vec{U}^q$ from \S \ref{sssec:U-turns}.  (See \S \ref{ssec:matrix-conventions} for matrix conventions.)  More generally, define for any  scalar $\nu$ a matrix $\vec{U}^q_\nu$ in $\mathrm{M}_n(\mathbb{C})$ by 
\begin{equation*}
	(\vec{U}^q_\nu)_i^j \overset{\text{def}}{=} \nu (-q)^{i-n} \delta_{i,n-j+1}.  
\end{equation*}
Note that the matrices $\vec{U}^q = \vec{U}^q_\nu$ agree for  $\nu = \bar{\sigma}_n q^{(n-1)/2}$.  
Similarly, define $b_V(\nu)$, $d_V(\nu)$, $b^\prime_V(\nu)$, and $d^\prime_V(\nu)$ for any scalar $\nu$ to be the dualities defined in \S \ref{ssec:left-duality} and \ref{ssec:right-duality} (see also \S \ref{ssec:formulas}).  
Recall from \S \ref{sssec:crossings} the ordered bases $\beta_{V^*, V}$ and $\beta_{V, V^*}$ of $V^* \otimes V$ and $V \otimes V^*$.    When written in terms of the basis $\beta_{V, V^*}$ (and the basis $\left\{ 1 \right\}$ of $\mathbb{C}$), the left duality $b_V(\nu) : \mathbb{C} \to V \otimes V^*$ becomes a $n^2 \times 1$ matrix with coefficients $[b_V(\nu)]_{ij}^1$.  Similarly, when written in terms of these bases, the dualities $d_V(\nu) : V^* \otimes V \to \mathbb{C}$, $b^\prime_V(\nu) : \mathbb{C} \to V^* \otimes V$, and $d^\prime_V(\nu) : V \otimes V^* \to \mathbb{C}$ become  $1 \times n^2$, $n^2 \times 1$, and $1 \times n^2$ matrices with coefficients $[d_V(\nu)]_1^{ij}$, $[b^\prime_V(\nu)]_{ij}^1$, and $[d^\prime_V(\nu)]_1^{ij}$.  

\begin{lemma}[defining property of the symmetric duality parameter $\nu$]
\label{fact:matrices-for-dualities}
	The following equalities of matrix coefficients,
\begin{gather*}
	[b^\prime_V(\nu)]^1_{ij} = ( \vec{U}^q_\nu )_j^i,\,\,
[d^\prime_V(\nu)]^{ij}_1 = ( (\bar{\zeta}_n)^{-1} \vec{U}^q_\nu )_j^i,
\,\,
[b_V(\nu)]^1_{ij} = ( (\vec{U}^q_\nu)^\mathrm{T} )_i^j,\,\,
[d_V(\nu)]^{ij}_1 =( (\bar{\zeta}_n)^{-1} (\vec{U}^q_\nu)^\mathrm{T} )_i^j,
\end{gather*}
hold if and only if $\nu = \nu_{\pm} = \pm \bar{\sigma}_n q^{(n-1)/2}$, where $\bar{\sigma}_n =  q^{(1-n^2)/2n}$.  
\end{lemma}

\begin{proof}
	We display the case of $d^\prime_V(\nu)$.  The other computations are similar.  	
	Put $v_{ij} = (\beta_{V, V^*})_{ij} = e^i \otimes (-q)^{n-j} e^*_{n-j+1}$.  We calculate
\begin{gather*}
	[d^\prime_V(\nu)]^{ij}_1 
=  d^\prime_V(\nu)(v_{ij})
=  (-q)^{n-j} d^\prime_V(\nu)(e^i \otimes e^*_{n-j+1})
  =  (-q)^{n-j} (-1)^{n-1} \nu^{-1} q^{n-2i+1} \delta_{i, n-j+1} \\
=  (-1)^{j-1} q^{n-j} \nu^{-1} q^{2j-n-1} \delta_{i, n-j+1}
  = (-1)^{j-1} q^{j-1} \nu^{-1} \delta_{i, n-j+1},
\end{gather*}
as well as 
\begin{gather*}
	( (\bar{\zeta}_n)^{-1} \vec{U}^q_\nu )_j^i
=  (-1)^{n-1} \left| \bar{\zeta}_n \right|^{-1} \nu (-q)^{j-n} \delta_{j, n-i+1} = (-1)^{j-1} q^{j-1} q^{1-n} \left| \bar{\zeta}_n \right|^{-1} \nu  \delta_{i, n-j+1},
\end{gather*}
where we define $\left| \bar{\zeta}_n \right| = +q^{(1-n^2)/n}$.  We gather
\begin{gather*}
[d^\prime_V(\nu)]^{ij}_1 = ( (\bar{\zeta}_n)^{-1} \vec{U}^q_\nu )_j^i
\,\, \Leftrightarrow \,\,
\nu^{-1} = q^{1-n} \left| \bar{\zeta}_n \right|^{-1} \nu
\,\, \Leftrightarrow \,\,
\pm \bar{\sigma}_n q^{(n-1)/2} = \nu.  \qedhere
\end{gather*}
\end{proof}

\begin{remark}
\label{rem:discussion-of-duality-parameter-choice}
	Compare the four analogous matrix equations in \S \ref{sssec:U-turns}; see Figures \ref{fig:decreasing-U-turns} and  \ref{fig:increasing-U-turns}.  We will see the reason for the transposition of the indices $i$, $j$ in the first two equalities of Lemma \ref{fact:matrices-for-dualities}  when we discuss diagrammatics below.  Essentially, this is because the tails of the U-turns in Figure \ref{fig:decreasing-U-turns}, corresponding (see \S \ref{ssec:RT-functor} and \ref{ssec:alternative-def-of-biangle-trace}) to the morphisms $b^\prime_V$ and $d^\prime_V$, are associated with the second tensor factor (measured from bottom to top).  The opposite is true for the U-turns in Figure \ref{fig:increasing-U-turns}, corresponding to the morphisms $b_V$ and $d_V$.  

The sign ambiguity in Lemma \ref{fact:matrices-for-dualities} is resolved by our need for the matrix $\vec{U}^q = \vec{U}^q_{\nu_+}$.   Why, in \S \ref{sssec:U-turns}, was the matrix $\vec{U}^q$ preferred over $-\vec{U}^q = \vec{U}^q_{\nu_-}$?  For instance, this was required for the quantum trace to satisfy the local isotopy Move (II); see Figure \ref{fig:Move-II} and Equation \eqref{eq:move-ii-example}.  (Assuming  we ask for the local monodromy matrices to have positive entries.)  

Note  that the first equation in Lemma \ref{fact:matrices-for-dualities}, for $b^\prime_V(\nu)$, uniquely determines the value  (including the sign) of the coribbon parameter $\bar{\zeta}_n$ (see Remark \ref{rem:atypical-coribbon-element}), independent of $\nu$.
\end{remark}

\subsection{Category \texorpdfstring{$\mathcal{R}$}{R} of ribbons and the Reshetikhin--Turaev functor \texorpdfstring{$F_V:\mathcal{R} \to \mathcal{C}_V$}{F-V R to C-V}}
\label{ssec:RT-functor}

\subsubsection{Category of ribbons}
\label{sssec:category-of-ribbons}

The universal ribbon category $\mathcal{R}$, also called the category of oriented ribbons, is defined exactly as in \cite[\S XIV.5.1]{Kassel95}.  Roughly speaking, the objects are collections of oriented ribbon ends, and the morphisms are isotopy classes of oriented ribbons $L$ matching this end data.  It is useful to, rather, think of ribbons as framed links, the link being the spine of the ribbon, and where the framing at a point on the link is normal to the tangent space of the ribbon at that point (in this way of thinking about the framing we differ from \cite{Kassel95}, but it is immaterial mathematically, so long as we work with framed links rather than ribbons).  

Ribbons live in the space $\mathbb{R}^2 \times [0, 1]$ having the usual $x$,$y$,$z$-coordinates.  However, when drawing \textit{ribbon diagrams} (which are, importantly, different from our previous pictures such as those appearing in Figures \ref{fig:decreasing-U-turns}-\ref{fig:single-strand-crossing-the-biangle}), the $x$-coordinate is drawn on the page horizontally right, the $z$-coordinate on the page vertically up, and the $y$-coordinate into the page.  By definition, the ribbon ends lying on the same boundary plane in $\mathbb{R}^2 \times \left\{0, 1\right\}$ are required to have distinct $x$-coordinates (and isotopy preserves this property).  The coordinates $x, y, z$ are called \textit{ribbon coordinates}. 

Ribbon diagrams always represent ribbons $L$ with the \textit{blackboard framing}, meaning the constant framing in the $(0, -1, 0)$ direction, that is, out of the page toward the eye of the reader.  Such a framing is always possible by introducing kinks into the ribbon.  Here, a positive kink (Figure \ref{fig:positive-kink-skein-relation}) replaces a full right-handed twist, and a negative kink (Figure \ref{fig:negative-kink-skein-relation}) a full left-handed twist \cite[\S X.8]{Kassel95}.   Note that ribbon ends lying on $\mathbb{R}^2 \times \left\{0, 1\right\}$ are also required to have this blackboard framing.  

\subsubsection{Reshetikhin--Turaev functor}
\label{sssec:rt-functor}

We will apply Theorem XIV.5.1 in \cite{Kassel95}.  This says that there is a unique functor $F_V$ from the category $\mathcal{R}$ of ribbons  to the ribbon category $\mathcal{C}_V$, which preserves the braiding,  duality, and twist, and satisfies the property that a single downward-pointing (namely, negative $z$ direction) ribbon end (a distinguished object in $\mathcal{R}$) is mapped to $V$.  (Consequently, upward-pointing ribbon ends are mapped to $V^*$.)  In particular, $F_V$ provides an  isotopy invariant  of stated oriented ribbons $L$  (see below).

Diagrammatically speaking,    we use exactly the same conventions for displaying morphisms in the category $\mathcal{C}_V$ as in \cite[Chapter XIV]{Kassel95}.  For example, in Figure \ref{fig:twists} we show how the twist morphisms are displayed diagrammatically \cite[Chapters X.8 and XIV.5.1]{Kassel95}; compare Figures \ref{fig:positive-kink-skein-relation}-\ref{fig:negative-kink-skein-relation}, as well as our   calculations in \S \ref{ssec:formulas}.  To help distinguish these ribbon diagrams from our previous pictures, as in Figures \ref{fig:decreasing-U-turns}-\ref{fig:single-strand-crossing-the-biangle}, we put a white arrow on the boundary axis $\mathbb{R} \times \left\{0\right\} \times \left\{ 0 \right\}$ indicating the positive $x$-direction.  

\begin{figure}[htb]
	\centering
	\includegraphics[width=.48\textwidth]{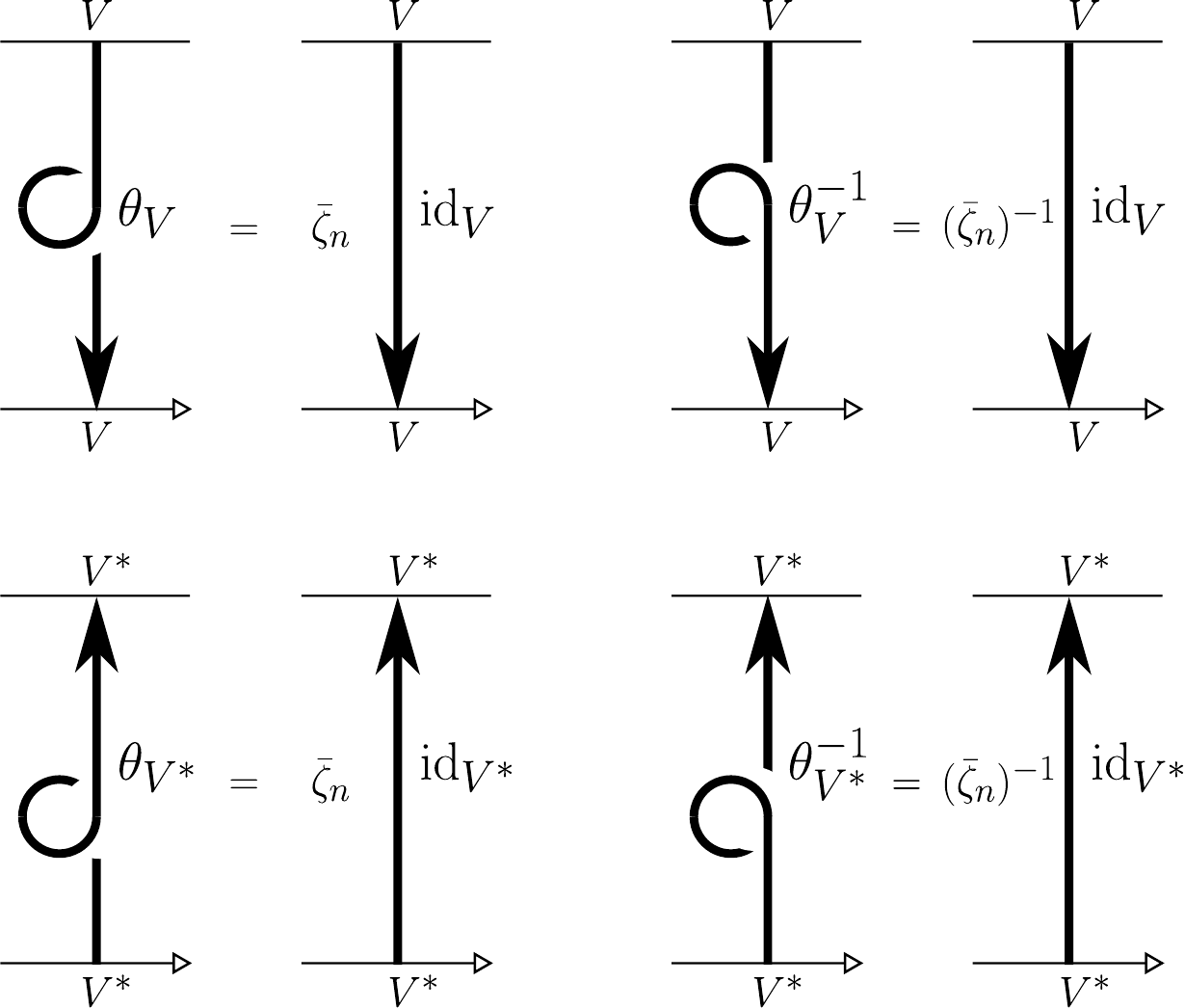}
	\caption{Ribbon diagrams for twist morphisms.}
	\label{fig:twists}
\end{figure}	

\subsection{Alternative definition of the biangle quantum trace map}
\label{ssec:alternative-def-of-biangle-trace}

We will prove part (A) of Proposition \ref{prop:reshetikhin-turaev}.  The strategy is to use the Reshetikhin--Turaev functor $F_V$ to give an alternative definition of the quantum trace map $\mathrm{Tr}^\omega_\mathfrak{B}$ for a biangle $\mathfrak{B}$, equivalent to the definition provided in \S \ref{ssec:biangles-and-the-reshetikhin-turaev-invariant}.  To do this, we need to be able to pass back and forth between the  more symmetric topological setting of framed links $K$ in the thickened biangle $\mathfrak{B} \times (0, 1)$ (which comes without any preferred parametrization--see the beginning of \S \ref{ssec:biangles-and-the-reshetikhin-turaev-invariant}), and the  less symmetric categorical setting of framed links $L$ in $\mathbb{R}^2\times[0,1]$ (where the parametrization matters--see, for example,  Figure \ref{fig:twists}).  

\subsubsection{`Turning your head'}
\label{sssec:alternative-definition}

To pass between the two settings, we  `turn our head'.  That is, instead of viewing the thickened biangle $\mathfrak{B} \times (0, 1)$ `from the top' (as in Figures \ref{fig:decreasing-U-turns}-\ref{fig:single-strand-crossing-the-biangle}), instead we view it `from the side'.  This can be done in two different ways, illustrated in Figure \ref{fig:DandG} (intuitively, from the perspective of Person G or that of Person D).  

More precisely, let the thickened biangle $\mathfrak{B} \times (0, 1) = [0, 1] \times \mathbb{R} \times (0, 1)$ have \textit{biangle coordinates} $X, Y, Z$ with respect to a choice of parametrization $\mathscr{P}$ (intuitively, this choice is only to tell Person G and Person D where to stand, but the  point is that where they stand does not matter, as they will see the same answer).  For example, in the left hand side of Figure \ref{fig:DandG}, the $X$-coordinate is drawn on the page horizontally right, the $Y$-coordinate on the page vertically up, and the $Z$-coordinate out of the page toward the eye of the reader.    Then,  by Figure \ref{fig:DandG} and our discussion of ribbon coordinates in \S \ref{sssec:category-of-ribbons}, if $x_\mathrm{G}$, $y_\mathrm{G}$, $z_\mathrm{G}$ denote the ribbon coordinates from Person G's perspective, the \textit{\textnormal{G}-(ribbon) coordinate transformation} $\varphi_\mathrm{G} : \mathfrak{B} \times (0,1) \overset{\sim}{\to} (0,1)\times\mathbb{R}\times[0,1]\hookrightarrow \mathbb{R}^2 \times [0, 1]$ is
\begin{equation*}
	\varphi_G(X, Y, Z) \overset{\text{def}}{=} (x_G, y_G, z_G) \overset{\text{def}}{=} (+Z, +Y, 1-X).
\end{equation*}
Similarly, the \textit{\textnormal{D}-(ribbon) coordinate transformation} $\varphi_\mathrm{D} : \mathfrak{B} \times (0,1) \overset{\sim}{\to} (0,1)\times\mathbb{R}\times[0,1]\hookrightarrow \mathbb{R}^2 \times [0, 1]$  is 
\begin{equation*}
	\varphi_D(X, Y, Z) \overset{\text{def}}{=} (x_D, y_D, z_D) \overset{\text{def}}{=} (+Z, -Y, +X).
\end{equation*}
Note that from either perspective the positive $x$-ribbon coordinate corresponds to the direction of increasing height in the thickened biangle $\mathfrak{B} \times (0, 1)$.  

\begin{figure}[htb]
	\centering
	\includegraphics[width=
	.775\textwidth]{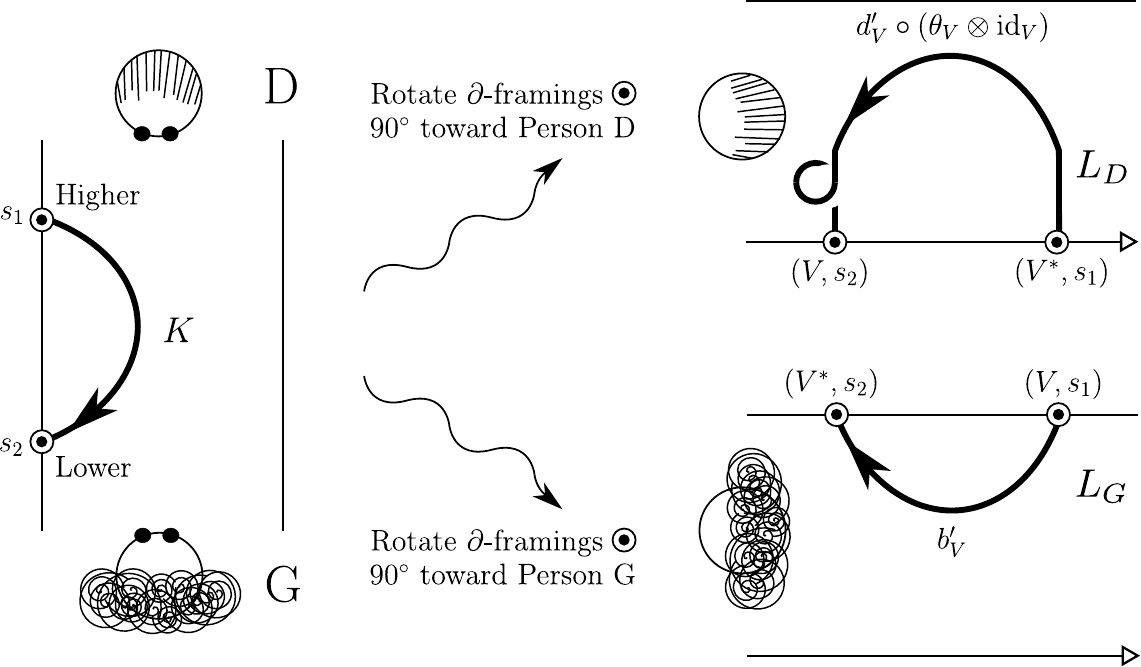}
	\caption{Calculating the biangle quantum trace by `turning your head'.  This can be done in two different ways, but if we choose the symmetric duality parameter $\nu  = \bar{\sigma}_n q^{(n-1)/2}$, depending on the square root of the (signed) coribbon element for the Hopf algebra $H=\mathrm{SL}_n^q$, then the result is independent of which perspective, that of Person G or Person D, is taken; see Lemma \ref{fact:matrices-for-dualities}.}
	\label{fig:DandG}
\end{figure}	

\subsubsection{Definition via the Reshetikhin--Turaev functor: topological setup}
\label{sssec:definition-via-the-reshetikhin-turaev-functor-topological-aspects}

Fix a stated framed oriented link $(K, s)$ in the thickened biangle $\mathfrak{B} \times (0, 1)$; see \S \ref{ssec:stated-links}.  Recall that by definition the framing on the boundary $\partial K$ points up in the vertical direction.  That is, the framing vector is $(X, Y, Z) = (0, 0, 1)$ in biangle coordinates (\S \ref{sssec:alternative-definition}); see the left hand side of Figure \ref{fig:DandG} where the  `bullseyes', indicating the tips of the framing vectors, point out of the page toward the eye of the reader.  Recall that, by definition, elements of $\partial K$ lying on the same boundary component of $\mathfrak{B} \times (0, 1)$ have distinct heights, namely, distinct Z-coordinates (Remark \ref{rem:higher-lower}).  

We now `turn our head' as in \S \ref{sssec:alternative-definition}, say from the perspective of Person G.  More precisely, by applying the $G$-coordinate transformation $\varphi_G$ we obtain a link $L_G = \varphi_G(K)$ in $\mathbb{R}^2 \times [0, 1]$.  However, $L_G$  is not yet a framed link, according to the definition in \S \ref{sssec:category-of-ribbons}.  Indeed, its framing vectors on the boundary $\mathbb{R}^2 \times \left\{ 0, 1 \right\}$ are all $(x_G, y_G, z_G) = (1, 0, 0)$ in ribbon coordinates.  To fix this, we rotate each framing vector 90 degrees toward Person G, yielding an appropriate framing vector $(x_G, y_G, z_G)=(0, -1, 0).$  We call the resulting framed link again $L_G$.  Similarly, by this process we obtain a, possibly different, framed link $L_D$ in $\mathbb{R}^2 \times [0, 1]$ from Person D's perspective; see Figure \ref{fig:DandG}.  We say that the new framed links $L_G$ and $L_D$ have been \textit{corrected}.  

Importantly, note that this process of correcting the framing may introduce a twist in the link.  Indeed, as an example, on the right hand side of Figure \ref{fig:DandG} the framed link $K$ acquires a right-handed twist from Person D's perspective.   On the other hand, there is no twist from Person G's perspective.  

Note also that the distinct Z-coordinates condition for the link boundary $\partial K$ is consistent with the distinct $x$-coordinates condition for $\partial L_G$ and $\partial L_D$ (see \S \ref{sssec:category-of-ribbons}).  

\subsubsection{Definition via the Reshetikhin--Turaev functor: algebraic setup}
\label{sssec:definition-via-the-reshetikhin-turaev-functor-algebraic-aspects}

From Person G's perspective, let $N_G^i$ for $i=0, 1$ denote the number of points of the corrected framed link $L_G$ lying on the boundary plane $\mathbb{R}^2 \times \left\{ i \right\}$.  The framed link $L_G$ comes with a pair of sequences $V_G^i = ((V_G^i)_j)_{j=1,2,\dots, N_G^i}$ of vector spaces $(V_G^i)_j \in \left\{ V, V^* \right\}$, where the sequence is ordered in the increasing $x_G$-direction; see \S \ref{sssec:rt-functor}.  In other words, the sequences $V_G^i$ come from evaluating the Reshetikhin--Turaev functor $F_V$ on the domain and codomain objects of the framed link $L_G$ viewed as a morphism in the category of ribbons $\mathcal{R}$; see \S \ref{sssec:category-of-ribbons}.    Moreover, the evaluation of the functor $F_V$ on the framed link $L_G$ provides a linear map 
\begin{equation*}
	F_V(L_G) : \bigotimes_{j=1, 2, \dots, N_G^0} (V_G^0)_j \to \bigotimes_{j=1, 2, \dots, N_G^1} (V_G^1)_j.
\end{equation*}  
Similarly, define $N_D^i$ for $i=0,1$, sequences  $(V^i_D)_{j=1,2,\dots,N_D^i}$ of vector spaces, and a linear map $F_V(L_D)$ from Person D's perspective.  

For example, on the right hand side of Figure \ref{fig:DandG} we see $V_G^1 = ((V_G^1)_1, (V_G^1)_2) = (V^*, V)$ and, as a degenerate case, $V_G^0 = (\mathbb{C})$.  The corresponding linear map is $F_V(L_G) = b^\prime_V : \mathbb{C} \to V^* \otimes V$ (see \cite[p.351]{Kassel95}).  On the other hand, from Person D's perspective, $V_D^1=(\mathbb{C})$ and $V_D^0 = ((V_D^0)_1, (V_D^0)_2) = (V, V^*)$ and the corresponding linear map is $F_V(L_D) = d^\prime_V \circ (\theta_V \otimes \mathrm{id}_V) : V \otimes V^* \to \mathbb{C}$ (compare Figure \ref{fig:twists}).  

  Let $v_i \in V$ be the basis vector $v_i = e^i$.  By abuse of notation, we also let $v_j \in V^*$ denote the covector $(-q)^{n-j} e^*_{n-j+1}$.  Note that $\left\{ v_i \right\}_{i=1,2,\dots,n}$ provides an ordered basis for either $V$ or $V^*$.  From Person G's perspective, say, we may then consider the  basis $\beta(V_G^i)$ for the tensor product $\bigotimes_{j=1, 2, \dots, N_G^i} (V_G^i)_j$ defined by
\begin{equation*}
	\beta(V_G^i)_{j_1 j_2 \cdots j_{N^i_G}} \overset{\text{def}}{=} v_{j_1} \otimes v_{j_2} \otimes \cdots \otimes v_{j_{N^i_G}}
\,\,
\left( j_1, j_2, \dots, j_{N^i_G} \in \left\{1, 2, \dots, n\right\}\right),
\end{equation*}
ordered as in \S \ref{ssec:matrix-conventions}.  Similarly, from Person D's perspective, we define a basis $\beta(V^i_D)$ for the tensor product $\bigotimes_{j=1, 2, \dots, N_D^i} (V_D^i)_j$.  Note that this procedure recovers the familiar bases $\beta_{V, V}$, $\beta_{V^*, V^*}$, $\beta_{V^*, V}$, and $\beta_{V, V^*}$ for $V \otimes V$, $V^* \otimes V^*$, $V^* \otimes V$, and $V \otimes V^*$ used in \S \ref{ssec:matrix-formulas} and \S \ref{sssec:crossings}.

\subsubsection{Alternative definition of the $\mathrm{SL}_n$-biangle quantum trace map via the Reshetikhin--Turaev functor}
\label{sssec:definition-via-the-reshetikhin-turaev-functor}

Recall that the framed link $K$ in the thickened biangle $\mathfrak{B} \times (0, 1)$ has in addition been equipped with a state $s$, meaning that to each point of the boundary $\partial K$ is associated a state-number in $\left\{ 1, 2, \dots, n \right\}$.  From Person G's perspective, say, mimicking the sequences $V_G^i = ((V_G^i)_j)_j$ we obtain a pair of sequences of state-numbers $s_G^i = ((s_G^i)_j)_{j=1,2,\dots,N_G^i}$ for $i=0,1$.  Similarly, from Person D's perspective, we obtain a pair of state sequences $s_D^i$.    Note $s^0_G=s^1_D$ and $s^1_G=s^0_D$.

For example, on the right hand side of Figure \ref{fig:DandG} we see $s_G^1 = ((s_G^1)_1, (s_G^1)_2) = (s_2, s_1)$ and $s_G^0 = \emptyset$ from Person G's perspective.  We also see $s_D^1 = \emptyset$ and $s_D^0 = ((s_D^0)_1, (s_D^0)_2) = (s_2, s_1)$ from Person D's perspective.  (Here and in Figure \ref{fig:DandG}, the names of the states $s_1, s_2$ were arbitrary, taken from \S \ref{sssec:U-turns}, and not intended to be ordered in any particular way.)  

\begin{definition}
\label{def:alternative-definition-of-the-biangle-quantum-trace-map}
	Choose one of the two parametrizations $\mathscr{P}$ of the thickened biangle $\mathfrak{B}\times(0,1)$, as in \S \ref{sssec:alternative-definition}.
	Given a stated framed oriented link $(K, s)$ in the thickened biangle $\mathfrak{B} \times (0, 1)$, from Person G's perspective (Figure \ref{fig:DandG}) define the complex number     $\mathrm{Tr}^\omega_\mathfrak{B}(K, s)	=  \mathrm{Tr}^\omega_\mathfrak{B}(K, s)(\mathscr{P})\in\mathbb{C}$, depending on the parametrization $\mathscr{P}$, by 
\begin{equation*}
	\mathrm{Tr}^\omega_\mathfrak{B}(K, s)
\overset{\text{def}}{=}  (\mathrm{Tr}^\omega_\mathfrak{B})_G(K, s)
\overset{\text{def}}{=}  [  F_V(L_G)  ]_{i_1 i_2 \cdots i_{N^1_G}}^{j_1 j_2 \cdots j_{N^0_G}},
\end{equation*}
where $[F_V(L_G)]$ denotes the matrix of the linear map $F_V(L_G)$ written in the bases $\beta(V^0_G)$ and $\beta(V^1_G)$, and with the state-numbers $i_k = (s^1_G)_k$ and $j_k = (s^0_G)_k$.  In addition, define the number  (also depending on $\mathscr{P}$)
\begin{equation*}
(\mathrm{Tr}^\omega_\mathfrak{B})_D(K, s)
\overset{\text{def}}{=}  [  F_V(L_D)  ]_{i_1 i_2 \cdots i_{N^1_D}}^{j_1 j_2 \cdots j_{N^0_D}},
\end{equation*}
from Person D's perspective (also Figure \ref{fig:DandG}), where $[F_V(L_D)]$ denotes the matrix of $F_V(L_D)$ written in the bases $\beta(V^0_D)$ and $\beta(V^1_D)$, and with the states $i_k = (s^1_D)_k$ and $j_k = (s^0_D)_k$.  
\end{definition}

\subsubsection{Finishing the proof of Proposition {\upshape\ref{prop:reshetikhin-turaev}}}
\label{sssec:finishing-the-proof}

\begin{lemma}[symmetry of the $\mathrm{SL}_n$-biangle quantum trace map]
\label{lem:main-lemma}
For a given parametrization $\mathscr{P}$ of $\mathfrak{B}\times(0,1)$, we have $(\mathrm{Tr}^\omega_\mathfrak{B})_G(K, s) =  (\mathrm{Tr}^\omega_\mathfrak{B})_D(K, s)$ for any stated framed oriented link $(K, s)$.
\end{lemma}

\begin{proof}
As a first case, consider the stated link $(K, s)$ displayed in Figure \ref{fig:DandG} (and Figure \ref{fig:U-turn-dec-cw}).  In \S \ref{sssec:definition-via-the-reshetikhin-turaev-functor-algebraic-aspects}, we saw that $F_V(L_G) = b^\prime_V : \mathbb{C} \to V^* \otimes V$ and $F_V(L_D) = d^\prime_V \circ (\theta_V \otimes \mathrm{id}_V) : V \otimes V^* \to \mathbb{C}$.  By Lemma \ref{fact:matrices-for-dualities} and the formula for the twist $\theta_V$ appearing in \S \ref{ssec:formulas}, we compute
\begin{equation*}
	 (\mathrm{Tr}^\omega_\mathfrak{B})_G(K, s)
=  [  F_V(L_G)  ]_{s_2 s_1}^{1}
=  [  b^\prime_V  ]_{s_2 s_1}^{1}
=  ( \vec{U}^q )_{s_1}^{s_2},
\end{equation*}
and
\begin{gather*}
	(\mathrm{Tr}^\omega_\mathfrak{B})_D(K, s)
=  [  F_V(L_D)  ]_{1}^{s_2 s_1}
=  [  d^\prime_V \circ (\theta_V \otimes \mathrm{id}_V)  ]_{1}^{s_2 s_1}
=  \bar{\zeta}_n [  d^\prime_V   ]_{1}^{s_2 s_1}
=  \bar{\zeta}_n  (\bar{\zeta}_n)^{-1} (  \vec{U}^q )_{s_1}^{s_2}
=  ( \vec{U}^q )_{s_1}^{s_2},
\end{gather*}
as desired.    (Recall the topological discussion surrounding Figure \ref{fig:DandG}, where it is important that we have chosen the symmetric duality parameter $\nu = \bar{\sigma}_n q^{(n-1)/2}$.)   The proof when the stated link $(K, s)$ is one of the three other kinds of U-turns (see Figures \ref{fig:U-turn-dec-ccw} and  \ref{fig:increasing-U-turns}) is similar (two of these U-turns acquire left-handed twists in $L_D$).  

  Next, consider the stated link $(K,s)$ displayed in Figure \ref{fig:cross-pos-same-over-to-lower}.  Note that no twists are introduced when either $L_G$ or $L_D$ are corrected, because their constituent arcs connect different boundary components of the thickened biangle.  We compute 
\begin{equation*}
	 (\mathrm{Tr}^\omega_\mathfrak{B})_G(K, s)
=  [  F_V(L_G)  ]_{s_3 s_4}^{s_1 s_2}     
=  [  \bar{c}_{V,V}  ]_{s_3 s_4}^{s_1 s_2}     
=  ( \vec{C}^q_\mathrm{same} )_{s_3 s_4}^{s_1 s_2},     
\end{equation*}
     see \cite[p.341]{Kassel95}, Remark \ref{rem:atypical-braiding}, and Fact \ref{fact:braiding-matrices},  as well as     
\begin{equation*}
	 (\mathrm{Tr}^\omega_\mathfrak{B})_D(K, s)
=  [  F_V(L_D)  ]_{s_1 s_2}^{s_3 s_4}     
=  [  \bar{c}_{V^*,V^*}  ]_{s_1 s_2}^{s_3 s_4}      
=  ( \vec{C}^q_\mathrm{same} )_{s_1 s_2}^{s_3 s_4}   
=  ( \vec{C}^q_\mathrm{same} )_{s_3 s_4}^{s_1 s_2},      
\end{equation*}
     see \cite[p.348]{Kassel95} and Fact \ref{fact:braiding-matrices}.    The proofs for $(K,s)$ as in Figures \ref{fig:cross-neg-same-over-to-higher}, \ref{fig:cross-pos-same-over-to-higher}, and \ref{fig:cross-neg-same-over-to-lower} are similar.  For the other type of crossing, say, the stated link $(K,s)$ displayed in Figure \ref{fig:cross-neg-opp-over-to-lower}, we compute in the same way     
\begin{gather*}
(\mathrm{Tr}^\omega_\mathfrak{B})_G(K, s)
=  [  F_V(L_G)  ]_{s_3 s_4}^{s_1 s_2}     
=  [  \bar{c}_{V,V^*}  ]_{s_3 s_4}^{s_1 s_2}     
=  (  \vec{C}^q_\mathrm{opp}  )_{s_3 s_4}^{s_1 s_2}        
\\         =  (  \vec{C}^q_\mathrm{opp}  )_{s_1 s_2}^{s_3 s_4}     
=  [  \bar{c}_{V,V^*}  ]_{s_1 s_2}^{s_3 s_4}     
=  [  F_V(L_D)  ]_{s_1 s_2}^{s_3 s_4}     
=  (\mathrm{Tr}^\omega_\mathfrak{B})_D(K, s).     
\end{gather*}
The proofs for $(K,s)$ as in Figures \ref{fig:cross-pos-opp-over-to-higher}, \ref{fig:cross-neg-opp-over-to-higher}, and \ref{fig:cross-pos-opp-over-to-lower} are similar, as well as for the trivial strand (possibly with twists), Figure \ref{fig:single-strand-crossing-the-biangle}.    

The argument for a general stated link $(K, s)$ follows from the previous special cases, together with the fact that the Reshetikhin--Turaev functor $F_V$ satisfies the State Sum Property (essentially by definition) and is an isotopy invariant.  Put $K$ into a bridge position (\S \ref{sssec:kinks-and-the-Reshetikhin-Turaev-invariant}) with respect to a partition $0=X_0<X_1<\dots<X_p=1$ in biangle-coordinates, call the resulting link $K^\prime$.  This determines bridge-positions for $L^\prime_G$ and $L^\prime_D$ with respect to partitions $0=(z_G)_0<(z_G)_1<\dots<(z_G)_p=1$ and $0=(z_D)_0<(z_D)_1<\dots<(z_D)_p=1$ in $G$- and $D$-coordinates.  Let $((L^\prime_G)_i,(s_G)_i)$ and $((L^\prime_D)_i,(s_D)_i)$ be the corresponding restricted stated links in $(0,1)\times\mathbb{R}\times[(z_G)_i,(z_G)_{i+1}]$ and $(0,1)\times\mathbb{R}\times[(z_D)_i,(z_D)_{i+1}]$, where for the moment the states $(s_G)_i$ and $(s_D)_i$ are chosen arbitrarily on the internal boundaries $z_G=(z_G)_i$ and $z_D=(z_D)_i$ for $0<i<p$.  Note that each component of $(L^\prime_G)_i$ or $(L^\prime_D)_i$ is either a U-turn, crossing, or trivial strand (possibly with twists) (we are abusing the word `component' here, since we are considering a crossing to be a single component).  Also, upon correction of $L^\prime_G$ and $L^\prime_D$, no components of any $(L^\prime_G)_i$ acquire twists, while all U-turn components (and no others) of the $(L^\prime_D)_i$ acquire twists, which are right-handed (resp. left-handed) when the U-turn component of $(L^\prime_D)_i$ has boundary with $z_D$-coordinate $(z_D)_{i}$ (resp. $(z_D)_{i+1}$), namely, is a \textit{cap} (resp. \textit{cup}).  By the coordinate transformation $x_G=x_D$, $y_G=-y_D$, $z_G=1-z_D$, the link $(L^\prime_G)_i$ is related to the link $(L^\prime_D)_{p-1-i}$ as follows:  (1)  as un-oriented links, they are the same, except cups (resp. caps) of the former flip to become caps (resp. cups) of the latter (with twists as described just above);  (2)  the orientations of U-turns, thought of as pointing along the positive or negative $x$-direction, are preserved;   (3)  the orientations of same (resp. opposite) direction crossings are flipped (resp. preserved);  and, (4)  the orientations of trivial strands are flipped.  Write $(s_G)_i=(((s_G)_i)^0,((s_G)_i)^1)$  (as in \S \ref{sssec:definition-via-the-reshetikhin-turaev-functor}, note $((s_G)_0)^0=s_G^0$ and $((s_G)_{p-1})^1=s_G^1$), and $((s_G)_i)^j|_{((L^\prime_G)_i)_k}$ the restriction of these states to a component $((L^\prime_G)_i)_k$ of $(L^\prime_G)_i$, and similarly for Person D.  If it happens that $((s_G)_i)^0=((s_D)_{p-1-i})^1$ and $((s_G)_i)^1=((s_D)_{p-1-i})^0$, then by the previous special cases we have
\begin{gather*}
     [F_V((L^\prime_G)_i)]_{((s_G)_i)^1}^{((s_G)_i)^0}     
=  \prod_k [F_V(((L^\prime_G)_i)_k)]_{((s_G)_i)^1|_{((L^\prime_G)_i)_k}}^{((s_G)_i)^0|_{((L^\prime_G)_i)_k}}    
=  \prod_k [F_V(((L^\prime_D)_{p-1-i})_k)]_{((s_G)_i)^0|_{((L^\prime_D)_{p-1-i})_k}}^{((s_G)_i)^1|_{((L^\prime_D)_{p-1-i})_k}}     
\\          =  [F_V((L^\prime_D)_{p-1-i})]_{((s_G)_i)^0}^{((s_G)_i)^1}    
=  [F_V((L^\prime_D)_{p-1-i})]_{((s_D)_{p-1-i})^1}^{((s_D)_{p-1-i})^0}.     
\end{gather*}
     Since $s^0_G=s^1_D$ and $s^1_G=s^0_D$ (\S \ref{sssec:definition-via-the-reshetikhin-turaev-functor}),  it follows by the State Sum Formula that     
\begin{gather*}
     [F_V(L^\prime_G)]_{s_G^1}^{s_G^0}   
=  \sum_{\text{compatible }(s_G)_i}  [F_V((L^\prime_G)_0)]_{((s_G)_0)^1}^{s_G^0}     [F_V((L^\prime_G)_1)]_{((s_G)_1)^1}^{((s_G)_1)^0}       \cdots  [F_V((L^\prime_G)_{p-1})]_{s_G^1}^{((s_G)_{p-1})^0}     
\\          =  \sum_{\text{compatible }(s_G)_i}  [F_V((L^\prime_D)_{p-1})]_{s_G^0}^{((s_G)_0)^1}          
[F_V((L^\prime_D)_{p-2})]_{((s_G)_1)^0}^{((s_G)_1)^1}      \cdots  [F_V((L^\prime_D)_{0})]_{((s_G)_{p-1})^0}^{s_G^1}          
\\     =  \sum_{\text{compatible }(s_D)_i}  [F_V((L^\prime_D)_{0})]_{((s_D)_{0})^1}^{s_D^0}      \cdots [F_V((L^\prime_D)_{p-2})]_{((s_D)_{p-2})^1}^{((s_D)_{p-2})^0}       [F_V((L^\prime_D)_{p-1})]_{s_D^1}^{((s_D)_{p-1})^0}     
=  [F_V(L^\prime_D)]_{s_D^1}^{s_D^0}.     
\end{gather*}
     Lastly, by isotopy invariance,  
\begin{equation*}
     (\mathrm{Tr}^\omega_\mathfrak{B})_G(K, s)
=  [  F_V(L_G)  ]_{s_G^1}^{s_G^0}     
=  [F_V(L^\prime_G)]_{s_G^1}^{s_G^0}    
=  [F_V(L^\prime_D)]_{s_D^1}^{s_D^0}     
=  [  F_V(L_D)  ]_{s_D^1}^{s_D^0}     
=   (\mathrm{Tr}^\omega_\mathfrak{B})_D(K, s).  \qedhere     
\end{equation*}
 \end{proof}
 
\begin{proof}[Proof of Proposition {\upshape\ref{prop:reshetikhin-turaev}}]
Let $\mathrm{Tr}^\omega_\mathfrak{B}(K, s)(\mathscr{P})$ be defined as in Definition \ref{def:alternative-definition-of-the-biangle-quantum-trace-map}, depending on the parametrization $\mathscr{P}$ of the thickened biangle $\mathfrak{B}\times(0,1)$.  We first argue that this is independent of the choice of $\mathscr{P}$.  Indeed, if $\mathscr{P}^\prime$ is the other parametrization, then 
\begin{equation*}
\mathrm{Tr}^\omega_\mathfrak{B}(K, s)(\mathscr{P}^\prime)
=  (\mathrm{Tr}^\omega_\mathfrak{B})_{G}(K, s)(\mathscr{P}^\prime)
=  (\mathrm{Tr}^\omega_\mathfrak{B})_{D}(K, s)(\mathscr{P})
=  (\mathrm{Tr}^\omega_\mathfrak{B})_{G}(K, s)(\mathscr{P})
=  \mathrm{Tr}^\omega_\mathfrak{B}(K, s)(\mathscr{P}),     
\end{equation*}
where the second equation is immediate from the definition, and the third equation is by Lemma \ref{lem:main-lemma}.  

It is then straightforward to check that $\mathrm{Tr}^\omega_\mathfrak{B}(K, s)$, so defined via the Reshetikhin--Turaev functor $F_V$, agrees with the definition given in \S \ref{ssec:biangles-and-the-reshetikhin-turaev-invariant} (this is, more or less, implicit in the proof of Lemma \ref{lem:main-lemma}).  In particular, since  $F_V$ is  isotopy invariant, so is $\mathrm{Tr}^\omega_\mathfrak{B}(K, s)$, which is what remained to be proved.  
\end{proof}

\section{Computer calculations}
\label{sec:the-appendix}

In the Mathematica code appearing at the end of this article:  section 2.1 verifies the claims for the local $\mathrm{SL}_4$ example of \S \ref{sssec:n=4-example}; section 2.2 verifies the claims for the local $\mathrm{SL}_3$ example of \S \ref{sssec:n=3-example}; section 3.2.1 verifies the claims for Moves (I), (I.b), (I.c), (I.d) of \S \ref{sssec:move-1}; section 3.2.2 verifies the claims for Moves (II), (II.b) of \S \ref{sssec:move-2}; section 3.2.3 verifies the claims for Move (III) of \S \ref{sssec:move-3}; and, section 3.2.4 verifies the claims for Move (IV) of \S \ref{sssec:move-4}.   

\bibliographystyle{alpha}
\bibliography{references.bib}
\addresseshere



\section*{Supplementary material}

In[1]:=    (* SECTION 1: DEFINITIONS NEEDED FOR QUANTUM CALCULATIONS *)
        (* number of coordinates *)
            (* the following definition of n is for Section 2;
a different definition is given in Section 3 *)
n = 12;
        (* Poisson structure matrix *)
        (* Z3=1,Z2=2,Z1=3,Z3′=4,Z2′=5,Z1′=6,X1=7,X2=8,X3=9,Z3′′=10,Z2′′=11,Z1′′=12 *)
        (* Encodes, e.g., Z3*Z2=q^(1)Z2*Z3, Z3*Z3′=q^(2)Z3′*Z3, Z2*Z3=q^(-1)Z3*Z2 *)
            (* the following definition of P is for Section 2;
a different definition is given in Section 3 *)
P = {{0, 1, 0, 2, 0, 0, 0, -2, 0, 0, 0, 0},
    {-1, 0, 1, 0, 0, 0, -2, 2, 0, 0, 0, 0},
    {0, -1, 0, 0, 0, 0, 2, 0, 0, 0, 0, -2},
    {-2, 0, 0, 0, -1, 0, 0, 2, 0, 0, 0, 0},
    {0, 0, 0, 1, 0, -1, 0, -2, 2, 0, 0, 0},
    {0, 0, 0, 0, 1, 0, 0, 0, -2, 2, 0, 0},
    {0, 2, -2, 0, 0, 0, 0, -2, 2, 0, -2, 2},
    {2, -2, 0, -2, 2, 0, 2, 0, -2, 0, 0, 0},
    {0, 0, 0, 0, -2, 2, -2, 2, 0, -2, 2, 0},
    {0, 0, 0, 0, 0, -2, 0, 0, 2, 0, -1, 0},
    {0, 0, 0, 0, 0, 0, 2, 0, -2, 1, 0, -1},
    {0, 0, 2, 0, 0, 0, -2, 0, 0, 0, 1, 0}};
        (* i[]: For encoding a (monic) monomial in the n variables *)
        (* E.g., Z3^(1/2)*Z2^(-1/2)*Z1^(1/4)*Z3 is encoded by {i[1,1/2],i[2,-1/2],i[3,1/4],i[1,1]} *)
i[z__] := {z}[[1]] + i {z}[[2]];
        (* a[]: Weyl ordering [[.,.]] coefficient for a monomial *)
        (* E.g., since [[Z3*Z2]] = q^(-1/2)Z3*Z2, then a[{i[1,1],i[2,1]}]=q^(-1/2) *)
a[w__] := (tempw = Flatten[w];
    temp = Sum[Sum[ -(1/2) *Im[tempw[[i]]] Im[tempw[[j]]] P[[Re[tempw[[i]]], Re[tempw[[j]]]]],
        {j, i + 1, Length[tempw]}], {i, 1, Length[tempw]}]; q^temp);
(* Test: *)
(* a[{i[1,1],i[2,1]}] *)
        (* b[]: Coefficient resulting from re-ordering a monomial according to the order 1,2,...,n above *)
        (* E.g., since Z2*Z3=q^(-1)Z3*Z2, then b[{i[2,1],i[1,1]}]=q^(-1), but b[{i[1,1],i[2,1]}]=1 *)
b[w__] := (temp = 0;
    tempw = Flatten[w];
    Do[ Do[If[Re[tempw[[Length[tempw] - i1 + 1]]] == n - i3 + 1,
        Do[If[Re[tempw[[i2]]] ≥ n - i3 + 1,  , temp = temp + 1 *Im[tempw[[Length[tempw] - i1 + 1]]]
            Im[tempw[[i2]]] P[[Re[tempw[[Length[tempw] - i1 + 1]]], Re[tempw[[i2]]]]]],
        {i2, Length[tempw] - i1 + 2, Length[tempw]}]], {i1, Length[tempw]}], {i3, n}]; q^temp);
(* Test: *)
(* b[{i[2,1],i[1,1]}] *)
(* b[{i[1,1],i[2,1]}] *)
(* b[{i[2,1],i[1,1],i[2,1]}] *)
        (* c[]: Prints a possibly re-ordered monomial as if the variables commuted *)
        (* E.g., prints either Z3*Z2 or Z2*Z3 when Z2*Z3 is entered, based on Mathematica's whims *)
            (* the following definition of c[] is for Section 2;
a different definition is given in Section 3 *)
c[w__] := (temp = ConstantArray[0, n];
    tempw = Flatten[w];
    Do[ temp[[Re[tempw[[i]]]]] = temp[[Re[tempw[[i]]]]] + Im[tempw[[i]]], {i, Length[tempw]}];
        (* the following depends on the ordering of the n variables *)
        (* Z3=1,Z2=2,Z1=3,Z3′=4,Z2′=5,X1=7,X2=8,X3=9,Z3′′=10,Z2′′=11,Z1′′=12 *)
    Z3^temp[[1]] *Z2^temp[[2]] *Z1^temp[[3]] * Z3′^temp[[4]] *Z2′^temp[[5]] *Z1′^temp[[6]] *
    X1^temp[[7]] *X2^temp[[8]] *X3^temp[[9]] *Z3′′^temp[[10]] *Z2′′^temp[[11]] *Z1′′^temp[[12]]);
(* Test: *)

```
(* c[{i[2,1],i[1,1]}] *)
(* c[{i[1,1],i[2,1]}] *)
(* c[{i[2,1],i[1,1],i[2,1]}] *)
    (* f[m1,m2]: Computes [[m_1]]*[[m_2]] = q^r m,
    where m should be viewed as if in the order 1,2,...,n,
    but, as for c[], Mathematica may not present the monomial in the proper order,
    so, for instance, if the output is q^(-1) Z_2*Z_3,
    the correct output is actually q^(-1) Z_3*Z_2 *)
f[x_, y_] := a[x]*a[y]*b[{x, y}]*c[{x, y}];
(* Test: *)
(* f[{i[2,1]},{i[1,1]}] *)
(* f[{i[1,1]},{i[2,1]}] *)
(* f[{i[2,1]},{i[1,1],i[2,1]}] *)
    (* g[m], like f[m1,m2] but for only one monomial m *)
        (* only used in Section 3 *)
g[x_] := a[x]*b[x]*c[x];
```

In[9]:=     **(\* SECTION 2: QUANTUM LEFT AND RIGHT MATRICES \*)**

In[10]:= **(\* SECTION 2.1:  n=**
    **4:  CHECKING COMMUTATION RELATIONS FOR QUANTIZED 2x2 SL4 SUB-MATRICES \*)**

In[11]:= **(\* SECTION 2.1.1: n=**
    **4:  CHECKING COMMUTATION RELATIONS FOR QUANTIZED 2x2 SUB-MATRIX OF LEFT SL4 MATRIX \*)**
```
    (* Encoding 2x2 sub-matrix, e.g. a := [[a_1]]+[[a_2]]+[[a_3]] *)
    (* Recall *)
    (* Z3=1,Z2=2,Z1=3,Z3′=4,Z2′=5,Z1′=6,X1=7,X2=8,X3=9,Z3″=10,Z2″=11,Z1″=12 *)
a1 = {i[1, 1/4], i[2, 2/4], i[3, 3/4], i[4, 1/4], i[5, -2/4],
    i[6, -1/4], i[7, -1/4], i[8, -2/4], i[9, -1/4], i[10, 0], i[11, 0], i[12, 0]};
a2 = {i[1, 1/4], i[2, 2/4], i[3, 3/4], i[4, 1/4], i[5, -2/4], i[6, -1/4],
    i[7, -1/4], i[8, 2/4], i[9, -1/4], i[10, 0], i[11, 0], i[12, 0]};
a3 = {i[1, 1/4], i[2, 2/4], i[3, 3/4], i[4, 1/4], i[5, -2/4], i[6, -1/4],
    i[7, 3/4], i[8, 2/4], i[9, -1/4], i[10, 0], i[11, 0], i[12, 0]};
b1 = {i[1, 1/4], i[2, 2/4], i[3, 3/4], i[4, -3/4], i[5, -2/4], i[6, -1/4],
    i[7, -1/4], i[8, -2/4], i[9, -1/4], i[10, 0], i[11, 0], i[12, 0]};
c1 = {i[1, 1/4], i[2, 2/4], i[3, -1/4], i[4, 1/4], i[5, -2/4], i[6, -1/4],
    i[7, -1/4], i[8, -2/4], i[9, -1/4], i[10, 0], i[11, 0], i[12, 0]};
c2 = {i[1, 1/4], i[2, 2/4], i[3, -1/4], i[4, 1/4], i[5, -2/4], i[6, -1/4],
    i[7, -1/4], i[8, 2/4], i[9, -1/4], i[10, 0], i[11, 0], i[12, 0]};
d1 = {i[1, 1/4], i[2, 2/4], i[3, -1/4], i[4, -3/4], i[5, -2/4], i[6, -1/4],
    i[7, -1/4], i[8, -2/4], i[9, -1/4], i[10, 0], i[11, 0], i[12, 0]};
```

In[18]:=   (* Checking relations *)
  (* da-ad = (q-q^(-1))bc *)
  x = Expand[(f[d1, a1] + f[d1, a2] + f[d1, a3]) - (f[a1, d1] + f[a2, d1] + f[a3, d1])]
  y = Expand[(q - q^(-1)) * (f[b1, c1] + f[b1, c2])]
  x == y

Out[18]= $-\dfrac{\sqrt{Z1}\ Z2\ \sqrt{Z3}}{q^{3/2}\ \sqrt{X1}\ \sqrt{X3}\ \sqrt{Z1'}\ Z2'\ \sqrt{Z3'}} + \dfrac{\sqrt{q}\ \sqrt{Z1}\ Z2\ \sqrt{Z3}}{\sqrt{X1}\ \sqrt{X3}\ \sqrt{Z1'}\ Z2'\ \sqrt{Z3'}} -$
$\dfrac{\sqrt{q}\ \sqrt{Z1}\ Z2\ \sqrt{Z3}}{\sqrt{X1}\ X2\ \sqrt{X3}\ \sqrt{Z1'}\ Z2'\ \sqrt{Z3'}} + \dfrac{q^{5/2}\ \sqrt{Z1}\ Z2\ \sqrt{Z3}}{\sqrt{X1}\ X2\ \sqrt{X3}\ \sqrt{Z1'}\ Z2'\ \sqrt{Z3'}}$

Out[19]= $-\dfrac{\sqrt{Z1}\ Z2\ \sqrt{Z3}}{q^{3/2}\ \sqrt{X1}\ \sqrt{X3}\ \sqrt{Z1'}\ Z2'\ \sqrt{Z3'}} + \dfrac{\sqrt{q}\ \sqrt{Z1}\ Z2\ \sqrt{Z3}}{\sqrt{X1}\ \sqrt{X3}\ \sqrt{Z1'}\ Z2'\ \sqrt{Z3'}} -$
$\dfrac{\sqrt{q}\ \sqrt{Z1}\ Z2\ \sqrt{Z3}}{\sqrt{X1}\ X2\ \sqrt{X3}\ \sqrt{Z1'}\ Z2'\ \sqrt{Z3'}} + \dfrac{q^{5/2}\ \sqrt{Z1}\ Z2\ \sqrt{Z3}}{\sqrt{X1}\ X2\ \sqrt{X3}\ \sqrt{Z1'}\ Z2'\ \sqrt{Z3'}}$

Out[20]= True

In[21]:=   (* bc = cb *)
  x = Expand[(f[b1, c1] + f[b1, c2])]
  y = Expand[(f[c1, b1] + f[c2, b1])]
  x == y

Out[21]= $\dfrac{\sqrt{Z1}\ Z2\ \sqrt{Z3}}{\sqrt{q}\ \sqrt{X1}\ \sqrt{X3}\ \sqrt{Z1'}\ Z2'\ \sqrt{Z3'}} + \dfrac{q^{3/2}\ \sqrt{Z1}\ Z2\ \sqrt{Z3}}{\sqrt{X1}\ X2\ \sqrt{X3}\ \sqrt{Z1'}\ Z2'\ \sqrt{Z3'}}$

Out[22]= $\dfrac{\sqrt{Z1}\ Z2\ \sqrt{Z3}}{\sqrt{q}\ \sqrt{X1}\ \sqrt{X3}\ \sqrt{Z1'}\ Z2'\ \sqrt{Z3'}} + \dfrac{q^{3/2}\ \sqrt{Z1}\ Z2\ \sqrt{Z3}}{\sqrt{X1}\ X2\ \sqrt{X3}\ \sqrt{Z1'}\ Z2'\ \sqrt{Z3'}}$

Out[23]= True

In[24]:=   (* ca = qac *)
  x = Expand[(f[c1, a1] + f[c1, a2] + f[c1, a3] + f[c2, a1] + f[c2, a2] + f[c2, a3])]
  y = Expand[q * (f[a1, c1] + f[a2, c1] + f[a3, c1] + f[a1, c2] + f[a2, c2] + f[a3, c2])]
  x == y

Out[24]= $\dfrac{\sqrt{Z1}\ Z2\ \sqrt{Z3}\ \sqrt{Z3'}}{\sqrt{X1}\ \sqrt{X3}\ \sqrt{Z1'}\ Z2'} + \dfrac{\sqrt{Z1}\ Z2\ \sqrt{Z3}\ \sqrt{Z3'}}{q^2\ \sqrt{X1}\ \sqrt{X3}\ \sqrt{Z1'}\ Z2'} + \dfrac{\sqrt{X1}\ \sqrt{Z1}\ Z2\ \sqrt{Z3}\ \sqrt{Z3'}}{\sqrt{X3}\ \sqrt{Z1'}\ Z2'} +$
$\dfrac{q^2\ \sqrt{Z1}\ Z2\ \sqrt{Z3}\ \sqrt{Z3'}}{\sqrt{X1}\ X2\ \sqrt{X3}\ \sqrt{Z1'}\ Z2'} + \dfrac{X2\ \sqrt{Z1}\ Z2\ \sqrt{Z3}\ \sqrt{Z3'}}{q^4\ \sqrt{X1}\ \sqrt{X3}\ \sqrt{Z1'}\ Z2'} + \dfrac{\sqrt{X1}\ X2\ \sqrt{Z1}\ Z2\ \sqrt{Z3}\ \sqrt{Z3'}}{q^2\ \sqrt{X3}\ \sqrt{Z1'}\ Z2'}$

Out[25]= $\dfrac{\sqrt{Z1}\ Z2\ \sqrt{Z3}\ \sqrt{Z3'}}{\sqrt{X1}\ \sqrt{X3}\ \sqrt{Z1'}\ Z2'} + \dfrac{\sqrt{Z1}\ Z2\ \sqrt{Z3}\ \sqrt{Z3'}}{q^2\ \sqrt{X1}\ \sqrt{X3}\ \sqrt{Z1'}\ Z2'} + \dfrac{\sqrt{X1}\ \sqrt{Z1}\ Z2\ \sqrt{Z3}\ \sqrt{Z3'}}{\sqrt{X3}\ \sqrt{Z1'}\ Z2'} +$
$\dfrac{q^2\ \sqrt{Z1}\ Z2\ \sqrt{Z3}\ \sqrt{Z3'}}{\sqrt{X1}\ X2\ \sqrt{X3}\ \sqrt{Z1'}\ Z2'} + \dfrac{X2\ \sqrt{Z1}\ Z2\ \sqrt{Z3}\ \sqrt{Z3'}}{q^4\ \sqrt{X1}\ \sqrt{X3}\ \sqrt{Z1'}\ Z2'} + \dfrac{\sqrt{X1}\ X2\ \sqrt{Z1}\ Z2\ \sqrt{Z3}\ \sqrt{Z3'}}{q^2\ \sqrt{X3}\ \sqrt{Z1'}\ Z2'}$

Out[26]= True

In[27]:= `(* dc = qcd *)`
`x = Expand[(f[d1, c1] + f[d1, c2])]`
`y = Expand[q * (f[c1, d1] + f[c2, d1])]`
`x == y`

Out[27]= $\dfrac{Z2\sqrt{Z3}}{\sqrt{X1}\;\sqrt{X3}\;\sqrt{Z1}\;\sqrt{Z1_1}\;Z2_1\;\sqrt{Z3_1}} + \dfrac{q^2\,Z2\sqrt{Z3}}{\sqrt{X1}\;X2\;\sqrt{X3}\;\sqrt{Z1}\;\sqrt{Z1_1}\;Z2_1\;\sqrt{Z3_1}}$

Out[28]= $\dfrac{Z2\sqrt{Z3}}{\sqrt{X1}\;\sqrt{X3}\;\sqrt{Z1}\;\sqrt{Z1_1}\;Z2_1\;\sqrt{Z3_1}} + \dfrac{q^2\,Z2\sqrt{Z3}}{\sqrt{X1}\;X2\;\sqrt{X3}\;\sqrt{Z1}\;\sqrt{Z1_1}\;Z2_1\;\sqrt{Z3_1}}$

Out[29]= True

In[30]:= `(* ba = qab *)`
`x = Expand[(f[b1, a1] + f[b1, a2] + f[b1, a3])]`
`y = Expand[q * (f[a1, b1] + f[a2, b1] + f[a3, b1])]`
`x == y`

Out[30]= $\dfrac{Z1^{3/2}\,Z2\sqrt{Z3}}{\sqrt{X1}\;\sqrt{X3}\;\sqrt{Z1_1}\;Z2_1\;\sqrt{Z3_1}} + \dfrac{\sqrt{X1}\;Z1^{3/2}\,Z2\sqrt{Z3}}{\sqrt{X3}\;\sqrt{Z1_1}\;Z2_1\;\sqrt{Z3_1}} + \dfrac{q^2\,Z1^{3/2}\,Z2\sqrt{Z3}}{\sqrt{X1}\;X2\;\sqrt{X3}\;\sqrt{Z1_1}\;Z2_1\;\sqrt{Z3_1}}$

Out[31]= $\dfrac{Z1^{3/2}\,Z2\sqrt{Z3}}{\sqrt{X1}\;\sqrt{X3}\;\sqrt{Z1_1}\;Z2_1\;\sqrt{Z3_1}} + \dfrac{\sqrt{X1}\;Z1^{3/2}\,Z2\sqrt{Z3}}{\sqrt{X3}\;\sqrt{Z1_1}\;Z2_1\;\sqrt{Z3_1}} + \dfrac{q^2\,Z1^{3/2}\,Z2\sqrt{Z3}}{\sqrt{X1}\;X2\;\sqrt{X3}\;\sqrt{Z1_1}\;Z2_1\;\sqrt{Z3_1}}$

Out[32]= True

In[33]:= `(* db = qbd *)`
`x = Expand[(f[d1, b1])]`
`y = Expand[q * (f[b1, d1])]`
`x == y`

Out[33]= $\dfrac{q^2\,\sqrt{Z1}\;Z2\sqrt{Z3}}{\sqrt{X1}\;X2\;\sqrt{X3}\;\sqrt{Z1_1}\;Z2_1\;Z3_1^{3/2}}$

Out[34]= $\dfrac{q^2\,\sqrt{Z1}\;Z2\sqrt{Z3}}{\sqrt{X1}\;X2\;\sqrt{X3}\;\sqrt{Z1_1}\;Z2_1\;Z3_1^{3/2}}$

Out[35]= True

In[36]:= (* **SECTION 2.1.2: n=
4: CHECKING COMMUTATION RELATIONS FOR QUANTIZED 2x2 SUB-MATRIX OF RIGHT SL4 MATRIX** *)

(* **Encoding 2x2 sub-matrix, e.g. d := [[d_1]]+[[d_2]]+[[d_3]]** *)

(* **Recall** *)

(* **Z3=1,Z2=2,Z1=3,Z3′=4,Z2′=5,Z1′=6,X1=7,X2=8,X3=9,Z3″=10,Z2″=11,Z1″=12** *)

a1 = {i[1, 1/4], i[2, -1/2], i[3, -1/4], i[4, 0], i[5, 0],
   i[6, 0], i[8, 1/4], i[7, 1/2], i[9, 1/4], i[10, 1/4], i[11, 1/2], i[12, 3/4]};

b1 = {i[1, 1/4], i[2, -1/2], i[3, -1/4], i[4, 0], i[5, 0], i[6, 0], i[8, 1/4],
   i[7, -1/2], i[9, 1/4], i[10, 1/4], i[11, 1/2], i[12, -1/4]};

b2 = {i[1, 1/4], i[2, -1/2], i[3, -1/4], i[4, 0], i[5, 0], i[6, 0], i[8, 1/4],
   i[7, 1/2], i[9, 1/4], i[10, 1/4], i[11, 1/2], i[12, -1/4]};

c1 = {i[1, -3/4], i[2, -1/2], i[3, -1/4], i[4, 0], i[5, 0], i[6, 0], i[8, 1/4],
   i[7, 1/2], i[9, 1/4], i[10, 1/4], i[11, 1/2], i[12, 1/4]};

d1 = {i[1, -3/4], i[2, -1/2], i[3, -1/4], i[4, 0], i[5, 0], i[6, 0], i[8, -3/4],
   i[7, -1/2], i[9, 1/4], i[10, 1/4], i[11, 1/2], i[12, -1/4]};

d2 = {i[1, -3/4], i[2, -1/2], i[3, -1/4], i[4, 0], i[5, 0], i[6, 0], i[8, 1/4],
   i[7, -1/2], i[9, 1/4], i[10, 1/4], i[11, 1/2], i[12, -1/4]};

d3 = {i[1, -3/4], i[2, -1/2], i[3, -1/4], i[4, 0], i[5, 0], i[6, 0], i[8, 1/4],
   i[7, 1/2], i[9, 1/4], i[10, 1/4], i[11, 1/2], i[12, -1/4]};

In[43]:= (* **da-ad = (q-q^(-1))bc** *)
x = Expand[(f[d1, a1] + f[d2, a1] + f[d3, a1]) - (f[a1, d1] + f[a1, d2] + f[a1, d3])]
y = Expand[(q - q^(-1)) * (f[b1, c1] + f[b2, c1])]
x = y

Out[43]= $-\dfrac{\sqrt{X2}\ \sqrt{X3}\ \sqrt{Z1''}\ Z2''\ \sqrt{Z3''}}{q\ \sqrt{Z1}\ Z2\ \sqrt{Z3}} + \dfrac{q\ \sqrt{X2}\ \sqrt{X3}\ \sqrt{Z1''}\ Z2''\ \sqrt{Z3''}}{\sqrt{Z1}\ Z2\ \sqrt{Z3}} -$
$\dfrac{X1\ \sqrt{X2}\ \sqrt{X3}\ \sqrt{Z1''}\ Z2''\ \sqrt{Z3''}}{q\ \sqrt{Z1}\ Z2\ \sqrt{Z3}} + \dfrac{q\ X1\ \sqrt{X2}\ \sqrt{X3}\ \sqrt{Z1''}\ Z2''\ \sqrt{Z3''}}{\sqrt{Z1}\ Z2\ \sqrt{Z3}}$

Out[44]= $-\dfrac{\sqrt{X2}\ \sqrt{X3}\ \sqrt{Z1''}\ Z2''\ \sqrt{Z3''}}{q\ \sqrt{Z1}\ Z2\ \sqrt{Z3}} + \dfrac{q\ \sqrt{X2}\ \sqrt{X3}\ \sqrt{Z1''}\ Z2''\ \sqrt{Z3''}}{\sqrt{Z1}\ Z2\ \sqrt{Z3}} -$
$\dfrac{X1\ \sqrt{X2}\ \sqrt{X3}\ \sqrt{Z1''}\ Z2''\ \sqrt{Z3''}}{q\ \sqrt{Z1}\ Z2\ \sqrt{Z3}} + \dfrac{q\ X1\ \sqrt{X2}\ \sqrt{X3}\ \sqrt{Z1''}\ Z2''\ \sqrt{Z3''}}{\sqrt{Z1}\ Z2\ \sqrt{Z3}}$

Out[45]= True

In[46]:= (* **bc = cb** *)
x = Expand[(f[b1, c1] + f[b2, c1])]
y = Expand[(f[c1, b1] + f[c1, b2])]
x = y

Out[46]= $\dfrac{\sqrt{X2}\ \sqrt{X3}\ \sqrt{Z1''}\ Z2''\ \sqrt{Z3''}}{\sqrt{Z1}\ Z2\ \sqrt{Z3}} + \dfrac{X1\ \sqrt{X2}\ \sqrt{X3}\ \sqrt{Z1''}\ Z2''\ \sqrt{Z3''}}{\sqrt{Z1}\ Z2\ \sqrt{Z3}}$

Out[47]= $\dfrac{\sqrt{X2}\ \sqrt{X3}\ \sqrt{Z1''}\ Z2''\ \sqrt{Z3''}}{\sqrt{Z1}\ Z2\ \sqrt{Z3}} + \dfrac{X1\ \sqrt{X2}\ \sqrt{X3}\ \sqrt{Z1''}\ Z2''\ \sqrt{Z3''}}{\sqrt{Z1}\ Z2\ \sqrt{Z3}}$

Out[48]= True

```
In[49]:=    (* ca = qac *)
           x = Expand[(f[c1, a1])]
           y = Expand[q * (f[a1, c1])]
           x == y
```

$$\text{Out[49]=} \quad \frac{X1 \sqrt{X2} \sqrt{X3} \ Z1''^{3/2} \ Z2'' \sqrt{Z3''}}{\sqrt{q} \ \sqrt{Z1} \ Z2 \ \sqrt{Z3}}$$

$$\text{Out[50]=} \quad \frac{X1 \sqrt{X2} \sqrt{X3} \ Z1''^{3/2} \ Z2'' \sqrt{Z3''}}{\sqrt{q} \ \sqrt{Z1} \ Z2 \ \sqrt{Z3}}$$

Out[51]=  True

```
In[52]:=    (* dc = qcd *)
           x = Expand[(f[d1, c1] + f[d2, c1] + f[d3, c1])]
           y = Expand[q * (f[c1, d1] + f[c1, d2] + f[c1, d3])]
           x == y
```

$$\text{Out[52]=} \quad \frac{\sqrt{X3} \ \sqrt{Z1''} \ Z2'' \sqrt{Z3''}}{\sqrt{q} \ \sqrt{X2} \ \sqrt{Z1} \ Z2 \ Z3^{3/2}} + \frac{\sqrt{X2} \ \sqrt{X3} \ \sqrt{Z1''} \ Z2'' \sqrt{Z3''}}{\sqrt{q} \ \sqrt{Z1} \ Z2 \ Z3^{3/2}} + \frac{X1 \sqrt{X2} \ \sqrt{X3} \ \sqrt{Z1''} \ Z2'' \sqrt{Z3''}}{\sqrt{q} \ \sqrt{Z1} \ Z2 \ Z3^{3/2}}$$

$$\text{Out[53]=} \quad \frac{\sqrt{X3} \ \sqrt{Z1''} \ Z2'' \sqrt{Z3''}}{\sqrt{q} \ \sqrt{X2} \ \sqrt{Z1} \ Z2 \ Z3^{3/2}} + \frac{\sqrt{X2} \ \sqrt{X3} \ \sqrt{Z1''} \ Z2'' \sqrt{Z3''}}{\sqrt{q} \ \sqrt{Z1} \ Z2 \ Z3^{3/2}} + \frac{X1 \sqrt{X2} \ \sqrt{X3} \ \sqrt{Z1''} \ Z2'' \sqrt{Z3''}}{\sqrt{q} \ \sqrt{Z1} \ Z2 \ Z3^{3/2}}$$

Out[54]=  True

```
In[55]:=    (* ba = qab *)
           x = Expand[(f[b1, a1] + f[b2, a1])]
           y = Expand[q * (f[a1, b1] + f[a1, b2])]
           x == y
```

$$\text{Out[55]=} \quad \frac{q^{3/2} \sqrt{X2} \ \sqrt{X3} \ \sqrt{Z1''} \ Z2'' \sqrt{Z3} \ \sqrt{Z3''}}{\sqrt{Z1} \ Z2} + \frac{q^{3/2} \ X1 \sqrt{X2} \ \sqrt{X3} \ \sqrt{Z1''} \ Z2'' \sqrt{Z3} \ \sqrt{Z3''}}{\sqrt{Z1} \ Z2}$$

$$\text{Out[56]=} \quad \frac{q^{3/2} \sqrt{X2} \ \sqrt{X3} \ \sqrt{Z1''} \ Z2'' \sqrt{Z3} \ \sqrt{Z3''}}{\sqrt{Z1} \ Z2} + \frac{q^{3/2} \ X1 \sqrt{X2} \ \sqrt{X3} \ \sqrt{Z1''} \ Z2'' \sqrt{Z3} \ \sqrt{Z3''}}{\sqrt{Z1} \ Z2}$$

Out[57]=  True

```
In[58]:=    (* db = qbd *)
           x = Expand[(f[d1, b1] + f[d2, b1] + f[d3, b1] + f[d1, b2] + f[d2, b2] + f[d3, b2])]
           y = Expand[q * (f[b1, d1] + f[b1, d2] + f[b1, d3] + f[b2, d1] + f[b2, d2] + f[b2, d3])]
           x == y
```

$$\text{Out[58]=} \quad \frac{\sqrt{X3} \ Z2'' \sqrt{Z3''}}{\sqrt{q} \ \sqrt{X2} \ \sqrt{Z1} \ \sqrt{Z1''} \ Z2 \ \sqrt{Z3}} + \frac{\sqrt{X3} \ Z2'' \sqrt{Z3''}}{\sqrt{q} \ X1 \sqrt{X2} \ \sqrt{Z1} \ \sqrt{Z1''} \ Z2 \ \sqrt{Z3}} + \frac{\sqrt{X2} \ \sqrt{X3} \ Z2'' \sqrt{Z3''}}{\sqrt{q} \ \sqrt{Z1} \ \sqrt{Z1''} \ Z2 \ \sqrt{Z3}} +$$

$$\frac{q^{3/2} \sqrt{X2} \ \sqrt{X3} \ Z2'' \sqrt{Z3''}}{\sqrt{Z1} \ \sqrt{Z1''} \ Z2 \ \sqrt{Z3}} + \frac{\sqrt{X2} \ \sqrt{X3} \ Z2'' \sqrt{Z3''}}{\sqrt{q} \ X1 \sqrt{Z1} \ \sqrt{Z1''} \ Z2 \ \sqrt{Z3}} + \frac{q^{3/2} \ X1 \sqrt{X2} \ \sqrt{X3} \ Z2'' \sqrt{Z3''}}{\sqrt{Z1} \ \sqrt{Z1''} \ Z2 \ \sqrt{Z3}}$$

$$\text{Out[59]=} \quad \frac{\sqrt{X3} \ Z2'' \sqrt{Z3''}}{\sqrt{q} \ \sqrt{X2} \ \sqrt{Z1} \ \sqrt{Z1''} \ Z2 \ \sqrt{Z3}} + \frac{\sqrt{X3} \ Z2'' \sqrt{Z3''}}{\sqrt{q} \ X1 \sqrt{X2} \ \sqrt{Z1} \ \sqrt{Z1''} \ Z2 \ \sqrt{Z3}} + \frac{\sqrt{X2} \ \sqrt{X3} \ Z2'' \sqrt{Z3''}}{\sqrt{q} \ \sqrt{Z1} \ \sqrt{Z1''} \ Z2 \ \sqrt{Z3}} +$$

$$\frac{q^{3/2} \sqrt{X2} \ \sqrt{X3} \ Z2'' \sqrt{Z3''}}{\sqrt{Z1} \ \sqrt{Z1''} \ Z2 \ \sqrt{Z3}} + \frac{\sqrt{X2} \ \sqrt{X3} \ Z2'' \sqrt{Z3''}}{\sqrt{q} \ X1 \sqrt{Z1} \ \sqrt{Z1''} \ Z2 \ \sqrt{Z3}} + \frac{q^{3/2} \ X1 \sqrt{X2} \ \sqrt{X3} \ Z2'' \sqrt{Z3''}}{\sqrt{Z1} \ \sqrt{Z1''} \ Z2 \ \sqrt{Z3}}$$

Out[60]=  True

In[61]:= `(* SECTION 2.2:  n=`
`   3:  CHECKING COMMUTATION RELATIONS FOR QUANTIZED 2x2 SL3 SUB-MATRICES *)`

In[62]:= `        (* see Section 1 *)`
`n = 5;`
`        (* W=1,Z=2,X=3,Z₁=4,W₁=5 *)`
`P = {{0, -1, 2, 0, 0},`
`     {1, 0, -2, 0, 2},`
`     {-2, 2, 0, 2, -2},`
`     {0, 0, -2, 0, 1},`
`     {0, -2, 2, -1, 0}};`
`c[w__] := (temp = ConstantArray[0, n];`
`     tempw = Flatten[w];`
`     Do[ temp[[Re[tempw[[i]]]]] = temp[[Re[tempw[[i]]]]] + Im[tempw[[i]]], {i, Length[tempw]}];`
`        (* W1=1,Z1=2,W2=3,Z2=4,W3=5,Z3=6,X=7 *)`
`     W^temp[[1]] * Z^temp[[2]] * X^temp[[3]] * Z'^temp[[4]] * W'^temp[[5]]);`

In[65]:= `(* SECTION 2.2.1:  n=`
`   3:  CHECKING COMMUTATION RELATIONS FOR QUANTIZED 2x2 SUB-MATRIX OF LEFT SL3 MATRIX *)`

In[66]:= `        (* Encoding 2x2 sub-matrix *)`
`        (* Recall *)`
`        (* W=1,Z=2,X=3,Z₁=4,W₁=5 *)`
`a1 = {i[1, 2/3], i[2, 1/3], i[3, 2/3], i[4, -1/3], i[5, 1/3]};`
`a2 = {i[1, 2/3], i[2, 1/3], i[3, -1/3], i[4, -1/3], i[5, 1/3]};`
`b1 = {i[1, 2/3], i[2, 1/3], i[3, -1/3], i[4, -1/3], i[5, -2/3]};`
`c1 = {i[1, -1/3], i[2, 1/3], i[3, -1/3], i[4, -1/3], i[5, 1/3]};`
`d1 = {i[1, -1/3], i[2, 1/3], i[3, -1/3], i[4, -1/3], i[5, -2/3]};`

In[71]:= `        (* Checking relations *)`
`        (* da-ad = (q-q^(-1))bc *)`
`x = Expand[(f[d1, a1] + f[d1, a2]) - (f[a1, d1] + f[a2, d1])]`
`y = Expand[(q - q^(-1)) * f[b1, c1]]`
`x == y`

Out[71]= $-\dfrac{W^{1/3} Z^{2/3}}{q^{11/9} W'^{1/3} X^{2/3} Z'^{2/3}} + \dfrac{q^{7/9} W^{1/3} Z^{2/3}}{W'^{1/3} X^{2/3} Z'^{2/3}}$

Out[72]= $-\dfrac{W^{1/3} Z^{2/3}}{q^{11/9} W'^{1/3} X^{2/3} Z'^{2/3}} + \dfrac{q^{7/9} W^{1/3} Z^{2/3}}{W'^{1/3} X^{2/3} Z'^{2/3}}$

Out[73]= `True`

In[74]:= `        (* bc = cb *)`
`x = Expand[f[b1, c1]]`
`y = Expand[f[c1, b1]]`
`x == y`

Out[74]= $\dfrac{W^{1/3} Z^{2/3}}{q^{2/9} W'^{1/3} X^{2/3} Z'^{2/3}}$

Out[75]= $\dfrac{W^{1/3} Z^{2/3}}{q^{2/9} W'^{1/3} X^{2/3} Z'^{2/3}}$

Out[76]= `True`

In[77]:=     (* ca = qac *)
x = Expand[(f[c1, a1] + f[c1, a2])]
y = Expand[q * (f[a1, c1] + f[a2, c1])]
x == y

Out[77]= $\dfrac{W^{1/3}\,W'^{2/3}\,Z^{2/3}}{q^{13/18}\,X^{2/3}\,Z'^{2/3}} + \dfrac{q^{17/18}\,W^{1/3}\,W'^{2/3}\,X^{1/3}\,Z^{2/3}}{Z'^{2/3}}$

Out[78]= $\dfrac{W^{1/3}\,W'^{2/3}\,Z^{2/3}}{q^{13/18}\,X^{2/3}\,Z'^{2/3}} + \dfrac{q^{17/18}\,W^{1/3}\,W'^{2/3}\,X^{1/3}\,Z^{2/3}}{Z'^{2/3}}$

Out[79]= True

In[80]:=     (* dc = qcd *)
x = Expand[f[d1, c1]]
y = Expand[q * f[c1, d1]]
x == y

Out[80]= $\dfrac{Z^{2/3}}{q^{13/18}\,W^{2/3}\,W'^{1/3}\,X^{2/3}\,Z'^{2/3}}$

Out[81]= $\dfrac{Z^{2/3}}{q^{13/18}\,W^{2/3}\,W'^{1/3}\,X^{2/3}\,Z'^{2/3}}$

Out[82]= True

In[83]:=     (* ba = qab *)
x = Expand[(f[b1, a1] + f[b1, a2])]
y = Expand[q * (f[a1, b1] + f[a2, b1])]
x == y

Out[83]= $\dfrac{q^{23/18}\,W^{4/3}\,Z^{2/3}}{W'^{1/3}\,X^{2/3}\,Z'^{2/3}} + \dfrac{q^{17/18}\,W^{4/3}\,X^{1/3}\,Z^{2/3}}{W'^{1/3}\,Z'^{2/3}}$

Out[84]= $\dfrac{q^{23/18}\,W^{4/3}\,Z^{2/3}}{W'^{1/3}\,X^{2/3}\,Z'^{2/3}} + \dfrac{q^{17/18}\,W^{4/3}\,X^{1/3}\,Z^{2/3}}{W'^{1/3}\,Z'^{2/3}}$

Out[85]= True

In[86]:=     (* db = qbd *)
x = Expand[f[d1, b1]]
y = Expand[q * f[b1, d1]]
x == y

Out[86]= $\dfrac{q^{23/18}\,W^{1/3}\,Z^{2/3}}{W'^{4/3}\,X^{2/3}\,Z'^{2/3}}$

Out[87]= $\dfrac{q^{23/18}\,W^{1/3}\,Z^{2/3}}{W'^{4/3}\,X^{2/3}\,Z'^{2/3}}$

Out[88]= True

In[89]:= (* SECTION 2.2.2: n=
    3: CHECKING COMMUTATION RELATIONS FOR QUANTIZED 2x2 SUB-MATRIX OF RIGHT SL3 MATRIX *)

```
In[90]:=    (* Encoding 2x2 sub-matrix *)
            (* Recall *)
            (* W=1,Z=2,X=3,Z=4,W=5 *)
            a1 = {i[1, 1/3], i[2, 2/3], i[3, 1/3], i[4, 1/3], i[5, -1/3]};
            b1 = {i[1, 1/3], i[2, -1/3], i[3, 1/3], i[4, 1/3], i[5, -1/3]};
            c1 = {i[1, 1/3], i[2, 2/3], i[3, 1/3], i[4, -2/3], i[5, -1/3]};
            d1 = {i[1, 1/3], i[2, -1/3], i[3, 1/3], i[4, -2/3], i[5, -1/3]};
            d2 = {i[1, 1/3], i[2, -1/3], i[3, -2/3], i[4, -2/3], i[5, -1/3]};
```

```
In[95]:=    (* Checking relations *)
            (* da-ad = (q-q^(-1))bc *)
            x = Expand[(f[d1, a1] + f[d2, a1]) - (f[a1, d1] + f[a1, d2])]
            y = Expand[(q - q^(-1)) * f[b1, c1]]
            x = y
```

$$\text{Out[95]=} \quad -\frac{W^{2/3} X^{2/3} Z^{1/3}}{q^{11/9} W_r^{2/3} Z_r^{1/3}} + \frac{q^{7/9} W^{2/3} X^{2/3} Z^{1/3}}{W_r^{2/3} Z_r^{1/3}}$$

$$\text{Out[96]=} \quad -\frac{W^{2/3} X^{2/3} Z^{1/3}}{q^{11/9} W_r^{2/3} Z_r^{1/3}} + \frac{q^{7/9} W^{2/3} X^{2/3} Z^{1/3}}{W_r^{2/3} Z_r^{1/3}}$$

Out[97]=   True

```
In[98]:=    (* bc = cb *)
            x = Expand[f[b1, c1]]
            y = Expand[f[c1, b1]]
            x = y
```

$$\text{Out[98]=} \quad \frac{W^{2/3} X^{2/3} Z^{1/3}}{q^{2/9} W_r^{2/3} Z_r^{1/3}}$$

$$\text{Out[99]=} \quad \frac{W^{2/3} X^{2/3} Z^{1/3}}{q^{2/9} W_r^{2/3} Z_r^{1/3}}$$

Out[100]=  True

```
In[101]:=   (* ca = qac *)
            x = Expand[f[c1, a1]]
            y = Expand[q * f[a1, c1]]
            x = y
```

$$\text{Out[101]=} \quad \frac{q^{35/18} W^{2/3} X^{2/3} Z^{4/3}}{W_r^{2/3} Z_r^{1/3}}$$

$$\text{Out[102]=} \quad \frac{q^{35/18} W^{2/3} X^{2/3} Z^{4/3}}{W_r^{2/3} Z_r^{1/3}}$$

Out[103]=  True

In[104]:=      `(* dc = qcd *)`
     `x = Expand[(f[d1, c1] + f[d2, c1])]`
     `y = Expand[q*(f[c1, d1] + f[c1, d2])]`
     `x == y`

Out[104]= $\dfrac{q^{5/18}\,W^{2/3}\,Z^{1/3}}{W^{2/3}\,X^{1/3}\,Z^{4/3}} + \dfrac{q^{11/18}\,W^{2/3}\,X^{2/3}\,Z^{1/3}}{W^{2/3}\,Z^{4/3}}$

Out[105]= $\dfrac{q^{5/18}\,W^{2/3}\,Z^{1/3}}{W^{2/3}\,X^{1/3}\,Z^{4/3}} + \dfrac{q^{11/18}\,W^{2/3}\,X^{2/3}\,Z^{1/3}}{W^{2/3}\,Z^{4/3}}$

Out[106]= `True`

In[107]:=      `(* ba = qab *)`
     `x = Expand[f[b1, a1]]`
     `y = Expand[q*f[a1, b1]]`
     `x == y`

Out[107]= $\dfrac{W^{2/3}\,X^{2/3}\,Z^{1/3}\,Z^{2/3}}{q^{1/18}\,W^{2/3}}$

Out[108]= $\dfrac{W^{2/3}\,X^{2/3}\,Z^{1/3}\,Z^{2/3}}{q^{1/18}\,W^{2/3}}$

Out[109]= `True`

In[110]:=      `(* db = qbd *)`
     `x = Expand[(f[d1, b1] + f[d2, b1])]`
     `y = Expand[q*(f[b1, d1] + f[b1, d2])]`
     `x == y`

Out[110]= $\dfrac{q^{5/18}\,W^{2/3}}{W^{2/3}\,X^{1/3}\,Z^{2/3}\,Z^{1/3}} + \dfrac{W^{2/3}\,X^{2/3}}{q^{25/18}\,W^{2/3}\,Z^{2/3}\,Z^{1/3}}$

Out[111]= $\dfrac{q^{5/18}\,W^{2/3}}{W^{2/3}\,X^{1/3}\,Z^{2/3}\,Z^{1/3}} + \dfrac{W^{2/3}\,X^{2/3}}{q^{25/18}\,W^{2/3}\,Z^{2/3}\,Z^{1/3}}$

Out[112]= `True`

In[113]:= 
```
(* SECTION 3: MOVES (I)-(IV) *)
    (* see Section 1 *)
n = 7;
    (* W1=1,Z1=2,W2=3,Z2=4,W3=5,Z3=6,X=7, i.e. going around in order clockwise *)
P = {{0, -1, 0, 0, 0, -2, 2},
    {1, 0, 2, 0, 0, 0, -2},
    {0, -2, 0, -1, 0, 0, 2},
    {0, 0, 1, 0, 2, 0, -2},
    {0, 0, 0, -2, 0, -1, 2},
    {2, 0, 0, 0, 1, 0, -2},
    {-2, 2, -2, 2, -2, 2, 0}};
c[w__] := (temp = ConstantArray[0, n];
    tempw = Flatten[w];
    Do[ temp[[Re[tempw[[i]]]]] = temp[[Re[tempw[[i]]]]] + Im[tempw[[i]]], {i, Length[tempw]}];
    (* W1=1,Z1=2,W2=3,Z2=4,W3=5,Z3=6,X=7 *)
    W1^temp[[1]] *Z1^temp[[2]] *W2^temp[[3]] * Z2^temp[[4]] *W3^temp[[5]] *Z3^temp[[6]] *X^temp[[7]]);
```

In[116]:= (* SECTION 3.1: ENCODING THE THREE LEFT MATRICES
       (LOWER CASE LETTERS) AND THREE RIGHT MATRICES (CAPITAL LETTERS) *)
       (* W1=1,Z1=2,W2=3,Z2=4,W3=5,Z3=6,X=7 *)
       (* going from 1-edge to 2-edge *)
       a3 = {i[1, 2/3], i[2, 1/3], i[3, 1/3], i[4, 2/3], i[7, 2/3]};
    b31 = {i[1, 2/3], i[2, 1/3], i[3, 1/3], i[4, -1/3], i[7, -1/3]};
    b32 = {i[1, 2/3], i[2, 1/3], i[3, 1/3], i[4, -1/3], i[7, 2/3]};
     c3 = {i[1, 2/3], i[2, 1/3], i[3, -2/3], i[4, -1/3], i[7, -1/3]};
     e3 = {i[1, -1/3], i[2, 1/3], i[3, 1/3], i[4, -1/3], i[7, -1/3]};
     f3 = {i[1, -1/3], i[2, 1/3], i[3, -2/3], i[4, -1/3], i[7, -1/3]};
     i3 = {i[1, -1/3], i[2, -2/3], i[3, -2/3], i[4, -1/3], i[7, -1/3]};
       A3 = {i[1, 1/3], i[2, 2/3], i[3, 2/3], i[4, 1/3], i[7, 1/3]};
    D3 = {i[1, 1/3], i[2, 2/3], i[3, -1/3], i[4, 1/3], i[7, 1/3]};
    E3 = {i[1, 1/3], i[2, -1/3], i[3, -1/3], i[4, 1/3], i[7, 1/3]};
    G3 = {i[1, 1/3], i[2, 2/3], i[3, -1/3], i[4, -2/3], i[7, 1/3]};
    H31 = {i[1, 1/3], i[2, -1/3], i[3, -1/3], i[4, -2/3], i[7, -2/3]};
    H32 = {i[1, 1/3], i[2, -1/3], i[3, -1/3], i[4, -2/3], i[7, 1/3]};
    I3 = {i[1, -2/3], i[2, -1/3], i[3, -1/3], i[4, -2/3], i[7, -2/3]};
       (* going from 2-edge to 3-edge *)
       a1 = {i[3, 2/3], i[4, 1/3], i[5, 1/3], i[6, 2/3], i[7, 2/3]};
    b11 = {i[3, 2/3], i[4, 1/3], i[5, 1/3], i[6, -1/3], i[7, -1/3]};
    b12 = {i[3, 2/3], i[4, 1/3], i[5, 1/3], i[6, -1/3], i[7, 2/3]};
     c1 = {i[3, 2/3], i[4, 1/3], i[5, -2/3], i[6, -1/3], i[7, -1/3]};
     e1 = {i[3, -1/3], i[4, 1/3], i[5, 1/3], i[6, -1/3], i[7, -1/3]};
     f1 = {i[3, -1/3], i[4, 1/3], i[5, -2/3], i[6, -1/3], i[7, -1/3]};
     i1 = {i[3, -1/3], i[4, -2/3], i[5, -2/3], i[6, -1/3], i[7, -1/3]};
       A1 = {i[3, 1/3], i[4, 2/3], i[5, 2/3], i[6, 1/3], i[7, 1/3]};
    D1 = {i[3, 1/3], i[4, 2/3], i[5, -1/3], i[6, 1/3], i[7, 1/3]};
    E1 = {i[3, 1/3], i[4, -1/3], i[5, -1/3], i[6, 1/3], i[7, 1/3]};
    G1 = {i[3, 1/3], i[4, 2/3], i[5, -1/3], i[6, -2/3], i[7, 1/3]};
    H11 = {i[3, 1/3], i[4, -1/3], i[5, -1/3], i[6, -2/3], i[7, -2/3]};
    H12 = {i[3, 1/3], i[4, -1/3], i[5, -1/3], i[6, -2/3], i[7, 1/3]};
    I1 = {i[3, -2/3], i[4, -1/3], i[5, -1/3], i[6, -2/3], i[7, -2/3]};
       (* going from 3-edge to 1-edge *)
       a2 = {i[5, 2/3], i[6, 1/3], i[1, 1/3], i[2, 2/3], i[7, 2/3]};
    b21 = {i[5, 2/3], i[6, 1/3], i[1, 1/3], i[2, -1/3], i[7, -1/3]};
    b22 = {i[5, 2/3], i[6, 1/3], i[1, 1/3], i[2, -1/3], i[7, 2/3]};
     c2 = {i[5, 2/3], i[6, 1/3], i[1, -2/3], i[2, -1/3], i[7, -1/3]};
     e2 = {i[5, -1/3], i[6, 1/3], i[1, 1/3], i[2, -1/3], i[7, -1/3]};
     f2 = {i[5, -1/3], i[6, 1/3], i[1, -2/3], i[2, -1/3], i[7, -1/3]};
     i2 = {i[5, -1/3], i[6, -2/3], i[1, -2/3], i[2, -1/3], i[7, -1/3]};
       A2 = {i[5, 1/3], i[6, 2/3], i[1, 2/3], i[2, 1/3], i[7, 1/3]};
    D2 = {i[5, 1/3], i[6, 2/3], i[1, -1/3], i[2, 1/3], i[7, 1/3]};
    E2 = {i[5, 1/3], i[6, -1/3], i[1, -1/3], i[2, 1/3], i[7, 1/3]};
    G2 = {i[5, 1/3], i[6, 2/3], i[1, -1/3], i[2, -2/3], i[7, 1/3]};
    H21 = {i[5, 1/3], i[6, -1/3], i[1, -1/3], i[2, -2/3], i[7, -2/3]};
    H22 = {i[5, 1/3], i[6, -1/3], i[1, -1/3], i[2, -2/3], i[7, 1/3]};
    I2 = {i[5, -2/3], i[6, -1/3], i[1, -1/3], i[2, -2/3], i[7, -2/3]};

In[158]:= (* SECTION 3.2 *)
       (* SECTION 3.2.1:  CHECK OF MOVES (I) *)

In[159]:= (* MOVE (I) *)

In[160]:= `(* 11 *)`
`x = Expand[q ^ (-1 / 3) * f[A1, c1] - q ^ (-4 / 3) * (f[D1, b11] + f[D1, b12]) + q ^ (-7 / 3) * f[G1, a1]]`
`y = Expand[0]`
`x = y`

Out[160]= `0`

Out[161]= `0`

Out[162]= `True`

In[163]:= `(* 21 *)`
`x = Expand[q ^ (-1 / 3) * f[A1, f1] - q ^ (-4 / 3) * f[D1, e1]]`
`y = Expand[0]`
`x = y`

Out[163]= `0`

Out[164]= `0`

Out[165]= `True`

In[166]:= `(* 31 *)`
`x = Expand[q ^ (-1 / 3) * f[A1, i1]]`
`y = Expand[q ^ (-1 / 3)]`
`x = y`

Out[166]= $\dfrac{1}{q^{1/3}}$

Out[167]= $\dfrac{1}{q^{1/3}}$

Out[168]= `True`

In[169]:= `(* 12 *)`
`x = Expand[-q ^ (-4 / 3) * (f[E1, b11] + f[E1, b12]) + q ^ (-7 / 3) * (f[H11, a1] + f[H12, a1])]`
`y = Expand[0]`
`x = y`

Out[169]= `0`

Out[170]= `0`

Out[171]= `True`

In[172]:= `(* 22 *)`
`x = Expand[-q ^ (-4 / 3) * f[E1, e1]]`
`y = Expand[-q ^ (-4 / 3)]`
`x = y`

Out[172]= $-\dfrac{1}{q^{4/3}}$

Out[173]= $-\dfrac{1}{q^{4/3}}$

Out[174]= `True`

In[175]:= `(* 13 *)`
`x = Expand[q ^ (-7 / 3) * f[I1, a1]]`
`y = Expand[q ^ (-7 / 3)]`
`x = y`

Out[175]= $\dfrac{1}{q^{7/3}}$

Out[176]= $\dfrac{1}{q^{7/3}}$

Out[177]= `True`

In[178]:= `(* MOVE (I.b) *)`

In[179]:= `(* 11 *)`
`x = Expand[q ^ (-7 / 3) * f[c1, A1] - q ^ (-4 / 3) * (f[b11, D1] + f[b12, D1]) + q ^ (-1 / 3) * f[a1, G1]]`
`y = Expand[0]`
`x = y`

Out[179]= `0`

Out[180]= `0`

Out[181]= `True`

In[182]:= `(* 21 *)`
`x = Expand[q ^ (-7 / 3) * f[f1, A1] - q ^ (-4 / 3) * f[e1, D1]]`
`y = Expand[0]`
`x = y`

Out[182]= `0`

Out[183]= `0`

Out[184]= `True`

In[185]:= `(* 31 *)`
`x = Expand[q ^ (-7 / 3) * f[i1, A1]]`
`y = Expand[q ^ (-7 / 3)]`
`x = y`

Out[185]= $\dfrac{1}{q^{7/3}}$

Out[186]= $\dfrac{1}{q^{7/3}}$

Out[187]= `True`

In[188]:= `(* 12 *)`
`x = Expand[-q ^ (-4 / 3) * (f[b11, E1] + f[b12, E1]) + q ^ (-1 / 3) * (f[a1, H11] + f[a1, H12])]`
`y = Expand[0]`
`x = y`

Out[188]= `0`

Out[189]= `0`

Out[190]= `True`

In[191]:= `(* 22 *)`
`x = Expand[-q ^ (-4 / 3) * f[e1, E1]]`
`y = Expand[-q ^ (-4 / 3)]`
`x = y`

Out[191]= $-\dfrac{1}{q^{4/3}}$

Out[192]= $-\dfrac{1}{q^{4/3}}$

Out[193]= `True`

In[194]:= `(* 13 *)`
`x = Expand[q ^ (-1 / 3) * f[a1, I1]]`
`y = Expand[q ^ (-1 / 3)]`
`x = y`

Out[194]= $\dfrac{1}{q^{1/3}}$

Out[195]= $\dfrac{1}{q^{1/3}}$

Out[196]= `True`

In[197]:= `(* MOVE (I.c) *)`

In[198]:= `(* 31 *)`
`x = Expand[q ^ (-7 / 3) * f[a1, I1]]`
`y = Expand[q ^ (-7 / 3)]`
`x = y`

Out[198]= $\dfrac{1}{q^{7/3}}$

Out[199]= $\dfrac{1}{q^{7/3}}$

Out[200]= `True`

In[201]:= `(* 22 *)`
`x = Expand[-q ^ (4 / 3) * f[e1, E1]]`
`y = Expand[-q ^ (4 / 3)]`
`x = y`

Out[201]= $-q^{4/3}$

Out[202]= $-q^{4/3}$

Out[203]= `True`

In[204]:= `(* 32 *)`
`x = Expand[q ^ (7 / 3) * (f[b11, I1] + f[b12, I1]) - q ^ (4 / 3) * (f[e1, H11] + f[e1, H12])]`
`y = Expand[0]`
`x = y`

Out[204]= `0`

Out[205]= `0`

Out[206]= `True`

In[207]:= `(* 13 *)`
`x = Expand[q ^ (1 / 3) * f[i1, A1]]`
`y = Expand[q ^ (1 / 3)]`
`x = y`

Out[207]= $q^{1/3}$

Out[208]= $q^{1/3}$

Out[209]= `True`

In[210]:= `(* 23 *)`
`x = Expand[-q ^ (4 / 3) * f[f1, E1] + q ^ (1 / 3) * f[i1, D1]]`
`y = Expand[0]`
`x = y`

Out[210]= `0`

Out[211]= `0`

Out[212]= `True`

In[213]:= `(* 33 *)`
`x = Expand[q ^ (7 / 3) * f[c1, I1] - q ^ (4 / 3) * (f[f1, H11] + f[f1, H12]) + q ^ (1 / 3) * f[i1, G1]]`
`y = Expand[0]`
`x = y`

Out[213]= `0`

Out[214]= `0`

Out[215]= `True`

In[216]:= `(* MOVE (I.d) *)`

In[217]:= `(* 31 *)`
`x = Expand[q ^ (1 / 3) f[I1, a1]]`
`y = Expand[q ^ (1 / 3)]`
`x = y`

Out[217]= $q^{1/3}$

Out[218]= $q^{1/3}$

Out[219]= `True`

In[220]:= `(* 22 *)`
`x = Expand[-q^(4/3)*f[E1, e1]]`
`y = Expand[-q^(4/3)]`
`x == y`

Out[220]= $-q^{4/3}$

Out[221]= $-q^{4/3}$

Out[222]= `True`

In[223]:= `(* 32 *)`
`x = Expand[q^(1/3)*(f[I1, b11] + f[I1, b12]) - q^(4/3)*(f[H11, e1] + f[H12, e1])]`
`y = Expand[0]`
`x == y`

Out[223]= `0`

Out[224]= `0`

Out[225]= `True`

In[226]:= `(* 13 *)`
`x = Expand[q^(7/3)*f[A1, i1]]`
`y = Expand[q^(7/3)]`
`x == y`

Out[226]= $q^{7/3}$

Out[227]= $q^{7/3}$

Out[228]= `True`

In[229]:= `(* 23 *)`
`x = Expand[-q^(4/3)*f[E1, f1] + q^(7/3)*f[D1, i1]]`
`y = Expand[0]`
`x == y`

Out[229]= `0`

Out[230]= `0`

Out[231]= `True`

In[232]:= `(* 33 *)`
`x = Expand[q^(1/3)*f[I1, c1] - q^(4/3)*(f[H11, f1] + f[H12, f1]) + q^(7/3)*f[G1, i1]]`
`y = Expand[0]`
`x == y`

Out[232]= `0`

Out[233]= `0`

Out[234]= `True`

In[235]:= `(* SECTION 3.2.2:  CHECK OF MOVES (II) *)`

In[236]:= `(* MOVE (II) *)`

In[237]:=  (* 11 *)
x = Expand[(g[a3])]
y = Expand[(q^(-1/3) * f[A2, G1])]
x == y

Out[237]=  $q^{1/9} W1^{2/3} W2^{1/3} X^{2/3} Z1^{1/3} Z2^{2/3}$

Out[238]=  $q^{1/9} W1^{2/3} W2^{1/3} X^{2/3} Z1^{1/3} Z2^{2/3}$

Out[239]=  True

In[240]:=  (* 12 *)
x = Expand[(g[b31] + g[b32])]
y = Expand[(q^(-1/3) * (f[A2, H11] + f[A2, H12]))]
x == y

Out[240]=  $\dfrac{q^{5/18} W1^{2/3} W2^{1/3} Z1^{1/3}}{X^{1/3} Z2^{1/3}} + \dfrac{W1^{2/3} W2^{1/3} X^{2/3} Z1^{1/3}}{q^{13/18} Z2^{1/3}}$

Out[241]=  $\dfrac{q^{5/18} W1^{2/3} W2^{1/3} Z1^{1/3}}{X^{1/3} Z2^{1/3}} + \dfrac{W1^{2/3} W2^{1/3} X^{2/3} Z1^{1/3}}{q^{13/18} Z2^{1/3}}$

Out[242]=  True

In[243]:=  (* 13 *)
x = Expand[(g[c3])]
y = Expand[(q^(-1/3) * f[A2, I1])]
x == y

Out[243]=  $\dfrac{q^{4/9} W1^{2/3} Z1^{1/3}}{W2^{2/3} X^{1/3} Z2^{1/3}}$

Out[244]=  $\dfrac{q^{4/9} W1^{2/3} Z1^{1/3}}{W2^{2/3} X^{1/3} Z2^{1/3}}$

Out[245]=  True

In[246]:=  (* 21 *)
x = Expand[(0)]
y = Expand[(-q^(-4/3) * f[E2, D1]
        + q^(-1/3) * f[D2, G1])]
x == y

Out[246]=  0

Out[247]=  0

Out[248]=  True

In[249]:=    **(\* 22 \*)**
x = Expand[(g[e3])]
y = Expand[(-q^(-4/3) \* f[E2, E1]
    + q^(-1/3) \* (f[D2, H11] + f[D2, H12]))]
x == y

Out[249]=    $\dfrac{W2^{1/3} \, Z1^{1/3}}{q^{2/9} \, W1^{1/3} \, X^{1/3} \, Z2^{1/3}}$

Out[250]=    $\dfrac{W2^{1/3} \, Z1^{1/3}}{q^{2/9} \, W1^{1/3} \, X^{1/3} \, Z2^{1/3}}$

Out[251]=    True

In[252]:=    **(\* 23 \*)**
x = Expand[(g[f3])]
y = Expand[(q^(-1/3) \* f[D2, I1])]
x == y

Out[252]=    $\dfrac{Z1^{1/3}}{q^{1/18} \, W1^{1/3} \, W2^{2/3} \, X^{1/3} \, Z2^{1/3}}$

Out[253]=    $\dfrac{Z1^{1/3}}{q^{1/18} \, W1^{1/3} \, W2^{2/3} \, X^{1/3} \, Z2^{1/3}}$

Out[254]=    True

In[255]:=    **(\* 31 \*)**
x = Expand[(0)]
y = Expand[(q^(-7/3) \* f[I2, A1]
    - q^(-4/3) \* (f[H21, D1] + f[H22, D1])
    + q^(-1/3) \* f[G2, G1])]
x == y

Out[255]=    0

Out[256]=    0

Out[257]=    True

In[258]:=    **(\* 32 \*)**
x = Expand[(0)]
y = Expand[(-q^(-4/3) \* (f[H21, E1] + f[H22, E1])
    + q^(-1/3) \* (f[G2, H11] + f[G2, H12]))]
x == y

Out[258]=    0

Out[259]=    0

Out[260]=    True

In[261]:=     (* 33 *)
x = Expand[(g[i3])]
y = Expand[(q^(-1/3) * f[G2, I1])]
x == y

Out[261]= $\dfrac{1}{q^{2/9} \, W1^{1/3} \, W2^{2/3} \, X1^{1/3} \, Z1^{2/3} \, Z2^{1/3}}$

Out[262]= $\dfrac{1}{q^{2/9} \, W1^{1/3} \, W2^{2/3} \, X1^{1/3} \, Z1^{2/3} \, Z2^{1/3}}$

Out[263]= True

In[264]:= **(* Move (II.b) *)**

In[265]:= (* 11 *)
x = Expand[q^(-1/3) * f[a2, c1]]
y = Expand[g[A3]]
x == y

Out[265]= $\dfrac{W1^{1/3} \, W2^{2/3} \, X1^{1/3} \, Z1^{2/3} \, Z2^{1/3}}{q^{2/9}}$

Out[266]= $\dfrac{W1^{1/3} \, W2^{2/3} \, X1^{1/3} \, Z1^{2/3} \, Z2^{1/3}}{q^{2/9}}$

Out[267]= True

In[268]:= (* 21 *)
x = Expand[q^(-1/3) * f[a2, f1]]
y = Expand[g[D3]]
x == y

Out[268]= $\dfrac{q^{11/18} \, W1^{1/3} \, X1^{1/3} \, Z1^{2/3} \, Z2^{1/3}}{W2^{1/3}}$

Out[269]= $\dfrac{q^{11/18} \, W1^{1/3} \, X1^{1/3} \, Z1^{2/3} \, Z2^{1/3}}{W2^{1/3}}$

Out[270]= True

In[271]:= (* 31 *)
x = Expand[q^(-1/3) * f[a2, i1]]
y = Expand[g[G3]]
x == y

Out[271]= $\dfrac{q^{4/9} \, W1^{1/3} \, X1^{1/3} \, Z1^{2/3}}{W2^{1/3} \, Z2^{2/3}}$

Out[272]= $\dfrac{q^{4/9} \, W1^{1/3} \, X1^{1/3} \, Z1^{2/3}}{W2^{1/3} \, Z2^{2/3}}$

Out[273]= True

In[274]:= `(* 12 *)`
`x = Expand[q ^ (-1 / 3) * (f[b21, c1] + f[b22, c1]) - q ^ (-4 / 3) * (f[e2, b11] + f[e2, b12])]`
`y = Expand[0]`
`x == y`

Out[274]= `0`

Out[275]= `0`

Out[276]= `True`

In[277]:= `(* 22 *)`
`x = Expand[q ^ (-1 / 3) * (f[b21, f1] + f[b22, f1]) - q ^ (-4 / 3) * f[e2, e1]]`
`y = Expand[g[E3]]`
`x == y`

Out[277]= $\dfrac{W1^{1/3}\, X1^{1/3}\, Z2^{1/3}}{q^{2/9}\, W2^{1/3}\, Z1^{1/3}}$

Out[278]= $\dfrac{W1^{1/3}\, X1^{1/3}\, Z2^{1/3}}{q^{2/9}\, W2^{1/3}\, Z1^{1/3}}$

Out[279]= `True`

In[280]:= `(* 32 *)`
`x = Expand[q ^ (-1 / 3) * (f[b21, i1] + f[b22, i1])]`
`y = Expand[g[H31] + g[H32]]`
`x == y`

Out[280]= $\dfrac{q^{11/18}\, W1^{1/3}}{W2^{1/3}\, X2^{2/3}\, Z1^{1/3}\, Z2^{2/3}} + \dfrac{W1^{1/3}\, X1^{1/3}}{q^{7/18}\, W2^{1/3}\, Z1^{1/3}\, Z2^{2/3}}$

Out[281]= $\dfrac{q^{11/18}\, W1^{1/3}}{W2^{1/3}\, X2^{2/3}\, Z1^{1/3}\, Z2^{2/3}} + \dfrac{W1^{1/3}\, X1^{1/3}}{q^{7/18}\, W2^{1/3}\, Z1^{1/3}\, Z2^{2/3}}$

Out[282]= `True`

In[283]:= `(* 13 *)`
`x = Expand[q ^ (-1 / 3) * f[c2, c1] - q ^ (-4 / 3) * (f[f2, b11] + f[f2, b12]) + q ^ (-7 / 3) * f[i2, a1]]`
`y = Expand[0]`
`x == y`

Out[283]= `0`

Out[284]= `0`

Out[285]= `True`

In[286]:= `(* 23 *)`
`x = Expand[q ^ (-1 / 3) * f[c2, f1] - q ^ (-4 / 3) * f[f2, e1]]`
`y = Expand[0]`
`x == y`

Out[286]= `0`

Out[287]= `0`

Out[288]= `True`

```
In[289]:= (* 33 *)
        x = Expand[q ^ (-1 / 3) * f[c2, i1]]
        y = Expand[g[I3]]
        x = y
```

$$Out[289]= \frac{q^{1/9}}{W1^{2/3} W2^{1/3} X^{2/3} Z1^{1/3} Z2^{2/3}}$$

$$Out[290]= \frac{q^{1/9}}{W1^{2/3} W2^{1/3} X^{2/3} Z1^{1/3} Z2^{2/3}}$$

$Out[291]=$ True

```
In[292]:= (* SECTION 3.2.3:  CHECK OF MOVE (III) EXAMPLE *)
             (* 11/11 *)
        x = Expand[f[a1, a1]]
        y = Expand[f[a1, a1]]
        x = y
```

$Out[292]= q^{4/9} W2^{4/3} W3^{2/3} X^{4/3} Z2^{2/3} Z3^{4/3}$

$Out[293]= q^{4/9} W2^{4/3} W3^{2/3} X^{4/3} Z2^{2/3} Z3^{4/3}$

$Out[294]=$ True

```
In[295]:=        (* 12/11 *)
        x = Expand[f[a1, b11] + f[a1, b12]]
        y = Expand[q * (f[b11, a1] + f[b12, a1]) + (1 - q^2) * (f[a1, b11] + f[a1, b12])]
        x = y
```

$$Out[295]= \frac{W2^{4/3} W3^{2/3} X^{1/3} Z2^{2/3} Z3^{1/3}}{q^{13/18}} + \frac{W2^{4/3} W3^{2/3} X^{4/3} Z2^{2/3} Z3^{1/3}}{q^{31/18}}$$

$$Out[296]= \frac{W2^{4/3} W3^{2/3} X^{1/3} Z2^{2/3} Z3^{1/3}}{q^{13/18}} + \frac{W2^{4/3} W3^{2/3} X^{4/3} Z2^{2/3} Z3^{1/3}}{q^{31/18}}$$

$Out[297]=$ True

```
In[298]:=        (* 13/11 *)
        x = Expand[f[a1, c1]]
        y = Expand[q * f[c1, a1] + (1 - q^2) * f[a1, c1]]
        x = y
```

$$Out[298]= \frac{q^{1/9} W2^{4/3} X^{1/3} Z2^{2/3} Z3^{1/3}}{W3^{1/3}}$$

$$Out[299]= \frac{q^{1/9} W2^{4/3} X^{1/3} Z2^{2/3} Z3^{1/3}}{W3^{1/3}}$$

$Out[300]=$ True

In[301]:=    (* **12/12** *)
            x = Expand[f[a1, e1]]
            y = Expand[f[e1, a1]]
            x == y

Out[301]=   $\dfrac{W2^{1/3}\,W3^{2/3}\,X^{1/3}\,Z2^{2/3}\,Z3^{1/3}}{q^{2/9}}$

Out[302]=   $\dfrac{W2^{1/3}\,W3^{2/3}\,X^{1/3}\,Z2^{2/3}\,Z3^{1/3}}{q^{2/9}}$

Out[303]=   True

In[304]:=    (* **13/12** *)
            x = Expand[f[a1, f1]]
            y = Expand[f[f1, a1]]
            x == y

Out[304]=   $\dfrac{q^{11/18}\,W2^{1/3}\,X^{1/3}\,Z2^{2/3}\,Z3^{1/3}}{W3^{1/3}}$

Out[305]=   $\dfrac{q^{11/18}\,W2^{1/3}\,X^{1/3}\,Z2^{2/3}\,Z3^{1/3}}{W3^{1/3}}$

Out[306]=   True

In[307]:=    (* **13/13** *)
            x = Expand[f[a1, i1]]
            y = Expand[f[i1, a1]]
            x == y

Out[307]=   $\dfrac{W2^{1/3}\,X^{1/3}\,Z3^{1/3}}{q^{2/9}\,W3^{1/3}\,Z2^{1/3}}$

Out[308]=   $\dfrac{W2^{1/3}\,X^{1/3}\,Z3^{1/3}}{q^{2/9}\,W3^{1/3}\,Z2^{1/3}}$

Out[309]=   True

In[310]:=    (* **12/21** *)
            x = Expand[0]
            y = Expand[(q - q^(-1)) * (f[e1, a1] - f[a1, e1])]
            x == y

Out[310]=   0

Out[311]=   0

Out[312]=   True

In[313]:=    (* **13/21** *)
            x = Expand[0]
            y = Expand[(q - q^(-1)) * (f[f1, a1] - f[a1, f1])]
            x == y

Out[313]=   0

Out[314]=   0

Out[315]=   True

```
In[316]:=     (* 13/31 *)
              x = Expand[0]
              y = Expand[(q - q^(-1)) * (f[i1, a1] - f[a1, i1])]
              x = y

Out[316]=  0

Out[317]=  0

Out[318]=  True
```

```
In[319]:=     (* 21/11 *)
              x = Expand[f[b11, a1] + f[b12, a1]]
              y = Expand[q * (f[a1, b11] + f[a1, b12])]
              x = y
```

$$Out[319]= \quad q^{5/18} \, W2^{4/3} \, W3^{2/3} \, X^{1/3} \, Z2^{2/3} \, Z3^{1/3} + \frac{W2^{4/3} \, W3^{2/3} \, X^{4/3} \, Z2^{2/3} \, Z3^{1/3}}{q^{13/18}}$$

$$Out[320]= \quad q^{5/18} \, W2^{4/3} \, W3^{2/3} \, X^{1/3} \, Z2^{2/3} \, Z3^{1/3} + \frac{W2^{4/3} \, W3^{2/3} \, X^{4/3} \, Z2^{2/3} \, Z3^{1/3}}{q^{13/18}}$$

Out[321]= True

```
In[322]:=     (* 22/11 *)
              x = Expand[f[b11, b11] + f[b12, b11] + f[b11, b12] + f[b12, b12]]
              y = Expand[f[b11, b11] + f[b12, b11] + f[b11, b12] + f[b12, b12]]
              x = y
```

$$Out[322]= \quad \frac{q^{10/9} \, W2^{4/3} \, W3^{2/3} \, Z2^{2/3}}{X^{2/3} \, Z3^{2/3}} + \frac{W2^{4/3} \, W3^{2/3} \, X^{1/3} \, Z2^{2/3}}{q^{17/9} \, Z3^{2/3}} + \frac{q^{1/9} \, W2^{4/3} \, W3^{2/3} \, X^{1/3} \, Z2^{2/3}}{Z3^{2/3}} + \frac{W2^{4/3} \, W3^{2/3} \, X^{4/3} \, Z2^{2/3}}{q^{26/9} \, Z3^{2/3}}$$

$$Out[323]= \quad \frac{q^{10/9} \, W2^{4/3} \, W3^{2/3} \, Z2^{2/3}}{X^{2/3} \, Z3^{2/3}} + \frac{W2^{4/3} \, W3^{2/3} \, X^{1/3} \, Z2^{2/3}}{q^{17/9} \, Z3^{2/3}} + \frac{q^{1/9} \, W2^{4/3} \, W3^{2/3} \, X^{1/3} \, Z2^{2/3}}{Z3^{2/3}} + \frac{W2^{4/3} \, W3^{2/3} \, X^{4/3} \, Z2^{2/3}}{q^{26/9} \, Z3^{2/3}}$$

Out[324]= True

```
In[325]:=     (* 23/11 *)
              x = Expand[f[b11, c1] + f[b12, c1]]
              y = Expand[q * (f[c1, b11] + f[c1, b12]) + (1 - q^2) * (f[b11, c1] + f[b12, c1])]
              x = y
```

$$Out[325]= \quad \frac{q^{17/18} \, W2^{4/3} \, Z2^{2/3}}{W3^{1/3} \, X^{2/3} \, Z3^{2/3}} + \frac{W2^{4/3} \, X^{1/3} \, Z2^{2/3}}{q^{1/18} \, W3^{1/3} \, Z3^{2/3}}$$

$$Out[326]= \quad \frac{q^{17/18} \, W2^{4/3} \, Z2^{2/3}}{W3^{1/3} \, X^{2/3} \, Z3^{2/3}} + \frac{W2^{4/3} \, X^{1/3} \, Z2^{2/3}}{q^{1/18} \, W3^{1/3} \, Z3^{2/3}}$$

Out[327]= True

In[328]:= **(\* 22/12 \*)**
**x = Expand[f[b11, e1] + f[b12, e1]]**
**y = Expand[q^(-1) \* (f[e1, b11] + f[e1, b12])]**
**x == y**

Out[328]= $\dfrac{W2^{1/3} W3^{2/3} Z2^{2/3}}{q^{7/18} X^{2/3} Z3^{2/3}} + \dfrac{W2^{1/3} W3^{2/3} X^{1/3} Z2^{2/3}}{q^{25/18} Z3^{2/3}}$

Out[329]= $\dfrac{W2^{1/3} W3^{2/3} Z2^{2/3}}{q^{7/18} X^{2/3} Z3^{2/3}} + \dfrac{W2^{1/3} W3^{2/3} X^{1/3} Z2^{2/3}}{q^{25/18} Z3^{2/3}}$

Out[330]= True

In[331]:= **(\* 23/12 \*)**
**x = Expand[f[b11, f1] + f[b12, f1]]**
**y = Expand[(f[f1, b11] + f[f1, b12]) + (q^(-1) - q) \* (f[e1, c1])]**
**x == y**

Out[331]= $\dfrac{W2^{1/3} Z2^{2/3}}{q^{5/9} W3^{1/3} X^{2/3} Z3^{2/3}} + \dfrac{q^{4/9} W2^{1/3} X^{1/3} Z2^{2/3}}{W3^{1/3} Z3^{2/3}}$

Out[332]= $\dfrac{W2^{1/3} Z2^{2/3}}{q^{5/9} W3^{1/3} X^{2/3} Z3^{2/3}} + \dfrac{q^{4/9} W2^{1/3} X^{1/3} Z2^{2/3}}{W3^{1/3} Z3^{2/3}}$

Out[333]= True

In[334]:= **(\* 23/13 \*)**
**x = Expand[f[b11, i1] + f[b12, i1]]**
**y = Expand[f[i1, b11] + f[i1, b12]]**
**x == y**

Out[334]= $\dfrac{q^{11/18} W2^{1/3}}{W3^{1/3} X^{2/3} Z2^{1/3} Z3^{2/3}} + \dfrac{W2^{1/3} X^{1/3}}{q^{7/18} W3^{1/3} Z2^{1/3} Z3^{2/3}}$

Out[335]= $\dfrac{q^{11/18} W2^{1/3}}{W3^{1/3} X^{2/3} Z2^{1/3} Z3^{2/3}} + \dfrac{W2^{1/3} X^{1/3}}{q^{7/18} W3^{1/3} Z2^{1/3} Z3^{2/3}}$

Out[336]= True

In[337]:= **(\* 21/21 \*)**
**x = Expand[f[e1, a1]]**
**y = Expand[f[a1, e1]]**
**x == y**

Out[337]= $\dfrac{W2^{1/3} W3^{2/3} X^{1/3} Z2^{2/3} Z3^{1/3}}{q^{2/9}}$

Out[338]= $\dfrac{W2^{1/3} W3^{2/3} X^{1/3} Z2^{2/3} Z3^{1/3}}{q^{2/9}}$

Out[339]= True

In[340]:=     (* 22/21 *)
x = Expand$[f[e1, b11] + f[e1, b12]]$
y = Expand$[q^{\wedge}(-1) * (f[b11, e1] + f[b12, e1]) + (1 - q^{\wedge}(-2)) * (f[e1, b11] + f[e1, b12])]$
x == y

Out[340]= $\dfrac{q^{11/18} \, W2^{1/3} \, W3^{2/3} \, Z2^{2/3}}{X^{2/3} \, Z3^{2/3}} + \dfrac{W2^{1/3} \, W3^{2/3} \, X^{1/3} \, Z2^{2/3}}{q^{7/18} \, Z3^{2/3}}$

Out[341]= $\dfrac{q^{11/18} \, W2^{1/3} \, W3^{2/3} \, Z2^{2/3}}{X^{2/3} \, Z3^{2/3}} + \dfrac{W2^{1/3} \, W3^{2/3} \, X^{1/3} \, Z2^{2/3}}{q^{7/18} \, Z3^{2/3}}$

Out[342]= True

In[343]:=     (* 23/21 *)
x = Expand$[f[e1, c1]]$
y = Expand$[(-q^{\wedge}2 + 2 - q^{\wedge}(-2)) * (f[e1, c1]) + (f[c1, e1]) + (q - q^{\wedge}(-1)) * ((f[f1, b11] + f[f1, b12]) - (f[b11, f1] + f[b12, f1]))]$
x =
  y

Out[343]= $\dfrac{q^{4/9} \, W2^{1/3} \, Z2^{2/3}}{W3^{1/3} \, X^{2/3} \, Z3^{2/3}}$

Out[344]= $\dfrac{q^{4/9} \, W2^{1/3} \, Z2^{2/3}}{W3^{1/3} \, X^{2/3} \, Z3^{2/3}}$

Out[345]= True

In[346]:=     (* 22/22 *)
x = Expand$[f[e1, e1]]$
y = Expand$[f[e1, e1]]$
x == y

Out[346]= $\dfrac{W3^{2/3} \, Z2^{2/3}}{q^{8/9} \, W2^{2/3} \, X^{2/3} \, Z3^{2/3}}$

Out[347]= $\dfrac{W3^{2/3} \, Z2^{2/3}}{q^{8/9} \, W2^{2/3} \, X^{2/3} \, Z3^{2/3}}$

Out[348]= True

In[349]:=     (* 23/22 *)
x = Expand$[f[e1, f1]]$
y = Expand$[q * f[f1, e1] + (1 - q^{\wedge}2) * f[e1, f1]]$
x == y

Out[349]= $\dfrac{Z2^{2/3}}{q^{19/18} \, W2^{2/3} \, W3^{1/3} \, X^{2/3} \, Z3^{2/3}}$

Out[350]= $\dfrac{Z2^{2/3}}{q^{19/18} \, W2^{2/3} \, W3^{1/3} \, X^{2/3} \, Z3^{2/3}}$

Out[351]= True

In[352]:=     (* **23/23** *)
    x = Expand[f[e1, i1]]
    y = Expand[f[i1, e1]]
    x == y

Out[352]= $\dfrac{q^{1/9}}{W2^{2/3}\,W3^{1/3}\,X^{2/3}\,Z2^{1/3}\,Z3^{2/3}}$

Out[353]= $\dfrac{q^{1/9}}{W2^{2/3}\,W3^{1/3}\,X^{2/3}\,Z2^{1/3}\,Z3^{2/3}}$

Out[354]= True

In[355]:=     (* **23/31** *)
    x = Expand[0]
    y = Expand[(q - q^(-1)) * ((f[i1, b11] + f[i1, b12]) - (f[b11, i1] + f[b12, i1]))]
    x == y

Out[355]= 0

Out[356]= 0

Out[357]= True

In[358]:=     (* **23/32** *)
    x = Expand[0]
    y = Expand[(q - q^(-1)) * ((f[i1, e1]) - (f[e1, i1]))]
    x == y

Out[358]= 0

Out[359]= 0

Out[360]= True

In[361]:=     (* **31/11** *)
    x = Expand[f[c1, a1]]
    y = Expand[q * f[a1, c1]]
    x == y

Out[361]= $\dfrac{q^{10/9}\,W2^{4/3}\,X^{1/3}\,Z2^{2/3}\,Z3^{1/3}}{W3^{1/3}}$

Out[362]= $\dfrac{q^{10/9}\,W2^{4/3}\,X^{1/3}\,Z2^{2/3}\,Z3^{1/3}}{W3^{1/3}}$

Out[363]= True

In[364]:=     (* **32/11** *)
    x = Expand[f[c1, b11] + f[c1, b12]]
    y = Expand[q * (f[b11, c1] + f[b12, c1])]
    x == y

Out[364]= $\dfrac{q^{35/18}\,W2^{4/3}\,Z2^{2/3}}{W3^{1/3}\,X^{2/3}\,Z3^{2/3}} + \dfrac{q^{17/18}\,W2^{4/3}\,X^{1/3}\,Z2^{2/3}}{W3^{1/3}\,Z3^{2/3}}$

Out[365]= $\dfrac{q^{35/18}\,W2^{4/3}\,Z2^{2/3}}{W3^{1/3}\,X^{2/3}\,Z3^{2/3}} + \dfrac{q^{17/18}\,W2^{4/3}\,X^{1/3}\,Z2^{2/3}}{W3^{1/3}\,Z3^{2/3}}$

Out[366]= True

In[367]:=      (* **33/11** *)
     **x = Expand[f[c1, c1]]**
     **y = Expand[f[c1, c1]]**
     **x == y**

Out[367]= $\dfrac{q^{16/9} \, W2^{4/3} \, Z2^{2/3}}{W3^{4/3} \, X2^{2/3} \, Z3^{2/3}}$

Out[368]= $\dfrac{q^{16/9} \, W2^{4/3} \, Z2^{2/3}}{W3^{4/3} \, X2^{2/3} \, Z3^{2/3}}$

Out[369]= True

In[370]:=      (* **32/12** *)
     **x = Expand[f[c1, e1]]**
     **y = Expand[f[e1, c1]]**
     **x == y**

Out[370]= $\dfrac{q^{4/9} \, W2^{1/3} \, Z2^{2/3}}{W3^{1/3} \, X2^{2/3} \, Z3^{2/3}}$

Out[371]= $\dfrac{q^{4/9} \, W2^{1/3} \, Z2^{2/3}}{W3^{1/3} \, X2^{2/3} \, Z3^{2/3}}$

Out[372]= True

In[373]:=      (* **33/12** *)
     **x = Expand[f[c1, f1]]**
     **y = Expand[q ^ (-1) * f[f1, c1]]**
     **x == y**

Out[373]= $\dfrac{q^{5/18} \, W2^{1/3} \, Z2^{2/3}}{W3^{4/3} \, X2^{2/3} \, Z3^{2/3}}$

Out[374]= $\dfrac{q^{5/18} \, W2^{1/3} \, Z2^{2/3}}{W3^{4/3} \, X2^{2/3} \, Z3^{2/3}}$

Out[375]= True

In[376]:=      (* **33/13** *)
     **x = Expand[f[c1, i1]]**
     **y = Expand[q ^ (-1) * f[i1, c1]]**
     **x == y**

Out[376]= $\dfrac{W2^{1/3}}{q^{5/9} \, W3^{4/3} \, X2^{2/3} \, Z2^{1/3} \, Z3^{2/3}}$

Out[377]= $\dfrac{W2^{1/3}}{q^{5/9} \, W3^{4/3} \, X2^{2/3} \, Z2^{1/3} \, Z3^{2/3}}$

Out[378]= True

In[379]:=     (\* 31/21 \*)
x = Expand[f[f1, a1]]
y = Expand[f[a1, f1]]
x == y

Out[379]= $\dfrac{q^{11/18} \, W2^{1/3} \, X^{1/3} \, Z2^{2/3} \, Z3^{1/3}}{W3^{1/3}}$

Out[380]= $\dfrac{q^{11/18} \, W2^{1/3} \, X^{1/3} \, Z2^{2/3} \, Z3^{1/3}}{W3^{1/3}}$

Out[381]= True

In[382]:=     (\* 32/21 \*)
x = Expand[f[f1, b11] + f[f1, b12]]
y = Expand[(f[b11, f1] + f[b12, f1]) + (q - q ^ (-1)) \* (f[e1, c1])]
x == y

Out[382]= $\dfrac{q^{13/9} \, W2^{1/3} \, Z2^{2/3}}{W3^{1/3} \, X^{2/3} \, Z3^{2/3}} + \dfrac{q^{4/9} \, W2^{1/3} \, X^{1/3} \, Z2^{2/3}}{W3^{1/3} \, Z3^{2/3}}$

Out[383]= $\dfrac{q^{13/9} \, W2^{1/3} \, Z2^{2/3}}{W3^{1/3} \, X^{2/3} \, Z3^{2/3}} + \dfrac{q^{4/9} \, W2^{1/3} \, X^{1/3} \, Z2^{2/3}}{W3^{1/3} \, Z3^{2/3}}$

Out[384]= True

In[385]:=     (\* 33/21 \*)
x = Expand[f[f1, c1]]
y = Expand[q ^ (-1) \* f[c1, f1] + (1 - q ^ (-2)) \* f[f1, c1]]
x == y

Out[385]= $\dfrac{q^{23/18} \, W2^{1/3} \, Z2^{2/3}}{W3^{4/3} \, X^{2/3} \, Z3^{2/3}}$

Out[386]= $\dfrac{q^{23/18} \, W2^{1/3} \, Z2^{2/3}}{W3^{4/3} \, X^{2/3} \, Z3^{2/3}}$

Out[387]= True

In[388]:=     (\* 32/22 \*)
x = Expand[f[f1, e1]]
y = Expand[q \* f[e1, f1]]
x == y

Out[388]= $\dfrac{Z2^{2/3}}{q^{1/18} \, W2^{2/3} \, W3^{1/3} \, X^{2/3} \, Z3^{2/3}}$

Out[389]= $\dfrac{Z2^{2/3}}{q^{1/18} \, W2^{2/3} \, W3^{1/3} \, X^{2/3} \, Z3^{2/3}}$

Out[390]= True

In[391]:=  (* 33/22 *)
x = Expand[f[f1, f1]]
y = Expand[f[f1, f1]]
x == y

Out[391]=  $\dfrac{Z2^{2/3}}{q^{2/9}\,W2^{2/3}\,W3^{4/3}\,X2^{2/3}\,Z3^{2/3}}$

Out[392]=  $\dfrac{Z2^{2/3}}{q^{2/9}\,W2^{2/3}\,W3^{4/3}\,X2^{2/3}\,Z3^{2/3}}$

Out[393]=  True

In[394]:=  (* 33/23 *)
x = Expand[f[f1, i1]]
y = Expand[q^(-1) * f[i1, f1]]
x == y

Out[394]=  $\dfrac{1}{q^{19/18}\,W2^{2/3}\,W3^{4/3}\,X2^{2/3}\,Z2^{1/3}\,Z3^{2/3}}$

Out[395]=  $\dfrac{1}{q^{19/18}\,W2^{2/3}\,W3^{4/3}\,X2^{2/3}\,Z2^{1/3}\,Z3^{2/3}}$

Out[396]=  True

In[397]:=  (* 31/31 *)
x = Expand[f[i1, a1]]
y = Expand[f[a1, i1]]
x == y

Out[397]=  $\dfrac{W2^{1/3}\,X^{1/3}\,Z3^{1/3}}{q^{2/9}\,W3^{1/3}\,Z2^{1/3}}$

Out[398]=  $\dfrac{W2^{1/3}\,X^{1/3}\,Z3^{1/3}}{q^{2/9}\,W3^{1/3}\,Z2^{1/3}}$

Out[399]=  True

In[400]:=  (* 32/31 *)
x = Expand[f[i1, b11] + f[i1, b12]]
y = Expand[f[b11, i1] + f[b12, i1]]
x == y

Out[400]=  $\dfrac{q^{11/18}\,W2^{1/3}}{W3^{1/3}\,X2^{2/3}\,Z2^{1/3}\,Z3^{2/3}} + \dfrac{W2^{1/3}\,X^{1/3}}{q^{7/18}\,W3^{1/3}\,Z2^{1/3}\,Z3^{2/3}}$

Out[401]=  $\dfrac{q^{11/18}\,W2^{1/3}}{W3^{1/3}\,X2^{2/3}\,Z2^{1/3}\,Z3^{2/3}} + \dfrac{W2^{1/3}\,X^{1/3}}{q^{7/18}\,W3^{1/3}\,Z2^{1/3}\,Z3^{2/3}}$

Out[402]=  True

In[403]:=
```
(* 33/31 *)
x = Expand[f[i1, c1]]
y = Expand[q ^ (-1) * f[c1, i1] + (1 - q ^ (-2)) * f[i1, c1]]
x == y
```

Out[403]=
$$\frac{q^{4/9} \, W2^{1/3}}{W3^{4/3} \, X^{2/3} \, Z2^{1/3} \, Z3^{2/3}}$$

Out[404]=
$$\frac{q^{4/9} \, W2^{1/3}}{W3^{4/3} \, X^{2/3} \, Z2^{1/3} \, Z3^{2/3}}$$

Out[405]= True

In[406]:=
```
(* 32/32 *)
x = Expand[f[i1, e1]]
y = Expand[f[e1, i1]]
x == y
```

Out[406]=
$$\frac{q^{1/9}}{W2^{2/3} \, W3^{1/3} \, X^{2/3} \, Z2^{1/3} \, Z3^{2/3}}$$

Out[407]=
$$\frac{q^{1/9}}{W2^{2/3} \, W3^{1/3} \, X^{2/3} \, Z2^{1/3} \, Z3^{2/3}}$$

Out[408]= True

In[409]:=
```
(* 33/32 *)
x = Expand[f[i1, f1]]
y = Expand[q ^ (-1) * f[f1, i1] + (1 - q ^ (-2)) * f[i1, f1]]
x == y
```

Out[409]=
$$\frac{1}{q^{1/18} \, W2^{2/3} \, W3^{4/3} \, X^{2/3} \, Z2^{1/3} \, Z3^{2/3}}$$

Out[410]=
$$\frac{1}{q^{1/18} \, W2^{2/3} \, W3^{4/3} \, X^{2/3} \, Z2^{1/3} \, Z3^{2/3}}$$

Out[411]= True

In[412]:=
```
(* 33/33 *)
x = Expand[f[i1, i1]]
y = Expand[f[i1, i1]]
x == y
```

Out[412]=
$$\frac{1}{q^{8/9} \, W2^{2/3} \, W3^{4/3} \, X^{2/3} \, Z2^{4/3} \, Z3^{2/3}}$$

Out[413]=
$$\frac{1}{q^{8/9} \, W2^{2/3} \, W3^{4/3} \, X^{2/3} \, Z2^{4/3} \, Z3^{2/3}}$$

Out[414]= True

In[415]:= `(* SECTION 3.2.4:  CHECK OF MOVE (IV) EXAMPLE *)`
`        (* 11/11 *)`
`x = Expand[(f[a3, A2])]`
`y = Expand[q ^ (1/3) * (q ^ (-1) * f[A2, a3])]`
`x == y`

Out[415]= $q^{7/9} W1^{4/3} W2^{1/3} W3^{1/3} X Z1^{2/3} Z2^{2/3} Z3^{2/3}$

Out[416]= $q^{7/9} W1^{4/3} W2^{1/3} W3^{1/3} X Z1^{2/3} Z2^{2/3} Z3^{2/3}$

Out[417]= True

In[418]:= `        (* 12/11 *)`
`x = Expand[(f[a3, D2])]`
`y = Expand[q ^ (1/3) * (f[D2, a3])]`
`x == y`

Out[418]= $q^{23/18} W1^{1/3} W2^{1/3} W3^{1/3} X Z1^{2/3} Z2^{2/3} Z3^{2/3}$

Out[419]= $q^{23/18} W1^{1/3} W2^{1/3} W3^{1/3} X Z1^{2/3} Z2^{2/3} Z3^{2/3}$

Out[420]= True

In[421]:= `        (* 13/11 *)`
`x = Expand[(f[a3, G2])]`
`y = Expand[q ^ (1/3) * (f[G2, a3])]`
`x == y`

Out[421]= $\dfrac{q^{4/9} W1^{1/3} W2^{1/3} W3^{1/3} X Z2^{2/3} Z3^{2/3}}{Z1^{1/3}}$

Out[422]= $\dfrac{q^{4/9} W1^{1/3} W2^{1/3} W3^{1/3} X Z2^{2/3} Z3^{2/3}}{Z1^{1/3}}$

Out[423]= True

In[424]:= `        (* 12/12 *)`
`x = Expand[(f[a3, E2])]`
`y = Expand[q ^ (1/3) * (f[E2, a3])]`
`x == y`

Out[424]= $\dfrac{W1^{1/3} W2^{1/3} W3^{1/3} X Z1^{2/3} Z2^{2/3}}{q^{2/9} Z3^{1/3}}$

Out[425]= $\dfrac{W1^{1/3} W2^{1/3} W3^{1/3} X Z1^{2/3} Z2^{2/3}}{q^{2/9} Z3^{1/3}}$

Out[426]= True

In[427]:= `        (* 13/12 *)`
`x = Expand[(f[a3, H21] + f[a3, H22])]`
`y = Expand[q ^ (1/3) * (f[H21, a3] + f[H22, a3])]`
`x == y`

Out[427]= $\dfrac{W1^{1/3} W2^{1/3} W3^{1/3} Z2^{2/3}}{q^{1/18} Z1^{1/3} Z3^{1/3}} + \dfrac{W1^{1/3} W2^{1/3} W3^{1/3} X Z2^{2/3}}{q^{19/18} Z1^{1/3} Z3^{1/3}}$

Out[428]= $\dfrac{W1^{1/3} W2^{1/3} W3^{1/3} Z2^{2/3}}{q^{1/18} Z1^{1/3} Z3^{1/3}} + \dfrac{W1^{1/3} W2^{1/3} W3^{1/3} X Z2^{2/3}}{q^{19/18} Z1^{1/3} Z3^{1/3}}$

Out[429]= True

In[430]:= (\* 13/13 \*)
x = Expand[(f[a3, I2])]
y = Expand[q^(1/3) \* (f[I2, a3])]
x == y

Out[430]= $\dfrac{q^{7/9} \, W1^{1/3} \, W2^{1/3} \, Z2^{2/3}}{W3^{2/3} \, Z1^{1/3} \, Z3^{1/3}}$

Out[431]= $\dfrac{q^{7/9} \, W1^{1/3} \, W2^{1/3} \, Z2^{2/3}}{W3^{2/3} \, Z1^{1/3} \, Z3^{1/3}}$

Out[432]= True

In[433]:= (\* 11/21 \*)
x = Expand[(f[b31, A2] + f[b32, A2])]
y = Expand[q^(1/3) \* (q^(-1) \* (f[A2, b31] + f[A2, b32]))]
x == y

Out[433]= $\dfrac{q^{17/18} \, W1^{4/3} \, W2^{1/3} \, W3^{1/3} \, Z1^{2/3} \, Z3^{2/3}}{Z2^{1/3}} + \dfrac{W1^{4/3} \, W2^{1/3} \, W3^{1/3} \, X \, Z1^{1/3} \, Z3^{2/3}}{q^{1/18} \, Z2^{1/3}}$

Out[434]= $\dfrac{q^{17/18} \, W1^{4/3} \, W2^{1/3} \, W3^{1/3} \, Z1^{2/3} \, Z3^{2/3}}{Z2^{1/3}} + \dfrac{W1^{4/3} \, W2^{1/3} \, W3^{1/3} \, X \, Z1^{1/3} \, Z3^{2/3}}{q^{1/18} \, Z2^{1/3}}$

Out[435]= True

In[436]:= (\* 12/21 \*)
x = Expand[(f[b31, D2] + f[b32, D2])]
y = Expand[q^(1/3) \* (f[D2, b31] + f[D2, b32] + (q^(-1) - q) \* f[A2, e3])]
x == y

Out[436]= $\dfrac{W1^{1/3} \, W2^{1/3} \, W3^{1/3} \, Z1^{2/3} \, Z3^{2/3}}{q^{5/9} \, Z2^{1/3}} + \dfrac{q^{4/9} \, W1^{1/3} \, W2^{1/3} \, W3^{1/3} \, X \, Z1^{2/3} \, Z3^{2/3}}{Z2^{1/3}}$

Out[437]= $\dfrac{W1^{1/3} \, W2^{1/3} \, W3^{1/3} \, Z1^{2/3} \, Z3^{2/3}}{q^{5/9} \, Z2^{1/3}} + \dfrac{q^{4/9} \, W1^{1/3} \, W2^{1/3} \, W3^{1/3} \, X \, Z1^{2/3} \, Z3^{2/3}}{Z2^{1/3}}$

Out[438]= True

In[439]:= (\* 13/21 \*)
x = Expand[(f[b31, G2] + f[b32, G2])]
y = Expand[q^(1/3) \* (f[G2, b31] + f[G2, b32])]
x == y

Out[439]= $\dfrac{q^{11/18} \, W1^{1/3} \, W2^{1/3} \, W3^{1/3} \, Z3^{2/3}}{Z1^{1/3} \, Z2^{1/3}} + \dfrac{W1^{1/3} \, W2^{1/3} \, W3^{1/3} \, X \, Z3^{2/3}}{q^{7/18} \, Z1^{1/3} \, Z2^{1/3}}$

Out[440]= $\dfrac{q^{11/18} \, W1^{1/3} \, W2^{1/3} \, W3^{1/3} \, Z3^{2/3}}{Z1^{1/3} \, Z2^{1/3}} + \dfrac{W1^{1/3} \, W2^{1/3} \, W3^{1/3} \, X \, Z3^{2/3}}{q^{7/18} \, Z1^{1/3} \, Z2^{1/3}}$

Out[441]= True

In[442]:=   (* 21/21 *)
x = Expand[(f[e3, A2])]
y = Expand[q^(1/3) * (f[A2, e3])]
x == y

Out[442]= $\dfrac{q^{4/9} \, W1^{1/3} \, W2^{1/3} \, W3^{1/3} \, Z1^{2/3} \, Z3^{2/3}}{Z2^{1/3}}$

Out[443]= $\dfrac{q^{4/9} \, W1^{1/3} \, W2^{1/3} \, W3^{1/3} \, Z1^{2/3} \, Z3^{2/3}}{Z2^{1/3}}$

Out[444]= True

In[445]:=   (* 22/21 *)
x = Expand[(f[e3, D2])]
y = Expand[q^(1/3) * (q^(-1) * f[D2, e3])]
x == y

Out[445]= $\dfrac{W2^{1/3} \, W3^{1/3} \, Z1^{2/3} \, Z3^{2/3}}{q^{19/18} \, W1^{2/3} \, Z2^{1/3}}$

Out[446]= $\dfrac{W2^{1/3} \, W3^{1/3} \, Z1^{2/3} \, Z3^{2/3}}{q^{19/18} \, W1^{2/3} \, Z2^{1/3}}$

Out[447]= True

In[448]:=   (* 23/21 *)
x = Expand[(f[e3, G2])]
y = Expand[q^(1/3) * (f[G2, e3])]
x == y

Out[448]= $\dfrac{q^{1/9} \, W2^{1/3} \, W3^{1/3} \, Z3^{2/3}}{W1^{2/3} \, Z1^{1/3} \, Z2^{1/3}}$

Out[449]= $\dfrac{q^{1/9} \, W2^{1/3} \, W3^{1/3} \, Z3^{2/3}}{W1^{2/3} \, Z1^{1/3} \, Z2^{1/3}}$

Out[450]= True

In[451]:=   (* 12/22 *)
x = Expand[(f[b31, E2] + f[b32, E2])]
y = Expand[q^(1/3) * (f[E2, b31] + f[E2, b32])]
x == y

Out[451]= $\dfrac{W1^{1/3} \, W2^{1/3} \, W3^{1/3} \, Z1^{2/3}}{q^{1/18} \, Z2^{1/3} \, Z3^{1/3}} + \dfrac{W1^{1/3} \, W2^{1/3} \, W3^{1/3} \, X \, Z1^{2/3}}{q^{19/18} \, Z2^{1/3} \, Z3^{1/3}}$

Out[452]= $\dfrac{W1^{1/3} \, W2^{1/3} \, W3^{1/3} \, Z1^{2/3}}{q^{1/18} \, Z2^{1/3} \, Z3^{1/3}} + \dfrac{W1^{1/3} \, W2^{1/3} \, W3^{1/3} \, X \, Z1^{2/3}}{q^{19/18} \, Z2^{1/3} \, Z3^{1/3}}$

Out[453]= True

In[454]:= `(* 13/22 *)`
`x = Expand[(f[b31, H21] + f[b32, H21] + f[b31, H22] + f[b32, H22])]`
`y = Expand[q^(1/3) * (f[H21, b31] + f[H22, b31] + f[H21, b32] + f[H22, b32])]`
`x == y`

Out[454]= $\dfrac{W1^{1/3}\,W2^{1/3}\,W3^{1/3}}{q^{8/9}\,Z1^{1/3}\,Z2^{1/3}\,Z3^{1/3}} + \dfrac{q^{10/9}\,W1^{1/3}\,W2^{1/3}\,W3^{1/3}}{Z1^{1/3}\,Z2^{1/3}\,Z3^{1/3}} + \dfrac{q^{19/9}\,W1^{1/3}\,W2^{1/3}\,W3^{1/3}}{X\,Z1^{1/3}\,Z2^{1/3}\,Z3^{1/3}} + \dfrac{W1^{1/3}\,W2^{1/3}\,W3^{1/3}\,X}{q^{17/9}\,Z1^{1/3}\,Z2^{1/3}\,Z3^{1/3}}$

Out[455]= $\dfrac{W1^{1/3}\,W2^{1/3}\,W3^{1/3}}{q^{8/9}\,Z1^{1/3}\,Z2^{1/3}\,Z3^{1/3}} + \dfrac{q^{10/9}\,W1^{1/3}\,W2^{1/3}\,W3^{1/3}}{Z1^{1/3}\,Z2^{1/3}\,Z3^{1/3}} + \dfrac{q^{19/9}\,W1^{1/3}\,W2^{1/3}\,W3^{1/3}}{X\,Z1^{1/3}\,Z2^{1/3}\,Z3^{1/3}} + \dfrac{W1^{1/3}\,W2^{1/3}\,W3^{1/3}\,X}{q^{17/9}\,Z1^{1/3}\,Z2^{1/3}\,Z3^{1/3}}$

Out[456]= `True`

In[457]:= `(* 22/22 *)`
`x = Expand[(f[e3, E2])]`
`y = Expand[q^(1/3) * (q^(-1) * f[E2, e3])]`
`x == y`

Out[457]= $\dfrac{W2^{1/3}\,W3^{1/3}\,Z1^{2/3}}{q^{5/9}\,W1^{2/3}\,Z2^{1/3}\,Z3^{1/3}}$

Out[458]= $\dfrac{W2^{1/3}\,W3^{1/3}\,Z1^{2/3}}{q^{5/9}\,W1^{2/3}\,Z2^{1/3}\,Z3^{1/3}}$

Out[459]= `True`

In[460]:= `(* 23/22 *)`
`x = Expand[(f[e3, H21] + f[e3, H22])]`
`y = Expand[q^(1/3) * (f[H21, e3] + f[H22, e3])]`
`x == y`

Out[460]= $\dfrac{q^{11/18}\,W2^{1/3}\,W3^{1/3}}{W1^{2/3}\,Z1^{1/3}\,Z2^{1/3}\,Z3^{1/3}} + \dfrac{q^{29/18}\,W2^{1/3}\,W3^{1/3}}{W1^{2/3}\,X\,Z1^{1/3}\,Z2^{1/3}\,Z3^{1/3}}$

Out[461]= $\dfrac{q^{11/18}\,W2^{1/3}\,W3^{1/3}}{W1^{2/3}\,Z1^{1/3}\,Z2^{1/3}\,Z3^{1/3}} + \dfrac{q^{29/18}\,W2^{1/3}\,W3^{1/3}}{W1^{2/3}\,X\,Z1^{1/3}\,Z2^{1/3}\,Z3^{1/3}}$

Out[462]= `True`

In[463]:= `(* 13/23 *)`
`x = Expand[(f[b31, I2] + f[b32, I2])]`
`y = Expand[q^(1/3) * (f[I2, b31] + f[I2, b32])]`
`x == y`

Out[463]= $\dfrac{W1^{1/3}\,W2^{1/3}}{q^{1/18}\,W3^{2/3}\,Z1^{1/3}\,Z2^{1/3}\,Z3^{1/3}} + \dfrac{q^{17/18}\,W1^{1/3}\,W2^{1/3}}{W3^{2/3}\,X\,Z1^{1/3}\,Z2^{1/3}\,Z3^{1/3}}$

Out[464]= $\dfrac{W1^{1/3}\,W2^{1/3}}{q^{1/18}\,W3^{2/3}\,Z1^{1/3}\,Z2^{1/3}\,Z3^{1/3}} + \dfrac{q^{17/18}\,W1^{1/3}\,W2^{1/3}}{W3^{2/3}\,X\,Z1^{1/3}\,Z2^{1/3}\,Z3^{1/3}}$

Out[465]= `True`

In[466]:=   (* 23/23 *)
x = Expand[ (f[e3, I2]) ]
y = Expand[ q ^ (1/3) * (f[I2, e3]) ]
x = y

Out[466]=   $\dfrac{q^{4/9} \; W2^{1/3}}{W1^{2/3} \; W3^{2/3} \; X \; Z1^{1/3} \; Z2^{1/3} \; Z3^{1/3}}$

Out[467]=   $\dfrac{q^{4/9} \; W2^{1/3}}{W1^{2/3} \; W3^{2/3} \; X \; Z1^{1/3} \; Z2^{1/3} \; Z3^{1/3}}$

Out[468]=   True

In[469]:=   (* 11/31 *)
x = Expand[ (f[c3, A2]) ]
y = Expand[ q ^ (1/3) * (q ^ (-1) * f[A2, c3]) ]
x = y

Out[469]=   $\dfrac{q^{16/9} \; W1^{4/3} \; W3^{1/3} \; Z1^{2/3} \; Z3^{2/3}}{W2^{2/3} \; Z2^{1/3}}$

Out[470]=   $\dfrac{q^{16/9} \; W1^{4/3} \; W3^{1/3} \; Z1^{2/3} \; Z3^{2/3}}{W2^{2/3} \; Z2^{1/3}}$

Out[471]=   True

In[472]:=   (* 12/31 *)
x = Expand[ (f[c3, D2]) ]
y = Expand[ q ^ (1/3) * (f[D2, c3] + (q ^ (-1) - q) * f[A2, f3]) ]
x = y

Out[472]=   $\dfrac{q^{5/18} \; W1^{1/3} \; W3^{1/3} \; Z1^{2/3} \; Z3^{2/3}}{W2^{2/3} \; Z2^{1/3}}$

Out[473]=   $\dfrac{q^{5/18} \; W1^{1/3} \; W3^{1/3} \; Z1^{2/3} \; Z3^{2/3}}{W2^{2/3} \; Z2^{1/3}}$

Out[474]=   True

In[475]:=   (* 13/31 *)
x = Expand[ (f[c3, G2]) ]
y = Expand[ q ^ (1/3) * (f[G2, c3] + (q ^ (-1) - q) * f[A2, i3]) ]
x = y

Out[475]=   $\dfrac{W1^{1/3} \; W3^{1/3} \; Z3^{2/3}}{q^{5/9} \; W2^{2/3} \; Z1^{1/3} \; Z2^{1/3}}$

Out[476]=   $\dfrac{W1^{1/3} \; W3^{1/3} \; Z3^{2/3}}{q^{5/9} \; W2^{2/3} \; Z1^{1/3} \; Z2^{1/3}}$

Out[477]=   True

In[478]:=  (* 21/31 *)
x = Expand[(f[f3, A2])]
y = Expand[q^(1/3) * (f[A2, f3])]
x == y

Out[478]= $\dfrac{q^{23/18} \, W1^{1/3} \, W3^{1/3} \, Z1^{2/3} \, Z3^{2/3}}{W2^{2/3} \, Z2^{1/3}}$

Out[479]= $\dfrac{q^{23/18} \, W1^{1/3} \, W3^{1/3} \, Z1^{2/3} \, Z3^{2/3}}{W2^{2/3} \, Z2^{1/3}}$

Out[480]= True

In[481]:=  (* 22/31 *)
x = Expand[(f[f3, D2])]
y = Expand[q^(1/3) * (q^(-1) * f[D2, f3])]
x == y

Out[481]= $\dfrac{W3^{1/3} \, Z1^{2/3} \, Z3^{2/3}}{q^{2/9} \, W1^{2/3} \, W2^{2/3} \, Z2^{1/3}}$

Out[482]= $\dfrac{W3^{1/3} \, Z1^{2/3} \, Z3^{2/3}}{q^{2/9} \, W1^{2/3} \, W2^{2/3} \, Z2^{1/3}}$

Out[483]= True

In[484]:=  (* 23/31 *)
x = Expand[(f[f3, G2])]
y = Expand[q^(1/3) * (f[G2, f3] + (q^(-1) - q) * f[D2, i3])]
x == y

Out[484]= $\dfrac{W3^{1/3} \, Z3^{2/3}}{q^{19/18} \, W1^{2/3} \, W2^{2/3} \, Z1^{1/3} \, Z2^{1/3}}$

Out[485]= $\dfrac{W3^{1/3} \, Z3^{2/3}}{q^{19/18} \, W1^{2/3} \, W2^{2/3} \, Z1^{1/3} \, Z2^{1/3}}$

Out[486]= True

In[487]:=  (* 31/31 *)
x = Expand[(f[i3, A2])]
y = Expand[q^(1/3) * (f[A2, i3])]
x == y

Out[487]= $\dfrac{q^{4/9} \, W1^{1/3} \, W3^{1/3} \, Z3^{2/3}}{W2^{2/3} \, Z1^{1/3} \, Z2^{1/3}}$

Out[488]= $\dfrac{q^{4/9} \, W1^{1/3} \, W3^{1/3} \, Z3^{2/3}}{W2^{2/3} \, Z1^{1/3} \, Z2^{1/3}}$

Out[489]= True

In[490]:=     (\* **32/31** \*)

x = Expand$\left[\left(f\left[i3, D2\right]\right)\right]$

y = Expand$\left[q^{\wedge}(1/3) * \left(f\left[D2, i3\right]\right)\right]$

x == y

Out[490]= $\dfrac{W3^{1/3} Z3^{2/3}}{q^{1/18} W1^{2/3} W2^{2/3} Z1^{1/3} Z2^{1/3}}$

Out[491]= $\dfrac{W3^{1/3} Z3^{2/3}}{q^{1/18} W1^{2/3} W2^{2/3} Z1^{1/3} Z2^{1/3}}$

Out[492]= True

In[493]:=     (\* **33/31** \*)

x = Expand$\left[\left(f\left[i3, G2\right]\right)\right]$

y = Expand$\left[q^{\wedge}(1/3) * \left(q^{\wedge}(-1) * f\left[G2, i3\right]\right)\right]$

x == y

Out[493]= $\dfrac{W3^{1/3} Z3^{2/3}}{q^{8/9} W1^{2/3} W2^{2/3} Z1^{4/3} Z2^{1/3}}$

Out[494]= $\dfrac{W3^{1/3} Z3^{2/3}}{q^{8/9} W1^{2/3} W2^{2/3} Z1^{4/3} Z2^{1/3}}$

Out[495]= True

In[496]:=     (\* **12/32** \*)

x = Expand$\left[\left(f\left[c3, E2\right]\right)\right]$

y = Expand$\left[q^{\wedge}(1/3) * \left(f\left[E2, c3\right]\right)\right]$

x == y

Out[496]= $\dfrac{q^{7/9} W1^{1/3} W3^{1/3} Z1^{2/3}}{W2^{2/3} Z2^{1/3} Z3^{1/3}}$

Out[497]= $\dfrac{q^{7/9} W1^{1/3} W3^{1/3} Z1^{2/3}}{W2^{2/3} Z2^{1/3} Z3^{1/3}}$

Out[498]= True

In[499]:=     (\* **13/32** \*)

x = Expand$\left[\left(f\left[c3, H21\right] + f\left[c3, H22\right]\right)\right]$

y = Expand$\left[q^{\wedge}(1/3) * \left(f\left[H21, c3\right] + f\left[H22, c3\right]\right)\right]$

x == y

Out[499]= $\dfrac{W1^{1/3} W3^{1/3}}{q^{1/18} W2^{2/3} Z1^{1/3} Z2^{1/3} Z3^{1/3}} + \dfrac{q^{17/18} W1^{1/3} W3^{1/3}}{W2^{2/3} X Z1^{1/3} Z2^{1/3} Z3^{1/3}}$

Out[500]= $\dfrac{W1^{1/3} W3^{1/3}}{q^{1/18} W2^{2/3} Z1^{1/3} Z2^{1/3} Z3^{1/3}} + \dfrac{q^{17/18} W1^{1/3} W3^{1/3}}{W2^{2/3} X Z1^{1/3} Z2^{1/3} Z3^{1/3}}$

Out[501]= True

In[502]:=       (* 22/32 *)
x = Expand[(f[f3, E2])]
y = Expand[q^(1/3) * (q^(-1) * f[E2, f3])]
x == y

Out[502]= $\dfrac{q^{5/18} \, W3^{1/3} \, Z1^{2/3}}{W1^{2/3} \, W2^{2/3} \, Z2^{1/3} \, Z3^{1/3}}$

Out[503]= $\dfrac{q^{5/18} \, W3^{1/3} \, Z1^{2/3}}{W1^{2/3} \, W2^{2/3} \, Z2^{1/3} \, Z3^{1/3}}$

Out[504]= True

In[505]:=       (* 23/32 *)
x = Expand[(f[f3, H21] + f[f3, H22])]
y = Expand[q^(1/3) * (f[H21, f3] + f[H22, f3] + (q^(-1) - q) * f[E2, i3])]
x == y

Out[505]= $\dfrac{W3^{1/3}}{q^{5/9} \, W1^{2/3} \, W2^{2/3} \, Z1^{1/3} \, Z2^{1/3} \, Z3^{1/3}} + \dfrac{q^{4/9} \, W3^{1/3}}{W1^{2/3} \, W2^{2/3} \, X \, Z1^{1/3} \, Z2^{1/3} \, Z3^{1/3}}$

Out[506]= $\dfrac{W3^{1/3}}{q^{5/9} \, W1^{2/3} \, W2^{2/3} \, Z1^{1/3} \, Z2^{1/3} \, Z3^{1/3}} + \dfrac{q^{4/9} \, W3^{1/3}}{W1^{2/3} \, W2^{2/3} \, X \, Z1^{1/3} \, Z2^{1/3} \, Z3^{1/3}}$

Out[507]= True

In[508]:=       (* 32/32 *)
x = Expand[(f[i3, E2])]
y = Expand[q^(1/3) * (f[E2, i3])]
x == y

Out[508]= $\dfrac{q^{4/9} \, W3^{1/3}}{W1^{2/3} \, W2^{2/3} \, Z1^{1/3} \, Z2^{1/3} \, Z3^{1/3}}$

Out[509]= $\dfrac{q^{4/9} \, W3^{1/3}}{W1^{2/3} \, W2^{2/3} \, Z1^{1/3} \, Z2^{1/3} \, Z3^{1/3}}$

Out[510]= True

In[511]:=       (* 33/32 *)
x = Expand[(f[i3, H21] + f[i3, H22])]
y = Expand[q^(1/3) * (q^(-1) * (f[H21, i3] + f[H22, i3]))]
x == y

Out[511]= $\dfrac{W3^{1/3}}{q^{7/18} \, W1^{2/3} \, W2^{2/3} \, Z1^{4/3} \, Z2^{1/3} \, Z3^{1/3}} + \dfrac{q^{11/18} \, W3^{1/3}}{W1^{2/3} \, W2^{2/3} \, X \, Z1^{4/3} \, Z2^{1/3} \, Z3^{1/3}}$

Out[512]= $\dfrac{W3^{1/3}}{q^{7/18} \, W1^{2/3} \, W2^{2/3} \, Z1^{4/3} \, Z2^{1/3} \, Z3^{1/3}} + \dfrac{q^{11/18} \, W3^{1/3}}{W1^{2/3} \, W2^{2/3} \, X \, Z1^{4/3} \, Z2^{1/3} \, Z3^{1/3}}$

Out[513]= True

In[514]:=    (* 13/33 *)
x = Expand[(f[c3, I2])]
y = Expand[q^(1/3) * (f[I2, c3])]
x == y

Out[514]=  $\dfrac{W1^{1/3}}{q^{2/9} \, W2^{2/3} \, W3^{2/3} \, X \, Z1^{1/3} \, Z2^{1/3} \, Z3^{1/3}}$

Out[515]=  $\dfrac{W1^{1/3}}{q^{2/9} \, W2^{2/3} \, W3^{2/3} \, X \, Z1^{1/3} \, Z2^{1/3} \, Z3^{1/3}}$

Out[516]=  True

In[517]:=    (* 23/33 *)
x = Expand[(f[f3, I2])]
y = Expand[q^(1/3) * (f[I2, f3])]
x == y

Out[517]=  $\dfrac{1}{q^{13/18} \, W1^{2/3} \, W2^{2/3} \, W3^{2/3} \, X \, Z1^{1/3} \, Z2^{1/3} \, Z3^{1/3}}$

Out[518]=  $\dfrac{1}{q^{13/18} \, W1^{2/3} \, W2^{2/3} \, W3^{2/3} \, X \, Z1^{1/3} \, Z2^{1/3} \, Z3^{1/3}}$

Out[519]=  True

In[520]:=    (* 33/33 *)
x = Expand[(f[i3, I2])]
y = Expand[q^(1/3) * (q^(-1) * f[I2, i3])]
x == y

Out[520]=  $\dfrac{1}{q^{5/9} \, W1^{2/3} \, W2^{2/3} \, W3^{2/3} \, X \, Z1^{4/3} \, Z2^{1/3} \, Z3^{1/3}}$

Out[521]=  $\dfrac{1}{q^{5/9} \, W1^{2/3} \, W2^{2/3} \, W3^{2/3} \, X \, Z1^{4/3} \, Z2^{1/3} \, Z3^{1/3}}$

Out[522]=  True



\end{document}